\documentclass[10pt,reqno]{amsart}
\usepackage{times}
\usepackage[T1]{fontenc}
\usepackage{amssymb}
\usepackage{amsmath,amsthm}
\usepackage{amsfonts}
\usepackage{leftidx}
\usepackage{mathrsfs}
\usepackage{enumerate}
\usepackage{abstract}
\usepackage{stmaryrd}
\usepackage{ulem}
\usepackage{graphicx}
 \usepackage{hyperref}

\usepackage[top=0.52in, bottom=0.52in, left=0.8in, right=0.8in]{geometry}

\def\R{\mathbb{R}}

\def\d{|\nabla|}

\def\p{\partial}

\def\vo{\vspace{1\baselineskip}}

\def\h{\frac{1}{2}}

\def\be{\begin{equation}}
\def\ee{\end{equation}}

\newtheorem{theorem}{Theorem}[section]

\newtheorem{lemma}{Lemma}[section]
\newtheorem{proposition}{Proposition}[section]
\theoremstyle{definition}
\newtheorem{definition}{Definition}[section]

\theoremstyle{remark}
\newtheorem{remark}{Remark}[section]

\numberwithin{equation}{section}
\begin{document}
 \title[3D Relativistic Vlasov-Nordstr\"om system]{Propagation of  regularity and long time behavior of   $3D$ massive relativistic transport equation I: Vlasov-Nordstr\"om system}
\author{Xuecheng Wang}
\address{YMSC, Tsinghua  University, Beijing, China,  10084
}
\email{xuecheng@tsinghua.edu.cn,\quad xuecheng.wang.work@gmail.com 
}

\thanks{}

\maketitle

\begin{abstract}

In this paper, we introduce a new set of vector fields for the relativistic transport equation, which is applicable for general Vlasov-Wave type coupled systems. By  combining the strength of Klainerman vector field method and  Fourier method, we  prove  global regularity and scattering  for the $3D$ massive relativistic Vlasov-Nordstr\"om system for small initial data without any compact support assumption, which is widely used in the literature for the study of Vlasov solution.

\end{abstract}
  \tableofcontents
\section{Introduction}

\subsection{An overview}\label{overviewsubsection}

In plasma physics and also in general relativity, the dynamics of many physical systems can be  described by Vlasov-wave type coupled systems of equations. For example, the motions of electrons and protons, which are particles,  in solar wind are subjected to the electromagnetic field created by the particles  themselves.  The dynamics of particles and electromagnetic field can be described by the Vlasov-Maxwell system in the  collisionless  case. Moreover, in general relativity,  the motion of  a collection of collisionless particles in the framework of Einstein's general relativity is described by the Einstein-Vlasov system, which is also a Vlasov-wave type coupled system. For our interest, we restrict ourself to the three dimensions case. 

An interesting  problem, which is also widely studied in the literature, is to understand the long term regularity for these  Vlasov-wave type coupled systems. Most of previous results imposed compact support assumption on the initial data. The long term regularity problem for unrestricted data is less studied.

 In this paper and its companion \cite{wang}, based on two new observations, we prove  global regularity for the $3D$ massive relativistic Vlasov-Nordstr\"om   system (see subsection \ref{introductionvn}) and the $3D$ massive relativistic Vlasov-Maxwell  system (see subsection \ref{introductionvm})  without imposing   compact support assumption on the small initial data. 

The first observation   is that there exists  a new set of vector fields for the massive relativistic Vlasov equation, which depend on a geometric observation about the light cone  in $(x,v)$-space instead of solely in physical space. 
The second observation, which will be elaborated in \cite{wang},  is that the spatial derivative in the rotational in $v$ direction, i.e., $v/|v|\times \nabla_x$, plays a role of null structure for the \textit{Vlasov-wave} type nonlinearity of the Vlasov equation.

In this paper,    we will elaborate the first observation and the construction of the new set of vector fields, which  helps to control the    energy of Vlasov part near the light cone. Moreover, we introduce a framework for the Vlasov-wave type coupled system that allows us to   combine the  strength of Klainerman vector field method and  Fourier method. As a result,   we prove small data global regularity for the $3D$ relativistic Vlasov-Nodstr\"om system without compact support condition.  The Fourier method implemented here  is motivated from the method developed in the study of relatively simpler Vlasov-Poisson system in  \cite{wang3}.  We remark that,  thanks to the smallness of   coefficient ``$m^2/\sqrt{m^2+|v|^2}$'' in the relativistic Vlasov-Nodstr\"om system (\ref{vlasovnordstrom}), there is no need to exploit the null structure    mentioned  in the second observation.

Since the benefit of good coefficients is not available for the  relativistic Vlasov-Maxwell system, in  \cite{wang},  we will elaborate the second observation and show how to exploit the hidden null structure by using a Fourier method, which relates it to the structure of time resonance set.

\subsection{The Vlasov-Nordstr\"om system}\label{introductionvn}
The Vlasov-Nordstr\"om system describes the collective motion of collisionless particles interacting by means of their own self-generated gravitational forces     under the assumption that   the gravitational forces are mediated by a scalar field.  This system was proposed by  Calogero \cite{calogero1} as a \textit{laboratory} substitution for the more complicated  and also   more physical  Einstein-Vlasov system. We refer readers to \cite{calogero1,calogero2,calogero3} for more detailed introduction.

   Mathematically speaking, the $3D$ relativistic Vlasov-Nordstr\"om system reads as follows, 
\be\label{vlasovnordstrom}
(\textup{RVN})\quad \left\{\begin{array}{l}
(\p^2_t - \Delta) \phi = \displaystyle{\int_{\R^3} \frac{m^2  f}{\sqrt{m^2+|v|^2}} d v } \\
 
\displaystyle{\p_t f + \hat{v}\cdot \nabla_x f - \big((\p_t + \hat{v}\cdot \nabla_x)\phi(t,x)\big)(4 f + v\cdot \nabla_v f )- \frac{m^2}{\sqrt{m^2+|v|^2}}\nabla_x \phi \cdot \nabla_v f =0},\\
 
f(0,x,v)=f_0(x,v), \quad \phi(0,x)=\phi_0(x), \quad \p_t \phi(0,x)=\phi_1(x),\\
\end{array}\right.
\ee
where $\hat{v}=v/\sqrt{m^2+|v|^2}$, ``$m$'' denotes the \textit{mass} of particles,  $\phi(t,x)$ denotes the scalar gravitational field generated by the particles and $f(t,x,v)$ denotes the density distribution of the particles. After normalizing the mass of particles, we only need to consider the case when $m\in\{0,1\}$, which corresponds to the massless  case and the massive case respectively. 

If  the initial particle density $f(0,x,v)$ has compact support,  Calogero \cite{calogero2} proved that    the $3D$ RVN system (\ref{vlasovnordstrom})  has a unique global classical solution.      However, no asymptotic behavior nor the decay estimate of the scalar field were obtained in \cite{calogero2}. If the initial data is small and moreover the initial scalar field has  a compact support, Friedrich \cite{friedrich} proved  the global existence of the solution   and   the decay estimates for the scalar field. Note that the high order derivatives of solution   were not studied in \cite{calogero2,friedrich}. 

 Recently,  in the spirit of the Klainerman's  classic vector field method\cite{Klainerman1,Klainerman2} for the nonlinear wave equations,  Fajman-Joudioux-Smulevici \cite{smulevic1,smulevic2} proposed a very interesting modified vector field method to study the propagation of   regularity for the solution of the RVN system. For the $3D$ \textit{massless} RVN system, i.e., $m=0$ in (\ref{vlasovnordstrom}),  and the \textit{massive}  RVN system in dimension $n$, $n\geq 4$, they proved global regularity for suitably small initial data. In particular, the important \textit{compact support assumption} was removed for these scenarios.

For the $3D$ massive RVN system,   it was not clear whether the compact support assumption assumed in  \cite{calogero2,friedrich} can be removed and the propagation of regularity can be studied at the same time. A recent interesting result by Fajman-Joudioux-Smulevici \cite{smulevic2} shows that  the $3D$ massive relativistic RVN can be solved forwardly in the sense of the hyperboloid foliation of the space-time, which in particular implies the case with compact support assumption.  

 For the more physical Einstein-Vlasov system,  Taylor\cite{taylor} proved that global stability holds for the massless case  if the initial small particle density has compact support in ``$x$'' and also in ``$v$''. For the massive case, global stability  also holds if the small initial particle density has compact support in ``$x$'' and in ``$v$'', see Lindblad-Taylor \cite{lindblad}; or the small perturbation of the Minkowski spacetime is compact, see Fajman-Joudioux-Smulevici \cite{smulevic3}. See also \cite{hwang,smulevici4} for related work on the Vlasov-Poisson system.

\subsection{The Vlasov-Maxwell system}\label{introductionvm}
In plasma physics,  a sufficiently diluted ionized gas or solar wind can be considered as a collisionless plasma, which means that  the collision effect between the particles is  not as  important  as the electromagnetic force. The dynamics of the   collisionless plasma is described the    Vlasov-Maxwell system. We are interested in the physical massive relativistic case, in which the  the speed of light is finite,  particles are massive and the speed of massive particles   is strictly less than the speed of light.

After normalizing the mass of particles and the speed of light to be one, the relativistic  Vlasov-Maxwell system with given initial data reads as follows,
 \be\label{mainequation}
(\textup{RVM})\qquad \left\{\begin{array}{c}
\p_t f + \hat{v} \cdot \nabla_x f + (E+ \hat{v}\times B)\cdot \nabla_v f =0,\\

\nabla \cdot E = 4 \pi \displaystyle{\int  f(t, x, v) d v}, \qquad \nabla \cdot B =0, \\

\p_t E = \nabla \times B - 4\pi \displaystyle{\int f(t, x, v) \hat{v} d v }, \quad \p_t B  =- \nabla\times E,\\

f(0,x,v)=f_0(x,v), \quad E(0,x)=E_0(x), \quad B(0,x)=B_0(x),
\end{array}\right.
\ee
where  $f(t,x,v)$ denotes the density distribution function of particles, $(E, B)$ stands for the classic electromagnetic field, and $\hat{v}$ denotes the relativistic speed. More precisely,  $\hat{v}:= \displaystyle{v/\sqrt{1+|v|^2}}$.

 The Cauchy problem of  long term regularity for the $3D$ RVM system   has been considered by many authors. A remarkable result   obtained by the Glassey-Strauss \cite{glassey3} says that the  classical solution can be globally  extended  as long as the particle density has compact support in $v$ for all the time. A new proof of this result based on Fourier analysis was   given by Klainerman-Staffilani \cite{Klainerman3}, which adds a new perspective to the study of $3D$ RVM system. See also \cite{alfonso,glasseys2, pallard2} for improvements of this result. An interesting line of research is the continuation criterion for the global existence of the Vlasov-Maxwell system instead of assuming the compact support in ``$v$'' assumption. In \cite{glassey4}, if the initial data decay at rate $|v|^{-7}$ as $|v|\rightarrow \infty$ and the assumption that $\| (1+|v|)f(t,x,v)\|_{L^\infty_x L^1_v}$ remains bounded form all time, then the lifespan of the solution can be continued.
 Recently,    an interesting new continuation criterion was given by Luk-Strain \cite{luk}, which says that a regular solution can be extended as long as $\|(1+|v|^2)^{\theta/2} f(t,x,v)\|_{L^q_x L_v^1}$ remains bounded for $\theta > 2/q, 2< q \leq +\infty$.  See also Kunze\cite{kunze}, Pallard\cite{pallard3}, and Patel\cite{patel} for the recent improvements  on the continuation criterion. See also \cite{glasseys1,luk2,luk} for     results in other physical dimensions. 

Although the assumptions in above mentioned results don't depend on the size of data, the assumptions are    imposed for  all   time, which are strong.  One can   ask whether it is possible to obtain global solution by only imposing assumptions on the initial data. The first positive result was given by Glassey-Strauss \cite{glassey2}. It, roughly speaking,  says that  if   the initial particle density $f(0,x,v)$ has a compact supports in both ``$x$'' and ``$v$'' and also the electromagnetic field $(E(0), B(0))$ has compact support in ``$x$'', and moreover the initial data is suitably small, then there exists a unique classical solution. An improvement by  Schaffer \cite{schaeffer1}    shows   that a similar result as in \cite{glassey2} also holds without compact support assumption   in ``$v$'' but with compact support assumption in ``$x$'' for both the initial particle density and the electromagnetic field. 

It is also interesting to   ask for the $3D$ RVM  that whether the regularity of solution can be unconditionally propagated for all the time and whether sharp decay estimates hold.  This question can be answered  in higher physical dimension $n\geq 4$ for small initial data, see   Bigorgne\cite{bigorgne}. For the two dimensional case, global regularity for large data has been established by Luk-Strain \cite{luk}.

 The detailed analysis of small data global regularity for the $3D$ relativistic  Vlasov-Maxwell system will be carried out in \cite{wang}. In this paper, we mainly use it as a comparison for the study of the Vlasov-Nordstr\"om system. 

\subsection{A rough statement of the main result and the outline of this  paper}
Since the precise statement of the main theorem,  Theorem \ref{precisetheorem},  depends on two sets of vector fields    discussed later, we postpone it to subsection \ref{precisetheoremsubsect}. In this subsection, to give a sense,  we provide a rough theorem for the $3D$ relativistic Vlasov-Nordstr\"om system.

\begin{theorem}[A rough statement]
Given    any smooth  suitably small localized initial particle density $f(0,x,v)$ and initial data of  scalar field $\big(\phi_0(x),   \phi_1(x))$, where the particle density $f(0,x,v)$ decays fast but \textit{polynomially} as $(x,v)$ goes to infinity. Then the $3D$ massive relativistic Vlasov-Nordstr\"om \textup{(\ref{vlasovnordstrom})} admits a unique global solution. The regularity of initial data can be globally propagated and the nonlinear solution scatters to a linear solution in a low regularity space. Moreover, the high order energy of solution only grows sub-polynomially and the scalar field and its derivatives decay sharply at rate $1/ (1+|t|)$ over time.  
\end{theorem}

The rest of this paper is organized as follows.

\begin{enumerate}
	\item[$\bullet$] In section \ref{preliminary}, we define notations used in this paper, prove two basic $L^\infty_x$-type linear estimates, and introduce   \textit{profiles} for the RVN system (\ref{vlasovnordstrom}), which suggest  us to study the  derivative ``$D_v:=\nabla_v - t\nabla_v \hat{v}\cdot\nabla_x$'' associated with the profile of the Vlasov part.
	\item[$\bullet$]
   In section \ref{construction}, we introduce the concept of inhomogeneous modulation of light cone in $(x,v)$-space,  construct two sets of vector fields and decompose the bulk derivative ``$D_v$'' in terms of the new set of vector fields. In particular, the coefficients of the decompositions of $D_v$ are related to the inhomogeneous modulation.
   \item[$\bullet$] In section \ref{setup}, we set up the high order energy estimates for the Vlasov-Nordstr\"om system, give a more precise statement of the main theorem of this paper, Theorem \ref{precisetheorem}, and use a bootstrap argument to prove our main theorem. 
   \item[$\bullet$] In section \ref{energyestimatewave}, we control the increment of both the high order energy and the low order energy over time  for the profiles of  the nonlinear wave equation.
   \item[$\bullet$] In section \ref{energyestimatevlasov}, we control the increment of both the high order energy and the low order energy over time  for the profiles of  the particle distribution function.  
   \item[$\bullet$]    In Appendix \ref{toolkit}, for the sake of readers,  we give detailed computations for the  commutation rules between the bulk derivative $D_v$ and the new set of vector fields as well as the case when the vector fields act on  the inhomogeneous modulation function.   

\end{enumerate}
\vo  

\noindent \textbf{Acknowledgment}\qquad This project was initiated when I was a semester postdoc at ICERM, Brown University. I thank   Yan Guo for suggesting this direction in kinetic theory and thank ICERM for the warm hospitality. Also, I thank Lingbing He and Pin Yu for several helpful conversations. The author is supported by a start-up grant at Tsinghua University. 

\section{Preliminaries}\label{preliminary}
 For any two numbers $A$ and $B$, we use  $A\lesssim B$, $A\approx B$,  and $A\ll B$ to denote  $A\leq C B$, $|A-B|\leq c A$, and $A\leq c B$ respectively, where $C$ is an absolute constant and $c$ is a sufficiently small absolute constant. We use $A\sim B$ to denote the case when $A\lesssim B$ and $B\lesssim A$.   For an integer $k\in\mathbb{Z}$, we use ``$k_{+}$'' to denote $\max\{k,0\}$ and  use ``$k_{-}$'' to denote $\min\{k,0\}$. For any two vectors $\xi, \eta\in \R^3$, we use $\angle(\xi, \eta)$ to denote the angle between ``$\xi$'' and ``$\eta$''. Moreover, we use the convention that $\angle(\xi, \eta)\in [0, \pi]$. 

 For  $f\in L^1(\R^3)$, we use both $\widehat{f}(\xi)$ and $\mathcal{F}(f)(\xi)$ to denote the Fourier transform of $f$, which is defined as follows, 
\[
\mathcal{F}(f)(\xi)= \int_{\R^3} e^{-ix \cdot \xi} f(x) d x.
\]
We use $\mathcal{F}^{-1}(g)$ to denote the inverse Fourier transform of $g(\xi)$. Moreover, for a distribution function $f:\R_x^3 \times \R_v^3\rightarrow \mathbb{C}$,  we treat ``$v$''   as a fixed parameter and use the following notation to denote the Fourier transform of $f(x,v)$ in ``$x$'', 
\[ 
\widehat{f}( \xi, v):=\int_{\R^3} e^{-i x\cdot \xi} f( x,v) d x. 
\]

We  fix an even smooth function $\tilde{\psi}:\R \rightarrow [0,1]$, which is supported in $[-3/2,3/2]$ and equals to ``$1$'' in $[-5/4, 5/4]$. For any $k\in \mathbb{Z}$, we define
\[
\psi_{k}(x) := \tilde{\psi}(x/2^k) -\tilde{\psi}(x/2^{k-1}), \quad \psi_{\leq k}(x):= \tilde{\psi}(x/2^k)=\sum_{l\leq k}\psi_{l}(x), \quad \psi_{\geq k}(x):= 1-\psi_{\leq k-1}(x).
\]
Moreover, we use  $P_{k}$, $P_{\leq k}$ and $P_{\geq k}$ to denote the projection operators  by the Fourier multipliers $\psi_{k}(\cdot),$ $\psi_{\leq k}(\cdot)$ and $\psi_{\geq k }(\cdot)$ respectively. We   use  $f_{k}(x)$ to abbreviate $P_{k} f(x)$. For   $f\in L^1(\R^3)$,   we define
\be\label{signnotation}
f^{+}:=f,\quad P_{+}[f]:=f , \quad f^{-}:= \bar{f}, \quad P_{-}[f]:=\bar{f}.
\ee

Define
\be\label{highhighlow}
\chi_k^1:=\{(k_1,k_2): |k_1-k_2|\leq 10, k\leq \max\{k_1,k_2\}+1\},  
\ee
\be\label{lowhigh}
 \chi_k^2:=\{(k_1,k_2): k_2<  k_1-10, |k-k_1|\leq 1\},\quad  \chi_k^3:=\{(k_1,k_2): k_1<  k_2-10, |k-k_2|\leq 1\},
\ee
where $\chi_k^1$ corresponds to the High $\times$ High type interaction and $\chi_k^2$ and $\chi_k^3$ correspond to the High $\times$ Low type interaction and the Low $\times$ High type interaction respectively. We define  unit vectors in $\R^3$ as follows, 
\[
e_1:=(1,0,0), \quad e_2:=(0,1,0), \quad e_3:=(0,0,1).
\]
Moreover, we define the  following  vectors, 
\be\label{eqn16}
 X_i =  e_i\times x,\quad   V_i=e_i\times v, \quad \hat{V}_i=e_i \times \hat{v},\quad \tilde{V}_i=e_i\times \tilde{v},   \quad \tilde{v}:=\frac{v}{|v|},\quad \tilde{v}_i:= \tilde{v}\cdot e_i,\quad \hat{v}_i:=\hat{v}\cdot e_i. 
\ee
where $  i=1,2,3.$ As a result of direct computations,   the following equalities  hold for $\forall u, v\in \R^3$,
\be\label{jan17eqn1}
u = \tilde{v} \tilde{v}\cdot u   + \sum_{i=1,2,3}\tilde{V}_i \tilde{V}_i \cdot u, \quad 
\tilde{v}\cdot\nabla_v \hat{v}= \frac{\tilde{v}}{(1+|v|^2)^{3/2}}, \quad \tilde{V}_i\cdot\nabla_v \hat{v}= \frac{\tilde{V}_i}{(1+|v|^2)^{1/2}}, \quad i\in\{1,2,3\}.
\ee
For any $k\in \mathbb{Z}$, we define   	 $\mathcal{S}^\infty_k$-norm associated with symbols as follows, 
\be\label{symbolnorm}
\|m(\xi)\|_{\mathcal{S}^\infty_k}:=  \sum_{|\alpha|=0, 1,\cdots,10 }2^{|\alpha| k}\|\mathcal{F}^{-1}[\nabla_\xi^\alpha m(\xi)\psi_k(\xi)]\|_{L^1}.
\ee
Moreover, we define a class of symbol as follows, 
\be
\mathcal{S}^\infty:=\{m(\xi): \quad\| m(\xi)\|_{\mathcal{S}^\infty}:=\sup_{k\in \mathbb{Z}} \| m(\xi)\|_{\mathcal{S}^\infty_k} < \infty \}. 
\ee

 \begin{definition}\label{definitioninversederivative}
We define a linear operator as follows, 
\be\label{sepeqn732}
Q_i:=-R_i \d^{-1}, Q:=(Q_1, Q_2, Q_3), \quad i\in \{1,2,3\},
\ee
where $R_i$, $i\in \{1,2,3\}$, denote the Riesz transforms.  Hence, we have 
\be\label{sepeqn734}
Id = \nabla\cdot Q. 
\ee
 \end{definition} 
  It is well known that the density of the distribution function decays over time. Now, there are several ways to prove this fact, e.g., performing change of variables, using the vector fields method. We refer readers to a recent result  by Wong \cite{wong} for more detailed discussion. In \cite{wang3},  we provided another proof  for this fact by using a Fourier transform method. More precisely, the following Lemma holds, 

\begin{lemma}\label{decayestimateofdensity}
For any fixed $a(v)\in\{v,\hat{v}\}$,  $x\in \R^3$, $a, t\in \R$, s.t, $|t|\geq 1$, $a> -3$, and any given symbol $m(\xi,v)\in L^\infty_v \mathcal{S}^\infty$, the following decay estimate holds, 
\[
 \big|\int_{\R^3}\int_{\R^3} e^{i x\cdot \xi + i t a(v)  \cdot \xi} m(\xi, v)|\xi|^{a}\widehat{g}(t, \xi, v) d v d\xi   \big|\lesssim \sum_{|\alpha|\leq 5+\lfloor a\rfloor} \big(\sum_{|\beta|\leq 5+\lfloor a\rfloor}   \|\nabla_v^\beta m(\xi,v)\|_{L^\infty_v\mathcal{S}^\infty}\big)
 \]
\be\label{densitydecay}
\times  \big[|t|^{-3-a}    \|(1+|v|)^{5+|a|} \nabla_v^\alpha \widehat{g}(t,0,v) \|_{L^1_v } +|t|^{-4-a} \| (1+|v|)^{5+|a|} (1+|x|)\nabla_v^\alpha g(t,x,v)\|_{L^1_x L^1_v}\big].
\ee
\end{lemma}
\begin{proof}
See \cite{wang3}[Lemma 3.1].
\end{proof}

 In  later argument, instead of studying the nonlinear wave equation, we will reduce it to a nonlinear half wave equation, which is convenient to study on the Fourier side. 
Hence, we provide a $L^\infty_x$-type decay estimate for the linear half wave equation  in the following Lemma. 
\begin{lemma}[The linear decay estimate]\label{twistedlineardecay}
For any   $\mu \in\{+,-\}$, the following estimate holds, 
\[
\big|\int_{\R^3}  e^{ i x\cdot \xi-i \mu t |\xi|  }     m(\xi) \widehat{f}(\xi) \psi_k(\xi) d \xi \big|   \lesssim   \min\{2^{ k_{-}}, (1+|t|+|x|)^{-1} \} 2^{k}\|m(\xi)\|_{\mathcal{S}^\infty_k} 
\]
\be\label{noveqn555}
\times \big(\sum_{|\alpha|\leq 1} 2^{k} \|\widehat{\nabla_x^\alpha f}(t, \xi)\psi_k(\xi)\|_{L^\infty_\xi} + 2^{2k} \|\nabla_\xi \widehat{  f}(t, \xi)\psi_k(\xi)\|_{L^\infty_\xi} \big).
\ee
\end{lemma}
\begin{proof}
By using the volume of support of $\xi$, we have
\be\label{nove562}
\big|\int_{\R^3}  e^{ i x\cdot \xi-i \mu t |\xi|  }    m(\xi) \widehat{f}(\xi) \psi_k(\xi) d \xi \big| \lesssim 2^{3k} \| m(\xi)\|_{\mathcal{S}^\infty_k}\|\widehat{f}(\xi)\psi_k(\xi)\|_{L^\infty_\xi}.
\ee
Hence finishing the proof of the first part of the desired estimate (\ref{noveqn555}). 

It remains to prove the second part of the desired estimate (\ref{noveqn555}).Based on the possible size of $t$ and $x$, we separate into two cases as follows. 

\noindent \textbf{Case} $1$:\quad If $|x|\geq 3(1+|t|)$ or $|t|\leq 1$.	
 
 For this case, we do integration by parts in $\xi$ once. As a result, we have
\[
\int_{\R^3}  e^{ i x\cdot \xi-i \mu t |\xi|  }     m(\xi) \widehat{f}(\xi) \psi_k(\xi) d \xi = \int_{\R^3}  e^{ i x\cdot \xi-i \mu t |\xi|  } i \nabla_\xi\cdot \big[  \frac{x- \mu t\xi/|\xi|}{\big|x- \mu t\xi/|\xi| \big|^2}     m(\xi) \widehat{f}(\xi) \psi_k(\xi)\big] d \xi. 
\]
After using the volume of support of $\xi$ for the above equality, we have
\be\label{feb13eqn310}
\big|
\int_{\R^3}  e^{ i x\cdot \xi-i \mu t |\xi|  }     m(\xi) \widehat{f}(\xi) \psi_k(\xi) d \xi \big| \lesssim \frac{2^k \|m(\xi)\|_{\mathcal{S}^\infty_k } }{|x|}\big(2^{ k}\|\widehat{f}(t, \xi)\psi_k(\xi)\|_{L^\infty_\xi} + 2^{2k} \|\nabla_\xi \widehat{f}(t, \xi)\psi_k(\xi)\|_{L^\infty_\xi} \big).
\ee
Now,  our desired estimate  (\ref{noveqn555}) holds from the above estimate and the estimate (\ref{nove562}) if  $|x|\geq 3(1+|t|)$ or $|t|\leq 1$.	

\noindent \textbf{Case} $2$:\quad If $|x|\leq 3(1+|t|)$ or $|t|\geq 1$.	

  Note that $\nabla_\xi(x\cdot\xi-\mu t|\xi|)=0$ if and only if \mbox{${\xi}/{|\xi|} = \mu {x}/{t} = \mu{x}/{|x|} := \xi_{0} $}.   Let $\bar{l} $ be the least integer such that  $2^{ \bar{l} } \geq   2^{ -k/2}(1+|t|)^{-1/2}$. From the volume of support of $\xi$, we have
\[
 \Big| 
\int_{\R^3}    e^{ i x\cdot \xi-i \mu t |\xi|  }  
 \widehat{f}(\xi) m(\xi) \psi_{k}(\xi) \psi_{\leq  \bar{l} }(\angle(\xi,\xi_0))
  d \xi \Big|\lesssim 2^{3 k+ 2\bar{l}}\|m(\xi)\|_{\mathcal{S}^\infty_k}\|\widehat{f}(\xi)\psi_k(\xi)\|_{L^\infty_\xi} \]
  \be\label{noveqn550}
  \lesssim (1+|t|)^{-1} 2^{  2k}  \|m(\xi)\|_{\mathcal{S}^\infty_k}   \|\widehat{f}(t, \xi)\psi_k(\xi)\|_{L^\infty_\xi}.
\ee
 For the case when the angle is localized around $2^l$ where $l> \bar{l}$, we first do integration by parts in $\xi$ once. As a result, we have
\[
 \int_{\R^3}  e^{ i x\cdot \xi-i \mu t |\xi|  }  
   \widehat{f}(\xi) m(\xi) \psi_{k}(\xi) \psi_{l }(\angle(\xi,\xi_0))
  d \xi = I^1_l + I_l^2,
\]
where
\[
 I^1_l= \int_{\R^3}  e^{ i x\cdot \xi-i \mu t |\xi|  }   \frac{i}{t}  \frac{x/t + \mu\xi/|\xi|}{\big|x/t + \mu\xi/|\xi|\big|^2}\cdot
 \nabla_\xi   \widehat{f}(\xi) m(\xi) \psi_{k}(\xi) \psi_{l }(\angle(\xi,\xi_0))
  d \xi,
\]
\[
 I^2_l=  \int_{\R^3}   e^{ i x\cdot \xi-i \mu t |\xi|  }   \frac{i}{t}  
 \nabla_\xi\cdot\Big[\frac{x/t + \mu\xi/|\xi|}{\big|x/t + \mu\xi/|\xi|\big|^2}   m(\xi) \psi_{k}(\xi) \psi_{l }(\angle(\xi,\xi_0)) \Big] \widehat{f}(\xi)
  d \xi.
\]
From the volume of support of $\xi$, the following estimate holds for $I_l^1$, 
\be\label{noveqn741}
 |I^1_l|\lesssim (1+|t|)^{-1} 2^{-l}2^{3k+2l} \|m(\xi)\|_{\mathcal{S}^\infty_k} \|\nabla_\xi  \widehat{f}(\xi)\psi_k(\xi)\|_{L^\infty_\xi} \lesssim (1+|t|)^{-1} 2^{ 3k+l} \|m(\xi)\|_{\mathcal{S}^\infty_k} \|\nabla_\xi \widehat{f}(t, \xi)\psi_k(\xi)\|_{L^\infty_\xi}.
\ee
For $I_l^2$, we do integration by parts in ``$\xi$'' one more time. As a result, the following estimate holds after using the volume of support of $\xi$, 
\[
|I_l^2|\lesssim (1+|t|)^{-2} 2^{ -2l}2^{3k+2l}\|m(\xi)\|_{\mathcal{S}^\infty_k} \big(2^{-2k-2l}\| \widehat{f}(t, \xi)\psi_k(\xi)\|_{L^\infty_\xi} + 2^{- k- l}\|\nabla_\xi \widehat{f}(t, \xi)\psi_k(\xi)\|_{L^\infty_\xi} \big)
\]
\be\label{noveqn742}
\lesssim  (1+|t|)^{-2} 2^{ -2l}\|m(\xi)\|_{\mathcal{S}^\infty_k} \big( 2^{k} \|\widehat{f}(t, \xi)\psi_k(\xi)\|_{L^\infty_\xi} + 2^{2k} \|\nabla_\xi \widehat{f}(t, \xi)\psi_k(\xi)\|_{L^\infty_\xi} \big).
\ee
Hence, from (\ref{noveqn741}) and (\ref{noveqn742}), we have
\be\label{noveqn744}
\sum_{ \bar{l}< l \leq 2} |I_l^1|+|I_l^2|\lesssim (1+|t|)^{-1} 2^{  k}\|m(\xi)\|_{\mathcal{S}^\infty_k} \big( 2^{k} \|\widehat{f}(t, \xi)\psi_k(\xi)\|_{L^\infty_\xi} + 2^{2k} \|\nabla_\xi \widehat{f}(t, \xi)\psi_k(\xi)\|_{L^\infty_\xi} \big).
\ee
Now,  our desired estimate  (\ref{noveqn555}) holds from the estimates (\ref{nove562}), (\ref{feb13eqn310}), (\ref{noveqn550}) and (\ref{noveqn744}) for the case $|x|\leq 3(1+|t|)$ and $|t|\geq 1$.	
\end{proof}

 \subsection{The profiles of the relativistic Vlasov-Nordstr\"om  system }\label{firstlevelprofile}

The idea of studying the profile of the system instead of the original variables is not new, it has been widely used in the study of nonlinear dispersive equation recently. One benefit of  studying the profile is that the effect of linear flow has already been taken into account in the transformation, we can focus on the nonlinear effect. 

 In this subsection, we define the profiles of the Vlasov part and the scalar field part and obtain the evolution equations for the profiles over time. We will also perform similar procedures in subsection \ref{profileofvectorfield} when we apply vector fields on the Vlasov-Nordstr\"om  system.

 We first define the profile  ``$g(t, x, v) $'' of the particle distribution function  ``$f(t,x,v)$ ''as follows, 
\be\label{profileoff}
  g(t,x, v)= f(t, x+\hat{v}t , v),\quad \Longrightarrow f(t, x, v)= g(t, x - \hat{v}t, v).
\ee 
As a result of direct computations, 
 the profile $g(t,x,v)$ satisfies the following equation from the system of equations in (\ref{vlasovnordstrom}),
\be\label{eqnn5}
 \p_t g(t,x,v) =  (\p_t +\hat{v}\cdot \nabla_x) \phi(t,x+\hat{v}t) \big(4g(t,x,v)+ v\cdot D_v g(t,x,v)\big)
+ \frac{ \nabla_x \phi(t,x+\hat{v}t) }{\sqrt{1+|v|^2}}\cdot D_v  g(t,x, v). 
\ee
where
\be\label{pullback}
 {D}_v:= \nabla_v - t   \nabla_v \hat{v}\cdot \nabla_x ,\quad {D}_{v_i}:= \p_{v_i} - t  \p_{v_i} \hat{v}\cdot \nabla_x, \quad i \in\{1,2,3\}.
\ee
Hence, to control the nonlinear effect in (\ref{eqnn5}), it's crucial to understand the role of derivative ``$D_v$'' acts on the profile $g$, which will be elaborated in subsection \ref{studyofDv}.

Next, we  define the profile ``$h(t)$'' and the half wave $u(t)$ of the scalar field part in (\ref{vlasovnordstrom})   as follows, 
\[
 h(t):=e^{i t\d} u(t), \quad u(t):= (\p_t - i\d) \phi(t).
 \]
Note that, we can recover $\phi$ and $\p_t \phi$ from the half wave $u(t)$ and  the profile $h(t)$ as follows, 
\[
\p_t\phi= \frac{u(t)+\overline{u(t)}}{2}, \quad  \phi= \frac{-u(t)+\overline{u(t)}}{2i \d}:=\sum_{\mu\in\{+,-\}} c_{\mu} \d^{-1} u^{\mu}(t), \quad c_{\mu}:=i \mu/2, \quad u(t)= e^{-it\d} h(t). 
\]

In terms of the half wave $u(t)$, we can rewrite the equation satisfied by the profile $g(t,x,v)$ in  (\ref{eqnn5}) as follows,
\be\label{jan29eqn1}
\p_t g(t,x,v) =\sum_{\mu\in \{+,-\}}\big(  \frac{1}{2} + c_{\mu}\hat{v}\cdot R\big) u^{\mu}(t, x+\hat{v}t)\big(4g(t,x,v)+ v\cdot D_v g(t,x,v)\big)+ \frac{ c_{\mu}R u^{\mu}(t,x+\hat{v}t) }{\sqrt{1+|v|^2}} \cdot D_v  g(t,x, v ),
\ee 
where $ R:=\nabla_x/\d $ denotes the Riesz transforms.  Moreover, on the Fourier side, we have
\[
\p_t \widehat{g}(t,\xi, v) =  \sum_{\mu\in\{+,-\}}  \int_{\R^3} e^{it \hat{v}\cdot (\xi-\eta) - it \mu |\xi-\eta|  }\widehat{h^\mu}(t, \xi-\eta)  \big( a^1_{\mu}(v,\xi-\eta)  \widehat{g}(t, \eta, v)   
\]
\be\label{feb20eqn1}
+ a^2_{\mu}(v,\xi-\eta)  \cdot\big(\nabla_v - i t \nabla_v \hat{v}\cdot \eta  \big)\widehat{g}(t, \eta, v) \big)  d \eta, 
\ee
where
\be\label{jan29eqn4}
a_\mu^1(v, \xi)= \big(2+ i 4 c_{\mu}\hat{v}\cdot \xi |\xi|^{-1} \big), \quad a_\mu^2(v, \xi)= \frac{v}{2}\big(1+ i 2 c_{\mu}\hat{v}\cdot \xi |\xi|^{-1} \big) + \frac{i   c_{\mu} \xi}{\sqrt{1 +|v|^2} |\xi|}. 
\ee

From the system of equations in (\ref{vlasovnordstrom}), we can derive the equation satisfied by $u(t)$ as follows, 
\be
(\p_t + i \d) u(t) = \int_{\R^3} \frac{1 }{\sqrt{1+|v|^2} }  f (t, x,v ) d v = \int_{\R^3} \frac{1 }{\sqrt{1+|v|^2} }  g (t, x-\hat{v} t,v ) d v. 
\ee
 On the Fourier side, we have
\be\label{jan29eqn2}
\p_t \widehat{h}(t, \xi ) = \int_{\R^3} e^{i t |\xi| - it \hat{v}\cdot \xi}  \frac{1 }{\sqrt{1 +|v|^2} }  \widehat{g}(t, \xi,  v) dv. 
\ee

Note that, the nonlinearity of the above equation is linear with respect to the Vlasov part, which, generally speaking, very hard to be controlled directly in the energy estimate. 

However, we observe that  the nonlinearity in (\ref{jan29eqn2}) is actually oscillating  in time. To take the advantage of the oscillation of the phase ``$|\xi|-\hat{v}\cdot \xi$'' in (\ref{jan29eqn2}) over time, instead of controlling the increment of the profile $h(t)$ over time, we control the following \textit{modified profile}, 
\[
\widehat{\widetilde{h}}(t, \xi ):= \widehat{h}(t, \xi ) + \int_{\R^3} e^{i t |\xi| - it \hat{v}\cdot \xi} \frac{i}{|\xi| - \hat{v}\cdot \xi }  \frac{1 }{\sqrt{1 +|v|^2} } \widehat{g}(t, \xi,  v) dv .
\]

Recall the   equations (\ref{feb20eqn1}) and (\ref{jan29eqn2}).  After doing integration by parts in $v$ once,  we can derive the equation satisfied by the modified profile $\widetilde{h}(t)$ as follows, 
\[
\p_t \widehat{\widetilde{h}}(t, \xi )= \int_{\R^3} e^{i t |\xi| - it \hat{v}\cdot \xi}  \frac{1 }{\sqrt{1 +|v|^2} }   \frac{ i }{|\xi| - \hat{v}\cdot \xi } \p_t\widehat{g}(t, \xi,  v)   dv 
\]
\be\label{jan29eqn5}
= \sum_{\mu\in\{+,-\}} \int_{\R^3} \int_{\R^3} e^{it |\xi| - it \mu |\xi-\eta| - i t\hat{v}\cdot \eta} a^3_{\mu}(v,\xi, \xi-\eta)   \widehat{h^\mu}(t, \xi-\eta)  \widehat{g}(t, \eta, v)  d \eta d v
\ee
where
\be\label{jan29eqn7}
a^3_{\mu}(v,\xi, \xi-\eta):=\frac{i   a_{\mu}^1(v, \xi-\eta)  }{\sqrt{1 +|v|^2}\big(|\xi|-v\cdot\xi\big)   } - \nabla_v\cdot \big(\frac{i  a_{\mu}^2(v, \xi-\eta)}{\sqrt{1 +|v|^2}\big(|\xi|-v\cdot\xi \big)} \big).
\ee
where $a_{\mu}^i(  v,\xi )$, $i\in\{1,2\}$, are defined in (\ref{jan29eqn4}).

\section{Constructing vector fields  for the relativistic Vlasov-wave type coupled system}\label{construction}

In this section, based an observation  on the light cone $C_t:=\{(x,v): x,v\in \R^3, |t|-|x+t\hat{v}|=0\}$ in $(x,v)$ space, we   construct a new set of vector fields, which will be used to decompose the bulk derivative $D_v$ defined in  (\ref{pullback}). Before that, we first introduce a set of classic vector fields which commutes with both   the Vlasov equation and the nonlinear wave equation. The classic set of vector fields enables us to obtain decay estimate from the energy estimate for the wave equation, which is well known  as the Klainerman-Sobolev embedding.

Recall (\ref{eqn16}). For any $   i,j =1,2,3, $ we define the first set of vector fields for the Vlasov-Nordstr\"om  system  and  the Vlasov-Maxwell system as follows,
\be\label{dec18eqn5}
S:=t \p_t + x \cdot\nabla_x,\quad \Omega_{i,j}=x_i\p_{x_j}-x_j\p_{x_i},\quad  \Omega_i= X_i \cdot \nabla_x,  \quad \tilde{\Omega}_{i}:=  {V}_i \cdot \nabla_v +X_i \cdot \nabla_x,\quad 
\ee
\be\label{lorentz}
 \quad L_i:= t \p_{x_i} + x_i \p_t, \quad \tilde{L}_i: = t \p_{x_i} + x_i \p_t + \sqrt{1+|v|^2} \p_{v_i},\,\, L:=(L_1, L_2, L_3), \,\, \tilde{L}:=(\tilde{L}_1, \tilde{L}_2, \tilde{L}_3),
\ee
where ``$S$'', ``$\Omega_{i,j}$'', and ``$L_i$'' are the well-known scaling vector field, rotational vector fields, and the Lorentz vector fields, which all commutate with the linear operator of the nonlinear wave equation, see the classic works of Klainerman\cite{Klainerman1,Klainerman2} for the introduction of the original vector field method. Note that $\Omega_{i,j}\in\{\mu \Omega_i, \mu\in\{+,-\}, i\in\{1,2,3\}\}$ for any $i,j\in\{1,2,3\}$.

As pointed out in Fajman-Joudioux-Smulevici \cite{smulevic1} that the vector fields $S, \tilde{\Omega}_i$, and $\tilde{L}_i$ all commutate with the linear operator of the relativistic   Vlasov equation. More precisely, we have 
\be\label{eqn825}
[\p_t + \hat{v}\cdot \nabla_x , S]=\p_t + \hat{v}\cdot \nabla_x,\quad [\p_t + \hat{v}\cdot \nabla_x, \tilde{\Omega}_i]=0,\quad  [\p_t + \hat{v}\cdot \nabla_x, \tilde{L}_i]= \hat{v}_i\big(\p_t + \hat{v}\cdot \nabla_x\big),
\ee
\be\label{dec27eqn1}
 [\p_t + \hat{v}\cdot \nabla_x, \tilde{\Omega}_i]=[\p_t +\hat{v}\cdot \nabla_x, V_i\cdot \nabla_v +X_i \cdot \nabla_x]=(\hat{v}\cdot \nabla_x(X_i))\cdot \nabla_x -  \hat{V}_i \cdot \nabla_x=0,
\ee

We define the first set of vector fields for the distribution function $f(t,x,v)$ as follows, 
\be\label{firstsetfordist}
 \mathcal{P}_1:=\{S,\tilde{\Omega}_i, \tilde{L}_i, \p_{x_i}, i\in \{1,2,3\} \}. 
\ee
Correspondingly, we define the following set of vector fields for the nonlinear wave part  as follows, 
 \be\label{firstsetforeb}
 \mathfrak{P}_1:=\{S,  {\Omega}_i,  {L}_i, \p_{x_i}, i\in \{1,2,3\} \}.
 \ee
\begin{lemma}\label{somebasicidentity}
For any $t\in \R, x\in \R^3$, the following equalities hold, 
\be\label{sepeqn190}
(t^2-|x|^2)\p_t = t S - x\cdot L, \quad (t^2-|x|^2)\p_i = \sum_{j=1,2,3}- x_j \Omega_{ij} + t L_i - x_i S, \quad i \in \{1,2,3\}. 
\ee
\end{lemma}
\begin{proof}
Desired identities follow from direct computations.
\end{proof}

Unfortunately,  we cannot represent the bulk derivative ``$D_v$''
 as a ``\textit{good}'' linear combination of vector fields defined previously, i.e., $\nabla_x$, $S$, $\tilde{\Omega}_i$, and $\tilde{L}_i$, in the sense that one of coefficients is of size ``$t$'', which is too big to control in the long run. This is also why we seek for a new set of vector fields.

\subsection{A new set of vector fields}\label{studyofDv}

Our new set of vector fields   is inspired from  the following identity regards the light cone $C_t:=\{(x,v): x,v\in \R^3, |t|-|x+t\hat{v}|=0\}$ in $(x,v)$ space.  
\begin{lemma}\label{essentialidentity}
The following identity holds,
\be\label{eqn20}
t^2-|x+\hat{v}t|^2 = \big(\frac{t}{ {1+|v|^2}}-\frac{\omega_{+}(x,v)}{\sqrt{1+|v|^2}}\big) \big( {t}-{\sqrt{1+|v|^2}}{\omega}_{-}(x,v)\big),
\ee
where
\be\label{eqn19}
\omega_+(x,v)= x\cdot v +\sqrt{(x\cdot {v})^2  + |x|^2 },\quad 	 {\omega_{-}}(x,v) =  x\cdot v -\sqrt{(x\cdot {v})^2  + |x|^2 }. 
\ee
Moreover, the following rough estimates hold,
\be\label{eqn50}
 {\omega_{-}}(x,v)\lesssim \frac{-|x|}{1+|v|}, \quad 0\leq {\omega_{+} }(x,v) \lesssim |x|(1+|v|).
\ee
\end{lemma}
\begin{proof}
Note that 
\[
t^2-|x+\hat{v}t|^2 =    \frac{t^2}{1+|v|^2} -\frac{2 t x\cdot v}{\sqrt{1+|v|^2}} - |x|^2 
=    \big(\frac{t}{\sqrt{1+|v|^2}}- \omega_{+}(x,v)\big) \big(\frac{t}{\sqrt{1+|v|^2}}- \omega_{-}(x,v)\big),\quad
\]
where
\be\label{sepeqn19}
\omega_{+}(x,v)= x\cdot v +\sqrt{(x\cdot {v})^2  + |x|^2 },\quad 	 {\omega_{-}}(x,v) =  x\cdot v -\sqrt{(x\cdot {v})^2  + |x|^2 }. 
\ee
From the above detailed formulas, the desired estimate  (\ref{eqn50})  holds straightforwardly. 
 \end{proof}
 From the equality (\ref{eqn20}) and the estimate (\ref{eqn50}), we know that ``$|t| -|x+\hat{v} t| =0$'' if and only if   ``$ {t}/{(1+|v|^2)}- {\omega_{+}(x,v)}/(\sqrt{1+|v|^2})  =0  $''. This fact motivates us to define the \textit{modulation} of the light cone  in $(x,v)$-space as follows, which   plays the role of distance to the light cone. 
 \begin{definition}
We define the following function as the \textit{ modulation} with respect to the ``light cone'' in $(x,v)$-space
\be\label{modulation}
d(t,x,v):=\frac{t}{1+|v|^2}-\frac{\omega_{+}(x,v)}{\sqrt{1+|v|^2}}.
\ee

Since we will only care about the  distance with respect  to the ``light cone'' when ``$|x| +|x\cdot v|\gtrsim 1$'', we define an \textit{inhomogeneous modulation}  $\tilde{d}(t,x,v)$ as follows, 
\be\label{inhomogeneousmodulation}
\tilde{d}(t,x,v):=  \frac{t}{1+|v|^2}-\frac{\omega (x,v)}{\sqrt{1+|v|^2}} ,
\ee
where  $\omega(x,v)$ is defined as follows,
\be\label{sepeqn975}
\omega(x, v)= \psi_{\geq 0}(|x|^2 +(x\cdot v)^2) \omega_{+}(x,v)=\psi_{\geq 0}(|x|^2 +(x\cdot v)^2)\big(x\cdot v +\sqrt{(x\cdot {v})^2  + |x|^2 }\big). 
\ee
From (\ref{eqn20}) and (\ref{eqn50}), we have
\be\label{nov43}
|\tilde{d}(t,x,v)| \lesssim 1+ |t-|x+\hat{v} t||.
\ee
\end{definition}

With the above definition of the inhomogeneous modulation function $\tilde{d}(t,x,v)$, the construction of  our new set of vector fields  is motivated from a good decomposition of the derivative $D_v$. Recall (\ref{pullback}). Instead of  ``naturally'' decompose $D_v$ into two parts, which are ``$\nabla_v$'' and ``$t \nabla_v \hat{v}\cdot \nabla_x$'', by using the inhomogeneous modulation defined in (\ref{inhomogeneousmodulation}),  we decompose $D_v$ into two parts  as follows,
\be\label{dec22eqn1}
D_v = \nabla_v - t \nabla_v \hat{v}\cdot\nabla_x=\underbrace{\nabla_v - \sqrt{1+|v|^2} \omega(x,v)\nabla_v  \hat{v}\cdot \nabla_x}_{\text{Part I}} - \underbrace{(t-  \sqrt{1+|v|^2} \omega(x,v) )\nabla_v  \hat{v}\cdot \nabla_x}_{\text{Part II}}. 
 \ee
The main intuition behind the above ``\textit{good}''  decomposition is that \textit{it is more promising to control the burden of extra modulation than the burden of extra ``$t$'' over time for the nonlinear wave solution}.

With the above motivation, we define, 
\be\label{eqq10}
K_v:= \nabla_v - \sqrt{1+|v|^2} \omega(x,v)\nabla_v  \hat{v}\cdot \nabla_x ,\quad S^v: = \tilde{v} \cdot \nabla_v, \quad S^x:= \tilde{v}  \cdot \nabla_x, \quad \Omega^v_i= \tilde{V}_i \cdot \nabla_v, \quad \Omega^x_i= \tilde{V}_i \cdot \nabla_x, 
\ee
where $i \in \{1,2,3\}$, $\tilde{v}$ and $\tilde{V}_i$  are defined in (\ref{eqn16}).

 Moreover, we   define a set of vector fields  as follows, 
\be\label{sepeqq2}
   \widehat{S}^v:= \tilde{v}\cdot K_v= S^v - \frac{\omega(x,v)}{ {1+|v|^2}} S^x, \quad \widehat{\Omega}^{v}_i:=  \tilde{V}_i \cdot K_v = \Omega_i^v - \omega(x,v) \Omega_i^x,\quad  K_{v_i}:=K_v \cdot e_i,\quad i\in\{1,2,3\}.
\ee
Note that we used the equalities in (\ref{jan17eqn1}) in the above equation.

We remark that the vector fields defined in (\ref{sepeqq2}) will be applied on the profile ``$g(t,x,v)$'' instead of the original distribution ``$f(t,x,v)$''. Also, it's not difficult to find the corresponding vector fields act  on the original distribution function $f(t,x,v)$. For example, 	from (\ref{profileoff}), we have
\[
K_{v } g(t,x,v) = \big(\nabla_{v }- \sqrt{1+|v|^2}\omega(x,v) \nabla_{v }\hat{v}\cdot \nabla_x \big) \big(f(t,x+\hat{v}t, v)\big)
\]
\[
= (\nabla_{v }f)(t, x+\hat{v}t, v) +(t-  \sqrt{1+|v|^2}\omega(x,v)) \nabla_{v }\hat{v}\cdot \nabla_x f(t,x+\hat{v}t, v)=: (\widetilde{K}_{v } f)(t, x+\hat{v}t, v),
\]
where
\be\label{pullbackvector}
\widetilde{K}_{v }:= \nabla_{v }+ (t-  \sqrt{1+|v|^2}\omega(x-\hat{v} t ,v)) \nabla_{v }\hat{v}\cdot \nabla_x.
\ee
As a result of direct computation, we know that  $\widetilde{K}_{v}$ commutates with the linear transport operator ``$\p_t + \hat{v} \cdot \nabla_x$'' of the relativistic  Vlasov equation.

With the above defined new set of vector fields, in the following Lemma,  we   decompose the bulk derivative ``$D_v$'' in terms of the new set of vector fields. 

 \begin{lemma}\label{twodecompositions}
 The following two decompositions holds for ``$D_v$'', 
 \be\label{eqn32}
D_v  =  (\tilde{v} \widehat{S}^v + \tilde{V}_i  \widehat{\Omega}^v_i) -\tilde{d}(t,x,v)\big( \frac{\tilde{v} S^x}{ \sqrt{1+|v|^2} } + \sum_{i=1,2,3} \sqrt{1+|v|^2} {\tilde{V}_i \Omega^x_i} \big), 
 \ee
\be\label{eqq17}
D_v= \tilde{v}\widehat{S}^v - \tilde{d}(t,x,v)   \frac{\tilde{v} S^x}{\sqrt{1+|v|^2}} +\sum_{i=1,2,3} \frac{\tilde{V}_i}{|v|} \tilde{\Omega}_i -\frac{\tilde{V}_i}{|v|}\big(  X_i\cdot \tilde{v} S^x +  \sum_{j=1,2,3} ( X_i + \hat{V}_i t )\cdot \tilde{V}_j \Omega^x_j\big).
\ee

 \end{lemma}

 \begin{proof}
 Recall (\ref{pullback}),  (\ref{eqq10}), and (\ref{sepeqq2}). From (\ref{jan17eqn1}),  we have
\be\label{eqq13}
D_v= \tilde{v} S^v -  \frac{t \tilde{v} S^x}{(1+|v|^2)^{3/2}} +\sum_{i=1,2,3} \tilde{V}_i \Omega^v_i  - \frac{t \tilde{V}_i \Omega^x_i }{\sqrt{1+|v|^2}}
\ee
\be\label{dec18eqn10}
= \sum_{i=1,2,3}  (\tilde{v} \widehat{S}^v + \tilde{V}_i  \widehat{\Omega}^v_i) -\tilde{d}(t,x,v)   \big( \frac{\tilde{v} S^x}{\sqrt{ 1+|v|^2 }} +  \sqrt{ 1+|v|^2 }{\tilde{V}_i \Omega^x_i} \big).
\ee
Hence, finishing the proof of (\ref{eqn32}).

Now, let's proceed to prove the desired equality (\ref{eqq17}).   Recall (\ref{dec18eqn5})  and (\ref{eqq10}), we have
\be\label{eqq14}
  \Omega^v_i   = \frac{1}{|v|} \tilde{\Omega}_i -   \frac{X_i}{|v|} \cdot \nabla_x. 
\ee
From (\ref{jan17eqn1}), (\ref{eqq13}) and (\ref{eqq14}), we have
\[
D_v= \tilde{v}\widehat{S}^v -  \tilde{d}(t,x,v)    \frac{\tilde{v} S^x}{\sqrt{1+|v|^2} } + \sum_{i=1,2,3}\frac{\tilde{V}_i}{|v|} \tilde{\Omega}_i -\frac{\tilde{V}_i}{|v|}( X_i + \hat{V}_i t )\cdot \nabla_x  
\]
\[
= \tilde{v}\widehat{S}^v - \tilde{d}(t,x,v)   \frac{\tilde{v} S^x}{\sqrt{1+|v|^2}} + \sum_{i=1,2,3}\frac{\tilde{V}_i}{|v|} \tilde{\Omega}_i -\frac{\tilde{V}_i}{|v|}( X_i + \hat{V}_i t )\cdot \big(\tilde{v} S^x  + \sum_{j=1,2,3} \tilde{V}_j \Omega^x_j\big) 
\]
\[
=\tilde{v}\widehat{S}^v - \tilde{d}(t,x,v)   \frac{\tilde{v} S^x}{\sqrt{1+|v|^2} } + \sum_{i=1,2,3} \frac{\tilde{V}_i}{|v|} \tilde{\Omega}_i -\frac{\tilde{V}_i}{|v|}\big[ X_i\cdot \tilde{v} S^x + \sum_{j=1,2,3} ( X_i +  \hat{V}_i t )\cdot \tilde{V}_j \Omega^x_j\big].
\]
Hence, finishing the proof of (\ref{eqq17}).
 \end{proof}
\begin{remark}
Because the coefficients of $\Omega_i^x$  in  the first decomposition (\ref{eqn32})  are of size ``$(1+|v|)|\tilde{d}(t,x,v)|$'', we   use it when ``$|v|$'' is relatively small and use the second decomposition (\ref{eqq17}) when ``$|v|$'' is relatively large.

 Thanks to the coefficient $m^2/\sqrt{1+|v|^2}$ in the Vlasov-Nordstr\"om system, see (\ref{vlasovnordstrom}), the first decomposition of ``$D_v$'' in (\ref{eqn32}) is sufficient for the Vlasov-Nordstr\"om system. We will use both decompositions  for the  Vlasov-Maxwell system because the benefit of good coefficient is no longer available, see (\ref{mainequation}).

\end{remark}

Motivated from   the two decompositions in the above Lemma, we define the following  set of vector fields, which act on \textit{the profile } $g(t,x,v)$, 
\be\label{dec28eqn1}
\mathfrak{P}_2:=\{\Gamma_i, \quad i\in\{1,\cdots, 17\}\}, 
\ee
where 
\be\label{dec26eqn5}
\Gamma_1= \psi_{\geq 1}(|v|) \widehat{S}^v, \quad \Gamma_2:=\psi_{\geq 1}(|v|)S^x, \quad  \Gamma_{i+2}:=\psi_{\geq 1}(|v|)\widehat{\Omega}^v_i, \quad  \Gamma_{i+5}:=\psi_{\geq 1}(|v|) \Omega^x_i, 
\ee
\be\label{dec26eqn6}
\Gamma_{i+8}:=\psi_{\leq 0}(|v|) K_{v_i},\quad \Gamma_{i+11}:=\psi_{\leq 0}(|v|) \p_{x_i},\quad \Gamma_{i+14}:=\tilde{\Omega}_i, \quad i=1,2,3. 
\ee

Correspondingly, we define the following associated set of  vector fields which act on the original distribution function $f(t,x,v)$, 
\be\label{secondsetfordist}
\mathcal{P}_2:=\{\widehat{\Gamma}_i, \quad i\in\{1,\cdots, 17\}\}, 
\ee
where
\be\label{dec27eqn21}
\widehat{\Gamma}_1= \psi_{\geq 1}(|v|) \tilde{v}\cdot \widetilde{K}_v, \quad\widehat{\Gamma}_2:=\psi_{\geq 1}(|v|)S^x, \quad  \Gamma_{i+2}:=\psi_{\geq 1}(|v|)\tilde{V}_i \cdot \widetilde{K}_v, \quad  \widehat{\Gamma}_{i+5}:=\psi_{\geq 1}(|v|) \Omega^x_i, 
\ee
\be\label{dec27eqn22}
\widehat{\Gamma}_{i+8}:=\psi_{\leq 0}(|v|) \widetilde{K}_{v_i},\quad \widehat{\Gamma}_{i+11}:=\psi_{\leq 0}(|v|) \p_{x_i},\quad \widehat{\Gamma}_{i+14}:=\tilde{\Omega}_i, \quad i=1,2,3. 
\ee
 
For the convenience of notation,  we don't distinguish these two sets of vector fields ($\mathfrak{P}_2$ and  $\mathcal{P}_2$) if there is no confusion. For simplicity, we define a set of  notations 	  to  represent the above defined vector fields uniformly.  
 
\begin{definition}

For any vectors $e=(e_1, \cdots, e_n)\in R^{n}$ and $f=(f_1,\cdots, f_m) \in \R^m$,   {where}\, $e_1, \cdots, e_n, f_1, \cdots, f_m \in \R$,   
we define
\[
 e\circ f :=(e_1,\cdots, e_n,f_1, \cdots, f_m),\quad |e|:=\sum_{i=1,\cdots, n}|e_i|,\quad \Longrightarrow |e\circ f|=|e|+|f|.
\]Let
\[
\mathcal{A}:=\{\vec{a}: \vec{a}\in\{0,1\}^{10}, |\vec{a}|=0,1\}, \quad \vec{0}:=(0,\cdots, 0),
\]
\[
  \vec{a}_i:=(0,\cdots,\underbrace{1}_{\text{$i$-th} },\cdots,0),\quad  \textit{if\,\,}   \vec{0}, \vec{a}_i\in \mathcal{A}, \quad \mathcal{B}:=\cup_{k\in\mathbb{N}_{+}}\mathcal{A}^k. 
\]
\be\label{dec28eqn28}
\Gamma^{\vec{0}} :=Id, \quad \Gamma^{\vec{a}_1}:=S, \quad \Gamma^{\vec{a}_{i+1}}:=\p_{x_i}, \quad \Gamma^{\vec{a}_{i+4}}:=\Omega_{i}, \quad \Gamma^{\vec{a}_{i+7}}:=L_i, \quad i=1,2,3,
\ee
\be\label{dec28eqn29}
\tilde{\Gamma}^{\vec{0}} :=Id, \quad \tilde{\Gamma}^{\vec{a}_1}:=  {S}, \quad \tilde{\Gamma}^{\vec{a}_{i+1}}:=\p_{x_i}, \quad \tilde{\Gamma}^{\vec{a}_{i+4}}:=\tilde{\Omega}_{i}, \quad \tilde{\Gamma}^{\vec{a}_{i+7}}:=\tilde{L}_i, \quad i=1,2,3.
\ee
We  represent the high order derivatives of the first set of vector field $\mathfrak{P}_1$ and $\mathcal{P}_1$ (see (\ref{firstsetforeb}) and (\ref{firstsetfordist})) through composition as follows, 
\be\label{dec28eqn30}
  \Gamma^{\alpha_1\circ\alpha_2}:=\Gamma^{\alpha_1}\Gamma^{\alpha_2},\quad \tilde{\Gamma}^{\alpha_1\circ\alpha_2}:=\tilde{\Gamma}^{\alpha_1}\tilde{\Gamma}^{\alpha_2} \quad \alpha_1, \alpha_2\in \mathcal{B}.
\ee
\end{definition}
\begin{definition}
We define
\[
\mathcal{K}:=\{\vec{e}:\,\,\vec{e}\in \{0,1\}^{17}, |\vec{e}|=0,1\},\quad \vec{0}:=(0,\cdots, 0), 
\vec{e}_i:=(0,\cdots,\underbrace{1}_{\text{$i$-th} },\cdots,0),\quad \textit{if\,\,} \vec{0}, \vec{e}_i\in \mathcal{K}, \]
\[
  \mathcal{S}:=\cup_{k\in\mathbb{N}_{+}}\mathcal{K}^{k},\quad 
\Lambda^{\vec{ {0}}}:=Id, \quad \widehat{\Gamma}^{\vec{0}}:= Id,  \quad \Lambda^{\vec{e}_i}:=\Gamma_i, \quad \widehat{\Gamma}^{\vec{e}_i}:= \widehat{\Gamma}_i,  \quad \Gamma_i\in \mathfrak{P}_2, \widehat{\Gamma}_i\in \mathcal{P}_2,  \quad  \vec{e}_i \in \mathcal{K},
\]
where ``$\mathcal{P}_2$'' is defined in (\ref{secondsetfordist}) and ``$\mathfrak{P}_2$'' is defined in (\ref{dec28eqn1}). Hence, we can represent the high order derivatives of the second set of vector fields for the profile ``$g(t,x,v)$'' or the original distribution function ``$f(t,x,v)$'' through composition as follows, 
\[ 
\Lambda^{e\circ f}:= \Lambda^{e}\Lambda^{f},\quad \widehat{\Gamma}^{e\circ f}:= \widehat{\Gamma}^{e}\widehat{\Gamma}^{f},  \quad e, f\in \mathcal{S}.
\]
\end{definition}

\begin{definition}
For any $\kappa, \gamma \in \mathcal{S}$, we define the equivalence relation between ``$ \kappa$'' and ``$ \gamma$'' as follows,
\be\label{equivalencerelation}
 {\kappa}\thicksim  {\gamma}\, \textup{and}\,  \Lambda^{\kappa}\thicksim  \Lambda^{\gamma}, \,\, \textup{if and only if\,} \Lambda^{\kappa} h(x,v)= \Lambda^{\gamma} h(x,v)\,\,\textup{for any differentiable function\,} h(x,v) ,
\ee
\be\label{nontequivalencerelation}
  {\kappa}\nsim  {\gamma}\, \textup{and}\,  \Lambda^{\kappa}\nsim  \Lambda^{\gamma},\,\,   \textup{if and only if\,\,} \Lambda^{\kappa} h(x,v)\neq \Lambda^{\gamma} h(x,v)\,\,\textup{for all non-constant differentiable function\,} h(x,v).
\ee
Very similarly, we can define the corresponding equivalence relation for $\alpha_1, \alpha_2\in \mathcal{B}$. 
Note that, for any $\beta\in \mathcal{S}$ and $\alpha \in \mathcal{B}$, there exists a unique expansion such that 
\be\label{uniqueexpansion}
 {\beta} \thicksim {\iota_1\circ \cdots\iota_{|\beta|}},\quad \iota_i\in \mathcal{K}, |\iota_i|=1,\quad i\in\{1,\cdots, |\beta|\},
\ee
\be\label{uniqueexpansionBeta}
\alpha \thicksim \gamma_{1}\circ  \cdots \gamma_{|\alpha|}, \quad \gamma_i\in \mathcal{A}, |\gamma_i|=1, \quad i \in \{1,\cdots, |\alpha|\}.
\ee
\end{definition}
With the above notation, without the complexity caused the constant coefficients, we can represent the Leibniz rule by the equality as follows,
\be\label{Leibnizrule}
\Lambda^{\beta}(fg) = \sum_{\beta_1, \beta_2\in \mathcal{S}, \beta_1+\beta_2=\beta} \Lambda^{\beta_1} f \Lambda^{\beta_2} g,\quad \beta\in \mathcal{S},
\ee
where $f$ and $g$ are two smooth functions.

To capture the effect of different sizes of coefficients in the decompositions of ``$D_v$'' in (\ref{eqn32}) and (\ref{eqq17}) and effect of good efficients in the Vlasov-Nordstr\"om system, we  define  an index ``$c_{\textup{vn}}(\iota)$'' for the Vlasov-Nordstr\"om system.

As a comparison, we also define the corresponding index ``$c_{\textup{vm}}(\iota)$'' for the Vlasov-Maxwell system. For simplicity, we will not redo similar computations  and will use part of results obtained in this paper directly in \cite{wang}. Hence, related results will be formulated in both indexes.

\begin{definition}\label{definitiongoodderivative}
For any $\iota\in \mathcal{K} $, we define  an index ``$c_{\textup{vn}}(\iota)$'' for the Vlasov-Nordstr\"om system (\ref{vlasovnordstrom}) and   indexes ``$c_{\textup{vm}}(\iota)$'' for the Vlasov-Maxwell system  (\ref{mainequation}) as follow, 
 \be\label{countingnumber}
c_{\textup{vn}}(\iota) =\left\{\begin{array}{ll}
1 & \textup{if\,} \Lambda^\iota \thicksim \widehat{S}^v  \\
-1 & \textup{if\,} \Lambda^\iota \thicksim \widehat{\Omega}_i, i\in\{1,2,3\} \\
0 & \textup{otherwise} \\
\end{array}\right. ,  \quad  c_{\textup{vm}}(\iota) =\left\{\begin{array}{ll}
1 & \textup{if\,} \Lambda^\iota \thicksim \widehat{S}^v \textup{or}\, \Omega_i^x, i\in\{1,2,3\}\\
\\
0 & \textup{otherwise}\\
\end{array}\right..
\ee 
Moreover, for any $\beta\in \mathcal{S}$, we have $\beta\thicksim \iota_1\circ \cdots \iota_{|\beta|},  \iota_i\in \mathcal{K}/\{\vec{0}\}$. We define
\be\label{feb8eqn41}
c_{\textup{vn}}(\beta)= \sum_{i=1,\cdots, |\beta|}c_{\textup{vn}}(\iota_i),\quad c_{\textup{vm}}(\beta)= \sum_{i=1,\cdots, |\beta|}c_{\textup{vm}}(\iota_i), \quad i(\beta)=\sum_{i=1,\cdots,|\beta|} i(\iota_i).
\ee
 \end{definition}

  With the above defined notation,    we  can  reformulate the results in Lemma \ref{twodecompositions} systematically as follows. 	
 \begin{lemma}\label{twodecompositionlemma}
The following two decompositions for ``$D_v$'' holds, 
\be\label{summaryoftwodecomposition}
D_v=\sum_{\rho\in \mathcal{K},|\rho|=1}d_{\rho}(t,x,v) \Lambda^\rho= \sum_{\rho\in \mathcal{K},|\rho|=1}e_{\rho}(t,x,v) \Lambda^\rho,
\ee
where the detailed formulas of coefficients $d_{\rho}(t,x,v)$ and $e_{\rho}(t,x,v)$ are given as follow, 
\be\label{sepeq947}
 {d}_{\rho}(t,x,v)=\left\{\begin{array}{ll}
\tilde{v} \psi_{\geq -1}(|v|)& \textup{if\,\,} \Lambda^\rho \thicksim \psi_{\geq   1}(|v|) \widehat{S}^v \\
\tilde{v} \tilde{d}(t,x,v) (1+|v|^2)^{-1/2}  \psi_{\geq -1}(|v|) & \textup{if\,\,} \Lambda^\rho \thicksim \psi_{\geq 1}(|v|) {S}^x\\
\tilde{V}_i  \psi_{\geq -1}(|v|) & \textup{if\,\,} \Lambda^\rho \thicksim  \psi_{\geq 1}(|v|) \widehat{\Omega}_i^v , i=1,2,3\\
\tilde{V}_i \tilde{d}(t,x,v) (1+|v|^2)^{ 1/2}  \psi_{\geq -1}(|v|) & \textup{if\,\,} \Lambda^\rho \thicksim   \psi_{\geq 1}(|v|) {\Omega}^x_i , i=1,2,3\\
  \psi_{\leq 2 }(|v|) &  \textup{if\,\,} \Lambda^\rho \thicksim   \psi_{\leq 0}(|v|) K_{v_i} , i=1,2,3	\\
- \tilde{d}(t,x,v)(1+|v|^2)\nabla_v\hat{v}_i \psi_{\leq 2}(|v|) &  \textup{if\,\,} \Lambda^\rho \thicksim   \psi_{\leq 0}(|v|) \p_{x_i} , i=1,2,3 \\
0 & \textup{if\,\,}  \Lambda^\rho \thicksim   \tilde{\Omega}_i , i=1,2,3 \\
 \end{array}\right. ,
\ee
\be\label{sepeq932}
 {e}_{\rho}(t,x,v)=\left\{\begin{array}{ll}
\tilde{v} \psi_{\geq -1}(|v|) & \textup{if\,\,} \Lambda^\rho \thicksim \psi_{\geq 1}(|v|) \widehat{S}^v \\
\displaystyle{-\psi_{\geq -1}(|v|)\big(\frac{ \tilde{d}(t,x,v) \tilde{v}}{(1+|v|^2)^{1/2}}+\frac{\tilde{V}_i(X_i\cdot \tilde{v})}{|v|}}\big)  &   \textup{if\,\,} \Lambda^\rho \thicksim \psi_{\geq 1}(|v|) {S}^x\\
0 &  \textup{if\,\,} \Lambda^\rho \thicksim  \psi_{\geq  1}(|v|) \widehat{\Omega}_i^v , i=1,2,3\\
-\psi_{\geq -1 }(|v|) |v|^{-1}{\tilde{V}_j}(X_j+\hat{V}_j t )\cdot \tilde{V}_i  & \textup{if\,\,} \Lambda^\rho \thicksim   \psi_{\geq 1}(|v|) {\Omega}^x_i , i=1,2,3\\
\psi_{\leq  2}(|v|)&   \textup{if\,\,} \Lambda^\rho \thicksim   \psi_{\leq 0}(|v|) K_{v_i} , i=1,2,3	\\
- \psi_{\leq  2}(|v|) \tilde{d}(t,x,v)(1+|v|^2)\nabla_v\hat{v}_i &  \textup{if\,\,} \Lambda^\rho \thicksim   \psi_{\leq 0}(|v|) \p_{x_i} , i=1,2,3 \\
   \psi_{\geq -1}(|v|)|v|^{-1}{\tilde{V}_i} & \textup{if\,\,}  \Lambda^\rho \thicksim   \tilde{\Omega}_i , i=1,2,3 \\
\end{array}\right. .
\ee
Moreover, the following estimates hold,
\[
\sum_{\rho\in \mathcal{K}, |\rho|=1} \big| (1+|v|)^{-1-c_{\textup{vn}}(\rho)} d_{\rho}(t,x,v)\big| + \big| (1+|v|)^{ 1-c_{\textup{vn}}(\rho)} \tilde{v}\cdot d_{\rho}(t,x,v)\big| + \big| (1+|v|)^{-c_{\textup{vm}}(\rho)} d_{\rho}(t,x,v)\big|
\]
\be\label{jan15eqn2}
\lesssim 1+|\tilde{d}(t,x,v)|. 
\ee
 
 \end{lemma}
 \begin{proof}
The above results follow directly from the results  in Lemma \ref{twodecompositions}, which will be used for the case when $|v|\gtrsim 1$ and the validity of    the following decomposition, which will be used  for the case when $|v|\lesssim 1$,
\[
D_v= K_v-(1+|v|^2)\tilde{d}(t,x,v)\nabla_v(\hat{v}\cdot \nabla_x).
\]
From the detailed formula of $d_{\rho}(t,x,v)$ in \textup{(\ref{sepeq947})} and the definition of $c_{\textup{vn}}(\iota)$ and $c_{\textup{vm}}(\iota)$ in (\ref{countingnumber}), it is easy to verify that the desired estimate (\ref{jan15eqn2}) holds straightforwardly.
 \end{proof}

 \subsection{Commutation rules}

In this subsection, we mainly obtain two types of commutation rules.

The first type of commutation rules apply to the commutation between the classic vector fields associated with the wave equation in $\mathfrak{P}_1$ (see (\ref{firstsetforeb})), and general Fourier multiplier operators. This type of commutation rules appears when we try to prove that  the scalar field $\phi$ not only    decays sharply over time but also has extra $1/(1+||t|-|x||)$ decay rate with respect to the light cone.

The second type of commutation rules apply to the commutation between    the vector fields defined in previous subsection and the derivative ``$D_v$''. This type of commutation rules appears when we do high order energy estimate for the profile $g(t,x,v)$ of the Vlasov part.

\begin{definition}
For any linear operator $T$ and any $\alpha\in \mathcal{B}, |\alpha|=1$, we use the following notation to denote the commutator between $T$ and $\Gamma^\alpha \in \mathfrak{P}_1$,
\be\label{octeqn6330}
T_{\alpha}:= (T)_{\alpha}:= [\Gamma^\alpha, T].
\ee
We use  ``$T_k$ '' to denote $T\circ P_k$, where $k\in \mathbb{Z}$. 
\end{definition}

  We    explicitly compute the commutator ``$T_\alpha$'' if ``$T$'' is a {Fourier multiplier operator} as follows. 
 \begin{lemma}\label{commutation1basicgeneral}
For any Fourier multiplier operator $T$, which has a Fourier multiplier $m(\xi) $, we have 
\be\label{octeqn6340}
T_{\alpha}= \widetilde{T}_{\alpha}\circ \hat{T}_{\alpha}, \quad 
\ee
where 
\be
\hat{T}_\alpha=\left\{\begin{array}{ll}
\p_t & \textup{if\,\,} \Gamma^\alpha\in\{L_1,L_2,L_3\} \\
 {Id}  & \textup{otherwise},\\
\end{array}\right.,\quad \widetilde{T}_{\alpha}= \left\{\begin{array}{ll}
 -\mathcal{F}^{-1}\circ\big(\xi\cdot \nabla_\xi m(\xi)\big)\circ \mathcal{F} & \textup{if\,\,} \Gamma^\alpha=S,
\\
0 &  \textup{if\,\,} \Gamma^\alpha=\p_{x_i}, i =1,2,3,
\\
 \mathcal{F}^{-1}\circ\big((e_i\times \xi)\cdot \nabla_\xi m(\xi)\big)\circ \mathcal{F}  &  \textup{if\,\,} \Gamma^\alpha=\Omega_i, i =1,2,3,\\
 
   \mathcal{F}^{-1}\circ\big(i\p_{\xi_i} m(\xi)\big)\circ \mathcal{F}    &  \textup{if\,\,} \Gamma^\alpha=L_i, i =1,2,3.\\
\end{array}\right.
\ee
\end{lemma} 
\begin{proof}
  Since $\p_{x_i}$ commutes with $T$, it is easy to see that the commutator is zero for this case. Now, we consider the case when $\Gamma^\alpha=S$. Note that the following equality holds on the Fourier side, 
\[
\mathcal{F}[[S, T]f ](t,\xi)=  - \big(3 +\xi\cdot\nabla_\xi\big)\big(m(\xi)\widehat{f}(t,\xi) \big) + m(\xi) \big(3 +\xi\cdot\nabla_\xi\big)\widehat{f}(t,\xi) =  -\big(\xi\cdot\nabla_\xi m(\xi)\big)\widehat{f}(t,\xi).
\]
If $\Gamma^\alpha =\Omega_i$, $i\in\{1,2,3\}$, we know that the following equality holds on the Fourier side,
\[
\mathcal{F}[[\Omega_i, T]f](t, \xi)=  (e_i \times \xi)\cdot\nabla_\xi\big(m(\xi)\widehat{f}(t,\xi) \big) -  m(\xi) (e_i \times \xi)\cdot\nabla_\xi\widehat{f}(t,\xi)
= \big((e_i \times \xi)\cdot\nabla_\xi m(\xi)\big) \widehat{f}(t,\xi).
\]
Lastly, we consider the case when $\Gamma^\alpha = L_i, i \in \{1,2,3\}$. For this case, we have
\[
\mathcal{F}[[L_i, T]f](t, \xi)= \mathcal{F}[ x_iT(\p_t f ) - T(x_i\p_t f )  ](t, \xi)
\]
\[
= i\p_{\xi_i}\big(m(\xi) \widehat{\p_t f}(t, \xi)\big)-  i m(\xi) \p_{\xi_i}  \widehat{\p_t f}(t, \xi) = i\p_{\xi_i}m(\xi) \widehat{\p_t f}(t, \xi).
\]
Hence finishing the proof.

\end{proof}

Recall the equality (\ref{sepeqn190}) in Lemma \ref{somebasicidentity}.  We use the following notation to represent it systematically,
\be\label{definitionofcoefficient1}
(|t|^2-|x|^2)\p_i = \sum_{\alpha\in\mathcal{B}, |\alpha|=1}c_{\alpha,i}(t,x) \Gamma^\alpha, i\in\{1,2,3\}, \quad  c_{\alpha}(t,x)=(c_{\alpha,1} (t,x), c_{\alpha,2} (t,x), c_{\alpha,3} (t,x)),
\ee
where ``$c_\alpha(t,x)$'' denotes the unique determined vector coefficient associated with $\Gamma^\alpha$ in (\ref{sepeqn190}). Moreover, we have 
\be\label{noveqn151}
\sum_{\alpha\in \mathcal{B}, |\alpha|=1}|c_{\alpha}(t,x)| +|t\p_t c_{\alpha}(t,x)| \lesssim (|t|+|x|),\quad \sum_{\alpha\in \mathcal{B}, |\alpha|=1}|\nabla_x c_{\alpha}(t,x)| \lesssim 1.
\ee

With the above notation and the commutation rules in Lemma \ref{commutation1basicgeneral}, in Lemma \ref{tradethreetimes}, we prove a Fourier version   of the equality (\ref{sepeqn190}) in Lemma \ref{somebasicidentity}.  It enables us to prove that the scalar field decays at rate $1/((1+|t|+|x|)(1+||t|-|x||))$, see Lemma \ref{sharplinftydecay} for more details.

\begin{lemma}\label{tradethreetimes}
For any given Fourier multiplier operator $T$ with Fourier symbol $m(\xi)$,  
the following equalities hold for any $k\in \mathbb{Z},$ 
\be\label{noveqn171}
(|t|^2-|x|^2)^3 T_k[f](t,x)= \sum_{
\begin{subarray}{c}
i=0,1,2,
\alpha\in \mathcal{B}, |\alpha|\leq 3 
\end{subarray} } \tilde{c}_{\alpha}^i(t,x) \tilde{T}^i_{k,\alpha}(\p_t^i f^\alpha)+(|t|^2-|x|^2)e_{\alpha}(t, x) \tilde{T}_{k,\alpha}^3((\p_t^2-\Delta)f),
\ee
\[
(|t|^2-|x|^2)^3 T_k[\p_t  f ](t,x) =  \sum_{
\begin{subarray}{c}
i=0,1,2, 
\alpha\in \mathcal{B}, |\alpha|\leq 3\\
\end{subarray}} \mathfrak{c}_{\alpha}^i(t,x) \widehat{T}^i_{k,\alpha}(\p_t^i f^\alpha)(t,x)+ (|t|^2-|x|^2)\mathfrak{e}_{\alpha}^1(t,x)   \]
\be\label{feb12eqn39}
  \times \widehat{T}_{k,\alpha}^3((\p_t^2-\Delta)f^{\alpha})(t,x) + (|t|^2-|x|^2)^2\mathfrak{e}_{\alpha}^2(t,x) \widehat{T}_{k,\alpha}^4((\p_t^2-\Delta)f^{\alpha})(t,x),
\ee
where   	 the coefficients $\tilde{c}_{\alpha}^i(t,x),  \mathfrak{c}_{\alpha}^i(t,x), i=0,1,2 $,  $e_\alpha(t,x)$, $\mathfrak{e}^1_\alpha(t,x)$, and $\mathfrak{e}^2_\alpha(t,x)$  satisfy the estimates in  \textup{(\ref{noveqn152})}, \textup{(\ref{noveqn157})}, \textup{(\ref{feb12eqn2})}, \textup{(\ref{feb12eqn10})}, and \textup{(\ref{feb12eqn15})}.

 Moreover, the symbols $\tilde{m}_{k,\alpha}^i(\xi) $   of the Fourier multiplier operators ``$\tilde{T}_{k,\alpha}^i(\cdot)$'', $i\in\{0,1,2,3\}$, and the symbols $\widehat{m}_{k,\alpha}^i(\xi)$ of the Fourier multiplier operators $\widehat{T}_{k,\alpha}^i(\cdot)$ satisfy the estimates in \textup{(\ref{noveqn141})} and \textup{(\ref{feb12eqn1})} respectively. 
\end{lemma}
\begin{remark}
We trade $|t|-|x|$ three times in (\ref{noveqn171}) and (\ref{feb12eqn39}) because of the process  of optimizing in the proof of Lemma \ref{sharplinftydecay}. 
\end{remark}
\begin{proof}
Recall the definition of commutator in (\ref{octeqn6330}) and the definition of operator ``$Q$'' in (\ref{sepeqn732}). From the equality (\ref{definitionofcoefficient1}), we have 
\[
(|t|^2-|x|^2)T_k[f](t,x  )= (|t|^2-|x|^2)\nabla_x \cdot Q\circ T_k(f)(t, x)=\sum_{\alpha\in \mathcal{B}, |\alpha|=1} c_{\alpha}(t,x)\cdot \Gamma^\alpha Q\circ T_k(f)(t, x)
\]
\be
=\sum_{\alpha\in \mathcal{B}, |\alpha|=1} c_{\alpha}(t,x)\cdot\big[ Q\circ T_k(\Gamma^\alpha f)(t, x) + \big(Q \circ T_k\big)_\alpha(f)(t, x)\big].
\ee
Very similarly, after doing this process one more time, we have 
\[
(|t|^2-|x|^2)^2 T_k[f](t,x  )= \sum_{\alpha\in \mathcal{B}, |\alpha|=1} c_{\alpha}(t,x)\cdot\Big[(|t|^2-|x|^2)\nabla\cdot Q\circ\big[ Q\circ T_k(\Gamma^\alpha f)  +  \big(Q \circ T_k\big)_\alpha(f)\big](t, x)\Big]
\]
\[
= \sum_{\alpha_1, \alpha_2\in \mathcal{B}, |\alpha_1|=|\alpha_2|=1} c_{\alpha_1}(t,x)\cdot\Big[ c_{\alpha_2}(t, x)\cdot\big[Q\circ Q\circ T_k( \Gamma^{\alpha_2\circ \alpha_1} f) + Q\circ(Q\circ T_k)_{\alpha_1}(\Gamma^{\alpha_2} f )\]
\be\label{noveqn811}
+(Q\circ Q\circ T_k)_{\alpha_2}(\Gamma^{\alpha_1} f ) +(Q\circ (Q\circ T_k)_{\alpha_1})_{\alpha_2}(f) \big](t,x).
\ee
 Note that the following commutation rule holds for any linear Fourier multiplier  operator $K$, 
\be\label{noveqn202}
[\Gamma^\alpha, K\circ\p_t] = K_\alpha \circ \p_t  + K[\Gamma^\alpha, \p_t],\quad [\Gamma^\alpha, \p_t]=\left\{\begin{array}{ll}
-\p_t & \textup{if\,\,} \Gamma^\alpha = S\\
-\p_{x_i} & \textup{if\,\,} \Gamma^\alpha = L_i,i\in\{1,2,3\}\\
0 & \textup{otherwise}.
\end{array}\right. 
\ee
Therefore, from the equality in (\ref{noveqn811}), we have
\be\label{noveqn876}
(|t|^2-|x|^2)^2 T_k[f](t,x  )=\sum_{\alpha\in \mathcal{B}, |\alpha|\leq 2}\sum_{i=0,1,2} \hat{c}_{\alpha}^i(t,x) T^i_{k,\alpha}(\p_t^i f^\alpha),
\ee
where the symbol $m^i_{k, \alpha}(\xi)$ of the Fourier multiplier $T_{k,\alpha}^i$ and the coefficient $\hat{c}_\alpha^i(t,x)$ satisfy the following estimate, 
\be\label{noveqn851}
\sum_{\alpha\in \mathcal{B}, |\alpha|\leq 2}\sum_{i=0,1,2} 2^{i k }\| m^i_{k, \alpha}(\xi)\|_{\mathcal{S}^\infty_k}\lesssim 2^{-2k }, \quad \sum_{\alpha\in \mathcal{B}, |\alpha|\leq 2}|\hat{c}_{\alpha}^i(t,x)| + (  |t|+|x|) |\nabla_x \hat{c}^i_{\alpha}(t,x)|\lesssim (|t|+|x|)^2.
\ee

Lastly, we do this process one more time.   Note that  the commutator   contains the time derivative ``$\p_t$'' when $\Gamma^\alpha = L_i, i\in\{1,2,3\}$. For simplicity, we don't want to introduce an operator with a third order time derivative   ``$\p_t^3$''. Hence, we need to be  more careful when commutating with a second order commutator term. More precisely, from (\ref{noveqn876}), we have 
\[
(|t|^2-|x|^2)^3 T_k[f](t,x  ) 
=\sum_{\alpha\in \mathcal{B}, |\alpha|\leq 2}\sum_{i=0,1 } \hat{c}_{\alpha}^i(t,x) (|t|^2-|x|^2)\nabla_x\cdot Q\circ T^i_{k,\alpha}(\p_t^i f^\alpha)
\]
\[
+ \hat{c}_{\alpha}^2(t,x) (|t|^2-|x|^2)\nabla_x \cdot \nabla_x T^2_{k,\alpha}(   f^\alpha)+  (|t|^2-|x|^2) \hat{c}_{\alpha}^2(t,x)T^2_{k,\alpha}((\p_t^2-\Delta)  f^\alpha),
\]
\[
= \sum_{\alpha\in \mathcal{B}, |\alpha|\leq 2}\sum_{\rho\in \mathcal{B}, |\rho|=1}\sum_{i=0,1 } \hat{c}_{\alpha}^i(t,x) c_{\rho}(t,x)\cdot\big(Q\circ T_{k,\alpha}^i(\Gamma^{\rho}\p_t^i f^\alpha) +  (Q\circ T_{k,\alpha}^i)_{\rho}(\p_t^i f^\alpha) \big)
\]
\be\label{noveqn135}
+ \hat{c}_{\alpha}^2(t,x)   c_{\rho}(t,x) \cdot \big[(\nabla_x T^2_{k,\alpha})(  \Gamma^\rho f^\alpha) + (\nabla_x T^2_{k,\alpha})_{\rho}(    f^\alpha)\big]+  (|t|^2-|x|^2) \hat{c}_{\alpha}^2(t,x)T^2_{k,\alpha}((\p_t^2-\Delta)  f^\alpha).
\ee

 From  the results in Lemma \ref{commutation1basicgeneral},  (\ref{noveqn202}),   and (\ref{noveqn135}), we know that the following equality holds for some uniquely determined coefficients  $\tilde{c}_{\alpha}^i(t, x)$ and $e_{\alpha}(t,x)$, 
\[
(|t|^2-|x|^2)^3 T_k[f](t,x)= \sum_{
\begin{subarray}{c}
i=0,1,2,
\alpha\in \mathcal{B}, |\alpha|\leq 3 
\end{subarray}}   \tilde{c}_{\alpha}^i(t,x) \tilde{T}^i_{k,\alpha}(\p_t^i f^\alpha)(t, x) 
+(|t|^2-|x|^2)e_{\alpha}(t, x) \tilde{T}_{k,\alpha}^3((\p_t^2-\Delta)f)(t,x),
\]
where the symbol $\tilde{m}_{k,\alpha}^i(\xi)$ of the Fourier multiplier operator ``$\tilde{T}_{k,\alpha}^i(\cdot)$ '', $i\in\{0,1,2,3\}$, satisfies the following estimate, 
\be\label{noveqn141}
\sum_{i=0,1,2 } 2^{ik }\|\tilde{m}_{k,\alpha}^i(\xi)  \|_{\mathcal{S}^\infty_k} \lesssim  2^{-3k}, \quad \| \tilde{m}_{k,\alpha}^3(\xi) \|_{\mathcal{S}^\infty_k}\lesssim 2^{-4k}.
\ee
Moreover, for $i\in\{0,1,2\}, \alpha \in \mathcal{B}, |\alpha|\leq 3$, the coefficients $\tilde{c}_{\alpha}^i(t,x)$ and $e_{\alpha}(t,x)$   satisfy the following estimates, 
\be\label{noveqn152}
 |\tilde{c}_{\alpha}^i(t,x)|+ |t\p_t \tilde{c}_{\alpha}^i(t,x)|\lesssim (|t|+|x|)^3,\quad  |e_{\alpha}(t, x)|+  |t\p_t e_{\alpha}(t,x)|\lesssim (|t|+|x|)^2,
\ee
\be\label{noveqn157}
  |\nabla_x \tilde{c}_{\alpha}^i(t,x)|\lesssim (|t|+|x|)^2, \quad   |\nabla_x e_{\alpha}(t, x)|\lesssim (|t|+|x|),
\ee
which can be derived from the estimate (\ref{noveqn151}). 

With minor modifications, we can   derive the following equality
\[
(|t|^2-|x|^2)^3 T_k[\p_t  f ](t,x) =  \sum_{
\begin{subarray}{c}
i=0,1,2,
\alpha\in \mathcal{B}, |\alpha|\leq 3\\
\end{subarray}} \mathfrak{c}_{\alpha}^i(t,x) \widehat{T}^i_{k,\alpha}(\p_t^i f^\alpha)(t,x) \]
\be\label{feb12eqn30}
+ (|t|^2-|x|^2)\mathfrak{e}_{\alpha}^1(t,x) \widehat{T}_{k,\alpha}^3((\p_t^2-\Delta)f^{\alpha})(t,x)  + (|t|^2-|x|^2)^2\mathfrak{e}_{\alpha}^2(t,x) \widehat{T}_{k,\alpha}^4((\p_t^2-\Delta)f^{\alpha})(t,x),
\ee
where the symbols $\widehat{m}_{k,\alpha}^i(\xi)$ of the Fourier multiplier operators $\widehat{T}^i_{k,\alpha}(\cdot)$, $i\in\{0,1,2,3,4\}$, satisfy the following estimate, 
\be\label{feb12eqn1}
\sum_{i=0,1,2} 2^{i k } \|\widehat{m}_{k,\alpha}^i(\xi)\|_{\mathcal{S}^\infty_k} \lesssim 2^{-2k}, \quad  \|\widehat{m}_{k,\alpha}^3(\xi)\|_{\mathcal{S}^\infty_k} \lesssim 2^{-3k}, \quad \|\widehat{m}_{k,\alpha}^4(\xi)\|_{\mathcal{S}^\infty_k} \lesssim 2^{-2k}. 
\ee
Moreover, the uniquely determined  coefficients $\mathfrak{c}_{\alpha}^i(t,x)$, $i\in\{0,1,2\}$,  $\mathfrak{e}_{\alpha}^1(t,x)$, and $\mathfrak{e}_{\alpha}^2(t,x)$, satisfy the following estimates, 
\be\label{feb12eqn2}
\sum_{i=0,1,2} |\mathfrak{c}_{\alpha}^i(t,x)| + |t\p_t \mathfrak{c}_{\alpha}^i(t,x)| \lesssim (|t|+|x|)^3, 
\ee 
\be\label{feb12eqn10}
 |\mathfrak{e}_{\alpha}^1(t,x)| +  |t\p_t \mathfrak{e}_{\alpha}^1(t,x)| \lesssim (|t|+|x|)^2, \quad |\mathfrak{e}_{\alpha}^2(t,x)|  + |t\p_t \mathfrak{e}_{\alpha}^2(t,x)|  \lesssim (|t|+|x|), 
\ee
\be\label{feb12eqn15}
\sum_{i=0,1,2} |\nabla_x \mathfrak{c}_{\alpha}^i(t,x)|\lesssim (|t|+|x|)^2, \quad  |\nabla_x \mathfrak{e}_{\alpha}^1(t,x)|\lesssim |t|+|x|, \quad |\nabla_x \mathfrak{e}_{\alpha}^2(t,x)|\lesssim 1.
\ee
\end{proof}

Now, we proceed to the second type of commutation rules. For simplicity, we  define a set of vector fields as follows, 
\be\label{lowhighvdecomposition}
  X_1:= \psi_{\geq 1}(|v|)\tilde{v}\cdot D_v, \quad X_{i+1}=  \psi_{\geq 1}(|v|) \tilde{V}_i \cdot D_v,\quad  
X_{i+4}=\psi_{\leq 0}(|v|) D_{v_i}, \quad i=1,2,3,	
\ee
 From (\ref{lowhighvdecomposition}), we have 
 \be\label{noveq1}
 D_v = \tilde{v} X_1 + \tilde{V}_i X_{i+1} + e_i X_{i+4} :=\sum_{i=1,\cdots 7} \alpha_i(v) X_i,
 \ee
 where 
 \be\label{dec28eqn10}
\alpha_1(v):= \psi_{\geq -1}(|v|)\tilde{v}, \quad \alpha_{i+1}(v):= \psi_{\geq -1}(|v|)\tilde{V}_i, \quad \alpha_{i+4}(v):= \psi_{\leq 2}(|v|){e}_i,\quad i=1,2,3. 
\ee

As a basic step, we first consider the first order commutation rule, i.e.,  the case when $\rho
\in \mathcal{K}, |\rho|=1$. 
 \begin{lemma}\label{firstordercommutationlemma} 
For any $\rho \in \mathcal{K}, |\rho|=1$, and $i\in \{1,\cdots, 7\}$, the following commutation rules holds, 
\be\label{feb14eqn1}
[X_i, \Lambda^\rho]= \sum_{\kappa\in \mathcal{K},|\kappa|=1} \tilde{d}_{\rho, i}^{\kappa}(t,x, v) \Lambda^\kappa, \quad  \tilde{d}_{\rho, i}^{\kappa}(t,x, v)  :=  \tilde{c}_{i}^{\rho, \kappa}( x,v) \tilde{d}(t,x,v) + \hat{c}_{i}^{\rho, \kappa}( x,v), 
\ee
where the coefficients $\tilde{c}_{i}^{\rho, \kappa}(t,x,v) $ and $\hat{c}_{i}^{\rho, \kappa}(t,x,v)$ satisfy the following  rough estimates,
\be\label{feb14eqn2}
 |\tilde{c}_{i}^{\rho, \kappa}( x,v)  | +  |\hat{c}_{i}^{\rho, \kappa}( x,v)  |  \lesssim \min\{(1+|v|)^{1+c_{\textup{vn}}(\kappa)-c_{\textup{vn}}(\rho)},(1+|v|)^{ c_{\textup{vm}}(\kappa)-c_{\textup{vm}}(\rho)}\},
\ee
\be\label{feb14eqn3}
| \Lambda^\beta \big( \tilde{c}_{i}^{\rho, \kappa}( x,v)\big)|+ | \Lambda^\beta \big( \hat{c}_{i}^{\rho, \kappa}( x,v)\big)|\lesssim (1+|v|)^{|\beta|+2} (1+|x|)^{|\beta|+2}, \quad \beta \in \mathcal{S}.
\ee
In particular, for the case when $i=1,$ the following improved estimate holds, 
\be\label{feb14eqn4}
 |\tilde{c}_{1}^{\rho, \kappa}(  x,v)  | +  |\hat{c}_{1}^{\rho, \kappa}( x,v)  |  \lesssim  (1+|v|)^{-1+c_{\textup{vn}}(\kappa)-c_{\textup{vn}}(\rho)}.
\ee
Moreover, if   $i(\kappa)-i(\rho)>0$, where $i(\kappa)$ denotes the total number of vector fields $\Omega_i^x$ in $\Lambda^\kappa$,  then the following improved estimate holds for the coefficients of the commutation rule in \textup{(\ref{feb14eqn1})},  
\be\label{april5eqn3}
  |\hat{c}_i^{\rho, \kappa}(x,v)|\lesssim  (1+|v|)^{-1+c_{\textup{vm}}(\kappa)-c_{\textup{vm}}(\rho)}.
\ee
\end{lemma}

\begin{proof}
Postponed to  Appendix \ref{toolkit}. See the proof of Lemma \ref{proofoffirstordercommutationrule} in Appendix \ref{toolkit}.
 \end{proof}

 In the process of commutation between high order vector fields and $X_i$, it is unavoidable  that the vector field $\Lambda^\kappa, \kappa\in \mathcal{K},$ might act on the coefficients. From the equality (\ref{feb14eqn1}) and the estimate (\ref{feb14eqn3}) in Lemma \ref{firstordercommutationlemma}, in the sense of tracking the growth rate with respect to time, we only need to consider the case when $\Lambda^\kappa$ hits the inhomogeneous modulation $\tilde{d}(t,x,v)$. The result is summarized in the following Lemma.

\begin{lemma}\label{derivativesofcoefficient}
For any $\rho \in \mathcal{K}, |\rho|=1$, the following equality holds, 
\be\label{dec26eqn1}
\Lambda^\rho(\tilde{d}(t, x,v)) :=\tilde{d}_{\rho}(x,v)= \tilde{d}(t, x,v)e^{\rho}_1( x,v) + e^{\rho}_2( x,v),
\ee
where the coefficients $b_{\rho}^1(t,x,v)$ and $b_{\rho}^2(t,x,v)$ are some explicit coefficients and satisfy the following estimate, 
\be\label{jan23eqn11}
\sum_{i=1,2}|e^{\rho}_i(  x,v)| \lesssim 1, \quad  \sum_{i=1,2}|\Lambda^{\beta}(e^{\rho}_i( x,v))|\lesssim (1+|x|)^{|\beta|}(1+|v|)^{|\beta|}, \quad \beta\in \mathcal{S}. 
\ee
\end{lemma}
\begin{proof}
Postponed to Appendix \ref{toolkit}. See the proof of Lemma \ref{derivativeofmodulation} in Appendix \ref{toolkit}.
\end{proof}

Now we are ready to introduce the high order commutation rules, which  are basic tools to compute the equation satisfied by the high order derivatives of the profile   ``$g(t,x,v)$''.

We will be very  precise about the estimate   of  the top order coefficients, which matter very much in the energy estimate. However, the estimate of  the lower order coefficients in (\ref{sepeqn524}), i.e., the case when $|\kappa|\leq |\beta|-1$, are rough because we will set a hierarchy for   different orders of the derivatives of the profile $g(t,x,v)$. 

\begin{lemma}\label{summaryofhighordercommutation}
For any $i\in \{1,\cdots,7\}$, and $\beta\in \mathcal{S},$
we have
\be\label{noveq521}
[X_i, \Lambda^\beta]= Y_i^\beta +   \sum_{\kappa\in \mathcal{S}, |\kappa|\leq |\beta|-1 } \big[ \tilde{d}(t,x,v)\tilde{e}_{\beta,i}^{\kappa, 1}(x,v) +\tilde{e}_{\beta,i}^{\kappa, 2}(x,v)\big]  \Lambda^\kappa,
\ee
where $Y_i^\beta$ denotes the top order commutators and it is given as follows, 
\be\label{sepeqn522}
Y_i^\beta= \sum_{\kappa\in \mathcal{S}, |\kappa|=|\beta|,|i(\kappa)-i(\beta)|\leq 1   } \big[  \tilde{d}(t,x,v)\tilde{e}_{\beta,i}^{\kappa, 1}(x,v) +\tilde{e}_{\beta,i}^{\kappa, 2}(x,v)\big]  \Lambda^\kappa,
\ee
where $i(\kappa)$ denotes the total number of vector fields $\Omega_i^x$ in $\Lambda^\kappa$.

For  any $  i\in \{1, \cdots, 7\}$, and $\kappa\in \mathcal{S}$,   the following   estimates hold for the coefficients   $\tilde{e}_{\beta,i}^{\kappa, 1}(x,v) $ and  $\tilde{e}_{\beta,i}^{\kappa,2}(x,v) $,
\be\label{sepeqn524}
 |\tilde{e}_{\beta,i}^{\kappa, 1}(x,v) |+|\tilde{e}_{\beta,i}^{\kappa, 2}(x,v) |\lesssim (1+|x|)^{|\beta|-|\kappa| +2} (1+|v|)^{|\beta|-|\kappa| +4}, \quad \textup{when\,\,}  |\kappa|\leq |\beta|-1,
\ee
 \be\label{sepeqn904}
|\tilde{e}_{\beta,i}^{\kappa, 1}(x,v)| + |\tilde{e}_{\beta,i}^{\kappa, 2}(x,v) |\lesssim \min\{(1+|v|)^{1+ c_{\textup{vn}}(\kappa)-c_{\textup{vn}}(\beta)}, (1+|v|)^{  c_{\textup{vm}}(\kappa)-c_{\textup{vm}}(\beta)}\},\,\, \textup{when}\, |\kappa|=|\beta|,
\ee
\be\label{jan15eqn41}
 |\Lambda^\rho \tilde{e}_{\beta,i}^{\kappa, 1}(x,v) |+|\Lambda^{\rho} \tilde{e}_{\beta,i}^{\kappa, 2}(x,v) |\lesssim (1+|x|)^{|\rho|+ |\beta|-|\kappa| +2} (1+|v|)^{|\rho|+|\beta|-|\kappa| +4}.
\ee
In particular, the following improved estimate holds if $i=1$,
 \be\label{jan31eqn110}
|\tilde{e}_{\beta,i}^{\kappa, 1}(x,v)| + |\tilde{e}_{\beta,i}^{\kappa, 2}(x,v) |\lesssim (1+|v|)^{ -1+ c_{\textup{vn}}(\kappa)-c_{\textup{vn}}(\beta)},\quad \textup{when\,\,} |\kappa|=|\beta|. 
\ee
 Moreover, if $i(\kappa)-i(\beta)>0$ and $|\kappa|=|\beta|$,  where $i(\kappa)$ denotes the total number of vector fields $\Omega_i^x$ in $\Lambda^\kappa$, then   the following improved estimate holds for the coefficients $\tilde{e}_{\beta,i}^{\kappa, 2}(x,v)$ of the commutation rule in \textup{(\ref{noveq521})}, 
\be\label{april5eqn1}
 |\tilde{e}_{\beta,i}^{\kappa, 2}(x,v) |\lesssim (1+|v|)^{-1+c_{\textup{vm}}(\kappa)-c_{\textup{vm}}(\beta)}.
\ee
\end{lemma}

\begin{proof}
From the equality (\ref{feb14eqn1}) in Lemma \ref{firstordercommutationlemma}, we have
\[
[X_i, \Lambda^\beta]= [X_i, \Lambda^{\iota_1}\cdots \Lambda^{\iota_{|\beta|-1}} ] \Lambda^{\iota_{|\beta|}} + \Lambda^{\iota_1}\cdots \Lambda^{\iota_{|\beta|-1}}[X_i, \Lambda^{\iota_{|\beta|}}]
\]
\[
= [X_i, \Lambda^{\iota_1}\cdots   \Lambda^{\iota_{|\beta|-1}} ] \Lambda^{\iota_{|\beta|}} + \Lambda^{\iota_1}\cdots \Lambda^{\iota_{|\beta|-1}}\big( \sum_{\gamma_1\in \mathcal{K}, |\gamma_1|=1}\tilde{d}_{\iota_{|\beta|}, i}^{\gamma_1}(t,x,v) \Lambda^{\gamma_1} \big)
\]
\[
= [X_i, \Lambda^{\iota_1}\cdots   \Lambda^{\iota_{|\beta|-1}} ] \Lambda^{\iota_{|\beta|}} + \sum_{\kappa_1,\rho_1\in \mathcal{S}, \kappa_1+\rho_1=\iota_1\circ\cdots \circ\iota_{|\beta|-1}} \sum_{\gamma_1 \in \mathcal{K},|\gamma_1|=1}\big( \Lambda^{\kappa_1}( \tilde{d}^{\gamma_1}_{\iota_{|\beta|},i}(t,x,v))  \Lambda^{\rho_1}\Lambda^{\gamma_1} \big).
\]
By induction, from the above equality,  we have 
\[ 
[X_i, \Lambda^{\beta}]=     \sum_{j=2\cdots,|\beta|-1 }\sum_{
\begin{subarray}{c}
 \kappa_j, \rho_j\in \mathcal{S}, |\kappa_j|\geq 1,
\kappa_j+\rho_j=\iota_{1}\circ\cdots \circ\iota_{|\beta|-j}\\
\end{subarray}
 } \sum_{\gamma_j\in \mathcal{K}, |\gamma_j|=1} \Lambda^{\kappa_j}( \tilde{d}^{\gamma_j}_{\iota_{|\beta|-j+1},i}(t,x,v))   \Lambda^{\rho_j\circ \gamma_j\circ \iota_{|\beta|-j+2}\circ \cdots \iota_{|\beta|}}  
\]
\be\label{highordercommutation}
+ \sum_{
\begin{subarray}{c}
 \kappa_1, \rho_1\in \mathcal{S}, |\kappa_1|\geq 1,
\kappa_1+\rho_1=\iota_{1}\circ\cdots \circ\iota_{|\beta|-1}\\
\end{subarray}
 } \sum_{\gamma_1\in \mathcal{K},|\gamma_1|=1}   \Lambda^{\kappa_1}( \tilde{d}^{\gamma_1}_{\iota_{|\beta|},i}(t,x,v))   \Lambda^{\rho_1\circ \gamma_1}  
+Y_i^\beta ,
\ee
where $\tilde{d}_{\iota,i}^\gamma(t,x,v)$ is defined in (\ref{feb14eqn1}) and   $Y_i^\beta$ is defined as follows, 
\[
Y_i^\beta:= \sum_{\gamma\in \mathcal{K}, |\gamma|=1}
  \tilde{d}^{\gamma }_{\iota_{1},i}(t,x,v)    \Lambda^{\gamma \circ \iota_2\circ \cdots \circ \iota_{|\beta|}} +  \tilde{d}^{\gamma }_{\iota_{|\beta|},i}(t,x,v)  \Lambda^{\iota_{1}\circ\cdots\iota_{|\beta|-1} \circ \gamma} 
\]
\be\label{sepeqn510}
 + \sum_{i=2\cdots,|\beta|-1 } \tilde{d}_{\iota_{|\beta|-i +1}, i}^{\gamma}(t,x,v) \Lambda^{\iota_{1}\circ\cdots\iota_{|\beta|-i}\circ \gamma \circ \iota_{|\beta|-i+2}\circ\cdots \iota_{|\beta|}}.
\ee
From (\ref{highordercommutation}) and (\ref{dec26eqn1}), we have
\[ 
[X_i, \Lambda^{\beta}]=     \sum_{
\begin{subarray}{c}
j=2\cdots,|\beta|-1\\
 \kappa_j, \rho_j\in \mathcal{S}, |\kappa_i|\geq 1\\
\kappa_j+\rho_j=\iota_{1}\circ\cdots \circ\iota_{|\beta|-i}\\
\end{subarray}
 } \sum_{\gamma_j\in \mathcal{K}, |\gamma_i|=1}  \big(\tilde{d}(t,x,v)e_{\iota_{|\beta|-j+1},i}^{\gamma_j,\kappa_j;1}(x,v)+ e_{\iota_{|\beta|-j+1},i}^{\gamma_j,\kappa_j;2}(x,v) \big)    \Lambda^{\rho_j\circ \gamma_j\circ \iota_{|\beta|-j+2}\circ \cdots \iota_{|\beta|}}\]
 \[  + \sum_{
\begin{subarray}{c}
 \kappa_1, \rho_1\in \mathcal{S}, |\kappa_1|\geq 1, 
\kappa_1+\rho_1=\iota_{1}\circ\cdots \circ\iota_{|\beta|-1}\\
\end{subarray}
 } \sum_{\gamma_1\in \mathcal{K},|\gamma_1|=1}  \big(\tilde{d}(t,x,v)e_{\iota_{|\beta|},i}^{\gamma_1,\kappa_1;1}(x,v) 
  + e_{\iota_{|\beta|},i}^{\gamma_1,\kappa_1;2}(x,v) \big)   \Lambda^{\rho_1\circ \gamma_1}  
+Y_i^\beta.
\]

Hence,  our desired equality (\ref{noveq521}) holds for some determined coefficients $\tilde{e}_{\beta,k}^{\kappa, 1}(x,v)$ and $\tilde{e}_{\beta,k}^{\kappa, 2}(x,v)$, whose explicit formulas are not pursued here. 

Recall (\ref{sepeqn510}).  Our desired equality (\ref{sepeqn522}) and  desired estimates (\ref{sepeqn904}), (\ref{jan31eqn110}), and (\ref{april5eqn1}) hold   from the decomposition  (\ref{feb14eqn1}) and the estimates (\ref{feb14eqn2}), (\ref{feb14eqn4}), and (\ref{april5eqn3}) in Lemma \ref{firstordercommutationlemma}. The desired estimate (\ref{sepeqn524}) follows from (\ref{feb14eqn1}), (\ref{feb14eqn2}), and  (\ref{feb14eqn3}) in Lemma \ref{firstordercommutationlemma} and (\ref{dec26eqn1}) and (\ref{jan23eqn11}) in Lemma \ref{derivativesofcoefficient}.
\end{proof}	
 
\section{Set-up of the energy estimate}\label{setup}

\subsection{The equation satisfied by the high order derivatives of the profile $g(t,x,v)$}

In this subsection, our main goal is to compute the equation satisfied by the high order derivatives of the Vlasov-Nordstr\"om system. 

Recall (\ref{vlasovnordstrom}). For the sake of readers, we restate the equation satisfied by ``$f(t,x,v)$'' as follows,
\[
\p_t f + \hat{v}\cdot \nabla_x f =  \big((\p_t + \hat{v}\cdot \nabla_x)\phi(t,x)\big)(4 f + v\cdot \nabla_v f )+ \frac{1}{\sqrt{1+|v|^2}}\nabla_x \phi \cdot \nabla_v f 
\]
For any $\alpha \in \mathcal{B}$, we define 
\be\label{dec29eqn1}
f^{\alpha}(t,x,v):=\tilde{\Gamma}^\alpha f(t, x,v), \quad \phi^\beta(t, x):= \Gamma^\beta \phi(t,x). 
\ee
Note that the following equality holds,  
\[
\tilde{\Gamma}^\alpha\big((\p_t +\hat{v}\cdot \nabla_v) f \big)=  \sum_{\beta, \gamma\in \mathcal{B}, \beta+\gamma=\alpha} \tilde{\Gamma}^\beta \big((\p_t + \hat{v}\cdot \nabla_x)\phi(t,x)\big)\tilde{\Gamma}^\gamma  (4 f + v\cdot \nabla_v f ) 
+ \frac{1}{\sqrt{1+|v|^2}}\Gamma^\beta\big( \nabla_x \phi\big) \cdot \tilde{\Gamma}^\gamma \nabla_v f.
\]
As a result of direct computations, the following commutation rules hold for any $i,j\in\{1,2,3\}$,
\be\label{commutationrules1}
[\p_{v_i}, S]=0, \,\,\, [\p_{v_i}, \tilde{\Omega}_j]= \p_{v_i} V_j\cdot \nabla_v,  \,\,\, [\p_{v_i}, \tilde{L}_j] = \frac{v_i}{\sqrt{1	+|v|^2}} \p_{v_j}, 
\ee
\be\label{feb12eqn51}
[\p_{x_i}, S]=\p_{x_i}, \quad [\p_{x_i}, \tilde{\Omega}_j]= \p_{x_i} X_j\cdot \nabla_x, \quad [\p_{x_i}, \tilde{L}_j]= \delta_{ij}\p_t,
\ee
\be\label{jan31eqn21}
[v\cdot \nabla_v, \tilde{\Omega}_j]=0,\quad [v\cdot\nabla_v, S]=0,\quad [v\cdot\nabla_v, \tilde{L}_j]=\frac{-1}{\sqrt{1+|v|^2}}\p_{v_j}. 
\ee
From the commutation rules in (\ref{eqn825}), and  (\ref{dec27eqn1}) and the above commutation rules,  we  have
\[
 (\p_t +\hat{v}\cdot \nabla_v) f^\alpha  =  \sum_{\beta, \gamma\in \mathcal{B}, |\beta|+|\gamma|\leq |\alpha|}     \big( (\p_t+\hat{v}\cdot \nabla_x)\phi^\beta (t,x)   \big)\big( a_{\alpha;\beta, \gamma}^1(v) f^\gamma+ a_{\alpha;\beta, \gamma}^2(v) v\cdot\nabla_v f^\gamma\big)  \]
\be\label{jan29eqn81}
+   \big(  a_{\alpha;\beta, \gamma}^3(v)  \nabla_x \phi^\beta (t,x) + a_{\alpha;\beta, \gamma}^4(v)  \p_t \phi^\beta (t,x)  \big)  \cdot  \nabla_v f^\gamma,
\ee
where $a_{\alpha;\beta, \gamma}^i(v), i\in\{1,2,3,4\}$, are some determined coefficients, whose explicit formulas are not pursued here.  
The following rough estimate holds for any $\beta, \gamma\in\mathcal{B}$, s.t., $|\beta|+|\gamma|\leq |\alpha|$, 
\be\label{jan31eqn91}
|a_{\alpha;\beta, \gamma}^1(v) |+ |a_{\alpha;\beta, \gamma}^2(v)| +(1+|v|)\big( |a_{\alpha;\beta, \gamma}^3(v)| + |a_{\alpha;\beta, \gamma}^4(v)|\big) \lesssim 1. 
\ee

Similar to the profiles defined in subsection \ref{firstlevelprofile}, we  define the profile of    $f^{\alpha}(t,x,v)$ as follows, 
\[
g^\alpha(t,x,v):= f^\alpha(t,x+\hat{v}t,v), \quad \Longrightarrow f^\alpha(t,x,v)= g^\alpha(t,x-\hat{v}t,v). 
\]
From (\ref{jan29eqn81}), we can compute the equation satisfied by the profile $g^\alpha(t,x,v)$ as follows, 
 \[
 \p_t g^\alpha(t,x,v)=     \sum_{\beta, \gamma\in \mathcal{B}, |\beta|+|\gamma|\leq |\alpha|}     \big( (\p_t\phi^\beta   +\hat{v}\cdot \nabla_x\phi^\beta )(t,x+\hat{v}t)   \big)\big( a_{\alpha;\beta, \gamma}^1(v) g^\gamma  + a_{\alpha;\beta, \gamma}^2(v) v\cdot D_v g^\gamma\big)(t,x,v) \]
\be\label{jan30eqn1}
+   \big(  a_{\alpha;\beta, \gamma}^3(v)  \nabla_x \phi^\beta  + a_{\alpha;\beta, \gamma}^4(v)  \p_t \phi^\beta  \big)(t,x+\hat{v}t)   \cdot  D_v g^\gamma(t,x,v),
\ee

Now, we apply the second set of vector fields on $g^\alpha(t,x,v)$.  
For any $\beta \in \mathcal{S} $ and any $\alpha\in \mathcal{B}$, we define 
\be\label{jan29eqn20}
g^\alpha_\beta(t,x,v):= \Lambda^{\beta} g^\alpha(t,x,v),\quad \beta\thicksim \iota_1\circ\iota_2\circ\cdots\circ\iota_{|\beta|},\quad \iota_i\in \mathcal{K}, |\iota_i|=1,  \quad i=1,\cdots, |\beta|.
\ee
Note that $[\p_t, \Lambda^\beta]=0$. From (\ref{jan30eqn1}) and (\ref{jan29eqn20}), based on the order of derivatives, we classify the nonlinearities of  $\p_t g^\alpha_\beta(t, x,v)$ as follows, 
\be\label{jan29eqn21}
\p_t g^\alpha_\beta(t,x,v)= K (t,x+\hat{v}t,v)\cdot D_v g^\alpha_\beta(t,x,v)+ \textit{h.o.t}_{\beta}^\alpha(t,x,v) + \textit{l.o.t}_{\beta}^\alpha(t,x,v),
\ee
where
\be\label{jan31eqn1}
 K(t,x+\hat{v}t,v):= v\big( \p_t\phi(t, x+\hat{v}t) + \hat{v}\cdot \nabla_x\phi(t,x+\hat{v}t)\big)  + \frac{1}{\sqrt{1+|v|^2}} \nabla_x\phi(t,x+\hat{v}t), 
\ee
and  ``$\textit{h.o.t}_{\beta}^\alpha(t,x,v)$'' denotes all the terms in which the total number of derivatives act  on $g(t,x,v)$ is  ``$|\alpha|+|\beta|$''  and  ``$\textit{l.o.t}_{\beta}^\alpha(t,x,v)$'' denotes all the terms in which  the total number of derivatives  act  on $g(t,x,v)$  is strictly less than `` $|\alpha|+|\beta|$''.  We remind readers that  the case when the total number of derivatives act on the   scalar field is $|\alpha| +|\beta|$ is included in ``$\textit{l.o.t}_{\beta}^\alpha(t,x,v)$''.

For any $\alpha\in \mathcal{B}, \beta\in \mathcal{S},$ from the decomposition of $D_v$ in (\ref{noveq1}),  we   decompose ``$K (t,x+\hat{v}t,v)\cdot D_v g^\alpha_\beta(t,x,v)$''    as follows, 	
\be\label{april10eqn21}
K (t,x+\hat{v}t,v)\cdot D_v g^\alpha_\beta(t,x,v)=  \sum_{i=1,\cdots, 7}K_i(t,x+\hat{v}t,v)\cdot X_i g^\alpha_\beta(t,x,v),
\ee
where
\be\label{jan31eqn61}
K_1(t,x+\hat{v}t,v)=\psi_{\geq -1}(|v|) |v|\big( \p_t\phi(t, x+\hat{v}t) + \hat{v}\cdot \nabla_x\phi(t,x+\hat{v}t)\big)+ \frac{m^2\alpha_1(v)}{\sqrt{1+|v|^2}} \nabla_x\phi(t,x+\hat{v}t), 
\ee
\be\label{jan31eqn62}
K_{i+1} (t,x+\hat{v}t,v)= \frac{m^2\alpha_{i+1}(v)}{\sqrt{1+|v|^2}} \nabla_x\phi(t,x+\hat{v}t), \quad i\in\{1,2,3\}, 
\ee
\be\label{jan31eqn63}
K_{i+4}(t,x+\hat{v}t,v) = v_i\psi_{\leq 2}(|v|)\big( \p_t\phi(t, x+\hat{v}t) + \hat{v}\cdot \nabla_x\phi(t,x+\hat{v}t)\big)+ \frac{m^2\alpha_{i+4}(v)}{\sqrt{1+|v|^2}} \nabla_x\phi(t,x+\hat{v}t), \quad i\in\{1,2,3\}. 
\ee

Based on the source of the high order terms, recall (\ref{noveq1}),  we classify the high order terms $\textit{h.o.t}_{\beta}^\alpha(t,x,v)$ as follows, 
\be\label{jan30eqn100}
\textit{h.o.t}_{\beta }^\alpha(t,x,v)=\sum_{i=1,2,3} \textit{h.o.t}_{\beta;i}^\alpha(t,x,v)
\ee
where
\be\label{jan30eqn51}
\textit{h.o.t}_{\beta;1}^\alpha(t,x,v)=\sum_{
\begin{subarray}{c}
\iota+\kappa=\beta, |\iota|=1,
\iota, \kappa\in \mathcal{S}, i=1,\cdots,7\\
\end{subarray}} \Lambda^\iota(K_i(t,x+\hat{v}t, v)) X_i g^\alpha_\kappa(t,x, v),
\ee
\[
\textit{h.o.t}_{\beta;2}^\alpha(t,x,v)=
    \sum_{|\rho|  \leq 1,|\gamma|=|\alpha|-1}   a_{\alpha;\rho, \gamma}^2(v)   \big(  \p_t\phi^\rho  (t,x+\hat{v}t)+ \hat{v}\cdot \nabla_x \phi^\rho  (t,x+\hat{v}t) \big)  v\cdot D_v g_{\beta}^\gamma(t,x,v)\big)+   \big(  a_{\alpha;\rho, \gamma}^3(v) \]
\be\label{jan30eqn52}
 \times  \nabla_x \phi^\rho (t,x+\hat{v}t) +  a_{\alpha;\rho, \gamma}^4(v)   \p_t \phi^\rho (t,x+\hat{v}t) \big)  \cdot  D_v g_\beta^\gamma(t,x,v)+ 4 \big(  \p_t\phi  (t,x+\hat{v}t) + \hat{v}\cdot \nabla_x \phi  (t,x+\hat{v}t) \big)  g_\beta^\alpha(t,x,v),
\ee
 \be\label{jan30eqn53}
\textit{h.o.t}_{\beta;3}^\alpha(t,x,v)=  \sum_{i=1,\cdots, 7}   K_i (t,x+\hat{v}t,v)    Y_i^\beta  g^\alpha(t,x,v).
\ee
 
 Similarly, we classify the low order terms ``$\textit{l.o.t}_{\beta}^\alpha(t,x,v)$'' as follows,
 \be\label{jan31eqn151}
\textit{l.o.t}_{\beta}^\alpha(t,x,v)=\sum_{i=1,\cdots, 4} \textit{l.o.t}_{\beta;i}^\alpha(t,x,v),
\ee 
where
\be\label{jan31eqn152}
 \textit{l.o.t}_{\beta;1}^\alpha(t,x,v)=\sum_{i=1,\cdots, 7}   K_i (t,x+\hat{v}t,v)  \big([\Lambda^\beta, X_i]- Y_i^\beta \big)  g^\alpha(t,x,v) ,
\ee
\[
 \textit{l.o.t}_{\beta;2}^\alpha(t,x,v)=    \sum_{|\rho|  \leq 1,|\gamma|=|\alpha|-1}       a_{\alpha;\rho, \gamma}^2(v)   \big(  \p_t\phi^\rho  (t,x+\hat{v}t)+ \hat{v}\cdot \nabla_x \phi^\rho  (t,x+\hat{v}t) \big) [\Lambda^\beta, v\cdot D_v] g^\gamma(t,x,v)\big) 
\]
\[
  +   \big(  a_{\alpha;\rho, \gamma}^3(v)  \nabla_x \phi^\rho   +  a_{\alpha;\rho, \gamma}^4(v)  \p_t \phi^\rho \big) (t,x+\hat{v}t)  \cdot \alpha_i(v) [\Lambda^\beta, X_i] g^\gamma(t,x,v)
\]
\be\label{jan31eqn153}  +  \sum_{
\begin{subarray}{c}
\iota+\kappa=\beta, |\iota|=1,
\iota, \kappa\in \mathcal{S}, i=1,\cdots,7\\
\end{subarray}} \Lambda^\iota(  K_i(t,x+\hat{v}t,v)) [\Lambda^\kappa, X_i] g^\alpha(t,x, v)   , 
\ee
\[
 \textit{l.o.t}_{\beta;3}^\alpha(t,x,v)= \sum_{
\begin{subarray}{c}
 \iota, \kappa\in \mathcal{S}, \beta_1, \gamma_1, \beta_2, \gamma_2\in \mathcal{B},\\ 
  |\rho|+|\beta_1|> 11,  |\beta_1|+|\gamma_1|\leq |\alpha|\\
 |\rho|+|\beta_2|> 11, |\beta_2|+|\gamma_2|\leq |\alpha|\\
 \end{subarray} } \Lambda^{\rho}\big( (\p_t\phi^{\beta_1} (t,x+\hat{v}t)+\hat{v}\cdot \nabla_x\phi^{\beta_1} (t,x+\hat{v}t) )  \big)   \Lambda^{\kappa}\big( a_{\alpha;{\beta_1}, \gamma_1}^1(v)  \]
\be\label{jan31eqn154}  
 \times  g^{\gamma_1}(t,x,v)+ a_{\alpha;\beta_1, \gamma_1}^2(v) v\cdot D_v g^{\gamma_1}(t,x,v)\big)+   \Lambda^\rho \big(  a_{\alpha;\beta_2, \gamma_2}^3(v)  \nabla_x \phi^{\beta_2} (t,x+\hat{v}t)  \cdot\alpha_i(v) \big) \Lambda^\kappa \big(   X_i g^{\gamma_2}(t,x,v)\big),
 \ee
\[
 \textit{l.o.t}_{\beta;4}^\alpha(t,x,v)=  \sum_{
\begin{subarray}{c}
 \iota, \kappa\in \mathcal{S}, \beta_1, \gamma_1, \beta_2, \gamma_2\in \mathcal{B},\\ 
  |\rho|+|\beta_1|\leq 11,  |\beta_1|+|\gamma_1|\leq |\alpha|\\
  |\rho|+|\beta_2|\leq 11, |\beta_2|+|\gamma_2|\leq |\alpha|\\
|\gamma_1|\leq |\alpha|+|\rho|-1, |\gamma_2|\leq |\alpha|+|\rho|-2\\
 \end{subarray} }  \Lambda^{\rho}\big( (\p_t\phi^{\beta_1} (t,x+\hat{v}t)   \]
\[
 +\hat{v}\cdot \nabla_x\phi^{\beta_1} (t,x+\hat{v}t) )  \big) \Lambda^{\kappa}\big( a_{\alpha;{\beta_1}, \gamma_1}^1(v) g^{\gamma_1}(t,x,v) + a_{\alpha;\beta_2, \gamma_2}^2(v) v\cdot D_v g^{\gamma_2}(t,x,v)\big)\]
\be\label{jan31eqn155}   
+   \Lambda^\rho \big[\big( a_{\alpha;\beta_2, \gamma_2}^3(v)  \nabla_x \phi^{\beta_2} (t,x+\hat{v}t) + a_{\alpha;\beta_2, \gamma_2}^4(v)  \p_t\phi^{\beta_2} (t,x+\hat{v}t) \big) \cdot\alpha_i(v) \big] \Lambda^\kappa \big(   X_i g^{\gamma_2}(t,x,v)\big),
 \ee
where $\textit{l.o.t}_{\beta;1}^\alpha(t,x,v)$ arises from the low order commutator between $X_i$ and $\Lambda^\beta$,  see (\ref{noveq521}) in Lemma \ref{summaryofhighordercommutation}, $\textit{l.o.t}_{\beta;2}^\alpha(t,x,v)$ arises from the commutator between $X_i$ and $\Lambda^{\kappa}$, $\kappa \in \mathcal{S}, |\kappa|=|\beta|-1$ or between $X_i$ and $\Lambda^\beta$ when there is only one derivative hits on $K^i(t,x,v)$, $\textit{l.o.t}_{\beta;3}^\alpha(t,x,v)$ arises from the case when there are at least twelve derivatives hit on the  nonlinear wave part, and $\textit{l.o.t}_{\beta;4}^\alpha(t,x,v)$ denotes all the other low order terms, in which  there are at most eleven derivatives hit on the nonlinear wave part and the total number of derivatives hit on $g(t,x,v)$ is strictly less than $|\alpha|+|\beta|$.

To  study the term  of type $\Lambda^\rho(f(t,x+\hat{v}t))$ in $ \textit{l.o.t}_{\beta;3}^\alpha(t,x,v)$, see (\ref{jan31eqn154}), the following Lemma is helpful.

\begin{lemma}\label{decompositionofderivatives}
The following identity holds for any $\rho \in \mathcal{S},  $  
\be\label{sepeqn610}
\Lambda^{\rho}\big( f  (t,x+\hat{v}t  )\big)=\sum_{\iota\in \mathcal{B}, |\iota|\leq |\rho|}  {c}_{ \rho}^{\iota} (x,v)f^{  \iota}(t,x+\hat{v} t) 
	\ee
where  the coefficients $ {c}_{ \rho}^{\iota} (x,v) $, $\iota\in \mathcal{B}, |\iota|\leq |\rho|$,  satisfy the following estimate,
\be\label{sepeqn88}
| {c}_{ \rho}^{\iota} (x,v) |  \lesssim  (1+|x|)^{|\rho|-|\iota|}  (1+|v|)^{  |\rho|-|\iota|}\min\{ (1+|v|)^{-c_{\textup{vn} }(\rho) }, (1+|v|)^{|\rho|-c_{\textup{vm}}(\rho) }  \}. 
\ee
For any $\kappa\in \mathcal{S}, $ the following rough estimate holds 
\be\label{noveq781}
| \Lambda^{\kappa}\big( {c}_{ \rho}^{\iota} (x,v)\big)| \lesssim (1+|x|)^{|\kappa|+|\rho|-|\iota| } (1+|v|)^{|\kappa|+ |\rho|-|\iota| } \min\{ (1+|v|)^{-c_{\textup{vn} }(\rho) }, (1+|v|)^{|\rho|-c_{\textup{vm}}(\rho) }  \} .
\ee
Moreover, the following improved estimate holds if $\Lambda^\rho \nsim \Omega_i^x$  \textup{or} $ \widehat{\Omega}_i^v$, $i\in\{1,2,3\}, |\rho|=1$,
\be\label{april1steqn1}
|a_{\rho,\alpha}(t,x,v) |\lesssim (1+|v|)^{-c_{\textup{vm}}(\rho)}. 
\ee
\end{lemma}
\begin{proof}
 To  calculate  $\Lambda^{\rho}\big( f  (t,x+\hat{v}t  )\big), \rho \in \mathcal{S}$, we induct on the size of $|\rho|$. Since the case when $\rho=\vec{0}$ is trivial, we  first consider the case   $ |\rho|=1,$ i.e., $\rho\in \mathcal{K}/\vec{0}$. Recall (\ref{eqq10}). It is easy to see that the  following equalities hold from direct computations, 
\be\label{noveq751}
\widehat{S}^v  f(t, x+\hat{v} t )=\tilde{d}(t,x,v)(1+|v|^2)^{-1/2} S^x f(t,x+\hat{v} t),  
\ee
\be\label{noveq752}
\widehat{\Omega}_i^v  f(t, x+\hat{v} t )=\sqrt{1+|v|^2}{\tilde{d}(t,x,v)} \Omega_i^x f(t,x+\hat{v} t),\quad \tilde{\Omega}_i\big(f(t,x+\hat{v}t) \big) = (\Omega_i f)(t,x+\hat{v}t), 
\ee
\be\label{noveq753}
K_{v_i} f(t, x+\hat{v} t )= (1+|v|^2)\tilde{d}(t,x,v) \p_{v_i}\hat{v}\cdot \nabla_x f(t, x+\hat{v}t), \quad i \in\{1,2,3\}. 
\ee 
Recall the equality (\ref{eqn20}) in Lemma \ref{essentialidentity} and the equality   (\ref{sepeqn190}) in Lemma \ref{somebasicidentity}. We know that the following equality holds for some uniquely determined  coefficient $b_{\alpha}(t,x,v)$,
\[
\tilde{d}(t,x,v)\nabla_x f(t,x+\hat{v}t)= \sum_{\alpha\in \mathcal{B}, |\alpha|=1} b_{\alpha}(t,x,v) \Gamma^\alpha f(t,x+\hat{v}t), 
\]
where the coefficient $b_{\alpha}(t,x,v)$ satisfies the following rough estimate, 
\be\label{noveq761}
|b_{\alpha}(t,x,v)|\lesssim 1, \quad |\Lambda^{\beta} b_{\alpha}(t,x,v)|\lesssim (1+|x|)^{|\beta|}(1+|v|)^{|\beta|}, \quad \textup{where\,} \beta\in \mathcal{S}.
\ee
From the above rough estimate (\ref{noveq761}) and the equalities (\ref{noveq751}), (\ref{noveq752}), and (\ref{noveq753}), we know that our desired equality (\ref{sepeqn610}) and the desired estimates (\ref{sepeqn88}), (\ref{noveq781}) and (\ref{april1steqn1}) hold for the case  $ |\rho|=1.$

 Now, we proceed to  consider the case when $\rho\in \mathcal{S}, |\rho|> 1$. Recall (\ref{uniqueexpansion}), we have
\[
\rho\thicksim \iota_1\circ \cdots \iota_{|\rho|}, \quad \iota_i \in \mathcal{K}/\vec{0},\quad i =1,\cdots, |\rho|.
\]
From the equality (\ref{sepeqn610}) for the case $|\rho|=1$ and the above equivalence relation, we have
\[
\Lambda^\rho\big( f(t,x+\hat{v}t,v)\big)= \sum_{\iota \in \mathcal{B}, |\iota |=1} 
 \Lambda^{\iota_1\circ \cdots \iota_{|\rho|-1}}\Big(  a_{ \iota_{|\rho|}, \iota}(t,x,v) f^{\iota} (t, x+\hat{v}t)\Big). 
 \]
After keeping iterating the above    process,  our desired estimate (\ref{sepeqn88}) holds from iteration and the Leibniz rule.
	\end{proof}

 \subsection{The   modified profiles of the scalar field}\label{profileofvectorfield}
 Recall (\ref{dec28eqn28}), (\ref{dec28eqn29}), and (\ref{dec28eqn30}). We know that the following equality holds, 
\be\label{jan29eqn53}
{\Gamma}^\alpha- \tilde{\Gamma}^\alpha = \sum_{\beta, \gamma\in \mathcal{B}, |\beta|+|\gamma|\leq |\alpha|, |\beta|\geq 1} a_{\alpha;\beta, \gamma}(v)\cdot\nabla_v^\beta \tilde{\Gamma}^\gamma,
\ee
where ``$a_{\alpha;\beta, \gamma}(v)$'', $\beta, \gamma\in\mathcal{B}, $ are some determined coefficients, whose explicit formulas are not pursued here and the vector field ``$\nabla_v^\beta$'', $ \beta\in \mathcal{B}$, is defined as follows, 
\[
\nabla_v^\beta:= \nabla_v^{\gamma_1}\circ   \cdots \nabla_{v}^{\gamma_{|\beta|}}, \quad \beta \thicksim \gamma_1\circ\cdots \gamma_{|\beta|}, \quad \gamma_i\in \mathcal{A}, |\gamma_i|=1, i\in\{1,\cdots, |\beta|\},
\]
\be\label{feb12eqn81}
\nabla_v^\gamma =\left\{\begin{array}{ll}
 {V}_i\cdot \nabla_v & \textup{if}\,\, \gamma=\vec{a}_{i+4}, i=1,2,3 \\
 
\sqrt{1+|v|^2}\p_{v_i} & \textup{if}\,\, \gamma=\vec{a}_{i+7}, i=1,2,3 \\
 
Id & \textup{otherwise}, \\
\end{array}
\right.\quad \textup{where}\,\, \gamma\in \mathcal{A}.
\ee

 Due to the fact that $\nabla_v$  may hit the coefficients during the expansion, we have $|\beta|+ |\gamma|\leq |\alpha|$ instead of $|\beta|+ |\gamma|= |\alpha|$ in (\ref{jan29eqn53}). Let 
\be\label{jan29eqn51}
a_{\alpha:\vec{0},\alpha}(v):=1, \Longrightarrow 
\Gamma^\alpha = \sum_{\beta,\gamma\in \mathcal{B},|\beta|+|\gamma|\leq |\alpha| } a_{\alpha;\beta, \gamma}(v)\cdot\nabla_v^\beta \tilde{\Gamma}^\gamma.
\ee

With the above equality, we are ready to compute the equation satisfied by the high order derivatives of the scalar field, i.e.,  $\Gamma^\alpha \phi$. Recall (\ref{vlasovnordstrom}). From (\ref{jan29eqn51}), we have,
\be\label{jan29eqn70}
\Gamma^{\alpha}\big((\p_t^2 -\Delta) \phi\big) =  \sum_{\beta,\gamma\in \mathcal{B},|\beta|+|\gamma|\leq |\alpha| }     \int_{\R^3} a_{\alpha;\beta, \gamma}(v)\cdot\nabla_v^\beta \tilde{\Gamma}^\gamma\big(\frac{  f}{\sqrt{1+|v|^2}}\big)  d v . 
\ee 
After doing integration by parts in ``$v$'',  we   derive the equation satisfied by $\phi^\alpha$ as follows,
\be
(\p_t^2 -\Delta) \phi^\alpha= \sum_{ \gamma\in \mathcal{B}, |\gamma|\leq |\alpha| } \int_{\R^3} \tilde{a}_{\alpha; \gamma}(v)   f^\gamma    d v,  
\ee
where $a_{\alpha;\gamma}(v)$, $\gamma\in \mathcal{B}, |\gamma|\leq |\alpha|$, are some   coefficients, whose explicit formulas are not so important and will not be pursued here. From (\ref{feb12eqn81}),  we know that the  following equality and rough estimate holds,
\be\label{feb12eqn83}
\tilde{a}_{\alpha;\alpha}(v)=  ({ {1+|v|^2}})^{-1/2}, \quad |\tilde{a}_{\alpha;\gamma}(v)| \lesssim (1+|v|)^{|\alpha|-|\gamma|}. 
\ee

 Similar to the half wave $u(t)$ and  the profile  $h(t)$ defined in subsection \ref{firstlevelprofile}, we  define
\be\label{feb8eqn31}
u^\alpha(t):= (\p_t - i\d) \phi^\alpha(t),\quad h^\alpha(t):=e^{i t\d} u^\alpha(t).
 \ee
Hence, we can recover $\p_t\phi^\alpha$ and $\phi^\alpha$ from the half wave $u^\alpha(t)$ and the profile $h^\alpha(t)$ as follow, 
\be\label{feb8eqn32}
\p_t\phi^\alpha= \frac{u^\alpha(t)+\overline{u^\alpha(t)}}{2}, \quad  \phi^\alpha = \frac{-u^\alpha(t)+\overline{u^\alpha(t)}}{2i \d}=\sum_{\mu\in\{+,-\}} c_{\mu} \d^{-1}  (u^\alpha)^{\mu} (t), \quad  u^\alpha(t):=e^{-i t\d} h^\alpha(t).
\ee
From (\ref{jan29eqn70}),  we can derive the equation satisfied by $u^\alpha(t)$ as follows, 
\be
(\p_t + i \d) u^\alpha(t) =  \sum_{ \gamma\in \mathcal{B}, |\gamma|\leq |\alpha| } \int_{\R^3} \tilde{a}_{\alpha; \gamma}(v)   f^\gamma    d v =  \sum_{ \gamma\in \mathcal{B}, |\gamma|\leq |\alpha| } \int_{\R^3} \tilde{a}_{\alpha; \gamma}(v)   g^\gamma(t, x-\hat{v}t, v)  d v. 	
\ee
On the Fourier side, we have 
\be\label{jan29eqn2000}
\p_t \widehat{h^\alpha}(t, \xi ) = \sum_{ \gamma\in \mathcal{B}, |\gamma|\leq |\alpha| }   \int_{\R^3} \tilde{a}_{\alpha; \gamma}(v)  e^{i t |\xi| - it \hat{v}\cdot \xi} \widehat{g^\gamma}(t, \xi,  v) dv. 
\ee

Correspondingly, we can write the equation satisfied by $g^\alpha(t,x,v)$ in (\ref{jan30eqn1}) in terms of profiles  on Fourier side as follows, 
\[
  \p_t \widehat{g^\alpha}(t,\xi,v)   =  \sum_{\beta, \gamma\in \mathcal{B}, |\beta|+|\gamma|\leq |\alpha|} \sum_{\mu\in\{+,-\} }  \int_{\R^3} e^{it (\xi-\eta)\cdot \hat{v} - i t\mu |\xi-\eta|}\widehat{(h^\beta)^\mu}(t, \xi-\eta) \big[   \big(\h   + i c_{\mu } \hat{v} \cdot  \frac{\xi-\eta}{|\xi-\eta|}   \big)a_{\alpha;\beta, \gamma}^1(v) 
\]
\be\label{jan30eqn31}
\times \widehat{g^\gamma}(t, \eta, v)	 +  \big( (\h+ i c_{\mu}\hat{v}\cdot \frac{\xi-\eta}{|\xi-\eta|}) a_{\alpha;\beta, \gamma}^2(v) v +  \h a_{\alpha;\beta, \gamma}^4(v)   + i c_{\mu }a_{\alpha;\beta, \gamma}^3(v)    \frac{\xi-\eta}{|\xi-\eta|}   \big) \cdot \big(\nabla_v - it \eta \big)\widehat{g^\gamma}(t, \eta, v) \big]   d \eta.
\ee

Similar to the modified profile defined in subsection \ref{firstlevelprofile}, to take the advantage of the oscillation of the phase over time, instead of controlling the increment of the profile, we control the following \textit{modified profile}, 
\be\label{feb1eqn41}
\widehat{\widetilde{h^\alpha}}(t, \xi ):= \widehat{h^\alpha}(t, \xi ) + \sum_{ \gamma\in \mathcal{B}, |\gamma|\leq |\alpha| }  \int_{\R^3} e^{i t |\xi| - it \hat{v}\cdot \xi} \frac{i \tilde{a}_{\alpha; \gamma}(v) }{|\xi| - \hat{v}\cdot \xi } \widehat{g^\gamma}(t, \xi,  v) dv.
\ee
 Define 
 \be\label{feb1eqn10}
 \widetilde{\phi^\alpha}(t):= \frac{-e^{-it \d} \widetilde{h^\alpha}(t) + \overline{e^{-it \d} \widetilde{h^\alpha}(t) } }{2 i\d 	},\quad   \widetilde{\p_t \phi^\alpha}(t):= \frac{  e^{-it \d} \widetilde{h^\alpha}(t) + \overline{e^{-it \d} \widetilde{h^\alpha}(t) } }{2  	}, 
 \ee
\be\label{noveq510}
E_{\alpha;\gamma}(f)(t,x):=  \mathcal{F}^{-1}\big[\int_{\R^3} e^{-it \hat{v}\cdot \xi } \frac{ - i \tilde{a}_{\alpha; \gamma}(v) }{ (|\xi|-\hat{v}\cdot \xi )}  \widehat{f}(t,\xi, v) d v\big](x).
\ee
Hence, from (\ref{feb8eqn32}) and (\ref{feb1eqn41}), we have
\be\label{feb1eqn1}
\phi^\alpha(t)= \widetilde{\phi^\alpha}(t) -\sum_{ \gamma\in\mathcal{B}, |\gamma|\leq |\alpha|}  \d^{-1} \big(\textup{Im}\big[ E_{\alpha;\gamma} (g^\gamma)(t)\big]\big), \,\, \p_t \phi^\alpha(t) =  \widetilde{\p_t\phi^\alpha}(t) +\sum_{ \gamma\in\mathcal{B}, |\gamma|\leq |\alpha|}  \big(\textup{Re}\big[ E_{\alpha;\gamma} (g^\gamma)(t)\big]\big).
\ee 
Recall the   equations  (\ref{jan29eqn2000}), (\ref{feb1eqn41}),  and (\ref{jan30eqn31}).  After doing integration by parts in ``$v$'' once to move around the ``$\nabla_v$'' derivative in front of $\nabla_v \widehat{g^\gamma}(t,\eta, v)$,  we can derive the equation satisfied by the modified profile $\widehat{\widetilde{h}^\alpha}(t, \xi )$ as follows, 
\[
\p_t \widehat{\widetilde{h^\alpha} } (t, \xi )= \sum_{ \gamma\in \mathcal{B}, |\gamma|\leq |\alpha| }  \int_{\R^3} e^{i t |\xi| - it \hat{v}\cdot \xi} \frac{i \tilde{a}_{\alpha; \gamma}(v)  }{|\xi| - \hat{v}\cdot \xi } \p_t\widehat{g^\gamma}(t, \xi,  v) dv
\]
\[
= \sum_{
\begin{subarray}{c}
\gamma\in \mathcal{B}, |\gamma|\leq |\alpha|,\mu\in\{+,-\} \\ 
\beta, \kappa \in \mathcal{B}, |\beta|+|\kappa|\leq |\gamma|
\end{subarray} }    \int_{\R^3}\int_{\R^3}  e^{it |\xi| - it \mu |\xi-\eta| - i t\hat{v}\cdot \eta}   \widehat{g^\kappa}(t, \eta, v) \widehat{(h^\beta)^\mu}(t, \xi-\eta)  \big[  \frac{i \tilde{a}_{\alpha; \gamma}(v)  }{|\xi| - \hat{v}\cdot \xi } \big(\h   + i c_{\mu } \hat{v} \cdot  \frac{\xi-\eta}{|\xi-\eta|}   \big) 
\]
\[
 \times a_{\gamma;\beta, \kappa}^1(v)  
 -  \nabla_v \cdot  \big(  \frac{i \tilde{a}_{\alpha; \gamma}(v)  }{|\xi| - \hat{v}\cdot \xi }\big((\h+ i c_{\mu}\hat{v}\cdot \frac{\xi-\eta}{|\xi-\eta|}) a_{\gamma;\beta, \kappa}^2(v) v +  \h a_{\gamma;\beta, \kappa}^4(v)   + i c_{\mu }a_{\gamma;\beta, \kappa}^3(v)    \frac{\xi-\eta}{|\xi-\eta|}   \big)\big)   \big] d\eta d v
\]
\[
= \sum_{
\beta, \gamma\in \mathcal{B},  
 |\beta|+|\gamma|\leq |\alpha| 
 } \sum_{\mu\in\{+,-\} }\int_{\R^3}\int_{\R^3}  e^{it |\xi| - it \mu |\xi-\eta| - i t\hat{v}\cdot \eta} \big(  m_1(\xi,v) \widehat{a}_{\alpha;\beta, \gamma}^{\mu,1}(v)    +  m_2(\xi,v) \widehat{a}_{\alpha;\beta, \gamma}^{\mu,2}(v) \cdot \frac{\xi-\eta}{|\xi-\eta|}\big) 
 \]
\be\label{feb1eqn42}
 \times  \widehat{g^\gamma}(t, \eta, v) \widehat{(h^\beta)^\mu}(t, \xi-\eta)  d \eta d v,
 \ee
where $\widehat{a}_{\alpha;\beta, \gamma}^{\mu,i}(v )$,  $\mu\in\{+,-\}, i\in\{1,2\}$, $\beta, \gamma\in\mathcal{B},$ s.t., $|\beta|+|\gamma|\leq |\alpha|$, are some determined coefficients and $m_i(\xi,v)$, $i\in\{1,2\}$, are some determined symbol, whose explicit formulas are not pursued here. 

From the  rough estimates (\ref{jan31eqn91}) and (\ref{feb12eqn83}), we know  that the following rough estimate holds, 
 \be\label{feb12eqn100}
\sum_{i=1,2} | \widehat{a}_{\alpha;\beta, \gamma}^i (v )| \lesssim  (1+|v|)^{|\alpha|-|\gamma|}, \quad \sum_{i=1,2,|a|\leq 5} \|(1+|v|)^{-10}\nabla_v^a m_i(\xi,v)\|_{\mathcal{S}^\infty_k} \lesssim 2^{-k}, \quad k \in \mathbb{Z}. 
 \ee

 \subsection{Constructing the energy  for the Vlasov-Nordstr\"om system}

\subsubsection{Control of the profile $g(t,x,v)$} 
We define the high order energy for the profile $g(t,x,v)$ of the distribution function as follows,
\be\label{highorderenergy1}
E_{\textup{high}}^{f}(t):=  E_{\textup{high}}^{f;1}(t)+ E_{\textup{high}}^{f;2}(t), \quad E_{\textup{high}}^{f:1}(t):= \sum_{\alpha\in \mathcal{B}, \beta\in \mathcal{S}, |\alpha|+|\beta|= N_0}   \| \omega_{\beta}^{\alpha}( x, v)  g^\alpha_\beta (t,x,v)\|_{L^2_{x,v}} ,
\ee 
\be\label{highorderenergy2}
E_{\textup{high}}^{f:2}(t):= \sum_{\alpha\in \mathcal{B}, \beta\in \mathcal{S}, |\alpha|+|\beta|< N_0}   \| \omega_{\beta}^{\alpha}( x, v)  g^\alpha_\beta (t,x,v)\|_{L^2_{x,v}},
\ee 
where $g^\alpha_\beta(t,x,v)$ is defined in (\ref{jan29eqn20}) and the weight function $\omega_\beta^\alpha(x,v)$ is defined as follows, 
 \be\label{highorderweight}
\omega^\alpha_\beta(t,x, v) = (1 +|x |^2 + (x\cdot v)^2 +|v|^{20} )^{20N_0-10(|\alpha|+|\beta|)}   (1+|v|)^{c_{\textup{vn}}(\beta)},
\ee
 where the index  $c_{\textup{vn}}(\beta)$ is  defined in (\ref{feb8eqn41}). We separate out $E_{\textup{high}}^{f:1}(t)$ as the strictly top order energy.

 We choose the   weight function as in (\ref{highorderweight}) based on the following ideas and reasons: (i) We set up a hierarchy for the order of weight function $\omega^\alpha_\beta(t,x, v)$.   Note that the total number of derivatives act on $g_{\beta}^\alpha(t,x,v)$ is $|\alpha|+|\beta|$.  The more  derivatives  act the profile, the lower order weight function we use for $\omega^\alpha_\beta(t,x, v)$. (ii) Comparing with the ordinary derivatives of the profile, we expect that the good derivatives of the profile can propagate more weight in ``$|v|$''; (iii)  
  We  choose an anisotropic weight in ``$x$''  in the definition of  the weight function $\omega_{\beta}^\alpha(x,v)$ in (\ref{highorderweight}) to guarantee that the following Lemma holds, which plays an essential role in the energy estimate. 	
\begin{lemma}\label{derivativeofweightfunction}
For any $\alpha\in \mathcal{B}, \beta\in \mathcal{S},$ s.t., $|\alpha|+|\beta|\leq N_0$, the following estimate holds for any $x, v\in \mathbb{R}^3$, 
\be\label{feb8eqn51}
 \Big[\Big| \frac{v\cdot {D}_v \omega_{\beta}^{\alpha}( x, v)}{\omega_{\beta}^{\alpha}( x, v)} \Big| + \Big| \frac{{D}_v \omega_{\beta}^{\alpha}( x, v)}{\omega_{\beta}^{\alpha}( x, v)} \Big| \big] \frac{1}{1+||t|-|x+\hat{v} t||} \lesssim 1.
\ee 
\end{lemma}
\begin{remark}
 Essentially speaking, the   estimate (\ref{feb8eqn51}) says that, through a good choice of the weight function $\omega_{\beta}^\alpha(x,v)$, the loss of size  $t$, which comes from the coefficient of $D_v$, can be controlled by the  distance with respect to the light cone when $D_v$ hits the weight function.  The desired estimate (\ref{feb8eqn51}) is important in the energy estimate of the distribution function for the case when all derivatives hits on $\nabla_v f$, or equivalently $D_v g(t,x,v)$, see Proposition \ref{notbulkterm1} in section \ref{highordervlasovestimate}.
\end{remark}
\begin{remark}
 Due to the presence of coefficient ``$m^2/\sqrt{1+|v|^2}$'' in the Vlasov-Nordstr\"om system (see (\ref{vlasovnordstrom})), we don't need an estimate as strong as in (\ref{feb8eqn51}) to close the energy estimate for the Vlasov-Nordstr\"om system. Since the desired estimate (\ref{feb8eqn51}) is good for both the Vlasov-Nordstr\"om system and the Vlasov-Maxwell system, we stick to the weight function defined in (\ref{highorderweight}) for this paper and also the second part \cite{wang}.
\end{remark}
\begin{proof} 
Recall the decomposition of $D_v$ in  (\ref{eqq13}). We have
\[
 \frac{{D}_v \omega_{\beta}^{\alpha}( x, v)}{ \omega_{\beta}^{\alpha}( x, v)   } = \tilde{v}\big( \frac{ \tilde{v}\cdot \nabla_v \omega_{\beta}^{\alpha}( x, v)}{\omega_{\beta}^{\alpha}( x, v)} - \frac{t}{ ({1+|v|^2})^{3/2}} \frac{ \tilde{v}\cdot \nabla_x \omega_{\beta}^{\alpha}( x, v)}{\omega_{\beta}^{\alpha}( x, v)}\big)\]
 \be\label{feb8eqn61}
  + \sum_{i=1,2,3} \tilde{V}_i\big( \frac{ \tilde{V}_i\cdot \nabla_v \omega_{\beta}^{\alpha}( x, v)}{\omega_{\beta}^{\alpha}( x, v)} - \frac{t}{ ({1+|v|^2})^{1/2}} \frac{ \tilde{V}_i\cdot \nabla_x \omega_{\beta}^{\alpha}( x, v)}{\omega_{\beta}^{\alpha}( x, v)}\big).
\ee
From the explicit formula of $\omega_{\beta}^\alpha(x,v)$ in (\ref{highorderweight}), we have 
\be\label{feb8eqn62}
(1+|v|)\Big| \frac{ \tilde{v}\cdot \nabla_v \omega_{\beta}^{\alpha}( x, v)}{\omega_{\beta}^{\alpha}( x, v)}\Big| + \sum_{i=1,2,3} \Big| \frac{ \tilde{V}_i\cdot \nabla_v \omega_{\beta}^{\alpha}( x, v)}{\omega_{\beta}^{\alpha}( x, v)}\Big|\lesssim 1,
\ee
\be\label{feb8eqn63}
 \Big| \frac{ \tilde{v}\cdot \nabla_x \omega_{\beta}^{\alpha}( x, v)}{\omega_{\beta}^{\alpha}( x, v)}\Big|\lesssim \frac{x\cdot \tilde{v} + (x\cdot v)|v|}{1+|x|^2+(x\cdot v)^2 + |v|^{10}}\lesssim \frac{1+|v|}{1+|x|+|x\cdot v| +|v|^5},
 \ee
 \be\label{feb8eqn64}
     \sum_{i=1,2,3} \Big| \frac{ \tilde{V}_i\cdot \nabla_x \omega_{\beta}^{\alpha}( x, v)}{\omega_{\beta}^{\alpha}( x, v)}\Big|\lesssim  \frac{|x|}{1+|x|^2+(x\cdot v)^2 + |v|^{10}}\lesssim \frac{1 }{1+|x|+|x\cdot v| +|v|^5}. 
\ee
Recall the decomposition (\ref{feb8eqn61}). From the estimates (\ref{feb8eqn62}), (\ref{feb8eqn63}), and (\ref{feb8eqn64}), we know that the desired estimate estimate (\ref{feb8eqn51}) holds easily if $|x|\geq 3|t|$. 

It remains to consider the case when $|x|\leq 3 |t|$. For this case,  we have 
\be\label{feb8eqn70}
\frac{1}{1+||t|-|x+\hat{v}t||} \sim  \frac{1+ |t|}{1+|t|+ |t^2-|x+\hat{v} t|^2| } = \frac{1+ |t|}{\displaystyle{1+|t|+ \Big|\frac{t^2}{1+|v|^2} - \frac{2tx\cdot  v}{\sqrt{1+|v|^2}} - |x|^2 \Big| }}.
\ee
Based on the size of $|x|$ and $x\cdot v$, we separate into two cases as follows.

\begin{enumerate}
\item[$\bullet$] If $|x|\geq 2^{-10} |t|/(1+|v|)$ or $|x\cdot v |\geq 2^{-10} |t|/(1+|v|)$, then from the estimates (\ref{feb8eqn63}) and (\ref{feb8eqn64}),  we have 
\be\label{feb8eqn81}
\frac{|t|}{1+|v|^2}\Big| \frac{ \tilde{v}\cdot \nabla_x \omega_{\beta}^{\alpha}( x, v)}{\omega_{\beta}^{\alpha}( x, v)}\Big|  +\sum_{i=1,2,3} \frac{|t|}{1+|v|} \Big| \frac{ \tilde{V}_i\cdot \nabla_x \omega_{\beta}^{\alpha}( x, v)}{\omega_{\beta}^{\alpha}( x, v)}\Big|  \lesssim 1. 
\ee
\item[$\bullet$] If $|x|\leq 2^{-10} |t|/(1+|v|)$ and  $|x\cdot v |\leq 2^{-10} |t|/(1+|v|)$, then from the estimate (\ref{feb8eqn70}), we have 
\[
\frac{1}{1+||t|-|x+\hat{v}t||} \lesssim \frac{1+|v|^2}{1+|t|}. 
\]
Therefore, from the above estimate and the estimates (\ref{feb8eqn63}) and (\ref{feb8eqn64}),  we have, 
\be\label{feb8eqn82}
\Big[\frac{|t|}{1+|v|^2}\Big| \frac{ \tilde{v}\cdot \nabla_x \omega_{\beta}^{\alpha}( x, v)}{\omega_{\beta}^{\alpha}( x, v)}\Big|  +\sum_{i=1,2,3} \frac{|t|}{1+|v|} \Big| \frac{ \tilde{V}_i\cdot \nabla_x \omega_{\beta}^{\alpha}( x, v)}{\omega_{\beta}^{\alpha}( x, v)}\Big|\Big] \frac{1}{1+||t|-|x+\hat{v}t||} \lesssim 1. 
\ee
\end{enumerate}
To sum up, for any $x, v\in \R^3$,    our desired estimate (\ref{feb8eqn51}) holds  from the decomposition (\ref{feb8eqn61}) and the estimates (\ref{feb8eqn62}), (\ref{feb8eqn81}), and (\ref{feb8eqn82}).

\end{proof}

Similar to the study of the Vlasov-Poisson system in \cite{wang3}, from the decay estimate (\ref{densitydecay}) of the average of the distribution function in Lemma \ref{decayestimateofdensity}, we know  that   the zero frequency  of the distribution function plays the leading role in the decay estimate. With this intuition, to ensure $E_{\alpha;\gamma} (g^\gamma)(t)$ defined in (\ref{noveq510}) decays sharply, we define a lower order energy for the profile $g(t,x,v)$ as follows,  
\be\label{octeqn1896}
E_{\textup{low}}^{f}(t):=\sum_{\gamma\in\mathcal{B},  |\alpha| +|\gamma|\leq N_0 }   \| \widetilde{\omega_{\gamma}^{\alpha}}(  v)\big(\nabla_v^\alpha \widehat{ g^\gamma }(t,0,v)-\nabla_v\cdot  \widetilde{g}_{\alpha,\gamma}(t, v) \big)\|_{  L^2_v},\,\,\,\,  \widetilde{\omega_{\gamma}^{\alpha}}(  v):=  (1   +|v|^{ } )^{20N_0-10(|\alpha| +|\gamma|)} 
\ee
where  the correction term $  \widetilde{g}_{\alpha,\gamma}(t, v)$ is defined as follows,
\be\label{correctionterm2}
 \widetilde{g}_{\alpha,\gamma}(t, v):=\left\{\begin{array}{ll}
 \displaystyle{\int_0^t \int_{\R^3} K(s,x+\hat{v}s,v) \nabla_v^\alpha g^{\gamma  }(s,x,v) } d x  d s & \textup{if}\, |\alpha|+|\gamma| =N_0\\ 
 &\\
 0 & \textup{if}\, |\alpha| < N_0 ,\\ 
 \end{array}
 \right.
\ee
where $K(t,x+\hat{v}t,v)$ was defined in (\ref{jan31eqn1}). We  introduce  $ \widetilde{g}_{\alpha,\gamma}(t, v)$ in (\ref{octeqn1896}) for the purpose of avoiding losing derivatives in  the study of the time evolution of $\nabla_v^\alpha \widehat{ g^\gamma }(t,0,v)$.

\subsubsection{Control of the profiles and the modified profiles of the scalar field} 
For the nonlinear wave part,  we define a high order energy as follows, 
\[
 E_{\textup{high}}^{ \phi   }(t) :=   \sup_{k\in \mathbb{Z}}  \sum_{\alpha \in \mathcal{B}, |\alpha|\leq N_0}   2^{k}  \| \widehat{h^\alpha }(t,\xi)\psi_k(\xi) \|_{L^\infty_\xi}  +  2^{k}  \| \widehat{\widetilde{h^\alpha} }(t,\xi)\psi_k(\xi) \|_{L^\infty_\xi}   +     2^{k/2} \| \nabla_\xi	 \widehat{\widetilde{h^\alpha} }(t,\xi)\psi_k(\xi)\|_{L^2_\xi}
 \]
\be\label{highorderenergyphi}
  +  \| \widehat{h^\alpha }(t,\xi) \|_{L^2_\xi} +  \|   \widehat{\widetilde{h^\alpha} }(t,\xi) \|_{L^2_\xi}. 
\ee
The first part of energy $E_{\textup{high}}^{\phi}(t)$, which is stronger than $L^2$ at low frequencies, controls the low frequency part of the profiles $h^\alpha(t)$;  the second part of energy $E_{\textup{high}}^{\phi}(t)$, which has the same scaling level as the first part of energy $E_{\textup{high}}^{\phi}(t)$,  aims to control the first order weighted norm of the modified profiles $\widetilde{h^\alpha}(t)$  ;  the third part   of energy $E_{\textup{high}}^{\phi}(t)$, controls the high frequency part of the profiles $h^\alpha(t)$ and the modified profiles $\widetilde{h^\alpha}(t)$. 

Motivated from the linear decay estimate of half wave equation in Lemma \ref{twistedlineardecay}, to prove sharp decay estimate for the nonlinear solution,  we define a low order energy for the profiles $h^\alpha(t)$, $\alpha\in\mathcal{B}$, of the nonlinear scalar field as follows, 
 \be\label{secondorderloworder}
   E_{\textup{low}}^{ \phi}(t):=   \sum_{
\begin{subarray}{c}
n=0,1,2,3\\
   \alpha\in \mathcal{B}, |\alpha|\leq 20-3n\\
   \end{subarray} } \| h^\alpha (t)\|_{X_n} + (1+|t|)\| \p_t h^\alpha (t)\|_{X_n} + (1+|t|)^2\| \p_t \nabla_x(1+|\nabla_x|)^{-1} h^\alpha (t)\|_{X_n} ,
\ee
where the $X_n$-normed space  is defined as follows, 
 \be\label{definitionofXnorm}
\|h\|_{X_n}:=  \sup_{k\in \mathbb{Z}}2^{(n+1)k }\| \nabla_\xi^n \widehat{h }(t, \xi)\psi_k(\xi)\|_{L^\infty_\xi},\quad  n\in\{0,1,2,3\}.
\ee

We remark that we also controls the second order and the third order weighted norms of the profiles in the low order energy because these will be   helpful in the energy estimate for the case when the distribution function has the top order regularity inside the nonlinearity.

\subsection{A precise statement of main theorem}\label{precisetheoremsubsect}
With   previous preparations, we are ready to state the main theorem. 	

\begin{theorem}[A precise statement]\label{precisetheorem}
Let $N_0=200, $ $\delta\in (0, 10^{-9}]$. Suppose that the given initial data $(f_0(x,v),  $  $ \phi_0(x),  \phi_1(x))$ of the $3D$ relativistic Vlasov-Nordstr\"om system \textup{(\ref{vlasovnordstrom})} satisfies the following smallness assumption, 
\[
\sum_{|\alpha_1|+|\alpha_2|\leq N_0}  \| (1 +|x |^2 + (x\cdot v)^2 +|v|^{20} )^{30N_0} \nabla_v^{\alpha_1} \nabla_x^{\alpha_2 } f_0(x,v)\|_{L^2_x L^2_v} 
+\sum_{\alpha\in \mathcal{B}, |\alpha|\leq N_0}\sum_{n\in\{0,1,2,3\}} \| \Gamma^\alpha\big(\d \phi_0(x)\big)\|_{L^2} \]
\be\label{march1steqn1}
+\| \Gamma^\alpha\big(\d \phi_0(x)\big)\|_{X_n}+ \| \Gamma^\alpha\big(  \phi_1(x)\big)\|_{L^2} + \| \Gamma^\alpha\big(  \phi_1(x)\big)\|_{X_n}\leq \epsilon_0,
\ee
where the $X_n$-normed space is defined in \textup{(\ref{definitionofXnorm})}  and   $\epsilon_0$ is some sufficiently small constant. Then the relativistic Vlasov-Nordstr\"om system \textup{(\ref{vlasovnordstrom})} admits a global solution and scatters to a linear solution. Moreover, the following estimate  holds  over time, 
\be\label{energyestimate1}
\sup_{t\in[0,\infty)} (1+t)^{-\delta}\big[ E_{\textup{high}}^{f}(t) +  E_{\textup{high}}^{\phi}(t)  \big] + E_{\textup{low}}^{f}(t) +  E_{\textup{low}}^{\phi}(t)\lesssim \epsilon_0,
\ee
As   byproducts of the above estimate, we have the following decay estimates for the derivatives of the average of the distribution function and the derivatives of the scalar field, 
\be\label{desiredecayaverage}
\sup_{t\in[0,\infty)}\sum_{|\alpha|\leq N_0-20}(1+|t|)^{ (3+|\alpha|)/p} \Big|\int_{\R^3} \nabla_x^\alpha \big|f(t,x,v)\big|^p d v  \Big|^{1/p} \lesssim \epsilon_0, \quad \textup{where}\,\, p\in[1, \infty)\cap \mathbb{Z},
\ee
\be\label{decayingeneral}
\sup_{t\in[0,\infty)}\sum_{|\alpha|\leq 10} (1+|t| )(1+||t|-|x||)^{|\alpha|+1}\big( \big|\nabla_x^{\alpha}\p_t\phi(t,x)\big| + \big|\nabla_x^{\alpha}\nabla_x \phi(t,x)\big|\big)\lesssim \epsilon_0.
\ee
\end{theorem}

\subsection{Proof of the main theorem: A bootstrap argument}\label{proofofmaintheorem}
 In next  two sections,   we will estimate the increment of both the low order energy and the high order energy for the nonlinear wave part  and  the Vlasov part  with respect to time. Those estimates allow us to prove the desire theorem,  Theorem \ref{precisetheorem},  by using a bootstrap argument.

We state our bootstrap assumption as follows, 
\be\label{bootstrapassumption}
\sup_{t\in[0, T]} (1+t)^{-\delta}\big[ E_{\textup{high}}^{f;1}(t) +  E_{\textup{high}}^{\phi}(t)  \big] + (1+t)^{-\delta/2} E_{\textup{high}}^{f;2}(t) + E_{\textup{low}}^{f}(t) +  E_{\textup{low}}^{\phi}(t)\lesssim \epsilon_1:=\epsilon_0^{5/6},
\ee
 where, from the local theory,  $T>0$. 

 Recall the  definition of correction term in (\ref{correctionterm2}), From the linear decay estimate (\ref{noveqn935}) of the scalar field in Lemma \ref{sharplinftydecay}, we know that  the following estimate holds for any $t\in[0,T]$,   under the bootstrap assumption (\ref{bootstrapassumption}),
 
 \be\label{may17eqn51}
 \sum_{\gamma\in\mathcal{B},  |\alpha| +|\gamma|\leq N_0 }\|\widetilde{\omega_{ }^{\alpha}}(  v) \widetilde{g}_{\alpha,\gamma}(t, v) \|_{L^2_v} \lesssim \int_{0}^{t} (1+|s|)^{-2+2\delta} \epsilon_1^2 d s \lesssim \epsilon_0. 
 \ee
From the estimate (\ref{feb20eqn10}) in Proposition \ref{loworderenergy}, the estimate (\ref{sepeqn110}) in Proposition \ref{notbulkterm1}, the estimate (\ref{octeq501}) in Proposition \ref{loworderenergyprop1}, and the estimate (\ref{octe691}) in Proposition \ref{highorderenergyprop1}, the following improved estimate holds under the bootstrap assumption (\ref{bootstrapassumption}), 
\be
\sup_{t\in[0, T]}  (1+t)^{-\delta}\big[ E_{\textup{high}}^{f;1}(t) +  E_{\textup{high}}^{\phi}(t)  \big] + (1+t)^{-\delta/2} E_{\textup{high}}^{f;2}(t) + E_{\textup{low}}^{f}(t) +  E_{\textup{low}}^{\phi}(t)\lesssim \epsilon_0.
\ee
Therefore, we can   close the bootstrap argument and extend ``$T$'' to infinity. As a result, we have
\be\label{march2eqn20}
\sup_{t\in[0, \infty)}  (1+t)^{-\delta}\big[ E_{\textup{high}}^{f;1}(t) +  E_{\textup{high}}^{\phi}(t)  \big] + (1+t)^{-\delta/2} E_{\textup{high}}^{f;2}(t) + E_{\textup{low}}^{f}(t) +  E_{\textup{low}}^{\phi}(t)\lesssim \epsilon_0,
\ee
which   also implies  our desired estimate (\ref{energyestimate1}). 

The desired decay estimate (\ref{desiredecayaverage}) of the average of the derivatives of distribution function  follows directly from the estimate (\ref{march2eqn20}) and the decay estimate (\ref{densitydecay}) in Lemma \ref{decayestimateofdensity}. The desired decay estimate (\ref{decayingeneral}) of the derivatives of the scalar field  follows directly from the estimate (\ref{march2eqn20}), the equality (\ref{sepeqn190}) in Lemma \ref{somebasicidentity},  and the decay estimate (\ref{noveqn935}) in Lemma \ref{sharplinftydecay}. 

Lastly, we explain the scattering property of  the nonlinear solution.  From the estimate (\ref{feb20eqn10}) in Proposition \ref{loworderenergy} and the definition of ``$E_{\textup{low}}^{\phi}(t)$'' in (\ref{secondorderloworder}), we can  construct a limit for  the profiles of the Vlasov-Nordstr\"om  system by integrating the profile  with respect to  time up to infinity. After pulling back the limit along the linear flow, we have our desired scattering linear solution.

Hence finishing the proof of    Theorem \ref{precisetheorem}. 
\section{Energy estimates for the nonlinear wave part}\label{energyestimatewave}
This section is devoted to control both the low order energy  $E_{\textup{low}}^{\phi}(t)$    defined in (\ref{secondorderloworder})  and the high order energy $E_{\textup{high}}^{\phi}(t)$  defined in (\ref{highorderenergyphi}) of the profiles of the scalar field over time.  For the low order energy estimate, our main result is summarized  in Proposition \ref{loworderenergyprop1}. For the high order energy estimate, our main result is summarized in  Proposition \ref{highorderenergyprop1}. 

The main tools used  are some linear estimates and some bilinear estimates, which will be used as black boxes first in the proof of Proposition \ref{loworderenergyprop1} and the proof of Proposition \ref{highorderenergyprop1}. We postpone these linear estimates and bilinear estimates to  subsection \ref{linearbilinear1} and subsection \ref{linearbilinear2}.

 \begin{proposition}\label{loworderenergyprop1}
 Under the bootstrap assumption \textup{(\ref{bootstrapassumption})}, 
the following estimate holds for any $t\in[0, T]$,
\be\label{octeq501}
 E_{\textup{low}}^{\phi}(t) \lesssim    E_{\textup{low}}^f(t) + (1+|t|)^{-1}E_{\textup{high}}^f(t)+\epsilon_0.
\ee
\end{proposition}
\begin{proof}
 Recall (\ref{secondorderloworder}). We first estimate the $X_n$-norm of $\p_t h^\alpha(t)$. Recall (\ref{jan29eqn2000}). Form the estimate of coefficients in (\ref{feb12eqn83}),  the estimate (\ref{octe0303}) in Lemma \ref{Alinearestimate}, we  have 
 \[
    \sum_{0\leq n\leq 3}\sum_{\alpha \in \mathcal{B}, |\alpha|\leq 20-3n}  (1+|t|) \| \p_t h^\alpha (t)\|_{X_n} + (1+|t|)^2 \| \frac{\nabla_x}{1+|\nabla_x|} \p_t   h^\alpha(t)\|_{X_n}
 \]
\be\label{octeq792}
\lesssim   E_{\textup{low}}^f(t) + (1+|t|)^{-1}E_{\textup{high}}^f(t)\lesssim \epsilon_0.
\ee
Now, it remains to estimate the $X_n$-norm of $h^\alpha(t)$. Recall (\ref{feb1eqn41}).  As a result of direct computation, we know that the symbol $1/(|\xi|-\hat{v}\cdot \xi)$ verifies the estimate (\ref{roughestimateofsymbolgenarl2}).   From  the estimate of coefficients in (\ref{feb12eqn83}) and the estimate (\ref{octe0303}) in Lemma \ref{Alinearestimate}, we have
 \be\label{octe673}
 \sum_{n=0,1,2,3 }\sum_{|\alpha|\leq 20-3n} \|\widetilde{h^\alpha}(t)- h^\alpha(t)\|_{X_n} \lesssim   E_{\textup{low}}^{f}(t) +  (1+|t|)^{-1}E_{\textup{high}}^f(t)\lesssim \epsilon_0.
 \ee

 Hence, it would be sufficient to estimate the $X_n$-norm     of the modified profiles $\widetilde{h^\alpha}(t) $. Recall (\ref{feb1eqn42}) and (\ref{octeqn1923}). We know that $\p_t\widetilde{h^\alpha}(t,\xi) $ is a linear combination of bilinear forms defined in (\ref{octeqn1923}). Therefore, from the estimate   (\ref{octe364}) in Lemma \ref{bilinearinXnormed}, we have
\be\label{octe667}
\sum_{n=0,1,2,3 }\sum_{|\alpha|\leq 20-3n} \| \p_t {\widetilde{h}^\alpha}(t    )\|_{X_n} \lesssim (1+|t|)^{-2}   E_{\textup{low}}^{\phi}(t)  E_{\textup{high}}^{f}(t)\lesssim (1+|t|)^{-2+ \delta}  \epsilon_1^2. 
 \ee
Hence, from the above estimate (\ref{octe667}) and the estimate (\ref{octe673}), we have
\be\label{octe671}
 \sum_{n=0,1,2,3}\sum_{|\alpha|\leq 20-3n} \|\widetilde{h}^\alpha (t)\|_{X_n} +\| {h}^\alpha (t)\|_{X_n}  \lesssim  \epsilon_0 + \int_0^t (1+|s|)^{-2+\delta}  \epsilon_1^2 d s\lesssim\epsilon_0.
\ee
To sum up, our desired estimate (\ref{octeq501}) holds from the estimates (\ref{octeq792})         and (\ref{octe671}).  
\end{proof}
 
\begin{proposition}\label{highorderenergyprop1}
 Under the bootstrap assumption \textup{(\ref{bootstrapassumption})}, 
the following estimate holds for any $t\in[0, T]$,
\be\label{octe691}
   E_{\textup{high}}^{\phi}(t)\lesssim  E_{\textup{high}}^f(t) +(1+|t|)^{\delta}\epsilon_0. 
\ee
\end{proposition}
\begin{proof}
Recall (\ref{highorderenergyphi}). Based on the entries of $ E_{\textup{high}}^{\phi}(t)$, we divide the high order energy estimate of the scalar field into three parts as follow. 

$\bullet$ \textbf{Case} $1$:\quad  $L^\infty_\xi$-estimate of  $\widehat{h^\alpha}(t, \xi)$ and $\widehat{ \widetilde{h^\alpha}}(t, \xi)$.

 Recall (\ref{feb1eqn41}).  We know that the following estimate holds for any $\alpha\in \mathcal{B}, |\alpha|\leq N_0$,
\be\label{octe695}
\sup_{k\in \mathbb{Z}}2^{k} \|\big(\widehat{\widetilde{h^\alpha}}(t, \xi)-\widehat{ {h^\alpha}}(t, \xi)\big)\psi_k(\xi)\|_{L^\infty_\xi} \lesssim \sum_{\gamma\in \mathcal{B}, |\gamma|\leq |\alpha|} \| (1+|v|)^{5+4(|\alpha|-|\gamma|)} \widehat{g^\gamma}(t, \xi, v)\|_{L^\infty_\xi L^1_v} \lesssim E_{\textup{high}}^f(t).
\ee

Hence,  it would be sufficient to estimate the $L^\infty_\xi$-norm of the localized $\widehat{\widetilde{h_i^\alpha}}(t, \xi)$. Recall the equation satisfied by $\p_t \widehat{\widetilde{h_i^\alpha}}(t, \xi)$ in  (\ref{feb1eqn42}) . From the estimate (\ref{octe364}) in Lemma \ref{bilinearinXnormed}, which is used when the profile $h (t)$ has relatively more derivatives, and the estimate (\ref{octe103}) in Lemma \ref{highorderbilinearlemma1}, which is used when $g(t,x,v)$ has relatively more derivatives, we have
\be\label{octe700}
\sup_{k\in \mathbb{Z}} 2^k \|\p_t \widehat{\widetilde{h^\alpha}}(t, \xi)\psi_k(\xi)\|_{L^\infty_\xi } \lesssim (1+|t|)^{-2+\delta} \big( 
E_{\textup{high}}^{\phi}(t) + E_{\textup{low}}^{\phi}(t)\big) E_{\textup{high}}^f(t)\lesssim (1+|t|)^{-2+2\delta}\epsilon_1^2. 
\ee 
Therefore, from  the estimates (\ref{octe695}) and (\ref{octe700}), we have 
 
\be\label{octe741}
  \sum_{|\alpha|\leq N_0}  \sup_{k\in \mathbb{Z}}  2^{k} \|  \widehat{ \widetilde{h^\alpha}}(t, \xi) \psi_k(\xi)\|_{L^\infty_\xi}\lesssim \epsilon_0, \quad \sum_{|\alpha|\leq N_0}  \sup_{k\in \mathbb{Z}} 2^{k} \|  \widehat{ {h^\alpha}}(t, \xi) \psi_k(\xi)\|_{L^\infty_\xi}\lesssim  E_{\textup{high}}^f(t)  + \epsilon_0. 
\ee

 $\bullet$ \textbf{Case} $2$:\quad   $L^2$-estimate  of   $\widehat{h^\alpha}(t, \xi)$ and $\widehat{ \widetilde{h^\alpha}}(t, \xi)$.

From the   estimate (\ref{octe695}),  after   dyadically decomposing     frequency ``$\xi$'', we have  
\[
   \|\big(\widehat{\widetilde{h^\alpha}}(t, \xi)-\widehat{ {h^\alpha}}(t, \xi)\big) \|_{L^2_\xi} \]
   \[
   \lesssim     \sum_{k\geq 0} \sum_{|\gamma|\leq |\alpha|}   2^{-k} \|  (1+|v|)^{5+4(|\alpha|-|\gamma|)}\widehat{g^\gamma}(t, \xi, v)\|_{L^1_v L^2_\xi }  +\sum_{k\leq 0} 2^{3k/2}  \|\big(\widehat{\widetilde{h^\alpha}}(t, \xi)-\widehat{ {h^\alpha}}(t, \xi)\big)\psi_{k}(\xi) \|_{L^\infty_\xi} \]
\be\label{deceqn191}
\lesssim E_{\textup{high}}^f(t) +  \sum_{|\gamma|\leq |\alpha|}\| \omega_{\gamma}^{\vec{0}}(x,v) g^\gamma(t,x,v)\|_{L^2_x L^2_v}\lesssim E_{\textup{high}}^f(t).  
\ee
Hence, it would be sufficient to estimate the $L^2$-norm of $\widehat{\widetilde{h^\alpha}}(t, \xi)$. After using the first   estimate of (\ref{deceqn98}) in Lemma \ref{L2bilinearvw} for the case when there are more derivatives act on  the profile  $g(t,x,v)$  and using the second  estimate of (\ref{deceqn98}) for the case when there are more derivatives act on  the profile $h(t)$, we have 
\[
 \|\p_t \widehat{\widetilde{h^\alpha}}(t, \xi) \|_{L^2_\xi } \lesssim (1+t)^{-1} E_{\textup{high}}^f(t)  E_{\textup{low}}^{\phi}(t) + (1+t)^{-2} E_{\textup{high}}^f(t) E_{\textup{high}}^{\phi} (t)\lesssim (1+|t|)^{-1+\delta}\epsilon_1^2.
\]
From the above estimate    and the estimate (\ref{deceqn191}), we have
\be\label{dec30eqn21}
\sum_{\alpha\in \mathcal{B}, |\alpha|\leq N_0}   \| \widehat{ {h^\alpha}}(t, \xi) \|_{L^2_\xi } + \| \widehat{  \widetilde{h^\alpha}}(t, \xi) \|_{L^2_\xi }   \lesssim   E_{\textup{high}}^f(t) +\epsilon_0 +\int_{0}^t (1+|s|)^{-1+\delta} \epsilon_1^2 d s \lesssim  E_{\textup{high}}^f(t) + (1+|t|)^{\delta} \epsilon_0.
\ee
$\bullet$ \textbf{Case} $3$:\quad $L^2_\xi$-estimate of $\nabla_\xi \widehat{\widetilde{h^\alpha}  }(t, \xi)$.  

  Recall the equation satisfied by $\widetilde{h^\alpha} (t, \xi)$ in (\ref{feb1eqn42}). Based on the different size of $|\gamma|$ in  (\ref{feb1eqn42}), we use different strategy.  If $|\gamma|\geq N_0-	10$,   we apply the   estimates (\ref{octe406})
in Proposition \ref{firstorderweighthigh1} directly.

 If $|\gamma|\leq N_0-10$, then we first use the decomposition  (\ref{feb1eqn41}) for  $ \widehat{h^\beta }(t, \xi)$ and have the following equality,   
\[
\p_t \widehat{\widetilde{h}^\alpha}(t, \xi )= \sum_{
\beta, \gamma\in \mathcal{B},  
 |\beta|+|\gamma|\leq |\alpha| 
 } \sum_{\mu\in\{+,-\} }\big[\int_{\R^3}\int_{\R^3}  e^{it |\xi| - it \mu |\xi-\eta| - i t\hat{v}\cdot \eta} \big(    m_2(\xi,v) \widehat{a}_{\alpha;\beta, \gamma}^{\mu,2}(v) \cdot \frac{\xi-\eta}{|\xi-\eta|} + m_1(\xi,v)\]
 \[
 \times  \widehat{a}_{\alpha;\beta, \gamma}^{\mu,1}(v)     \big) \widehat{g^\gamma}(t, \eta, v) \widehat{(\widetilde{h}^\beta)^\mu} (t, \xi-\eta)  d \eta d v + \sum_{\kappa\in\mathcal{B},|\kappa|\leq |\beta| } \int_{\R^3}\int_{\R^3} \int_{\R^3}   e^{it |\xi| - it \hat{u}\cdot(\xi-\eta) - i t\hat{v}\cdot \eta} \big(   \widehat{a}_{\alpha;\beta, \gamma}^{\mu,2}(v) \cdot \frac{\xi-\eta}{|\xi-\eta|}\]
 \be\label{may15eqn500}
 \times  m_2(\xi,v)  + m_1(\xi,v)\widehat{a}_{\alpha;\beta, \gamma}^{\mu,1}(v)     \big) \widehat{g^\gamma}(t, \eta, v)   \frac{i\mu  \tilde{a}_{\beta; \kappa}(u) }{|\xi-\eta| - \mu \hat{u}\cdot (\xi-\eta) } \widehat{(g^\kappa)^\mu}(t, \xi-\eta,  u)  d \eta d v  d u\big].
 \ee
Then we apply the estimates (\ref{octe406}) and  (\ref{octe581}) in  Proposition \ref{firstorderweighthigh1}  for the first integral in (\ref{may15eqn500}) and apply the estimate (\ref{may15eqn410}) in Lemma \ref{twoaveragedistri} for the second integral in (\ref{may15eqn500}). 

To summarize, as a result of above different strategies for different scenarios, the following estimate holds for any fixed $k\in \mathbb{Z}$, 
\[
  2^{k/2}\|\p_t \nabla_\xi \widehat{\widetilde{h}^\alpha }(t, \xi)\psi_k(\xi)\|_{L^2_\xi }     \lesssim  \big( (1+|t|)^{-1}2^{k_{-} }+ (1+|t|)^{-2+\delta} \big) E_{\textup{low}}^{\phi}(t)  E_{\textup{high}}^{f}(t) \]
  \be 
  + (1+|t|)^{-2} E_{\textup{high}}^{f}(t)\big( E_{\textup{high}}^{f}(t) +  E_{\textup{high}}^{\phi}(t)  \big)\lesssim (1+|t|)^{-1+\delta} 2^{k_{-}} \epsilon_1^2 + (1+|t|)^{-2+2\delta}\epsilon_1^2. 
  \ee
Hence, from the above estimate, we have 
\be\label{octe742}
\sum_{|\alpha|\leq N_0}   2^{k/2}\|  \nabla_\xi \widehat{\widetilde{h}^\alpha }(t, \xi)\psi_k(\xi)\|_{L^2_\xi } \lesssim \epsilon_0 + \int_{0}^{t} \big[(1+s)^{-1+\delta} 2^{k_{-}}   + (1+s)^{-2+2\delta}\big]\epsilon_1^2 d s \lesssim \epsilon_0 + (1+t)^{\delta} 2^{k_{-} } \epsilon_0 .
 \ee 

 To sum up, recall again (\ref{highorderenergyphi}),   our desired estimate (\ref{octe691}) holds from the estimates (\ref{octe741}), (\ref{dec30eqn21}),  and (\ref{octe742}). 
\end{proof}

\subsection{Linear estimates and  a bilinear estimate   in the  $L^\infty_\xi$-type space}\label{linearbilinear1}
In this subsection, we will mainly prove several linear estimates  in the  $L^\infty_\xi$-type space, which will be used for the density type function, and   a bilinear estimate    in the  $L^\infty_\xi$-type space,  which will be used for the Vlasov-wave type interaction.

Our linear estimates are summarized in the following Lemma. 

\begin{lemma}\label{Alinearestimate}
For any given $n\in \mathbb{N}_{+}, n\leq 10$,   and any given symbol $m(\xi, v)$ that satisfies the following  estimate,
 \be\label{roughestimateofsymbolgenarl2}
 \sup_{k\in \mathbb{Z}}\sum_{i=0,1,\cdots,10,|a|\leq 15 }  2^{ik-(n-1)k} \|(1+|v|)^{-20-4i} \nabla_\xi^i\nabla_v^a m(\xi, v)\psi_k(\xi)\|_{L^\infty_\xi L^\infty_v}\lesssim 1,
\ee
then the following estimate holds for any $i\in\{0,1,2,3\},$  
\[
\| \int_{\R^3} e^{i t|\xi|- i\mu t\hat{v}\cdot \xi}    m(\xi, v) \widehat{g}(t, \xi, v) d v \|_{X_i}  \lesssim   \sum_{ |\alpha|\leq  i+n }  (1+|t|)^{-n} \|(1+|v|)^{30} \nabla_v^\alpha \widehat{g}(t,  0, v)  \|_{  L^1_v }
\]
\be\label{octe0303}
 +   \sum_{\beta\in \mathcal{S}, |\beta|\leq  i+n }  (1+|t|)^{-n-1} \| (1+|x|^2+|v|^2)^{20}\Lambda^\beta g(t,x,v)\|_{L^2_{x }L^2_v}.
 \ee
 Moreover, for any fixed $k\in \mathbb{Z} $ and any differentiable   function $\tilde{g}(t,v):\R_t\times \R_v^3\rightarrow \R^3$,  the following $L^\infty_\xi$-type estimate holds,
 \[
2^k \|\int_{\R^3} e^{i t|\xi|- i\mu t\hat{v}\cdot \xi}    m(\xi, v) \widehat{g}(t, \xi, v)\psi_{k}(\xi) d v\|_{L^\infty_\xi } \lesssim 2^{nk}\big(\|(1+|v|)^{20}\big(\widehat{g}(t,0,v)-\nabla_v\cdot \tilde{g}(t,v)\big)\|_{L^1_v} \]
\be\label{may17eqn12}
+ (1+|t|2^{k})\| (1+|v|)^{20} \tilde{g}(t,v)\|_{L^1_v} + 2^{k} \| (1+|x|+|v|)^{30} g(t,x,v)\|_{L^2_x L^2_v}\big).
 \ee
\end{lemma}
\begin{proof}
Note that for any $i\in\{0,1,2,3\},$ the following equality holds,
\be\label{octe510}
\nabla_\xi^i\big( \int_{\R^3} e^{i t|\xi|- i\mu t\hat{v}\cdot \xi} m(\xi, v) \widehat{g}(t, \xi, v) d v\big)= \sum_{0\leq l_1+l_2 \leq i} H^{\mu}_{i;l_1,l_2}(t, \xi),
\ee
where
\be\label{may14eqn20}
H^{\mu}_{i;l_1,l_2}(t, \xi) = \sum_{ 0\leq k\leq l_1} \int_{\R^3} e^{i t|\xi|- i\mu t\hat{v}\cdot \xi} c_{i,l_1}^{l_2,k}\big( i t \big)^{l_1} \big(\frac{\xi}{|\xi|}-\mu \hat{v}\big)^{l_1-k} \nabla_\xi^{i-l_1-l_2}\big(  \big(\frac{\xi}{|\xi|}-\mu \hat{v}\big)^{k}  m(\xi, v)\big)\nabla_\xi^{l_2} \widehat{g}(t, \xi, v)\big) d v,
\ee
where the  $c_{i,l_1}^{l_2,k}$ are some uniquely determined coefficients.    From the above detailed formula of $ H^{\mu}_{i;l_1,l_2}(t, \xi)$, we know that our desired estimate (\ref{octe0303}) is trivial if $|t|\leq 1$. From now on, we restrict ourself to the case when $|t|\geq 1$.

If $l_2=0$, then we separate out the zero frequency of ``$\widehat{g}(t,\xi,v)$'' as follows, 
\[
H^{\mu}_{i;l_1,0}(t, \xi) = \sum_{ 0\leq k\leq l_1} \int_{\R^3} e^{i t|\xi|- i\mu t\hat{v}\cdot \xi}  c_{i,l_1}^{l_2,k}\big( i t \big)^{l_1} \big(\frac{\xi}{|\xi|}-\mu \hat{v}\big)^{l_1-k} \nabla_\xi^{i-l_1}\big(  \big(\frac{\xi}{|\xi|}-\mu \hat{v}\big)^{k}  m(\xi, v)\big) \widehat{g}(t, 0, v)\big) d v
\]
\be\label{march2eqn1}
+ \int_0^1 \int_{\R^3} e^{i t|\xi|- i\mu t\hat{v}\cdot \xi} c_{i,l_1}^{l_2,k} \big( i t \big)^{l_1} \big(\frac{\xi}{|\xi|}-\mu \hat{v}\big)^{l_1-k} \nabla_\xi^{i-l_1}\big(  \big(\frac{\xi}{|\xi|}-\mu \hat{v}\big)^{k}  m(\xi, v) \xi\cdot \nabla_\xi \widehat{g}(t, s\xi, v)\big) d v d s.
\ee
 For the first integral of    $H^{\mu}_{i;l_1,0}(t, \xi)$ in (\ref{march2eqn1}), we do integration by parts in ``$v$'' $n+l_1$ times. For the second integral of  $H^{\mu}_{i;l_1,0}(t, \xi)$ in (\ref{march2eqn1}) and $H^{\mu}_{i;l_1,l_2}(t, \xi)$, $|l_2|>0$, in (\ref{may14eqn20}),   we do integration by parts in ``$v$'' $n+l_1+1$ times. As a result, from the estimate of symbol in (\ref{roughestimateofsymbolgenarl2}), the following estimate holds for any fixed $k\in \mathbb{Z},$ $i,l_1,$ and $l_2$,
\[
2^{(i+1)k}\| H^{\mu}_{i;l_1,l_2}(t, \xi)\psi_k(\xi) \|_{L^\infty_\xi} \lesssim  \sum_{|\beta|\leq n+j}  (1+|t|)^{-n} \|(1+|v|)^{ 30} \nabla_v^\beta \widehat{g}(t,  0, v)  \|_{  L^1_v }  +  \sum_{1\leq a\leq i-j} (1+|t|)^{-n-1} \]
\[
   \times \big(2^{(a-1)k}+ 2^{  ak}\big) \|(1+|v|)^{30} \nabla_\xi^a \nabla_v^\beta \widehat{g}(t, \xi, v)\psi_{k}(\xi) \|_{L^\infty_\xi L^1_v } \lesssim \sum_{|\alpha|\leq n +i } (1+|t|)^{-n} \|(1+|v|)^{30} \nabla_v^\beta \widehat{g}(t,  0, v)  \|_{  L^1_v }
 \]
 \be\label{octe511}
 + \sum_{\beta\in \mathcal{S}, |\beta|\leq  i+n }  (1+|t|)^{-n-1} \| (1+|x|^2+|v|^2)^{20}\Lambda^\beta g(t,x,v)\|_{L^2_{x }L^2_v}.
 \ee
 Hence  our desired estimate (\ref{octe0303}) holds from (\ref{octe510}) and the above estimate (\ref{octe511}).

 Now we proceed to prove the desired estimate (\ref{may17eqn12}). Note that the following equality holds after separating out the correction term $\tilde{g}(t,v)$ and doing integration by parts in $v$, 
 \[
 \int_{\R^3} e^{i t|\xi|- i\mu t\hat{v}\cdot \xi}    m(\xi, v) \widehat{g}(t, \xi, v)\psi_{k}(\xi) d v= \int_{\R^3} e^{i t|\xi|- i\mu t\hat{v}\cdot \xi}    m(\xi, v)\big( \widehat{g}(t, 0, v)- \nabla_v\cdot \tilde{g}(t,v)\big)\psi_{k}(\xi) d v
 \]
 \be
- \int_{\R^3}  \nabla_v\big(e^{i t|\xi|- i\mu t\hat{v}\cdot \xi}    m(\xi, v)\big)\cdot \tilde{g}(t,v) \psi_{k}(\xi) d v + \int_0^1\int_{\R^3} e^{i t|\xi|- i\mu t\hat{v}\cdot \xi}    m(\xi, v) \xi\cdot \nabla_\xi \widehat{g}(t, s\xi, v)\psi_{k}(\xi) d v ds. 
 \ee
 From the above equality, it is straightforward to see that the desired estimate (\ref{may17eqn12}) holds directly from the estimate of symbol in (\ref{roughestimateofsymbolgenarl2}). 
\end{proof}

Before proceeding to our desired bilinear estimates, we need some preparation for a general bilinear form, which captures the interaction between the Vlasov equation and the wave equation.

 For any fixed $l\in\{0,1\}$ and  any given symbol  $m (\xi, v) $   that satisfies the following estimate 
\be\label{generalestimateofsymbol}
 \sup_{k\in \mathbb{Z}}\sum_{n=0,1,2,3} \sum_{|\alpha|\leq 5}  2^{l k+nk} \|(1+|v|)^{-20}\nabla_\xi^n \nabla_{v}^\alpha m (\xi, v)\psi_k(\xi)\|_{  L^\infty_v \mathcal{S}^\infty_k}  \lesssim 1, 
 \ee
 we define   a bilinear operator  $T_{\mu}  (\cdot, \cdot)(t, \xi)$   as follows, 
\be\label{octeqn1923}
T_{\mu}  (h, f)(t, \xi):= \int_{\R^3}\int_{\R^3} e^{it|\xi|-i\mu t |\xi-\eta| - i  t \hat{v}\cdot \eta} m (\xi,   v)\widehat{h^\mu}(t, \xi-\eta) \widehat{f}(t,\eta, v) d \eta d v. 
\ee

 Based on the relative size between the frequencies of two inputs, we separate $T_{\mu}(h, f)(t, \xi)$ into two parts as follows, 
 \be\label{octe361}
 T_{\mu}(h, f)(t, \xi)= T_{\mu}^1(h, f)(t, \xi) + T_{\mu}^2(h, f)(t, \xi),
 \ee
 where
 \be\label{octe190}
T_{\mu}^1(h, f)(t, \xi)= \int_{\R^3}\int_{\R^3} e^{it|\xi|-i\mu t |\xi-\eta| - i  t \hat{v}\cdot \eta} m(\xi, v)\widehat{h^\mu}(t, \xi-\eta) \widehat{f}(t,\eta, v) \psi_{\geq -10}(|\xi- \eta|/|\xi|) d \eta d v,
 \ee
 \be\label{octe191}
T_{\mu}^2(h, f)(t, \xi)= \int_{\R^3}\int_{\R^3} e^{it|\xi|-i\mu t |\eta| - i  t \hat{v}\cdot(\xi-\eta) } m(\xi, v)\widehat{h^\mu}(t,  \eta) \widehat{f}(t,\xi-\eta, v) \psi_{< -10}(| \eta|/|\xi|) d \eta d v.
 \ee
 Note that the following equality holds for any $n\in \{0,1,2,3\}$,
 \be\label{octe362}
\nabla_\xi^n \big(T_{\mu}^i(h, f)(t, \xi) \big)= \sum_{0\leq j \leq n} T_{\mu,j}^{n,i}(t, \xi),\quad i=1,2,
 \ee
 where
 \[
 T_{\mu,j}^{n,1}(t, \xi)= \sum_{0\leq a\leq j} \int_{\R^3}\int_{\R^3}  e^{it|\xi|-i\mu t |\xi-\eta| - i  t \hat{v}\cdot \eta}   c_{a,j}^{n,1} (i t)^j \big(\frac{\xi}{|\xi|} -\mu \frac{\xi-\eta}{|\xi-\eta|}\big)^{a}	\nabla_\xi^{n-j}\big[\big(\frac{\xi}{|\xi|} -\mu \frac{\xi-\eta}{|\xi-\eta|}\big)^{j-a} \]
 \be\label{octe200}
 \times  m(\xi, v) \widehat{h^\mu}(t, \xi-\eta) \psi_{\geq -10}(| \xi-\eta|/|\xi|)\big)\big]	 \widehat{f}(t,\eta, v)  d \eta d v,
 \ee
 \[
 T_{\mu,j}^{n,2 }(t, \xi)= \sum_{0\leq a\leq j} \int_{\R^3}\int_{\R^3}  e^{it|\xi|-i\mu t | \eta| - i  t \hat{v}\cdot(\xi- \eta)}   c_{a,j}^{n,2} (i t)^j \big(\frac{\xi}{|\xi|} -\hat{v}\big)^{a}	\nabla_\xi^{n-j}\big[\big(\frac{\xi}{|\xi|} -\hat{v}\big)^{j-a}  m(\xi, v)\]
 \be\label{octe201}
 \times 	 \widehat{f}(t,\xi-\eta, v)  \psi_{ < -10}(| \eta|/|\xi|)\big)\big] \widehat{h^\mu}(t, \eta) d \eta d v,
 \ee
 where $ c_{a,j}^{n,1}$ and $ c_{a,j}^{n,2}$ are some determined constants, whose explicit formulas are not pursued here. 

Our desired bilinear estimate is summarized in the following Lemma. 

\begin{lemma}\label{bilinearinXnormed}
For any $n \in\{0,1,2,3 \}$, any $l\in\{0,1\}$, and  any given symbol ``$m(\xi, v)$'' that satisfies the estimate \textup{(\ref{generalestimateofsymbol})}, the following estimate holds for the bilinear form  $T_{\mu} (h, f)(t, \xi) $ defined in \textup{(\ref{octeqn1923})}, 
\[
\sup_{k\in \mathbb{Z}} 2^{(n+1)k} \| \nabla_\xi^n \big(T_{\mu} (h, f)(t, \xi) \big)\psi_{k}(\xi)\|_{L^\infty_\xi}\lesssim \sum_{\beta\in\mathcal{S}, |\beta|\leq n+3}   (1+|t|)^{-3+l} \big( \sum_{0\leq c \leq n } \sum_{0\leq b \leq n-c}\sum_{|\alpha|\leq c}\| h^{\alpha}\|_{X_b} \big)
\]
\be\label{octe364}
 \times       \| (1+|x|^2+|v|^2)^{20} \Lambda^{\beta} {f}(t, x,v) \|_{L^2_x  L^2_v } .
\ee
\end{lemma}
\begin{proof}
Recall the decompositions (\ref{octe361}) and (\ref{octe362}). To control $X_n$-norm of $T_{\mu} (h, f)(t, \xi)$, it would be sufficient to control $L^\infty$-norm of  $T_{\mu,j}^{n,1}(t, \xi)$ and $T_{\mu,j}^{n,2}(t, \xi)$.

  Recall (\ref{octe200}). We first do integration by parts in`` $\eta$'' $j$-times for $T_{\mu,j}^{n,1 }(t, \xi)$. As a result, we have
\[
T_{\mu,j}^{n,1}(t, \xi)= \sum_{0\leq a\leq j} \int_{\R^3}\int_{\R^3}  e^{it|\xi|-i\mu t |\xi-\eta| - i  t \hat{v}\cdot \eta}   c_{a,j}^{n,1} \nabla_\eta\cdot\Big[\frac{ i\mu \frac{\xi-\eta}{|\xi-\eta|}- i  \hat{v}}{\big|\frac{(\xi-\eta)}{|\xi-\eta|}-\mu\hat{v}\big|^2}\circ  \nabla_\eta\cdot\big[  \frac{i\mu \frac{\xi-\eta}{|\xi-\eta|}- i \hat{v}}{\big|\frac{(\xi-\eta)}{|\xi-\eta|}-\mu\hat{v}\big|^2}\circ 
 \cdots  
\circ\big[  \big(\frac{\xi}{|\xi|} -\mu \frac{\xi-\eta}{|\xi-\eta|}\big)^{a} \] 
\[
\circ\big[  \big(\frac{\xi}{|\xi|} -\mu \frac{\xi-\eta}{|\xi-\eta|}\big)^{a}  \nabla_\xi^{n-j}\big[\big(\frac{\xi}{|\xi|} -\mu \frac{\xi-\eta}{|\xi-\eta|}\big)^{j-a}   m(\xi, v) \widehat{h^\mu}(t, \xi-\eta) \psi_{\geq -10}(| \xi-\eta|/|\xi|)\big)\big]\widehat{f}(t,\eta, v)\big] \cdots \big]\Big]	   d \eta d v,
\]
  Now, we apply the dyadic decomposition for ``$\xi-\eta$'' and ``$\eta$''. As a result, we have
 \be\label{octe198}
 T_{\mu,j}^{n,1 }(t, \xi)\psi_k(\xi) = \sum_{(k_1,k_2)\in \chi_k^1\cup \chi_k^2} H_{k,k_1,k_2}^{\mu, n,j}(t, \xi), 
 \ee
 where 
\[
H_{k,k_1,k_2}^{\mu, n,j}(t, \xi)= \sum_{0\leq a\leq j} \int_{\R^3}\int_{\R^3}  e^{it|\xi|-i\mu t |\xi-\eta| - i  t \hat{v}\cdot \eta}   c_{a,j}^{n,1} \nabla_\eta\cdot\Big[ \frac{ i\mu \frac{\xi-\eta}{|\xi-\eta|}- i  \hat{v}}{\big|\frac{(\xi-\eta)}{|\xi-\eta|}-\mu\hat{v}\big|^2}\circ  \nabla_\eta\cdot\big[  \frac{i\mu \frac{\xi-\eta}{|\xi-\eta|}- i \hat{v}}{\big|\frac{(\xi-\eta)}{|\xi-\eta|}-\mu\hat{v}\big|^2} \]
\[\circ 
 \cdots \big[  \big(\frac{\xi}{|\xi|} -\mu \frac{\xi-\eta}{|\xi-\eta|}\big)^{a}	\nabla_\xi^{n-j}\big[\big(\frac{\xi}{|\xi|} -\mu \frac{\xi-\eta}{|\xi-\eta|}\big)^{j-a} m(\xi, v) \widehat{h^\mu}(t, \xi-\eta)	 \]
\be\label{octe231}
  \times \psi_{\geq -10}(| \xi-\eta|/|\xi|)\big)\big]\widehat{f}(t,\eta, v)\big]\circ\cdots \big]\Big]\psi_{k}(\xi)\psi_{k_1}(\xi-\eta)\psi_{k_2}(\eta)	   d \eta d v,
\ee

Recall (\ref{octe201}), after doing dyadic decomposition for $\xi-\eta$ and $\eta$, we have
\be\label{may23eqn43}
 T_{\mu,j}^{n,2 }(t, \xi)=\sum_{(k_1,k_2)\in \chi_k^3} K_{k,k_1,k_2}^{\mu, n,j}(t, \xi)
\ee
where
\[
 K_{k,k_1,k_2}^{\mu, n,j}(t,\xi)= \sum_{0\leq a\leq j} \int_{\R^3}\int_{\R^3}  e^{it|\xi|-i\mu t | \eta| - i  t \hat{v}\cdot(\xi- \eta)}   c_{a,j}^{n,2} (i t)^j \big(\frac{\xi}{|\xi|} -\hat{v}\big)^{a}	\nabla_\xi^{n-j}\big[\big(\frac{\xi}{|\xi|} -\hat{v}\big)^{j-a}  m(\xi, v)\]
 \be\label{octe281}
 \times 	 \widehat{f}(t,\xi-\eta, v)  \psi_{ < -10}(| \eta|/|\xi|)\big)\big] \widehat{h^\mu}(t, \eta) \psi_{k_1}(\eta)\psi_k(\xi)\psi_{k_2}(\xi-\eta) d \eta d v.
\ee

From the estimate of symbol ``$m(\xi, v)$'' in \textup{(\ref{generalestimateofsymbol})} and the above detailed formulas of $H_{k,k_1,k_2}^{\mu, n,j}(t, \xi)$ and $K_{k,k_1,k_2}^{\mu, n,j}(t,\xi)$, we know that our desired estimate (\ref{octe364}) is trivial if $|t|\leq 1$. From now on, we restrict ourself to the case   $|t|\geq 1$.

\noindent $\bullet$\quad The estimate of  $H_{k,k_1,k_2}^{\mu, n,j}$.

 Recall (\ref{octe198}) and (\ref{octe231}). Note that we have $(k_1,k_2)\in \chi_k^1\cup \chi_k^2$ for the case we are considering.  After using the volume of support of $\eta$ and the estimate of symbol $m(\xi, v)$ in (\ref{generalestimateofsymbol}), we have 
\[
 2^{(n+1)k} \| H_{k,k_1,k_2}^{\mu, n,j}\psi_k(\xi) \|_{L^\infty_\xi}\lesssim \sum_{|\alpha|\leq 4}  \sum_{0\leq |c| \leq n} 2^{(n+1-l)k  +3k_2 -4k_{2,+}}\| (1+|v|)^{4 n +10} \nabla_\xi^c \widehat{\nabla_x^\alpha f}(t, \xi, v)\psi_{k_2}(\xi)\|_{L^\infty_\xi L^1_v} \]
 \[
 \times  \big(\sum_{0\leq b\leq n-c } 2^{-b k  }\|\nabla_\xi^{n-c-b} \widehat{h}(t, \xi)\psi_{k_1}(\xi)\|_{L^\infty_\xi} \big)
 \]
 \be\label{octe300}
 \lesssim \sum_{|\beta|\leq 4}2^{ -k_1 + (1-l)k+  3k_{2,-} } \big( \sum_{0\leq c \leq n } \sum_{0\leq b \leq n-c}\sum_{\alpha \in \mathcal{B}, |\alpha|\leq c}\| h^{\alpha}(t)\|_{X_b} \big)  \|  (1+|x|^2+|v|^2)^{20} \Lambda^\beta f(t,x,v) \|_{L^2_xL^2_v}.
\ee
 After doing integration by parts in $v$ four times for ``$H_{k,k_1,k_2}^{\mu, n,j}$'', the following estimate holds,
\[
 2^{(n+1)k} \| H_{k,k_1,k_2}^{\mu, n,j}\psi_k(\xi) \|_{L^\infty_\xi}\lesssim   |t|^{-4} 2^{(n+1-l) k   -k_2} 	\Big[\sum_{0\leq c \leq n}\sum_{0\leq d \leq 4 }\big(\sum_{0\leq b\leq n-c } 2^{-b k  }\|\nabla_\xi^{n-c-b} \widehat{h}(t, \xi)\psi_{k_1}(\xi)\|_{L^\infty_\xi} \big)
\]
\[
\times \| (1+|v|)^{4 n +10}\nabla_v^d \nabla_\xi^c \widehat{f}(t, \xi, v)\psi_{k_2}(\xi)\|_{L^\infty_\xi L^1_v} \Big]  \]
 \be\label{octe301}
 \lesssim  |t|^{-4} 2^{-k_1+(1-l)k -k_2 } \big( \sum_{0\leq c \leq n } \sum_{0\leq b \leq n-c}\sum_{|\alpha|\leq c}\| h^{\alpha}\|_{X_b} \big) \big( \sum_{\beta\in \mathcal{S},|\beta|\leq 4}\| (1+|x|^2+|v|^2)^{20} \Lambda^{\beta}f(t,x,v) \|_{L^2_xL^2_v}\big).
\ee
Therefore, from (\ref{octe300}) and (\ref{octe301}), the following estimate holds for any fixed $k\in \mathbb{Z},$
\[
 \sum_{(k_1,k_2)\in \chi_k^1\cup \chi_k^2}2^{(n+1)k} \| H_{k,k_1,k_2}^{\mu, n,j}\psi_k(\xi) \|_{L^\infty_\xi} \lesssim \big(\sum_{k_2, 2^{k_2}\leq  |t|^{-1}   } 2^{(3-l)k_2} + \sum_{k_2, 2^{k_2}\geq |t|^{-1}  }  |t|^{-4} 2^{  -(1+l)k_2}  \big)\]
 \[
 \times \big( \sum_{0\leq c \leq n } \sum_{0\leq b \leq n-c}\sum_{|\alpha|\leq c}\| h^{\alpha}\|_{X_b} \big)   \big( \sum_{\beta\in \mathcal{S},|\beta|\leq 4}\| (1+|x|^2+|v|^2)^{20} \Lambda^{\beta}f(t,x,v) \|_{L^2_xL^2_v}\big)
 \]
 \be\label{may15eqn2}
 \lesssim  |t|^{-3+l}  \big( \sum_{0\leq c \leq n } \sum_{0\leq b \leq n-c}\sum_{|\alpha|\leq c}\| h^{\alpha}\|_{X_b} \big)   \big( \sum_{\beta\in \mathcal{S},|\beta|\leq 4}\| (1+|x|^2+|v|^2)^{20} \Lambda^{\beta}f(t,x,v) \|_{L^2_xL^2_v}\big).
\ee

\noindent $\bullet$\quad The estimate of  $K_{k,k_1,k_2}^{\mu, n,j}$.

Recall (\ref{may23eqn43}) and (\ref{octe281}).  From the estimate of symbol ``$m(\xi, v)$'' in \textup{(\ref{generalestimateofsymbol})}, the following estimate holds   after using the volume of support of ``$\eta$'', 
\[
2^{(n+1)k}\| K_{k,k_1,k_2}^{\mu, n,j}(t,\xi)\|_{L^\infty_\xi} \lesssim  |t|^{j} 2^{(n+1-l) k  + 3k_1 }\big[\sum_{0\leq b\leq n-j} 2^{-b k_2}\|(1+|v|)^{4n +10}\nabla_\xi^{n-j-b}\widehat{f}(t, \xi, v)\psi_{k_2}(\xi)\|_{L^\infty_\xi L^1_v } 
\]
\be\label{octe291}
\times \|\widehat{h}(t, \xi)\psi_{k_1}(\xi)\|_{L^\infty_\xi}\big] \lesssim   \sum_{\beta\in\mathcal{S}, |\beta|\leq n-j}  |t|^{j}  2^{  (j+1-l) k +2k_1} \| h\|_{X_0} \| (1+|x|^2+|v|^2)^{20} \Lambda^{\beta}{f}(t, x,v) \|_{L^2_x  L^2_v }.
\ee

Moreover, after doing integration by parts in ``$v$'' ``$(j+3)$'' times for ``$K_{k,k_1,k_2}^{\mu, n,j}(t,\xi)$''  in (\ref{octe281}), the following estimate holds after using the volume of support of ``$\eta$'',  
\[
2^{(n+1)k}\| K_{k,k_1,k_2}^{\mu, n,j}(t,\xi)\|_{L^\infty_\xi} \lesssim  (1+|t|)^{-3} 2^{ (n-j-l-2)k+3k_1}
\]
\[
\times\big[\sum_{0\leq b\leq n-j} \sum_{  |\alpha|\leq j+3} 2^{-b k_2} \|\widehat{h}(t, \xi)\psi_{k_1}(\xi)\|_{L^\infty_\xi} \|(1+|v|)^{4n +10}\nabla_v^{\alpha}\nabla_\xi^{n-j-b}\widehat{f}(t, \xi,v)\psi_{k_2}(\xi)\|_{L^\infty_\xi L^1_v }\big]\]
\be\label{octe321}
\lesssim \sum_{\beta\in\mathcal{S}, |\beta|\leq n+3}  (1+|t|)^{-3} 2^{2k_1-(2+l)k}  \| h\|_{X_0} \|  (1+|x|^2+|v|^2)^{20} \Lambda^{\beta} {f}(t, x,v) \|_{L^2_x  L^2_v }.
\ee

Combining the estimates (\ref{octe291}) and (\ref{octe321}), we have
\[
\sum_{(k_1,k_2)\in \chi_k^3} 2^{(n+1)k}\| K_{k,k_1,k_2}^{\mu, n,j}(t,\xi)\|_{L^\infty_\xi} \]
\[
 \lesssim \sum_{\beta\in\mathcal{S}, |\beta|\leq n+3}    \| h\|_{X_0} \| (1+|x|^2+|v|^2)^{20} \Lambda^{\beta} {f}(t, x,v) \|_{L^2_x  L^1_v } \min\big\{ 
  (1+|t|)^{j } 2^{  (j+3-l) k  }, (1+|t|)^{-3}2^{ -lk} \big\} \]
 \be\label{may15eqn1}
    \lesssim \sum_{\beta\in\mathcal{S}, |\beta|\leq n+3}    (1+|t|)^{-3+l}  \| h\|_{X_0} \|  (1+|x|^2+|v|^2)^{20}  \Lambda^{\beta} 	{f}(t, x,v) \|_{L^2_x L^2_v }.
\ee
To sum up, 
  our desired estimate (\ref{octe364}) holds from  the estimates (\ref{may15eqn2}) and (\ref{may15eqn1}).
\end{proof}

 \subsection{Bilinear estimates in the high order energy space: Vlasov-Wave type interaction}\label{linearbilinear2}
 In this subsection, we   mainly prove several bilinear estimates for the Vlasov-Wave type interaction  in different function spaces. As we have seen in the proof of Proposition \ref{loworderenergyprop1} and Proposition \ref{highorderenergyprop1}, these bilinear estimates   play important roles in  the estimate of  the high order energy  $E_{\textup{high}}^{\phi}(t)$ (see (\ref{highorderenergyphi})) of the nonlinear wave part.

We first prove a $L^\infty_\xi$-type bilinear estimate. 

   \begin{lemma}\label{highorderbilinearlemma1}
Given   any    symbol ``$m(\xi, v)$'' that satisfies the estimate \textup{(\ref{generalestimateofsymbol})} with $l=1$, 
  the following estimate holds for the bilinear form  $T_{\mu} (h, f)(t, \xi) $ defined in \textup{(\ref{octeqn1923})}, 
\be\label{octe103}
 \sup_{k \in \mathbb{Z}} 2^{ k }\| T_{\mu} (h, f)(t, \xi)\psi_k(\xi)\|_{L^\infty_\xi} \lesssim \sum_{n=0,1,2, \alpha\in \mathcal{B},  |\alpha|\leq 4} (1+|t|)^{-2+\delta}  \|  {h^\alpha}(t )\|_{X_n}   \| (1+|x|^2+|v|^2)^{20} f(t,x,v)\|_{L^2_x L^2_v}. 
\ee
\end{lemma}
\begin{proof}
Recall (\ref{octeqn1923}). After doing dyadic decompositions for two inputs, we have
 \be\label{octe100}
T_{\mu} (h, f)(t, \xi)\psi_k(\xi)=  \sum_{i=1,2,3}\sum_{(k_1,k_2)\in \chi_k^i} T_{k,k_1,k_2}^{\mu }(t, \xi), 
\ee
 where $\chi_k^i, i\in\{1,2,3\}$, are defined in (\ref{highhighlow}) and (\ref{lowhigh}) and the detailed formula of $T_{k,k_1,k_2}^{\mu }(t, \xi)$ is given as follows, 
 \be\label{octeq41}
 T_{k,k_1,k_2}^{\mu }(t, \xi)= \int_{\R^3}\int_{\R^3} e^{it|\xi|-i\mu t |\xi-\eta| - i  t \hat{v}\cdot \eta} m(\xi, v)\widehat{h^\mu}(t, \xi-\eta) \widehat{f}(t,\eta, v)\psi_{k_1}(\xi-\eta)\psi_{k_2}(\eta)\psi_k(\xi) d \eta d v.
 \ee

 From the above detailed formula,  our desired estimates (\ref{octe103}) and (\ref{deceqn98}) hold easily if $|t|\leq 1$. Therefore, we restrict ourself to the case when $|t|\geq 1$. 
For this case, we do integration by parts in $\eta$ once for $ T_{k,k_1,k_2}^{\mu }(t, \xi)$ in (\ref{octeq41}). As a result, we have
 \[
 T_{k,k_1,k_2}^{\mu }(t, \xi)= \int_{\R^3}\int_{\R^3} e^{it|\xi|-i\mu t |\xi-\eta| - i  t \hat{v}\cdot \eta}   m(\xi, v) \psi_k(\xi)\frac{i}{t} \nabla_\eta\cdot\big(\frac{\mu (\xi-\eta)/|\xi-\eta|-  \hat{v}}{|\mu (\xi-\eta)/|\xi-\eta|-  \hat{v}|^2}	\]
 \be\label{deceqn3}
 \times  \widehat{h^\mu}(t, \xi-\eta)  \widehat{f}(t,\eta, v) \psi_{k_1}(\xi-\eta)\psi_{k_2}(\eta)\big) d \eta d v,
 \ee
  From the estimate of symbol ``$m(\xi, v)$'' in (\ref{generalestimateofsymbol}) and the volume of support of $\eta$,   we have 
\[
	2^{k }\| T_{k,k_1,k_2}^{\mu}(t, \xi)\|_{L^\infty_\xi} \lesssim  |t|^{-1} 2^{ 3\min\{k_1,k_2\}}  \big(2^{-\min\{k_1,k_2\}}\|\widehat{h}(t, \xi)\psi_{k_1}(\xi)\|_{L^\infty_\xi} + \| \nabla_\xi\widehat{h}(t, \xi)\psi_{k_1}(\xi)\|_{L^\infty_\xi} \big)\]
\[
\times \big( \| (1+|v|)^{10} \widehat{f}(t, \xi, v)\psi_{k_2}(\xi)\|_{L^\infty_\xi L^1_v} +  2^{\min\{k_1,k_2\}}\| (1+|v|)^{10} \nabla_\xi \widehat{f}(t, \xi, v)\psi_{k_2}(\xi)\|_{L^\infty_\xi L^1_v} \big) 
\]
\be\label{octe56}
\lesssim  \sum_{n=0,1,\alpha\in \mathcal{B}, |\alpha|\leq 4} |t|^{-1}2^{\min\{k_1,k_2\}-2k_{1,+}} \| h^\alpha(t)\|_{X_n}\| (1+|x|^2+|v|^2)^{20} f(t,x,v)\|_{L^2_x L^2_v}. 
\ee
Similarly, the following estimate holds after doing integration by parts in ``$\eta$'' twice, 
\be\label{may15eqn900}
	2^{k }\| T_{k,k_1,k_2}^{\mu}(t, \xi)\|_{L^\infty_\xi} \lesssim     \sum_{n=0,1,2,\alpha\in \mathcal{B}, |\alpha|\leq 4} |t|^{-2} \| h^\alpha(t)\|_{X_n}\| (1+|x|^2+|v|^2)^{20} f(t,x,v)\|_{L^2_x L^2_v}. 
\ee
Combining  the    estimates  (\ref{octe56}) and (\ref{may15eqn900}), the following estimate holds for any fixed $k\in \mathbb{Z}, $
\[
 \sum_{i=1,2,3,(k_1,k_2)\in \chi_k^i } 2^{k}\|T_{k,k_1,k_2}^{\mu}(t, \xi)\|_{L^\infty_\xi}\lesssim \sum_{n=0,1,2,k_1\in \mathbb{Z}, |\alpha|\leq 4}     |t|^{-2+\delta }    2^{\delta k_1 -2\delta k_{1,+}}  \| (1+|x|^2+|v|^2)^{20} f(t,x,v)\|_{L^2_x L^2_v}
\]
\be\label{octe59}
\times\| h^\alpha(t)\|_{X_n} \lesssim    \sum_{n=0,1,2, \alpha\in \mathcal{B}, |\alpha|\leq 4}   |t|^{-2+\delta }   \big\|  {h^\alpha}(t )\|_{X_n}\| (1+|x|^2+|v|^2)^{20} f(t,x,v)\|_{L^2_x L^2_v}.
\ee 
Hence finishing the proof of the desired estimate (\ref{octe103}).

\end{proof}
Now, we prove a $L^2_\xi$-type bilinear estimate. 
\begin{lemma}\label{L2bilinearvw}
Given   any    symbol ``$m(\xi, v)$'' that satisfies the estimate \textup{(\ref{generalestimateofsymbol})} with $l=1$, 
  the following estimate holds for the bilinear form  $T_{\mu} (h, f)(t, \xi) $ defined in \textup{(\ref{octeqn1923})}, 
\[ 
  \| T_{\mu} (h, f)(t, \xi) \|_{L^2_\xi}\lesssim   \min \big\{  \sum_{n=0,1, \alpha\in \mathcal{B},  |\alpha|\leq 4}  (1+|t|)^{-1}  \|  {h^\alpha}(t )\|_{X_n}    \| (1+|x|^2+|v|^2)^{20} f(t,x,v)\|_{L^2_x L^2_v}, 
\]
\be\label{deceqn98}
  \sum_{\beta\in \mathcal{S}, |\beta|\leq 3} (1+|t|)^{-2}  \|  h(t)\|_{L^2}   \| (1+|x|^2+|v|^2)^{20}\Lambda^\beta f(t,x,v) \|_{L^2_{x,v}} \big\}.
\ee
\end{lemma}
\begin{proof}
Note that the decomposition (\ref{octe100}) and the equality (\ref{octeq41}) still hold and   our desired estimate  (\ref{deceqn98})  hold straightforwardly if $|t|\leq 1$. Hence, we restrict ourself to the case when $|t|\geq 1$.   Since our desired estimate (\ref{deceqn98}) consists of two estimates, we separate naturally into two parts as follow. 

\noindent $\bullet$\quad  Proof of the first estimate in (\ref{deceqn98}). \qquad Based on the possible sizes of $k,k_1,k_2$, we separate into three sub-cases as follow.

\textbf{Subcase} $1$:\qquad If $k\leq 0$. \qquad Note that, from   the   estimate (\ref{octe56})  and the volume of support of $\xi$, we have 
\[
\sum_{k\in\mathbb{Z}, k\leq 0} \sum_{(k_1,k_2)\in \chi_k^1 \cup \chi_k^2 \cup \chi_k^3} \|T_{k,k_1,k_2}^{\mu}(t, \xi)\|_{L^2_\xi} \lesssim  \sum_{k\in\mathbb{Z}, k\leq 0} \sum_{n=0,1, \alpha\in \mathcal{B}, |\alpha|\leq 4}  \| (1+|x|^2+|v|^2)^{20} f(t,x,v)\|_{L^2_x L^2_v}  \]
\be\label{feb9eqn20}
\times  |t|^{-1}  2^{ k/2} \big\| {h^\alpha}(t )\|_{X_n} \lesssim  \sum_{n=0,1,\alpha\in \mathcal{B},	 |\alpha|\leq 4}  |t|^{-1}    \big\| {h^\alpha}(t )\|_{X_n}\| (1+|x|^2+|v|^2)^{20} f(t,x,v)\|_{L^2_x L^2_v}.
\ee

\textbf{Subcase} $2$:\qquad If $k\geq 0$ and      $(k_1,k_2)\in \chi_k^1\cup \chi_k^2$. 

Note that  we have $k_1\geq k-10$ and  $k_2\leq k+10$ for the subcase we are considering. Recall (\ref{deceqn3}).  From the estimate of symbol ``$m(\xi, v)$'' in  (\ref{generalestimateofsymbol}) and the volume of support of ``$\xi$'' and ``$\eta$'', we have 
\[
\| T_{k,k_1,k_2}^{\mu }(t, \xi) \|_{L^2} \lesssim \sum_{\alpha\in \mathcal{B}, |\alpha|\leq 4} |t|^{-1}  2^{ k/2+3k_2-4k_{1,+}}\big( 2^{- k_2} \| \widehat{h^\alpha}(t, \xi)\psi_{k_1}(\xi)\|_{L^\infty_\xi} + \|\nabla_\xi \widehat{h^\alpha}(t, \xi)\psi_{k_1}(\xi)\|_{L^\infty_\xi} \big) 
\]
\[
\times \big( \| (1+|v|)^{10} \widehat{f}(t, \xi, v)\psi_{k_2}(\xi)\|_{L^\infty_\xi L^1_v} + 2^{k_2} \| (1+|v|)^{10} \nabla_\xi \widehat{f}(t, \xi, v)\psi_{k_2}(\xi)\|_{L^\infty_\xi L^1_v} \big)
\]
\be\label{deceqn7500}
\lesssim \sum_{n=0,1, \alpha\in \mathcal{B}, |\alpha|\leq 4}  |t|^{-1}  2^{ k_2-3k_{1,+ }/2 } \| {h^\alpha}(t  ) \|_{X_n}\| (1+|x|^2+|v|^2)^{20} f(t,x,v)\|_{L^2_x L^2_v}. 
\ee
From the above estimate, we have 
\be\label{deceqn77}
\sum_{k\in\mathbb{Z}, k\geq 0} \sum_{(k_1,k_2)\in \chi_k^1\cup \chi_k^2} \| T_{k,k_1,k_2}^{\mu }(t, \xi) \|_{L^2}
\lesssim \sum_{n=0,1,\alpha\in \mathcal{B},|\alpha|\leq 4}  |t|^{-1}   \| {h^\alpha}(t  ) \|_{X_n} \|  (1+|x|^2+|v|^2)^{20} f(t,x,v)\|_{L^2_x L^2_v}.
\ee
\textbf{Subcase} $3$:\qquad If $k\geq 0$ and      $(k_1,k_2)\in \chi_k^3$. 

  For this case, we  first switch the
role of ``$\xi-\eta$'' and ``$\eta$'' in (\ref{deceqn3}). Instead of using the volume of support of ``$\xi$'', we use the Minkowski inequality.  As a result,  from the estimate of symbol ``$m(\xi, v)$'' in (\ref{generalestimateofsymbol}), we have
\[
  \sum_{k\in \mathbb{Z}, k\geq 0, (k_1,k_2)\in \chi_k^3}   \| T_{k,k_1,k_2}^{\mu }(t, \xi) \|_{L^2}  
 \lesssim  \sum_{k\in \mathbb{Z}, k\geq 0} \sum_{k_1\leq k-5, |k-k_2|\leq 5}\sum_{\alpha\in \mathcal{B}, |\alpha|\leq 4}  |t|^{-1} 2^{ 3k_1-k-4k_{1,+}} 
\]
 \[\times\big(\| (1+|v|)^{10} \widehat{f}(t, \xi, v)\psi_{k_2}(\xi) \|_{L^1_v L^2_\xi }  
+ 2^{k_1}\| (1+|v|)^{10} \nabla_\xi \widehat{f}(t, \xi, v) \psi_{k_2}(\xi)\|_{L^1_v L^2_\xi } \big) \big( 2^{-k_1}\|\widehat{h^\alpha}(t, \xi)\psi_{k_1}(\xi)\|_{L^\infty_\xi}
\]
\be\label{deceqn158}
 +\|\nabla_\xi \widehat{h^\alpha}(t, \xi)\psi_{k_1}(\xi)\|_{L^\infty_\xi}   \big) \lesssim \sum_{n=0,1} \sum_{\alpha\in \mathcal{B}, |\alpha|\leq 4}   |t|^{-1}  \| h^\alpha(t)\|_{X_n}   \|  (1+|x|^2+|v|^2)^{20} f(t,x,v)\|_{L^2_x L^2_v} . 
\ee

To sum up, recall the decomposition (\ref{octe100}), our desired first estimate in  (\ref{deceqn98}) holds from the estimates (\ref{feb9eqn20}), (\ref{deceqn77}) and (\ref{deceqn158}). 

 \noindent $\bullet$\quad  Proof of the second estimate in (\ref{deceqn98}). \qquad   Based on the size of $k_1$ and $k_2$, we separate into three subcases as follows.

 \textbf{Subcase} $1$:\qquad If  $(k_1,k_2)\in \chi_k^1 $.
 
 For this case, we have $|k_1-k_2|\leq 10$ and $k\leq k_1+10$. Recall (\ref{octeq41}).   After using the volume of ``$\xi$'', the estimate of symbol ``$m(\xi, v)$'' in (\ref{generalestimateofsymbol}),  the $L^2-L^2$ type bilinear estimate, and the volume of support of ``$\eta$'',  we have
\[
\| T_{k,k_1,k_2}^{\mu }(t, \xi) \|_{L^2} \lesssim   2^{k/2 } \|  \widehat{h}(t, \xi )\psi_{k_1}(\xi) \|_{L^2 } \| (1+|v|)^{25}\widehat{ f}(t, \xi, v)\psi_{k_2}(t,\xi)\|_{L^2_\xi L^1_v} 
\]
\be\label{feb10eqn1}
\lesssim 2^{k/2+3k_2/2} \|  \widehat{h}(t, \xi )\psi_{k_1}(\xi) \|_{L^2 } \|  (1+|v|)^{25}\widehat{ f}(t, \xi, v)\psi_{k_2}(t,\xi)\|_{L^\infty_\xi L^1_v}.
\ee

Moreover,  after first  doing  integration by parts in ``$v$'' three times and then using the volume of ``$\xi$'' and ``$\eta$'',  the following estimate holds from  the estimate of symbol ``$m(\xi, v)$'' in (\ref{generalestimateofsymbol}),  
\be\label{feb10eqn2}
\| T_{k,k_1,k_2}^{\mu }(t, \xi) \|_{L^2} \lesssim \sum_{ |\alpha|\leq 3} |t|^{-3} 2^{k/2 -3k_2/2} \|  \widehat{h}(t, \xi )\psi_{k_1}(\xi) \|_{L^2_\xi } \| (1+|v|)^{25} \nabla_v^\alpha \widehat{ f}(t, \xi, v)\psi_{k_2}(t,\xi)\|_{L^\infty_\xi L^1_v}.
\ee

From (\ref{feb10eqn1}) and (\ref{feb10eqn2}), we have 
\[
\sum_{k\in \mathbb{Z}} \sum_{(k_1,k_2)\in \chi_k^1} \| T_{k,k_1,k_2}^{\mu }(t, \xi) \|_{L^2}\]
\[
 \lesssim \sum_{ \beta\in \mathcal{S}, |\beta|\leq 3} \big[\big(\sum_{k_1, 2^{k_1}\leq  |t|^{-1}} \sum_{k\leq k_1+10} 2^{k/2+3k_1/2} \big) + \big(\sum_{k_1, 2^{k_1}\geq  |t|^{-1}} \sum_{k\leq k_1+10}  |t|^{-3}
  2^{k/2 -3k_1/2} \big) \big]\| h(t)\|_{L^2} \]
\be\label{feb10eqn3}
 \times \| (1+|x|^2+|v|^2)^{20}\Lambda^\beta f(t,x,v) \|_{L^2_{x,v}}\lesssim  \sum_{ \beta\in \mathcal{S}, |\beta|\leq 3}   |t|^{-2}\| h(t)\|_{L^2} \| (1+|x|^2+|v|^2)^{20}\Lambda^\beta f(t,x,v) \|_{L^2_{x,v}}. 
\ee

 \textbf{Subcase} $2$:\qquad If  $(k_1,k_2)\in \chi_k^2 $.

  For this case we have $|k-k_1|\leq 10, k_2\leq k_1-5$. Recall (\ref{octeq41}).  After using the volume of ``$\eta$'' and the estimate of symbol ``$m(\xi, v)$'' in (\ref{generalestimateofsymbol}), the following estimate holds,
\be\label{deceqn134}
\| T_{k,k_1,k_2}^{\mu }(t, \xi) \|_{L^2} \lesssim 2^{-k+3k_2} \|  \widehat{h}(t, \xi)\psi_{k_1}(\xi)\|_{L^2_\xi }  \|(1+|v|)^{25} \widehat{f}(t, \xi, v)\|_{L^\infty_\xi L^1_v}.
\ee

Moreover,  after first doing  integration by parts in ``$v$'' three times and then using the volume of support of ``$\eta$ '',  the following estimate holds from  the estimate of symbol $m(\xi, v)$ in (\ref{generalestimateofsymbol}), 
\be\label{feb10eqn13}
\| T_{k,k_1,k_2}^{\mu }(t, \xi) \|_{L^2} \lesssim \sum_{|\alpha|\leq 3}  2^{-k   } |t|^{-3}  \|  \widehat{h}(t, \xi)\psi_{k_1}(\xi)\|_{L^2_\xi } \|(1+|v|)^{25}\nabla_v^\alpha \widehat{f}(t, \xi, v)\|_{L^\infty_\xi L^1_v}.
\ee
Therefore, from the orthogonality in $L^2_\xi$ and the estimates (\ref{deceqn134}) and (\ref{feb10eqn13}) derived above, we have
\[
\| \sum_{k\in \mathbb{Z}} \sum_{(k_1,k_2)\in \chi_k^2}  T_{k,k_1,k_2}^{\mu }(t, \xi) \|_{L^2}^2 \lesssim \sum_{|k-k_1|\leq 10}\big( \sum_{k_2\leq k_1-5} \| T_{k,k_1,k_2}^{\mu }(t, \xi) \|_{L^2}  \big)^2 \lesssim \sum_{\beta\in \mathcal{S}, |\beta|\leq 3}\sum_{k_1\in\mathbb{Z}}\big[
\]
\[
 \big( \sum_{ 2^{k_2}\leq |t|^{-1}} 2^{2k_2} + \sum_{  2^{k_2}\geq  |t|^{-1}}  |t|^{-3}2^{ -k_2}\big)  \|  \widehat{h}(t, \xi)\psi_{k_1}(\xi)\|_{L^2_\xi }   \|(1+|x|^2+|v|^2)^{20} \Lambda^\beta f(t,x,v) \|_{L^2_{x,v}}\big]^2  \]
\[
\lesssim \sum_{\beta\in \mathcal{S}, |\beta|\leq 3} \sum_{k_1 \in \mathbb{Z} }   |t|^{-4} \|  \widehat{h}(t, \xi)\psi_{k_1}(\xi)\|_{L^2_\xi }^2 \|(1+|x|^2+|v|^2)^{20} \Lambda^\beta f(t,x,v) \|_{L^2_{x,v}}^2\]
\be\label{feb10eqn31}
\lesssim \sum_{\beta\in \mathcal{S}, |\beta|\leq 3}  |t|^{-4}\| h(t)\|_{L^2}^2 \|(1+|x|^2+|v|^2)^{20} \Lambda^\beta f(t,x,v) \|_{L^2_{x,v}}^2. 
\ee

 \textbf{Subcase} $3$:\qquad If  $(k_1,k_2)\in \chi_k^3 $.

  For this case we have $|k-k_2|\leq 10$ and $k_1\leq k- 5$. Recall (\ref{octeq41}). On one hand, after using the $L^2-L^\infty$ type bilinear estimate,  the $L^\infty\rightarrow L^2$ type Sobolev embedding, the estimate of symbol ``$m(\xi, v)$'' in (\ref{generalestimateofsymbol}) and the volume of support of the frequency variable, we have 
\[
 \| T_{k,k_1,k_2}^{\mu }(t, \xi) \|_{L^2} \lesssim 2^{-k+3k_1/2 } \|  \widehat{h}(t, \xi)\psi_{k_1}(\xi)\|_{L^2_\xi }  \|(1+|v|)^{25} \widehat{f}(t, \xi, v)\psi_{k_2}(\xi) 	\|_{L^2_\xi L^1_v}  
\]
\be\label{feb10eqn21}
\lesssim  2^{-k+3k_1/2+3k_2/2} \|  \widehat{h}(t, \xi)\psi_{k_1}(\xi)\|_{L^2_\xi }  \|(1+|v|)^{25} \widehat{f}(t, \xi, v)\|_{L^\infty_\xi L^1_v}. 
\ee

On the other hand, we first do integration by parts in ``$v$'' three times and then  use the $L^2-L^\infty$ type bilinear estimate,  the $L^\infty\rightarrow L^2$ type Sobolev embedding, the estimate of symbol ``$m(\xi, v)$'' in (\ref{generalestimateofsymbol}) and the volume of support of the frequency. As a result, we have 
\[
 \| T_{k,k_1,k_2}^{\mu }(t, \xi) \|_{L^2} \lesssim \sum_{|\alpha|\leq 3}  |t|^{-3} 2^{-k -3k_2+ 3k_1/2 } \|  \widehat{h}(t, \xi)\psi_{k_1}(\xi)\|_{L^2_\xi }  \| (1+|v|)^{25} \nabla_v^\alpha \widehat{f}(t, \xi, v)\psi_{k_2}(\xi) \|_{L^2_\xi L^1_v}  
\]
\be\label{feb10eqn24}
\lesssim   \sum_{|\alpha|\leq 3}  |t|^{-3}  2^{-k + 3k_1/2-3k_2/2} \|  \widehat{h}(t, \xi)\psi_{k_1}(\xi)\|_{L^2_\xi }  \| (1+|v|)^{2 5} \nabla_v^\alpha \widehat{f}(t, \xi, v)\|_{L^\infty_\xi L^1_v}. 
\ee
 Therefore, from the estimates (\ref{feb10eqn21}) and (\ref{feb10eqn24}), we have 
 \[
\sum_{k\in \mathbb{Z}} \sum_{(k_1,k_2)\in \chi_k^3}   \| T_{k,k_1,k_2}^{\mu }(t, \xi) \|_{L^2}\]
\[
 \lesssim  \sum_{\beta\in \mathcal{S}, |\beta|\leq 3}  \big(\sum_{k_2, 2^{k_2}\leq  |t|^{-1}}\sum_{k_1\leq k_2-5} 2^{3k_1/2+k_2/2} + \sum_{k_2, 2^{k_2}\geq   |t|^{-1}}\sum_{k_1\leq k_2-5}  |t|^{-3}  2^{ 3k_1/2-5k_2/2}  \big) \|h(t)\|_{L^2}  
 \]
 \be\label{feb10eqn33}
 	 \times  \|(1+|x|^2+|v|^2)^{20} \Lambda^\beta f(t,x,v) \|_{L^2_{x,v}} \lesssim \sum_{\beta\in \mathcal{S}, |\beta|\leq 3} |t|^{-2} \|h(t)\|_{L^2}  \|(1+|x|^2+|v|^2)^{20} \Lambda^\beta f(t,x,v) \|_{L^2_{x,v}} .
 \ee

 To sum up, recall the decomposition (\ref{octe100}),  our desired second estimate in   (\ref{deceqn98}) holds  from the estimates (\ref{feb10eqn3}), (\ref{feb10eqn31}) and (\ref{feb10eqn33}).  Hence finishing the proof. 

 \end{proof}
 Our last bilinear estimates concern the weighted $L^2$-estimate, which is also the last part of the high order energy defined in (\ref{highorderenergyphi}). 
 Recall (\ref{octe361}), (\ref{octe190}), and (\ref{octe191}). We have
\be\label{octeq262}
\nabla_\xi\big(T_{\mu}(h, f)(t, \xi) \big)= O^1(t,\xi)+O^2(t,\xi),
\ee
where
\[
O^1(t,\xi)=\int_{\R^3}\int_{\R^3} e^{it|\xi|-i\mu t |\xi-\eta| - i  t \hat{v}\cdot \eta} \nabla_\xi\big(m(\xi, v)\widehat{h^\mu}(t, \xi-\eta)\psi_{\geq -10}(|\xi-\eta|/|\xi|) \big)\widehat{f}(t,\eta, v)  
\]
\be\label{octeq1}
+   e^{it|\xi|-i\mu t | \eta| - i  t \hat{v}\cdot (\xi-\eta)}\nabla_\xi \big( m(\xi, v)\psi_{< -10}(|\eta|/|\xi|) \widehat{f}(t,\xi-\eta, v)\big) \widehat{h^\mu}(t, \eta) d \eta d v ,
\ee
\[
O^2(t,\xi)=\int_{\R^3}\int_{\R^3} e^{it|\xi|-i\mu t |\xi-\eta| - i  t \hat{v}\cdot \eta} it\big(\frac{\xi}{|\xi|} -\mu \frac{\xi-\eta}{|\xi-\eta|} \big)  m(\xi, v)\widehat{h^\mu}(t, \xi-\eta)  \widehat{f}(t,\eta, v) \psi_{\geq -10}(|\xi-\eta|/|\xi|)  
\]
\be\label{octeq2}
+ e^{it|\xi|-i\mu t | \eta| - i  t \hat{v}\cdot (\xi-\eta)} it\big(\frac{\xi}{|\xi|} - \hat{v} \big)  m(\xi, v)\widehat{h^\mu}(t,  \eta)  \widehat{f}(t,\xi-\eta, v)\psi_{< -10}(|\eta|/|\xi|) d \eta d v.
\ee

For the sake of clarity, we first summarize our desired two weighted $L^2$-type bilinear estimates in the  following Proposition. Its proof consists of four parts, which will be elaborated in the next four Lemmas separately. 

\begin{proposition}\label{firstorderweighthigh1}
Given   any    symbol ``$m(\xi, v)$'' that satisfies the estimate \textup{(\ref{generalestimateofsymbol})} with $l=1$ and any fixed $k\in \mathbb{Z}$, 
  the following estimate holds for the bilinear form  $T_{\mu} (h, f)(t, \xi) $ defined in \textup{(\ref{octeqn1923})}, 
  \[
    2^{k/2}\|\nabla_\xi\big(T_{\mu}(h, f)(t, \xi) \big)\psi_k(\xi)\|_{L^2}\lesssim   \sum_{0\leq n \leq 3}\sum_{ \alpha\in\mathcal{B}, |\alpha|\leq 4} \big( (1+|t|)^{-1}2^{k_{-}   } + (1+|t|)^{-2+\delta}\big)  \| {h^\alpha}(t   )\|_{X_n} 
  \]
\be\label{octe406}
\times \| (1+|x|^2+|v|^2)^{20}  f(t, x, v)\|_{L^2_x L^2_v}.
\ee
Moreover, the following bilinear estimate also holds,
\[
  2^{k/2}\|\nabla_\xi\big(T_{\mu}(h, f)(t, \xi) \big)\psi_k(\xi)\|_{L^2}\lesssim  \sum_{\beta\in \mathcal{S}, |\beta|\leq 4}  (1+|t|)^{-2}\big(\sup_{k\in \mathbb{Z}} 2^{k} \| \widehat{h}(t, \xi)\psi_k(\xi)\|_{L^\infty_\xi} + 2^{k/2}\|\nabla_\xi \widehat{h}(t, \xi)\psi_k(\xi)\|_{L^2}\big)
\]
\be\label{octe581}
\times \|  (1+|x|^2+|v|^2)^{20} \Lambda^\beta f(t,x,v)\|_{L^2_x L^2_v}.
\ee
\end{proposition}
\begin{proof}
Recall (\ref{octeq262}), (\ref{octeq1}), and (\ref{octeq2}). The desired estimates (\ref{octe406}) and (\ref{octe581}) follows directly from   estimate (\ref{octe140}) in Lemma \ref{basichighenergylemma1}, estimate (\ref{octe540}) in Lemma \ref{highorderbilinear1},     estimate (\ref{octe149}) in Lemma \ref{basichighenergylemma2}, and estimate (\ref{octe580}) in Lemma \ref{highorderbilinear2}.
\end{proof}

\begin{lemma}\label{basichighenergylemma1}
For ``$O^1(t, \xi)$'' defined in \textup{(\ref{octeq1})},  the following estimate holds for any fixed $k\in \mathbb{Z}$,
\[
  2^{k/2}\| O^1(t, \xi) \psi_k(\xi)\|_{L^2}
\]
\be\label{octe140}
 \lesssim  \sum_{0\leq n \leq 3}\sum_{ \alpha\in\mathcal{B},  |\alpha|\leq 4} \big( (1+|t|)^{-1}2^{k_{-}  } + (1+|t|)^{-2+\delta}\big) \| {h^\alpha}(t   )\|_{X_n}   \| (1+|x|^2+|v|^2)^{20} f(t, x, v)\|_{L^2_x L^2_v}.
\ee
\end{lemma}
\begin{proof}
 Firstly,  we do dyadic decomposition for two frequencies `` $\xi-\eta$'' and ``$\eta$''. As a result, we have
 
\be\label{octe120}
O^1(t,\xi)\psi_k(\xi)=  \sum_{i=1,2,3}\sum_{(k_1,k_2)\in \chi_k^i} O^1_{k,k_1,k_2}(t, \xi)  ,
 \ee
 where $\chi_k^i, i\in\{1,2,3\}$ is defined in (\ref{highhighlow}) and (\ref{lowhigh}),
 \[
O^1_{k,k_1, k_2}(t, \xi)= \int_{\R^3}\int_{\R^3} e^{it|\xi|-i\mu t |\xi-\eta| - i  t \hat{v}\cdot \eta} \nabla_\xi\big(m(\xi, v)\widehat{h^\mu}(t, \xi-\eta)\psi_{\geq -10}(|\xi-\eta|/|\xi|) \big)\]
\be\label{octeq42}
\times \widehat{f}(t,\eta, v) \psi_{k_1}(\xi-\eta)\psi_{k_2}(\eta)\psi_k(\xi)  d \eta d v,\quad \textup{when\,} (k_1,k_2)\in \chi_k^1\cup \chi_k^2,
 \ee
  \[
O^1_{k,k_1, k_2}(t, \xi)= \int_{\R^3}\int_{\R^3} e^{it|\xi|-i\mu t | \eta| - i  t \hat{v}\cdot (\xi-\eta)} \nabla_\xi\big(m(\xi, v) \widehat{f}(t,\xi-\eta, v)\psi_{<  -10}(|\eta|/|\xi|) \big)\]
\be\label{octeq47}
\times \widehat{h^\mu}(t,  \eta) \psi_{k_1}( \eta)\psi_{k_2}(\xi-\eta)\psi_k(\xi)  d \eta d v,\quad \textup{when\,} (k_1,k_2)\in \chi_k^3.
 \ee
From the above detailed formulas of $O_{k,k_1,k_2}^1(t,\xi)$,  we know that our   desired estimate (\ref{octe140}) holds straightforwardly if $|t|\leq 1$. Therefore, from now on, we restrict ourself to the case when $|t|\geq 1$. For any fixed $k\in \mathbb{Z}$, we separate into two cases as follows. 

  $\bullet $\, \textbf{Case} 1:\quad If $(k_1,k_2)\in \chi_k^1\cup \chi_k^2$.\qquad For this case we have $k_1\geq k-10, k_2\leq k_1+10$. Recall  (\ref{octeq42}) and the estimate of symbol $m(\xi,v)$ in (\ref{generalestimateofsymbol}).  After  doing integration  by parts in ``$\eta$'' once, we have the following estimate, 
\[
 2^{ k/2 }\| O_{k,k_1,k_2}^{1}(t, \xi)\|_{L^2}
\lesssim 2^{2k} \| O_{k,k_1,k_2}^{1}(t, \xi)\|_{L^\infty_\xi}  \lesssim  (1+|t|)^{-1} 2^{ k+3k_2}\big[ \| (1+|v|)^{25} \widehat{f}(t, \xi, v)\psi_{k_2}(\xi)\|_{L^\infty_\xi L^1_v}
 \]
 \[
\times \big(
 2^{-k- k_2}\|\widehat{h}(t, \xi)\psi_{k_1}(\xi)\|_{L^\infty_\xi}+   2^{-\min\{k,k_2\} } \|\nabla_\xi \widehat{h}(t,\xi)\psi_{k_1}(\xi)\|_{L^\infty_\xi} +   \|\nabla_\xi^2 \widehat{h}(t,\xi)\psi_{k_1}(\xi)\|_{L^\infty_\xi}\big)    \]
 \[
 + \big( 2^{-\min\{k, k_2\} }\|\widehat{h}(t, \xi)\psi_{k_1}(\xi)\|_{L^\infty_\xi} + \|\nabla_\xi \widehat{h}(t, \xi)\psi_{k_1}(\xi)\|_{L^\infty_\xi}  \big)\| (1+|v|)^{25}\nabla_\xi \widehat{f}(t, \xi, v)\psi_{k_2}(\xi)\|_{L^\infty_\xi L^1_v} \big]\]
 \be\label{octe76}
\lesssim  \sum_{n=0,1,2}\sum_{\alpha \in \mathcal{B}, |\alpha|\leq 4}  (1+|t|)^{-1} 2^{   2k_2-k_1-2k_{1,+}}  \| {h^\alpha}(t)\|_{X_n} \|(1+|x|^2+|v|^2)^{20}  f(t,x,v)\|_{L^2_x L^2_v}.
\ee
Similarly, after doing integration by parts in ``$\eta$'' twice, the following estimate holds, 
\be\label{may15eqn22}
 2^{ k/2 }\| O_{k,k_1,k_2}^{1}(t, \xi)\|_{L^2}\lesssim \sum_{n=0,1,2,3}\sum_{\alpha \in \mathcal{B}, |\alpha|\leq 4}  (1+|t|)^{-2} 2^{k_2-k_1} \| {h^\alpha}(t)\|_{X_n} \|(1+|x|^2+|v|^2)^{20}  f(t,x,v)\|_{L^2_x L^2_v}. 
\ee
To sum up, after interpolating the estimates (\ref{octe76}) and (\ref{may15eqn22}), we have 
\be\label{octe137}
\sum_{(k_1, k_2)\in \chi_k^1\cup \chi_k^2}  2^{ k/2 }\| O_{k,k_1,k_2}^{1}(t, \xi)\|_{L^2} \lesssim  \sum_{\begin{subarray}{c}
n=0,1,2,3\\ 
\alpha \in \mathcal{B}, |\alpha|\leq 4\\ 
\end{subarray}}   (1+|t|)^{-2+\delta}   \| {h^\alpha}(t)\|_{X_n} \|(1+|x|^2+|v|^2)^{20}  f(t,x,v)\|_{L^2_x L^2_v}.
\ee
  $\bullet$\, \textbf{Case} 2: \quad If $(k_1,k_2)\in \chi_k^1\cup \chi_k^2$.\qquad   Recall (\ref{octeq47}). After doing integration by parts in ``$\eta$'' once, we have
\[
O^1_{k,k_1, k_2}(t, \xi)= \int_{\R^3}\int_{\R^3} e^{it|\xi|-i\mu t | \eta| - i  t \hat{v}\cdot (\xi-\eta)}\frac{1}{t}\nabla_\eta \cdot \Big[\frac{ -i\mu \eta/|\eta|+i \hat{v}}{|\mu \eta/|\eta|-\hat{v}|^2}   \nabla_\xi\big(m(\xi, v)\widehat{f}(t,\xi-\eta, v)\psi_{<  -10}(|\eta|/|\xi|) \big) \]
\[
\times  \widehat{h^\mu}(t,  \eta) \psi_{k_1}( \eta)\psi_{k_2}(\xi-\eta)\Big]\psi_k(\xi)  d \eta d v.
\]
If the derivative $\nabla_\xi$ doesn't hit on  $\widehat{f}(t, \xi-\eta, v)$ in the above integral, we use the volume of support of $\xi$ and $\eta$. Alternatively, if the derivative  $\nabla_\xi$ does hit on  $\widehat{f}(t, \xi-\eta, v)$, we first use the $L^2-L^\infty$ type bilinear estimate with  $\widehat{f}(t, \xi-\eta, v)$ put in $L^2$ and then use the volume of support of $\eta$. As a result, the following estimate holds, 
\[
2^{k/2}\|O^1_{k,k_1, k_2}(t, \xi)\|_{L^2}\lesssim  (1+|t|)^{-1} 2^{ -k/2+3k_1}\Big[\big( 2^{ -k_1}\|\widehat{h}(t,\xi)\psi_{k_1}(\xi)\|_{L^\infty_\xi} +   \| \nabla_\xi\widehat{h}(t, \xi)\psi_{k_1}(\eta)\|_{L^\infty_\xi}  \big) \]
\[
\times\big(2^{k/2} \| (1+|v|)^{25} \widehat{f}(t, \xi, v)\psi_{k_2}(\xi )\|_{ L^\infty_\xi L^1_v} +   \| (1+|v|)^{25} \nabla_\xi\widehat{f}(t, \xi, v)\psi_{k_2}(\xi )\|_{L^1_v L^2_\xi}
\]
\be\label{octe132}
+  2^{k_1} \| (1+|v|)^{25} \nabla_\xi^2 \widehat{f}(t, \xi, v)\psi_{k_2}(\xi )\|_{L^1_v L^2_\xi}\big)   \Big]. 
\ee
 From the above estimate (\ref{octe132}), we have 
\[
\sum_{(k_1,k_2)\in \chi_k^3}  2^{k/2}\|O^1_{k,k_1, k_2}(t, \xi)\|_{L^2} \lesssim \sum_{k_1\leq k} \sum_{n=0,1, \alpha\in \mathcal{B}, |\alpha|\leq 4 }  (1+|t|)^{-1} 	2^{ k_1 -2k_{1,+}}  \|   (1+|x|^2+|v|^2)^{20} f(t, x, v)\|_{L^2_x L^2_v}
\]
\be\label{octe138}
\times    \| {h^\alpha}(t)\|_{X_n}  
  \lesssim   \sum_{n=0,1, \alpha\in \mathcal{B}, |\alpha|\leq 4 }  (1+|t|)^{-1}2^{k_{-} } \| {h^\alpha}(t)\|_{X_n}    \|  (1+|x|^2+|v|^2)^{20}f(t, x, v)\|_{L^2_x L^2_v}.
\ee
Combining estimates (\ref{octe137}) and (\ref{octe138}),  our desired estimate (\ref{octe140}) holds.
\end{proof}
\begin{lemma}\label{highorderbilinear1}
For ``$O^1(t, \xi)$'' defined in \textup{(\ref{octeq1})},   the following estimate holds for any fixed $k\in \mathbb{Z}$,
\[
  2^{k/2}\| O^1(t, \xi) \psi_k(\xi)\|_{L^2}\lesssim\sum_{\beta\in\mathcal{S}, |\beta|\leq 3}   (1+|t|)^{-2}\big(\sup_{k\in \mathbb{Z}}2^{k}\|\widehat{h}(t, \xi)\psi_{k}(\xi)\|_{L^\infty_\xi} + 2^{k/2}\|\nabla_\xi \widehat{h}(t, \xi)\psi_k(\xi)\|_{L^2_\xi}  \big)
\] 
\be\label{octe540}
\times \| (1+|x|^2+|v|^2)^{20} \Lambda^\beta f(t, x, v)\|_{L^2_x L^2_v}.
\ee
\end{lemma}
\begin{proof}
Recall (\ref{octe120}), (\ref{octeq42}), and (\ref{octeq47}). Since the desired estimate (\ref{octe540}) is trivial when $|t|\leq 1$, we restrict ourself to the case when $|t|\geq 1$.  We separate into two cases as follows.

   $\bullet $\, \textbf{Case} 1:\quad    If $(k_1,k_2)\in \chi_k^1\cup \chi_k^3$.\qquad   For this case, we have $k_1\leq k_2+10$. Recall the estimate of symbol ``$m(\xi,v)$'' in (\ref{generalestimateofsymbol}).  On the one hand, after using the volume of support of $\xi$ and the Cauchy-Schwarz inequality for the integration with respect to`` $\eta$'', the following estimate holds, 
 \[
2^{ k/2 }\| O_{k,k_1,k_2}^{1}(t, \xi)\|_{L^2_\xi}
\lesssim 2^{k + 3k_1/2}\|(1+|v|)^{25}\widehat{f}(t, \xi, v)\psi_{k_2}(\xi)\|_{L^\infty_\xi L^1_v}\big(2^{-k} \|\widehat{h}(t, \xi)\psi_{k_1}(\xi)\|_{L^2_\xi}  +\|\nabla_\xi\widehat{h}(t, \xi)\psi_{k_1}(\xi)\|_{L^2_\xi}  \big)
 \]
\be\label{octe610}
 \lesssim 2^{k_1+\max\{k,k_1\}}\big( \sup_{k\in \mathbb{Z}} 2^{k}\|\widehat{h}(t, \xi)\psi_k(\xi)\|_{L^\infty_\xi} + 2^{k/2}\|\nabla_\xi \widehat{h}(t, \xi)\psi_k(\xi)\|_{L^2} \big) \|  (1+|x|^2+|v|^2)^{20} f(t,x, v )\|_{L^2_x L^2_v}.
\ee
On the other hand, if we do integration by parts in ``$v$'' three times first and then  use the volume of support of ``$\xi$''   and   $L^2-L^2$ type bilinear estimate,  then  we have, 
\[
2^{ k/2 }\| O_{k,k_1,k_2}^{1}(t, \xi)\|_{L^2_\xi}
\lesssim \sum_{|\beta|\leq 3}  |t|^{-3} 2^{  k-3k_2 +3k_1/2}\big(2^{-k} \|\widehat{h}(t, \xi)\psi_{k_1}(\xi)\|_{L^2_\xi}  +\|\nabla_\xi\widehat{h}(t, \xi)\psi_{k_1}(\xi)\|_{L^2_\xi}  \big)
\]
\[
\times \|(1+|v|)^{25} \nabla_v^\beta \widehat{f}(t, \xi, v)\psi_{k_2}(\xi)\|_{L^\infty_\xi L^1_v}\lesssim  \sum_{\beta\in \mathcal{S}, |\beta|\leq 3}  |t|^{-3}  2^{ -3k_2+k_1+\max\{k,k_1\}}\big( \sup_{k\in \mathbb{Z}} 2^{k}\|\widehat{h}(t, \xi)\psi_k(\xi)\|_{L^\infty_\xi} 
\]
\be\label{octe611}
  + 2^{k/2}\|\nabla_\xi \widehat{h}(t, \xi)\psi_k(\xi)\|_{L^2} \big) \|  (1+|x|^2+|v|^2)^{20} \Lambda^\beta  f(t,x, v)\|_{L^2_x L^2_v}.
\ee
 From (\ref{octe610}) and (\ref{octe611}), we have 
 \[
\sup_{(k_1,k_2)\in \chi_k^1\cup \chi_k^3} 2^{ k/2 }\| O_{k,k_1,k_2}^{1}(t, \xi)\|_{L^2_\xi}\]
\[
 \lesssim \sum_{\beta\in \mathcal{S}, |\beta|\leq 3}   \big(\sum_{k,2^{k}\leq  |t|^{-1}} 2^{2k} + \sum_{k,2^{k}\geq |t|^{-1}} 2^{ -k}  |t|^{-3}\big) \big( \sup_{k\in \mathbb{Z}} 2^{k}\|\widehat{h}(t, \xi)\psi_k(\xi)\|_{L^\infty_\xi} + 2^{k/2}\|\nabla_\xi \widehat{h}(t, \xi)\psi_k(\xi)\|_{L^2} \big)
 \]
\[
 \times \|  (1+|x|^2+|v|^2)^{20} \Lambda^\beta  f(t,x, v)\|_{L^2_x L^2_v}  \lesssim \sum_{\beta\in \mathcal{S}, |\beta|\leq 3}   |t|^{-2} \big( \sup_{k\in \mathbb{Z}} 2^{k}\|\widehat{h}(t, \xi)\psi_k(\xi)\|_{L^\infty_\xi} \]
\be\label{octe630}
+ 2^{k/2}\|\nabla_\xi \widehat{h}(t, \xi)\psi_k(\xi)\|_{L^2} \big)\|  (1+|x|^2+|v|^2)^{20}  \Lambda^\beta  f(t,x,  v)\|_{L^2_x L^2_v}.
\ee
 
  $\bullet $\, \textbf{Case} 2:\quad   If $(k_1,k_2)\in \chi_k^2$. \qquad  For this case, we have $k_2\leq k_1-10$ and $|k-k_1|\leq 10.$ 
Note that the following estimate holds after using the $L^2-L^\infty$ type bilinear estimate and  the volume of support of $\eta$, 
\[
2^{ k/2 }\| O_{k,k_1,k_2}^{1}(t, \xi)\|_{L^2_\xi}\]
\[
 \lesssim   2^{-k/2 +3k_2}\big(2^{-k}\|\widehat{h}(t, \xi)\psi_{k_1}(\xi)\|_{L^2} +\|\nabla_\xi\widehat{h}(t, \xi)\psi_{k_1}(\xi)\|_{L^2}   \big)\| (1+|v|)^{25}\widehat{f}(t, \xi, v)\psi_{k_2}(\xi)\|_{L^\infty_\xi L^1_v}  
\]
\be\label{octe621}
   \lesssim 2^{3k_2-k}\big(\sup_{k\in \mathbb{Z}} 2^{k}  \|\widehat{h}(t, \xi)\psi_{k_1}(\xi)\|_{L^2}+ 2^{k/2}\|\nabla_\xi\widehat{h}(t, \xi)\psi_{k_1}(\xi)\|_{L^2}  \big)\| (1+|x|^2+|v|^2)^{20}  f(t,x,  v)\|_{L^2_x L^2_v}.
\ee

Moreover, if we do integration by parts in ``$v$'' three times first and then using the $L^2-L^\infty$ type bilinear estimate and the   volume of support of the frequency of the input put in $L^\infty_x$, we have  
\[
2^{ k/2 }\| O_{k,k_1,k_2}^{1}(t, \xi)\|_{L^2_\xi}
\lesssim \sum_{|\beta|\leq 3} |t|^{-3} 2^{ - k/2 }\big(2^{-k} \|\widehat{h}(t, \xi)\psi_{k_1}(\xi)\|_{L^2_\xi}  +\|\nabla_\xi\widehat{h}(t, \xi)\psi_{k_1}(\xi)\|_{L^2_\xi}  \big)
\]
\[
\times \|(1+|v|)^{25} \nabla_v^\beta \widehat{f}(t, \xi, v)\psi_{k_2}(\xi)\|_{L^\infty_\xi L^1_v}\lesssim  \sum_{\beta\in \mathcal{S}, |\beta|\leq 3}   (1+|t|)^{-3}  2^{ -k}\big( \sup_{k\in \mathbb{Z}} 2^{k}\|\widehat{h}(t, \xi)\psi_k(\xi)\|_{L^\infty_\xi} 
\]
\be\label{octe622}
  + 2^{k/2}\|\nabla_\xi \widehat{h}(t, \xi)\psi_k(\xi)\|_{L^2} \big) \|  (1+|x|^2+|v|^2)^{20} \Lambda^\beta f(t,x, v )\|_{L^2_x L^2_v}.
\ee
From (\ref{octe621}) and (\ref{octe622}), we have
 \[
\sup_{(k_1,k_2)\in \chi_k^2} 2^{ k/2 }\| O_{k,k_1,k_2}^{1}(t, \xi)\|_{L^2_\xi} \lesssim \sum_{\beta\in \mathcal{S}, |\beta|\leq 3}   \big(\sum_{k,2^{k}\leq  |t|^{-1}} 2^{2k} + \sum_{k,2^{k}\geq  |t|^{-1}} 2^{ -k}  |t|^{-3}\big)  \]
\[  \times \big( \sup_{k\in \mathbb{Z}} 2^{k}\|\widehat{h}(t, \xi)\psi_k(\xi)\|_{L^\infty_\xi} 
  + 2^{k/2}\|\nabla_\xi \widehat{h}(t, \xi)\psi_k(\xi)\|_{L^2} \big)\|  (1+|x|^2+|v|^2)^{20} \Lambda^\beta f(t,x, v)\|_{L^2_x L^2_v}\]
\be\label{octe632}
  \lesssim \sum_{\beta\in \mathcal{S}, |\beta|\leq 3}   |t|^{-2} \big( \sup_{k\in \mathbb{Z}} 2^{k}\|\widehat{h}(t, \xi)\psi_k(\xi)\|_{L^\infty_\xi}+ 2^{k/2}\|\nabla_\xi \widehat{h}(t, \xi)\psi_k(\xi)\|_{L^2} \big) \|  (1+|x|^2+|v|^2)^{20} \Lambda^\beta f(t,x,   v)\|_{L^2_x L^2_v} .
\ee

To sum up, our desired estimate (\ref{octe540}) holds  
 from (\ref{octe630}) and (\ref{octe632}).
 \end{proof}
 \begin{lemma}\label{basichighenergylemma2}
For ``$O^2(t, \xi)$'' defined in \textup{(\ref{octeq2})},    the following estimate holds for any fixed $k\in \mathbb{Z}$,
\[
  2^{k/2}\| O^2(t, \xi) \psi_k(\xi)\|_{L^2}\lesssim  \sum_{0\leq n \leq 3}\sum_{ \alpha\in\mathcal{B},  |\alpha|\leq 3} \big( (1+|t|)^{-1}2^{k_{-}  } + (1+|t|)^{-2+\delta}\big) 
\]
\be\label{octe149}
 \times \| {h^\alpha}(t   )\|_{X_n}   \| (1+|x|^2+|v|^2)^{20} f(t, x, v)\|_{L^2_x L^2_v}.
 \ee
\end{lemma}
\begin{proof}
Recall (\ref{octeq2}). After doing dyadic decomposition for ``$\xi-\eta$'' and ``$\eta$'', we have
\[
O^2(t,\xi)\psi_k(\xi)= \sum_{i=1,2,3}\sum_{(k_1,k_2)\in \chi_k^i} O_{k,k_1,k_2}^2(t,\xi),
\]
where
 \[
 O_{k,k_1,k_2}^2(t,\xi)= \int_{\R^3}\int_{\R^3} e^{it|\xi|-i\mu t |\xi-\eta| - i  t \hat{v}\cdot \eta} it\big(\frac{\xi}{|\xi|} -\mu \frac{\xi-\eta}{|\xi-\eta|} \big)  m(\xi, v)\psi_{\geq -10}(|\xi-\eta|/|\xi|)  \widehat{h^\mu}(t, \xi-\eta) \]
 \be\label{octeq100}
 \times   \widehat{f}(t,\eta, v) \psi_k(\xi)\psi_{k_1}(\xi-\eta)\psi_{k_2}(\eta) d \eta d v, \quad \textit{if}\quad (k_1, k_2)\in \chi_k^1\cup \chi_k^2, 
 \ee
 \[
 O_{k,k_1,k_2}^2(t,\xi)= \int_{\R^3}\int_{\R^3}   e^{it|\xi|-i\mu t | \eta| - i  t \hat{v}\cdot (\xi-\eta)} it\big(\frac{\xi}{|\xi|} - \hat{v} \big)  m(\xi, v) \psi_{< -10}(|\eta|/|\xi|)	 \widehat{h^\mu}(t,  \eta) \]
 \be\label{feb21eqn21}
 \times  \widehat{f}(t,\xi-\eta, v)   \psi_k(\xi)\psi_{k_1}( \eta)\psi_{k_2}(\xi-\eta) d \eta d v, \quad \textit{if}\quad (k_1, k_2)\in \chi_k^3.  
 \ee
 From the above detailed formulas of ``$O^2_{k,k_1,k_2}$'',  our desired estimate (\ref{octe149}) holds straightforwardly when $|t|\leq 1$. Hence, from now on,  we restrict ourself to the case when $|t|\geq 1$.
 
 $\bullet $\, \textbf{Case} 1:\quad  If $(k_1,k_2)\in \chi_k^1\cup \chi_k^2$. \qquad For this case, we have $k_1\geq k-10$ and $k_2\leq k_1+10$.  Recall (\ref{octeq100}).   For this case, we do integration by parts in ``$\eta$'' twice for $O^2_{k,k_1,k_2}$. As a result, we have
\[
 O_{k,k_1,k_2}^2(t,\xi)= \int_{\R^3}\int_{\R^3} e^{it|\xi|-i\mu t | \xi-\eta| - i  t \hat{v}\cdot \eta} \frac{i}{t} \nabla_\eta \cdot \Big[ \frac{i \mu (\xi-\eta)/|\xi-\eta|-i  \hat{v}}{\big|\mu(\xi-\eta)/|\xi-\eta|- \hat{v}\big|^2} \nabla_\eta \cdot \big[ \frac{i \mu (\xi-\eta)/|\xi-\eta|-i  \hat{v}}{\big|\mu(\xi-\eta)/|\xi-\eta|- \hat{v}\big|^2}    \]
 \[
 \times  \big(\frac{\xi}{|\xi|} -\mu \frac{\xi- \eta}{| \xi- \eta|} \big) m(\xi, v)\psi_{\geq -10}(|\xi-\eta|/|\xi|) \widehat{h^\mu}(t,\xi-  \eta)  \widehat{f}(t,\eta, v) \psi_k(\xi)\psi_{k_1}(\xi-\eta)\psi_{k_2}(\eta) \big] \Big]d \eta d v.
\]
After using the volume of support of ``$\xi$'' and ``$\eta$'', the following estimate holds, 
\[
2^{k/2} \|  O_{k,k_1,k_2}^2(t,\xi)\|_{L^2_\xi} \lesssim  |t|^{-1} 2^{ k+3k_2}\big[\big(2^{-2k_2}\|\widehat{h}(t, \xi)\psi_{k_1}(\xi)\|_{L^\infty_\xi} + 2^{-k_2}\|\nabla_\xi\widehat{h}(t, \xi)\psi_{k_1}(\xi)\|_{L^\infty_\xi}\]
\[ +  \|\nabla_\xi^2\widehat{h}(t, \xi)\psi_{k_1}(\xi)\|_{L^\infty_\xi}  \big) \| (1+|v|)^{25} \widehat{f}(t, \xi, v)\psi_{k_2}(\xi)\|_{L^\infty_\xi L^1_v} + \big(2^{- k_2}\|\widehat{h}(t, \xi)\psi_{k_1}(\xi)\|_{L^\infty_\xi} +  \|\nabla_\xi\widehat{h}(t, \xi)\psi_{k_1}(\xi)\|_{L^\infty_\xi}\big) \]
\[ 
\times \| (1+|v|)^{25} \nabla_\xi\widehat{f}(t, \xi, v)\psi_{k_2}(\xi)\|_{L^\infty_\xi L^1_v} + \|\widehat{h}(t, \xi)\psi_{k_1}(\xi)\|_{L^\infty_\xi}  \| (1+|v|)^{25} \nabla_\xi^2\widehat{f}(t, \xi, v)\psi_{k_2}(\xi)\|_{L^\infty_\xi L^1_v}\big]
\]
\be\label{may15eqn41}
\lesssim |t|^{-1} 2^{ k_2-2k_{2,+}}\big(\sum_{n=0,1,2, |\alpha|\leq 4} \|  {h^\alpha}(t)\|_{X_n} \big)\|  (1+|x|^2+|v|^2)^{20} f(t, x, v)\|_{L^2_x L^2_v}.
\ee
Moreover, after doing integration by parts in ``$\eta$'' three times first and then  using the volume of support of $\xi$ and $\eta$, the following estimate holds, 
\be\label{may15eqn42}
 2^{k/2} \|  O_{k,k_1,k_2}^2(t,\xi)\|_{L^2_\xi} \lesssim \sum_{n=0,1,2,3, |\alpha|\leq 4}  |t|^{-2} \|  {h^\alpha}(t)\|_{X_n} \|  (1+|x|^2+|v|^2)^{20} f(t, x, v)\|_{L^2_x L^2_v}.
\ee
After interpolating the estimates (\ref{may15eqn41}) and (\ref{may15eqn42}),  we have
\be\label{octe402}
\sum_{(k_1, k_2)\in \chi_k^1\cup \chi_k^2} 2^{k/2} \|  O_{k,k_1,k_2}^2(t,\xi)\|_{L^2_\xi}\lesssim  \sum_{0\leq n\leq 3,  |\alpha|\leq 4}  |t|^{-2+\delta} \|  {h^\alpha}(t)\|_{X_n}  \|  (1+|x|^2+|v|^2)^{20} f(t, x, v)\|_{L^2_x L^2_v}.
 \ee 

  $\bullet $\, \textbf{Case} 2:\quad    If $(k_1,k_2)\in \chi_k^3$.\qquad  Recall (\ref{feb21eqn21}). For this case we have $k_1\leq k_2-10, |k-k_2|\leq 10$.  On one hand, the following estimate holds after using the volume of support of $\eta$ and the Minkowski inequality, 
\[
2^{k/2} \|  O_{k,k_1,k_2}^2(t,\xi)\|_{L^2_\xi} \lesssim 2^{ -k/2 + 3k_1}(1+|t|)^{ } \|\widehat{h}(t, \xi)\psi_{k_1}(\xi)\|_{L^\infty_\xi} \| (1+|v|)^{25} \widehat{f}(t, \xi, v)\psi_{k_2}(\xi) \|_{L^1_v L^2_\xi}
\]
\be\label{octe399}
\lesssim  (1+|t|) 2^{  2k_1} \min\{2^{-k/2}, 2^{k}\}  \| h(t)\|_{X_0}\| (1+|x|^2+|v|^2)^{20} f(t, x,v)\|_{L^2_x L^2_v}.
\ee
Alternatively,  we can do integration by parts in ``$\eta$'' three times first. As a result, the following estimate holds after using the volume of support of $\eta$ and the Minkowski inequality, 
\[
2^{k/2} \|  O_{k,k_1,k_2}^2(t,\xi)\|_{L^2_\xi}\lesssim  \sum_{0\leq |\beta|\leq 3} \sum_{0\leq |\alpha|\leq 3-|\beta|}     |t|^{-2} 2^{-k/2+3k_1-(3-|\beta|-|\alpha|) k_1} \| (1+|v|)^{25}\nabla_\xi^{\beta}\widehat{f}(t, \xi, v)\psi_{k_2}(\xi)\|_{L^1_v L^2_\xi}  \]
\[   
 \times     \|\nabla_\xi^{\alpha}\widehat{h}(t, \xi)\psi_{k_1}(\xi)\|_{L^\infty_\xi} \lesssim   \sum_{0\leq n\leq 3, \alpha\in \mathcal{B},  |\alpha|,|\beta|\leq 3}|t|^{-2} 2^{ -k/2- k_1}     \| (1+|v|)^{25}\nabla_\xi^{\beta}\widehat{f}(t, \xi, v)\psi_{k_2}(\xi)\|_{L^1_v L^2_\xi}\|  h^\alpha(t) \|_{X_n} 
\]
\be\label{octe400}
  \lesssim \sum_{0\leq n\leq 3, \alpha\in \mathcal{B},  |\alpha|\leq 3} |t|^{-2} 2^{ -k_1} \min\{2^{-k/2}, 2^{k}\} \| h^{\alpha}(t)\|_{X_i}   \| (1+|x|^2+|v|^2)^{20} f(t, x,v)\|_{L^2_x L^2_v}.
\ee
Therefore, from (\ref{octe399}) and (\ref{octe400}), we have
\[
\sum_{(k_1,k_2)\in \chi_k^3} 2^{k/2} \|  O_{k,k_1,k_2}^2(t,\xi)\|_{L^2_\xi} \lesssim  \sum_{0\leq n\leq 3}\sum_{|\alpha|\leq 3}  \big(\sum_{ 2^{k_1}\leq  |t|^{-1}}  |t|  2^{ 2k_1} + \sum_{  2^{k_1}\geq  |t|^{-1}}  |t|^{-2}2^{ -k_1} \big) \min\{2^{-k/2}, 2^{k}\} \| h^{\alpha}(t)\|_{X_n} \]
\be\label{octe403}
 \times  \| (1+|x|^2+|v|^2)^{20}f(t, x,v)\|_{L^2_x L^2_v} \lesssim  \sum_{0\leq n\leq 3, |\alpha|\leq 3}    |t|^{-1}2^{k_{-}} \| h^{\alpha}(t)\|_{X_n}   \| (1+|x|^2+|v|^2)^{20} f(t, x,v)\|_{L^2_x L^2_v}.
\ee

To sum up, our desired estimate (\ref{octe149}) holds after combining the  estimates (\ref{octe402}) and (\ref{octe403}).

\end{proof}

\begin{lemma}\label{highorderbilinear2}
For ``$O^2(t, \xi)$'' defined in \textup{(\ref{octeq2})},    the following estimate holds for any $k\in\mathbb{Z},$ 
\be\label{octe580} 
    2^{k/2}\| O^2(t, \xi) \psi_k(\xi)\|_{L^2}\lesssim\sum_{\beta\in\mathcal{S}, |\beta|\leq 4}    \sup_{k\in \mathbb{Z}}(1+|t|)^{-2} 2^{k}\|\widehat{h}(t, \xi)\psi_{k}(\xi)\|_{L^\infty_\xi} \| (1+|x|^2+|v|^2)^{20} \Lambda^\beta f(t, x, v)\|_{L^2_x L^2_v}.
\ee
\end{lemma}
 \begin{proof}
 Recall (\ref{octeq100}) and (\ref{feb21eqn21}).   From the estimate of symbol in (\ref{generalestimateofsymbol}),   the following estimate holds after using the volume of support of ``$\xi$'' and `` $\eta$'', 
\be\label{octe591}
 2^{k/2}\|O_{k,k_1,k_2}^2(t,\xi)\|_{L^2_\xi}\lesssim 2^{ k+3\min\{k_1,k_2\}}(1+|t|)  \|\widehat{h}(t, \xi)\psi_{k_1}(\xi)\|_{L^\infty_\xi} \|(1+|v|)^{25}\widehat{f}(t, \xi, v)\psi_{k_2}(\xi)\|_{L^\infty_\xi L^1_v}.
\ee

Moreover, after doing integration by parts in ``$v$'' four times first and then using the volume of support of ``$\xi$'' and `` $\eta$'', we have  
\[
 2^{k/2}\|O_{k,k_1,k_2}^2(t,\xi)\|_{L^2_\xi}
\]
\be\label{octe592}
\lesssim \sum_{ |\beta|\leq 4} (1+|t|)^{-3} 2^{ k+3\min\{k_1,k_2\}-4k_2}\|\widehat{h}(t, \xi)\psi_{k_1}(\xi)\|_{L^\infty_\xi}\|(1+|v|)^{25}\nabla_v^\beta \widehat{f}(t, \xi, v)\psi_{k_2}(\xi)\|_{L^\infty_\xi L^1_v}.
 \ee
Therefore, from the     estimates (\ref{octe591}) and  (\ref{octe592}), the following estimate holds for any $k\in \mathbb{Z} $, 
\[
 \sum_{(k_1,k_2)\in \chi_k^1\cup \chi_k^2\cup \chi_k^3}2^{k/2}\|O_{k,k_1,k_2}^2(t,\xi)\|_{L^2_\xi}\lesssim \sum_{   |\beta|\leq 4} \big(\sum_{k, 2^{k}\leq (1+|t|)^{-1} } (1+|t|) 2^{ 3k} +\sum_{k, 2^{k}\geq (1+|t|)^{-1}} (1+|t|)^{-3} 2^{ -k} \big)
\]
\[
\times \big(\sup_{k\in \mathbb{Z}}  2^{k}\|\widehat{h}(t, \xi)\psi_k(\xi)\|_{L^\infty_\xi} \big)\|(1+|v|)^{25}\nabla_v^\beta \widehat{f}(t, \xi, v)\psi_{k_2}(\xi)\|_{L^\infty_\xi L^1_v}
\]
\be
\lesssim \sum_{\beta\in \mathcal{S}, |\beta|\leq 4} (1+|t|)^{-2}\big(\sup_{k\in \mathbb{Z}}  2^{k}\|\widehat{h}(t, \xi)\psi_k(\xi)\|_{L^\infty_\xi} \big)\| (1+|x|^2+|v|^2)^{20} \Lambda^\beta {f}(t,x, v)\|_{L^2_x L^2_v}.
\ee
Hence finishing the proof of the desired estimate (\ref{octe580}).
 \end{proof}
\subsection{Bilinear estimates in the high order energy space: Vlasov-Vlasov type interaction}
In this subsection, we prove a bilinear estimate in the weighted $L^2$-type space for the Vlasov-Vlasov type interaction, which, more precisely, is the interaction between two density type functions. This  bilinear estimate has been  used in the high order energy estimate of the scalar field in the proof of Proposition \ref{highorderenergyprop1}. More precisely,  the estimate of the second integral in (\ref{may15eqn500}), which can be viewed as a linear combination of bilinear forms defined in (\ref{may15eqn60}).

 \begin{lemma}\label{twoaveragedistri}
For any symbols $m_1(\xi, v), m_2(\xi,v)$ that satisfy \textup{(\ref{generalestimateofsymbol})} with ``$l=1$'' and any two localized distribution functions $f, g:\R_{t}\times \R_x^3\times \R_v^3\longrightarrow \mathbb{R}$, we define a   bilinear operator  as follows, 
\be\label{may15eqn60}
 K^{\mu}  (g, f)(t, \xi):= \int_{\R^3}\int_{\R^3}\int_{\R^3} e^{it|\xi|-i\mu t \hat{u}\cdot (\xi-\eta) - i  t \hat{v}\cdot \eta} m_1(\xi,   v)m_2(\xi-\eta,   u)\widehat{g}(t, \xi-\eta, u) \widehat{f}(t,\eta, v) d \eta d u d v  .
\ee
Then the following bilinear estimate holds for any fixed $k\in \mathbb{Z}, $  
\[
2^{k/2}\|\nabla_\xi\big( K^{\mu}  (g, f)(t, \xi) \big)\psi_k(\xi) \|_{L^2_\xi }
\]
\be\label{may15eqn410}
 \lesssim \sum_{\beta\in \mathcal{S}, |\beta|\leq 5} (1+|t|)^{-2}\|(1+|x|^2+|v|^2)^{20}g(t,x,v)\|_{L^2_x L^2_v}\|(1+|x|^2+|v|^2)^{20} \Lambda^{\beta}f(t,x,v)\|_{L^2_x L^2_v}.
\ee
 \end{lemma}
\begin{proof}
Similar to the decomposition we did in (\ref{octeq262}), 
we first separate ``$\nabla_\xi( K^{\mu}  (g, f)(t, \xi))$'' into two parts and then   do dyadic decompositions for ``$\xi-\eta$'' and ``$\eta$''. As a result, the following decompositions hold, 
\be\label{may15eqn321}
\nabla_\xi( K^{\mu}  (g, f)(t, \xi))\psi_k(\xi) =\sum_{i=1,2} K_k^{i}  (g, f)(t, \xi), \,\, K_k^{i}  (g, f)(t, \xi)=\sum_{(k_1, k_2)\in \chi_k^l, l=1,2,3} K^{k;i}_{k_1,k_2}  (g, f)(t, \xi),
\ee
where 
\[
K^{k;1}_{k_1,k_2}  (g, f)(t, \xi) =\int_{\R^3}\int_{\R^3}\int_{\R^3}  e^{it|\xi|-i\mu t \hat{u}\cdot (\xi-\eta) - i  t \hat{v}\cdot \eta} \nabla_\xi\big(m_1(\xi, v)m_2(\xi-\eta, u)\widehat{g}(t, \xi-\eta, u)\psi_{\geq -10}(|\xi-\eta|/|\xi|)\big)\]
\be\label{may15eqn110}
\times  \widehat{f}(t,\eta, v) \psi_k(\xi)\psi_{k_1}(\xi-\eta)\psi_{k_2}(\eta)  d \eta d u d v, \quad \textit{if\,\,} (k_1, k_2)\in \chi_k^1\cup \chi_k^2, 
\ee
\[
K^{k;1}_{k_1,k_2}  (g, f)(t, \xi) =\int_{\R^3}\int_{\R^3} \int_{\R^3} e^{it|\xi|-i\mu t \hat{u}\cdot  \eta  - i  t \hat{v}\cdot(\xi-\eta)} \nabla_\xi\big(m_1(\xi, v)\widehat{f}(t, \xi-\eta, v) \psi_{ <-10}(| \eta|/|\xi|)\big)\]
\be\label{may15eqn111}
\times \widehat{g}(t,  \eta, u)   m_2( \eta, u) \psi_k(\xi)\psi_{k_1}( \eta)\psi_{k_2}(\xi-\eta)  d \eta d u d v, \quad \textit{if\,\,} (k_1, k_2)\in \chi_k^3, 
\ee
\[
K^{k;2}_{k_1,k_2}  (g, f)(t, \xi) = \int_{\R^3}\int_{\R^3}\int_{\R^3}  e^{it|\xi|-i\mu t \hat{u}\cdot (\xi-\eta) - i  t \hat{v}\cdot \eta} it\big(\frac{\xi}{|\xi|} -\mu  \hat{u} \big)  m_1(\xi, v)m_2(\xi-\eta, u) \widehat{g}(t, \xi-\eta, u)\]
\be\label{may15eqn112}
\times   \widehat{f}(t,\eta, v)  \psi_k(\xi)\psi_{k_1}(\xi-\eta)\psi_{k_2}(\eta)  d \eta d u d v, \quad   \textit{if\,\,} (k_1, k_2)\in \chi_k^1\cup \chi_k^2, 
\ee
\[
K^{k;2}_{k_1,k_2}  (g, f)(t, \xi) = \int_{\R^3}\int_{\R^3}\int_{\R^3}  e^{it|\xi|-i\mu t \hat{u}\cdot  \eta  - i  t \hat{v}\cdot(\xi-\eta)}  it\big(\frac{\xi}{|\xi|} -   \hat{v} \big)  m_1(\xi, v)m_2(  \eta, u) \widehat{g}(t,  \eta, u)\]
\be\label{may15eqn113}
\times   \widehat{f}(t,\xi-\eta, v)  \psi_k(\xi)\psi_{k_1}( \eta)\psi_{k_2}(\xi-\eta)  d \eta d u d v, \quad   \textit{if\,\,} (k_1, k_2)\in \chi_k^3.  
\ee
From the above detailed formulas,  our  desired estimate (\ref{may15eqn410}) holds straightforwardly if $|t|\leq 1$. Hence, we restrict ourself to the case when $|t|\geq 1$. 

Based on the size of $k_1$, we separate into two cases as follows. 

  $\bullet $\, \textbf{Case} 1:\quad  If $(k_1,k_2)\in \chi_k^1\cup\chi_k^2$. \qquad  For this case we have $k_1\geq k-10$ and $k_2\leq k_1+10$.

On one hand, from the size of support of ``$\xi$'' and ``$\eta$'' and the estimate of symbols ``$m_i(\xi, v)$'', $i\in\{1,2\}$, in (\ref{generalestimateofsymbol}), we have  
\[
2^{k/2}\|K^{k;1}_{k_1,k_2}  (g, f)(t, \xi)\|_{L^2_\xi}\lesssim 2^{2k-k-k_1-\min\{k,k_1\}+3k_2}\big(\|(1+|u|)^{25} \widehat{g}(t, \xi, u)\psi_{k_1}(\xi)\|_{L^\infty_\xi L^1_u}\]
\[
 + 2^{k_1} \| (1+|u|)^{25} \nabla_\xi \widehat{g}(t, \xi, u)\psi_{k_1}(\xi)\|_{L^\infty_\xi L^1_u} \big)  \|(1+|v|)^{25} \widehat{f}(t, \xi, v)\psi_{k_2}(\xi)\|_{L^\infty_\xi L^1_v} 
\]
\be\label{may15eqn71}
  \lesssim 2^{2k_2}\|(1+|x|^2+|v|^2)^{20} g(t,x,v)\|_{L^2_x L^2_v}\|(1+|x|^2+|v|^2)^{20} f(t,x,v)\|_{L^2_x L^2_v}. 
\ee
Following the similar strategy, recall (\ref{may15eqn112}),   the following estimate holds for $K^{k;2}_{k_1,k_2}  (g, f)(t, \xi)$, 
\be\label{may15eqn181}
2^{k/2}\|K^{k;2}_{k_1,k_2}  (g, f)(t, \xi)\|_{L^2_\xi}\lesssim  |t| 2^{ 3k_2 } \| (1+|x|^2+|v|^2)^{20} g(t,x,v)\|_{L^2_x L^2_v}\| (1+|x|^2+|v|^2)^{20}f(t,x,v)\|_{L^2_x L^2_v}. 
\ee

On the other hand,   we first do integration by parts in $v$ three times  for $K_{k_1,k_2}^{k;1}$ and do integration by parts in $v$ four times   for $K_{k_1,k_2}^{k;2}$ and then use the volume of support of $\xi $ and $\eta$. As a result,  we have 
\[
2^{k/2}\|K^{k;1}_{k_1,k_2}  (g, f)(t, \xi)\|_{L^2_\xi} +2^{k/2}\|K^{k;2}_{k_1,k_2}  (g, f)(t, \xi)\|_{L^2_\xi} \]
 
\be\label{may15eqn101} 
 \lesssim \sum_{\beta\in \mathcal{S}, |\beta|\leq 5} (1+|t|)^{-3} 2^{-k_2} \| (1+|x|^2+|v|^2)^{20} g(t,x,v)\|_{L^2_x L^2_v} \|(1+|x|^2+|v|^2)^{20}  \Lambda^\beta f(t,x,v)\|_{L^2_x L^2_v}.
\ee
To sum up, from the estimates (\ref{may15eqn71}), (\ref{may15eqn181}), and  (\ref{may15eqn101}),   the following estimate holds, 
\[
\sum_{(k_1,k_2)\in \chi_k^1\cup \chi_k^2}\sum_{i=1,2} 2^{k/2}\|K^{k;i}_{k_1,k_2}  (g, f)(t, \xi)\|_{L^2_\xi}\lesssim \sum_{\beta\in \mathcal{S}, |\beta|\leq 5}\big(\sum_{2^{k_2}\leq |t|^{-1}}(2^{2k_2}+|t|2^{3k_2}) + \sum_{2^{k_2}\geq |t|^{-1}} |t|^{-3}2^{ -k_2}  \big)
\]
\[
\times \|(1+|x|^2+|v|^2)^{20}g(t,x,v)\|_{L^2_x L^2_v}\| (1+|x|^2+|v|^2)^{20} \Lambda^{\beta}f(t,x,v)\|_{L^2_x L^2_v}
\]
\be\label{may15eqn341}
\lesssim \sum_{\beta\in \mathcal{S}, |\beta|\leq 5} |t|^{-2}  \|(1+|x|^2+|v|^2)^{20}  g(t,x,v)\|_{L^2_x L^2_v}\| (1+|x|^2+|v|^2)^{20}  \Lambda^{\beta}f(t,x,v)\|_{L^2_x L^2_v}. 
\ee

 $\bullet $\, \textbf{Case} 2:\quad If $(k_1,k_2)\in \chi_k^3$.\qquad  For this case, we have $k_1\leq k-10$ and $|k_2-k|\leq 10$.

On one hand, similar to the strategy used in obtaining 	 the estimates (\ref{may15eqn71}) and (\ref{may15eqn181}), from the size of support of $\xi$ and $\eta$ and the estimate of symbols $m_i(\xi, v)$ in (\ref{generalestimateofsymbol}),  we have
\[
\sum_{i=1,2}2^{k/2}\|K^{k;i}_{k_1,k_2}  (g, f)(t, \xi)\|_{L^2_\xi}\lesssim 2^{2k- k-k_1 +3k_1} \big(1+|t| 2^{k_1})\big(2^{-k_1}\| (1+|v|)^{25} \widehat{f}(t, \xi, v)\psi_{k_1}(\xi)\|_{L^\infty_\xi L^1_v} 
\]
\[
+ 2^{k } \| (1+|v|)^{25} \nabla_\xi \widehat{f}(t, \xi, v)\psi_{k_1}(\xi)\|_{L^\infty_\xi L^1_v} \big) \|(1+|u|)^{25} \widehat{g}(t, \xi, u)\psi_{k_2}(\xi)\|_{L^\infty_\xi L^1_u}\]
\be\label{may15eqn300}
 \lesssim \big(2^{k_1+k_2}+|t| 2^{2k_1+k_2}\big) \| (1+|x|^2+|v|^2)^{20}  g(t,x,v)\|_{L^2_x L^2_v}\| (1+|x|^2+|v|^2)^{20}  f(t,x,v)\|_{L^2_x L^2_v}.  
\ee

On the other hand, similar to the strategy used in the  estimate  (\ref{may15eqn101}),  we do integration by parts in $v$ three times for   $K_{k_1,k_2}^{k;1}$  in (\ref{may15eqn111}) and do integration by parts in $v$ four times  for $K_{k_1,k_2}^{k;2}$  in (\ref{may15eqn113}). As a result, the following estimate holds after using the volume of support of $\xi$ and $\eta$,  
\[
\sum_{i=1,2}2^{k/2}\|K^{k;i}_{k_1,k_2}  (g, f)(t, \xi)\|_{L^2_\xi}\lesssim \sum_{\beta\in \mathcal{S}, |\beta|\leq 5}  |t|^{-3} 2^{2k-k-2k_1+3k_1-3k_2}\| (1+|x|^2+|v|^2)^{20}  g(t,x,v)\|_{L^2_x L^2_v}
\]
\be\label{may15eqn301}
\times \| (1+|x|^2+|v|^2)^{20} \Lambda^{\beta}f(t,x,v)\|_{L^2_x L^2_v}.
\ee
From the estimates (\ref{may15eqn300}) and (\ref{may15eqn301}), the following estimate holds, 
\[
\sum_{(k_1, k_2)\in \chi_k^3}\sum_{i=1,2} 2^{k/2}\|K^{k;i}_{k_1,k_2}  (g, f)(t, \xi)\|_{L^2_\xi}\lesssim \sum_{\beta\in \mathcal{S}, |\beta|\leq 5}\big(\sum_{2^{k_2}\leq |t|^{-1}}\sum_{k_2\leq k_1+10}\big(2^{k_1+k_2}+|t| 2^{2k_1+k_2}\big) \]
\[
+  \sum_{2^{k_2}\geq |t|^{-1}}\sum_{k_2\leq k_1+10}|t|^{-3}2^{k_1-2k_2}   \big)\| (1+|x|^2+|v|^2)^{20} g(t,x,v)\|_{L^2_x L^2_v}\|(1+|x|^2+|v|^2)^{20} \Lambda^{\beta}f(t,x,v)\|_{L^2_x L^2_v}
\]
\be\label{may15eqn342}
\lesssim \sum_{\beta\in \mathcal{S}, |\beta|\leq 5} |t|^{-2}\| (1+|x|^2+|v|^2)^{20} g(t,x,v)\|_{L^2_x L^2_v}\| (1+|x|^2+|v|^2)^{20} \Lambda^{\beta}f(t,x,v)\|_{L^2_x L^2_v} .
\ee

To sum up, our desired estimate  (\ref{may15eqn410}) holds from the decomposition (\ref{may15eqn321}) and the estimates (\ref{may15eqn341}) and (\ref{may15eqn342}). 
\end{proof}
 \section{Energy estimate for the Vlasov part}\label{energyestimatevlasov}

 In this section, we estimate both the low order energy $E_{\textup{low}}^f(t)$ defined in (\ref{octeqn1896}) and the high order energy $E_{\textup{high}}^f(t)$ defined in  (\ref{highorderenergy1})  of the profile $g(t,x,v)$ for the Vlasov part.

    The main ingredients are some general linear estimates and bilinear estimates, which will be used as black boxes  first in subsection \ref{highordervlasovestimate} and subsection \ref{lowordervlasovsec}. We will prove these estimates in    subsection \ref{multilinearestimate}.

\subsection{The high order energy estimate for the Vlasov part}\label{highordervlasovestimate}

 Recall  (\ref{highorderenergy1}) and (\ref{jan29eqn21}).   As the result of direct computation,  the following equality holds  for any fixed $t\in[0,T]$, $\alpha\in \mathcal{B}, \beta \in \mathcal{S}$,  s.t.,  $|\alpha| +|\beta|\leq N_0$,  
\[
 \h \| \omega_{\beta}^{\alpha}( x, v)  g^\alpha_\beta (t ,x,v)\|_{L^2_{x,v}}^2- \h \| \omega_{\beta}^{\alpha}( x, v)  g^\alpha_\beta (0,x,v)\|_{L^2_{x,v}}^2\]
 \be\label{eqq50}
 =  \textup{Re}\Big[\int_{0}^{t} \int_{\R^3} \int_{\R^3}\big(\omega_{\beta}^{\alpha}( x, v)\big)^2  g^\alpha_\beta (t  ,x,v) \p_t g^\alpha_\beta(t,x,v) d x d v \Big]=  \sum_{i=1,2,3  }  \textup{Re}[I_{\beta;i}^{\alpha}], 
\ee
where
\be\label{sepeqn809}
 I_{\beta;1}^{\alpha} = \int_{0}^{t}  \int_{\R^3} \int_{\R^3}\big(\omega_{\beta}^{\alpha}( x, v)\big)^2  g^\alpha_\beta (s ,x,v) K(s,x+\hat{v}s,v)\cdot  {D}_v g^\alpha_{\beta}(s,x,v)  d x d v ds ,
\ee
\be\label{sepeqn813}
 I_{\beta;2}^{\alpha} =  \int_{0}^{t} \int_{\R^3} \int_{\R^3}\big(\omega_{\beta}^{\alpha}( x, v)\big)^2  g^\alpha_\beta (s ,x,v)  \textit{l.o.t}_{\beta}^\alpha(s,x,v) d x  dv d s,  
 \ee
  \be\label{sepeqn810}
 I_{\beta;3}^{\alpha} =  \int_{0}^{t}  \int_{\R^3} \int_{\R^3}\big(\omega_{\beta}^{\alpha}( x, v)\big)^2  g^\alpha_\beta (s ,x,v) \textit{h.o.t}_{\beta}^\alpha(s,x,v)   d x  dv ds . 
 \ee

The main result of  this subsection is  summarized in the  following Proposition. For the sake of readers and for clarity, we  give a concise proof of Proposition \ref{notbulkterm1} by separate out two independent Lemmas first and then prove the validities of these two Lemmas. 
\begin{proposition}\label{notbulkterm1}
 Under the bootstrap assumption \textup{(\ref{bootstrapassumption})},  the following estimate holds for any $t\in[0,T]$, 
\be\label{sepeqn110}
E_{\textup{high}}^{f;1}(t) \lesssim (1+t)^{\delta}\epsilon_0,\quad E_{\textup{high}}^{f;2}(t) \lesssim (1+t)^{\delta/2}\epsilon_0 
\ee
\end{proposition}
\begin{proof}
Recall the definition of the high order energy in  (\ref{highorderenergy1}) and the decomposition in (\ref{eqq50}).

We first estimate $I_{\beta;2}^\alpha$ and $I_{\beta;3}^\alpha$. Recall (\ref{sepeqn813}) and (\ref{sepeqn810}).  From the estimates  (\ref{jan15eqn21}) and (\ref{may17eqn71})  in Lemma \ref{fixedtimeestimate1} and the estimates (\ref{jan15eqn91}) and (\ref{may17eqn120}) in Lemma \ref{fixedtimeestimate2}, the following estimate holds from the $L_{x,v}^2-L^2_{x,v}$ type bilinear estimate, 
\be\label{march2eqn5}
\sum_{|\alpha|+|\beta |=N_0}|I_{\beta;2}^{\alpha}| + |I_{\beta;3}^{\alpha}| \lesssim \int_{0}^t (1+s)^{-1+2\delta}\epsilon_0 ds \lesssim (1+t)^{ 2\delta}\epsilon_0, \quad \sum_{|\alpha|+|\beta |< N_0}|I_{\beta;2}^{\alpha}| + |I_{\beta;3}^{\alpha}| \lesssim (1+t)^{\delta  }\epsilon_0.
\ee

Now, it remains to estimate $I_{\beta;1}^\alpha$. Recall (\ref{sepeqn809}). Note that
\[
D_v=\nabla_v - t\nabla_v\hat{v}\cdot\nabla_x, \quad \Longrightarrow g^\alpha_\beta (t ,x,v)    {D}_v g^\alpha_{\beta}(t,x,v) = \frac{1}{2} D_v\big(g^\alpha_\beta (t ,x,v)  \big)^2.
\]
Therefore, after doing integration by parts in $x$ and $v$, the following equality holds, 
\[
 I_{\beta;1}^{\alpha}= \int_{t_1}^{t_2} \int_{\R^3} \int_{\R^3} \h \big(\omega_{\beta}^{\alpha}( x, v)g^\alpha_\beta (t ,x,v)\big)^2 \big[ \frac{  2 K(t,x+\hat{v}t,v)\cdot  {D}_v \omega_{\beta}^{\alpha}( x, v)}{ \omega_{\beta}^{\alpha}( x, v)} +  {D}_v \cdot  K(t,x+\hat{v}t,v) \big]   d x dv.   \]
 
Recall  (\ref{jan31eqn1}). From the direct computation, the estimate (\ref{feb8eqn51}) in Lemma \ref{derivativeofweightfunction}, and   the decay estimate (\ref{noveqn935}) in Lemma \ref{sharplinftydecay}, we have 
\be\label{feb10eqn51}
\big|{D}_v \cdot  K(t,x+\hat{v}t,v) \big| + \Big| \frac{  K(t,x+\hat{v}t,v)\cdot  {D}_v \omega_{\beta}^{\alpha}( x, v)}{\omega_{\beta}^{\alpha}( x, v)}\Big| \lesssim (1+|t|)^{-1} E_{\textup{low}}^{\phi}(t)\lesssim (1+|t|)^{-1}\epsilon_1. 
\ee

To sum up,   our desired estimate (\ref{sepeqn110}) holds  from the above estimate (\ref{feb10eqn51}),  the $L^2_{x,v}-L^2_{x,v}-L^\infty_{x,v}$ type multi-linear estimate, and the estimate (\ref{march2eqn5}).
 \end{proof}

\begin{lemma}\label{fixedtimeestimate1}
 Under the bootstrap assumption \textup{(\ref{bootstrapassumption})},  the following estimate holds for any $t\in[0,T]$, 
\be\label{jan15eqn21}
 \sum_{\alpha\in \mathcal{B}, \beta\in \mathcal{S}, |\alpha|+|\beta|= N_0} \|\omega_{\beta}^\alpha(x,v) \textit{h.o.t}_\beta^\alpha(t,x,v)  \|_{L^2_x L^2_v} \lesssim (1+|t|)^{-1+\delta}\epsilon_0.
\ee
\be\label{may17eqn71}
\sum_{\alpha\in \mathcal{B}, \beta\in \mathcal{S}, |\alpha|+|\beta|<N_0}\|\omega_{\beta}^\alpha(x,v) \textit{h.o.t}_\beta^\alpha(t,x,v)  \|_{L^2_x L^2_v}\lesssim (1+|t|)^{-1+\delta/2}\epsilon_0.
\ee
\end{lemma}
\begin{proof}
 Recall (\ref{jan30eqn100}). Motivated from the the decomposition of $ \textit{h.o.t}_\beta^\alpha(t,x,v)$, we separate into three cases as follows. 

  $\bullet$ \textbf{Case} $1$:   \quad The estimate of $\textit{h.o.t}_{\beta;1}^\alpha(t,x,v) $. 

 Recall  the equations  (\ref{april10eqn21}), (\ref{jan30eqn51}), and  the first decomposition of $D_v$ in (\ref{summaryoftwodecomposition}) in Lemma \ref{twodecompositionlemma}. From the estimate   of coefficients in the estimate (\ref{jan15eqn2}) in Lemma \ref{twodecompositionlemma} and the estimate (\ref{sepeqn88}) in Lemma \ref{decompositionofderivatives}   and the decay estimate (\ref{noveqn935}) in Lemma \ref{sharplinftydecay}, we have 
\[
\sum_{|\alpha|+|\beta|=N_0}\|\omega_{\beta}^\alpha(x,v) \textit{h.o.t}_{\beta;1}^\alpha(t,x,v) \|_{L^2_{x,v}}  \lesssim \sum_{\rho, \iota\in \mathcal{S}, \gamma\in \mathcal{B}, |\rho| +|\gamma|\leq N_0, |\iota| \leq 1} \big( \|(1+|\tilde{d}(t,x,v)|) \p_t \phi^{\iota}(t,x+\hat{v}t)\|_{L^\infty_{x,v}}
\]
\be\label{jan31eqn131}
+ \|(1+|\tilde{d}(t,x,v)|) \nabla_x \phi^{\iota}(t,x+\hat{v}t)\|_{L^\infty_{x,v}} \big) \|\omega_{\rho}^\gamma( x,v)g_{\rho}^\gamma(t,x,v) \|_{L^2_{x,v}}\lesssim (1+|t|)^{-1} E_{\textup{high}}^{f }(t)E_{\textup{low}}^{\phi}(t).
\ee

 $\bullet$ \textbf{Case} $2$:   \quad  The estimate of $\textit{h.o.t}_{\beta;2}^\alpha(t,x,v) $. 

Recall the equation  (\ref{jan30eqn52})  and  the first decomposition of $D_v$ in (\ref{summaryoftwodecomposition}) in Lemma \ref{twodecompositionlemma}.  From the estimate   of coefficients in the estimate (\ref{jan15eqn2}) in Lemma \ref{twodecompositionlemma},  the estimate of coefficients in (\ref{jan31eqn91}) and    the decay estimate (\ref{noveqn935}) in Lemma \ref{sharplinftydecay}, we have
\[
\sum_{|\alpha|+|\beta|=N_0}\|\omega_{\beta}^\alpha(x,v) \textit{h.o.t}_{\beta;2}^\alpha(t,x,v) \|_{L^2_{x,v}}  \lesssim \sum_{\rho, \iota\in \mathcal{S}, \gamma\in \mathcal{B}, |\rho| +|\gamma|\leq N_0, |\iota| \leq 1}\big( \|(1+|\tilde{d}(t,x,v)|) \p_t \phi^{\iota}(t,x+\hat{v}t)\|_{L^\infty_{x,v}}
\]
\be\label{jan31eqn132}
+ \|(1+|\tilde{d}(t,x,v)|) \nabla_x \phi^{\iota}(t,x+\hat{v}t)\|_{L^\infty_{x,v}} \big) \|\omega_{\rho}^\gamma( x,v)g_{\rho}^\gamma(t,x,v) \|_{L^2_{x,v}}\lesssim (1+|t|)^{-1} E_{\textup{high}}^{f }(t)E_{\textup{low}}^{\phi}(t).
\ee

 $\bullet$ \textbf{Case} $3$:   \quad The estimate of $\textit{h.o.t}_{\beta;3}^\alpha(t,x,v) $. 

Recall (\ref{jan30eqn53}), the detailed formula of $Y^\beta_i$ in (\ref{sepeqn522}), and the detailed formula of $K_i(t,x,v), i\in\{1,\cdots, 7\},$ in (\ref{jan31eqn61}), (\ref{jan31eqn62}), and (\ref{jan31eqn63}).   From  the estimates of coefficients in  (\ref{sepeqn904}) and (\ref{jan31eqn110}) in Lemma \ref{summaryofhighordercommutation},  the $L^2_{x,v}-L^\infty_{x,v}$ type bilinear estimate and the   decay estimate (\ref{noveqn935}) in Lemma \ref{sharplinftydecay}, we have
\[
\sum_{|\alpha|+|\beta|=N_0}\|\omega_{\beta}^\alpha(x,v) \textit{h.o.t}_{\beta;3}^\alpha(t,x,v) \|_{L^2_{x,v}}   \lesssim \sum_{\rho, \iota\in \mathcal{S}, \gamma\in \mathcal{B}, |\rho| +|\gamma|\leq N_0, |\iota| \leq 1} \big( \|(1+|\tilde{d}(t,x,v)|) \p_t \phi^{\iota}(t,x+\hat{v}t)\|_{L^\infty_{x,v}}
\]
\be\label{jan31eqn133}
+ \|(1+|\tilde{d}(t,x,v)|) \nabla_x \phi^{\iota}(t,x+\hat{v}t)\|_{L^\infty_{x,v}} \big)  \|\omega_{\rho}^\gamma( x,v)g_{\rho}^\gamma(t,x,v) \|_{L^2_{x,v}}\lesssim (1+|t|)^{-1} E_{\textup{high}}^{f }(t)E_{\textup{low}}^{\phi}(t).
\ee

To sum up, our desired estimate (\ref{jan15eqn21}) holds from the estimates (\ref{jan31eqn131}), (\ref{jan31eqn132}), and (\ref{jan31eqn133}). With minor modifications, the desired estimate (\ref{may17eqn71}) holds very similarly. 
\end{proof}

\begin{lemma}\label{fixedtimeestimate2}
Under the bootstrap assumption \textup{(\ref{bootstrapassumption})}, 
the following estimate holds for any fixed  time $t\in [0,T]$, 
\be\label{jan15eqn91}
\sum_{\alpha\in \mathcal{B}, \beta\in \mathcal{S}, |\alpha|+|\beta|= N_0}\|\omega_{\beta}^\alpha(x,v)  \textit{l.o.t}_\beta^\alpha(t,x,v)  \|_{L^2_x L^2_v} \lesssim (1+|t|)^{-1+\delta} \epsilon_0.  
\ee
\be\label{may17eqn120}
\sum_{\alpha\in \mathcal{B}, \beta\in \mathcal{S}, |\alpha|+|\beta| < N_0}\|\omega_{\beta}^\alpha(x,v)  \textit{l.o.t}_\beta^\alpha(t,x,v)  \|_{L^2_x L^2_v} \lesssim (1+|t|)^{-1+\delta/2} \epsilon_0.  
\ee
\end{lemma}
\begin{proof}

Since the case   $|t|\leq 1$ is trivial, it would be sufficient to consider  the case   $|t|\geq 1$. Recall the decomposition of $\textit{l.o.t}_{\beta}^\alpha(t,x,v)$ in (\ref{jan31eqn151}). Based on the total number of derivatives acting on the scalar field, we separate into two cases as follows.

 $\bullet$ \textbf{Case} $1$:  \quad  The estimate of $ \textit{l.o.t}_{\beta;i}^\alpha(t,x,v),$ $i\in \{1,2,4\}$.

Recall (\ref{jan31eqn152}), (\ref{jan31eqn153}), and (\ref{jan31eqn155}). Note that there are at most ten derivatives hit on the nonlinear wave part. Recall the commutation rule between $\Lambda^\beta$ and $X_i$ in  (\ref{noveq521}) and the equality (\ref{sepeqn610}). From the estimate of coefficients in (\ref{sepeqn524}), (\ref{sepeqn904}), (\ref{jan15eqn41}), (\ref{sepeqn88})  and (\ref{noveq781}), the following estimate holds from the linear decay estimate (\ref{noveqn935}) in Lemma \ref{sharplinftydecay} and the $L^2_{x,v}-L^\infty_{x,v}$ type bilinear estimate, 
\[
\sum_{i=1,2,4} \sum_{\alpha\in \mathcal{B}, \beta\in \mathcal{S}, |\alpha|+|\beta|=  N_0}\|\omega_{\beta}^\alpha(x,v)  \textit{l.o.t}_{\beta;i}^\alpha(t,x,v)  \|_{L^2_x L^2_v} \]
\[
\lesssim  \sum_{\rho, \gamma\in \mathcal{B}, \kappa\in \mathcal{S}, |\gamma|+|\kappa|\leq N_0, |\rho|\leq 13} \sum_{u\in\{\p_t \phi^\rho, \nabla   \phi^\rho\}} \| \omega_{\kappa}^\gamma(x,v)  g_{\kappa}^\gamma(t,x,v)\|_{L^2_{x,v}} \| (1+|\tilde{d}(t,x,v)|) u(t,x+\hat{v}t)\|_{L^\infty_{x,v}}\]
\be\label{jan16eqn71}
 \lesssim (1+|t|)^{-1} E_{\textup{high}}^{f}(t)E_{\textup{low}}^{\phi}(t)\lesssim (1+|t|)^{-1+\delta} \epsilon_1^2.
\ee

 $\bullet$ \textbf{Case} $2$:  \quad  The estimate of $ \textit{l.o.t}_{\beta;3}^\alpha(t,x,v)$.

 The main difficulty  in this case is caused by the scenario when the scalar field  has the maximal number of derivatives, for which neither the decay rate in time nor the decay with respect to the light cone of the scalar field is available. For this case, we will use both the decay estimate of the density type function and  the fact that the  space-resonance set is relatively small. This fact, which is high nontrivial and also is the starting point of our bilinear estimates in subsection \ref{multilinearestimate},  will be elaborated in subsection \ref{multilinearestimate}.

Recall (\ref{jan31eqn154}). From the equality (\ref{sepeqn610}) in Lemma \ref{decompositionofderivatives},  we have
\[
 \textit{l.o.t}_{\beta;3}^\alpha(t,x,v)= \sum_{
\begin{subarray}{c}
 \iota, \kappa\in \mathcal{S}, \beta_1, \gamma_1, \beta_2, \gamma_2\in \mathcal{B},\\ 
  |\rho|+|\beta_1|> 11,  |\beta_1|+|\gamma_1|\leq |\alpha|\\
 |\rho|+|\beta_2|> 11, |\beta_2|+|\gamma_2|\leq |\alpha|\\
 \end{subarray} } \Lambda^{\rho}\big( (\p_t\phi^{\beta_1} (t,x+\hat{v}t)   \]
\[
 +\hat{v}\cdot \nabla_x\phi^{\beta_1} (t,x+\hat{v}t) )  \big) \Lambda^{\kappa}\big( a_{\alpha;{\beta_1}, \gamma_1}^1(v) g^{\gamma_1}(t,x,v) + a_{\alpha;\beta_1, \gamma_1}^2(v) v\cdot D_v g^{\gamma_1}(t,x,v)\big)\]
\be\label{jan15eqn81}	
+   \Lambda^\rho \big( \big(  a_{\alpha;\beta_2, \gamma_2}^3(v)  \nabla_x \phi^{\beta_2} +   a_{\alpha;\beta_2, \gamma_2}^4(v)  \p_t \phi^{\beta_2}   \big)(t,x+\hat{v}t)  \cdot\alpha_i(v) \big) \Lambda^\kappa \big(   X_i g^{\gamma_2}(t,x,v)\big),
\ee

 From the equalities (\ref{noveq521}) and  (\ref{sepeqn522}) in Lemma \ref{summaryofhighordercommutation} and the first decomposition of $D_v$ in (\ref{summaryoftwodecomposition}) in Lemma \ref{twodecompositionlemma}, we have
\[
\Lambda^{\kappa}\big(X_i g^{\gamma_2}(t, x,v)\big)= \big[ \alpha_i(v)\cdot D_v
\circ \Lambda^{\kappa} + [\Lambda^\kappa, X_i] \big] g^{\gamma_2}(t,x,v)= \sum_{\rho\in \mathcal{K}, |\rho|=1} \alpha_i(v)\cdot d_{\rho}(t, x, v) \Lambda^{\rho\circ \kappa} g^{\gamma_2}(t,x,v) \]
\be\label{noveq102}
+ Y_i^\kappa g^{\gamma_2}(t,x,v) + \sum_{\kappa'\in \mathcal{S}, |\kappa'|\leq |\kappa|-1} \big[ \tilde{d}(t,x,v)\tilde{e}_{\kappa,i}^{\kappa', 1}(x,v) +\tilde{e}_{\kappa,i}^{\kappa', 2}(x,v)\big]  \Lambda^{\kappa'}g^{\gamma_2}(t, x,v). 
\ee
 
From (\ref{jan15eqn81}) and  (\ref{noveq102}), and the detailed formula of $d_{\rho}(t,x,v)$ in (\ref{sepeq947}),  we can rewrite ``$\textit{l.o.t}_{\beta;3}^\alpha(t, x, v)$'' as follows
\[
\textit{l.o.t}_{\beta;3}^\alpha(t, x, v)= \sum_{
\begin{subarray}{c}
\rho \in\mathcal{S}, \kappa_1, \kappa_2\in \mathcal{B},|\kappa_2|\leq |\alpha|, |\rho|\leq |\beta|, \\
|\rho|+|\kappa_1|+|\kappa_2|\leq |\alpha|+|\beta| \\
|\rho|+|\kappa_2|\leq  |\alpha|+|\beta|-12\\
\end{subarray}}\big[\big(\tilde{d}(t,x,v)\widehat{e}^1_{\kappa_1, \kappa_2,\rho}(t,x,v)  +\widehat{e}^2_{\kappa_1, \kappa_2,\rho}(t,x,v)   \big) \p_t \phi^{\kappa_1}(t, x+\hat{v}t)
\]
\be\label{noveq130}
+ \big(\tilde{d}(t,x,v)\widehat{e}^3_{\kappa_1, \kappa_2,\rho}(t,x,v)  +\widehat{e}^4_{\kappa_1, \kappa_2,\rho}(t,x,v)   \big)\cdot \nabla_x \phi^{\kappa_1}(t, x+\hat{v}t)\big] g_{\rho}^{\kappa_2}(t,x,v),
\ee
where the coefficients $\widehat{e}^i_{\kappa_1, \kappa_2,\rho}(t,x,v), i\in\{1,2,3,4\} $, satisfy the following rough estimate, 
\be\label{noveq220}
\sum_{i=1,\cdots, 4} | \widehat{e}^i_{\kappa_1, \kappa_2,\rho}(t,x,v) |\lesssim (1+|x|^2+|v|^2)^{|\alpha|+2|\beta|-2|\rho|-|\kappa_2|+5},
\ee
which can be derived from the estimate of coefficients in (\ref{jan31eqn91}),  the estimate (\ref{sepeqn88}) in Lemma \ref{decompositionofderivatives} and the estimates (\ref{sepeqn524}),  (\ref{sepeqn904}), and  (\ref{jan15eqn41})  in Lemma \ref{summaryofhighordercommutation}.

 Recall (\ref{noveq130}). We first do dyadic decomposition for the nonlinear wave part. More precisely, the following decomposition holds, 
\be\label{jan16eqn67}
\textit{l.o.t}_{\beta;3 }^\alpha(t,x,v)=\sum_{k\in \mathbb{Z}} H_k(t,x,v),
\ee
where
\[
 H_k(t,x,v):= \sum_{
\begin{subarray}{c}
\rho \in\mathcal{S}, \kappa_1, \kappa_2\in \mathcal{B},|\kappa_2|\leq |\alpha|, |\rho|\leq |\beta|, \\
|\rho|+|\kappa_1|+|\kappa_2|\leq |\alpha|+|\beta| \\
|\rho|+|\kappa_2|\leq  |\alpha|+|\beta|-12\\
\end{subarray}}\big[\big(\tilde{d}(t,x,v)\widehat{e}^1_{\kappa_1, \kappa_2,\rho}(t,x,v)  +\widehat{e}^2_{\kappa_1, \kappa_2,\rho}(t,x,v)   \big) (\p_t\phi^{\kappa_1})_k (t, x+\hat{v}t)
\]
\[
+ \big(\tilde{d}(t,x,v)\widehat{e}^3_{\kappa_1, \kappa_2,\rho}(t,x,v)  +\widehat{e}^4_{\kappa_1, \kappa_2,\rho}(t,x,v)   \big)  \big(\nabla\phi^{\kappa_1}\big)_k (t, x+\hat{v}t)\big] g_{\rho}^{\kappa_2}(t,x,v).
\]

Based on the possible size of $k$,  we separate into the low frequency case and the high frequency case as follows. 

\noindent $\bullet$\quad If $k\leq 0$.

  Recall the decomposition (\ref{feb1eqn41}).   From the estimate (\ref{may17eqn12}) in Lemma \ref{Alinearestimate}, the estimate of correction term in (\ref{may17eqn51})  and the $L^\infty_\xi$-type estimate of the modified profile in (\ref{octe741}), we have
\be\label{may17eqn21}
\sum_{\kappa\in \mathcal{B}, |\kappa|\leq N_0}\|P_k[\p_t\phi^{\kappa}]\|_{L^2} + \|P_k[\nabla_x\phi^{\kappa}]\|_{L^2}\lesssim 2^{k/2} \epsilon_1 + |t|2^{3k/2}\epsilon_1. 
\ee
From   the estimate of coefficients in (\ref{noveq220}), after  using the $L^2_x-L^\infty_x L^2_v$ type estimate, the estimate (\ref{may17eqn21}), and the decay estimate (\ref{densitydecay}) in Lemma \ref{decayestimateofdensity}, the following estimate holds if $2^{k}\leq |t|^{-1}$,
  \be\label{jan16eqn61}
\|\omega_{\beta}^\alpha(x,v)H_k(t,x,v) \|_{L^2_xL^2_v}\lesssim \big(2^{k/2}  + |t|2^{3k/2} \big)\epsilon_1(1+|t|)^{-1/2} E_{\textup{high}}^{f;2}(t)\lesssim  |t|^{-1/2+\delta/2 }2^{k/2}\big(1  + |t|  2^{ k }\big)\epsilon_0.
 \ee
Note that, we estimate   the inhomogeneous modulation $\tilde{d}(t,x,v)$ roughly from the above by $1+|t|$ in the above estimate. 

If $ |t|^{-1} \leq 2^{k}\leq 1$, from the decompositions in (\ref{feb1eqn1}),  the following decomposition holds for $H_k$,
\be\label{jan16eqn69}
H_k(t,x,v)= H_k^1(t,x,v) + H_k^2(t,x,v), 
\ee
where
\[
 H_k^1(t,x,v) := \sum_{
\begin{subarray}{c}
\rho \in\mathcal{S}, \kappa_1, \kappa_2\in \mathcal{B},|\kappa_2|\leq |\alpha|, |\rho|\leq |\beta|, \\
|\rho|+|\kappa_1|+|\kappa_2|\leq |\alpha|+|\beta| \\
|\rho|+|\kappa_2|\leq  |\alpha|+|\beta|-12\\
\end{subarray}}\big[\big(\tilde{d}(t,x,v)\widehat{e}^1_{\kappa_1, \kappa_2,\rho}(t,x,v)  +\widehat{e}^2_{\kappa_1, \kappa_2,\rho}(t,x,v)   \big)  { \widetilde{(\p_t \phi^{\kappa_1})_{k} } }(t, x+\hat{v}t)
\]
\be\label{2020feb13eqn1}
+ \big(\tilde{d}(t,x,v)\widehat{e}^3_{\kappa_1, \kappa_2,\rho}(t,x,v)  +\widehat{e}^4_{\kappa_1, \kappa_2,\rho}(t,x,v)   \big)\cdot \big(\nabla_x\widetilde{\phi^{\kappa_1}}\big)_k (t, x+\hat{v}t)\big] g_{\rho}^{\kappa_2}(t,x,v),
\ee
\[
 H_k^2(t,x,v) := \sum_{
\begin{subarray}{c}
\rho \in\mathcal{S}, \kappa_1, \kappa_2, \eta\in \mathcal{B},|\kappa_2|\leq |\alpha|, |\rho|\leq |\beta|  \\
|\rho|+|\kappa_1|+|\kappa_2|\leq |\alpha|+|\beta|, |\eta|\leq |\kappa_1| \\
|\rho|+|\kappa_2|\leq  |\alpha|+|\beta|-12\\
\end{subarray}}\big[\big(\tilde{d}(t,x,v)\widehat{e}^1_{\kappa_1, \kappa_2,\rho}(t,x,v)  +\widehat{e}^2_{\kappa_1, \kappa_2,\rho}(t,x,v)   \big)\textup{Re}[ E_{\kappa_1;\eta} (g^{\eta}_k)(t, x+\hat{v}t)]
\]
\be\label{noveq291}
- \big(\tilde{d}(t,x,v)\widehat{e}^3_{\kappa_1, \kappa_2,\rho}(t,x,v)  +\widehat{e}^4_{\kappa_1, \kappa_2,\rho}(t,x,v)   \big)\cdot R\big( \textup{Im}\big( E_{\kappa_1;\eta}(g^{\eta}_k)(t, x+\hat{v}t)\big)\big) \big] g_{\rho}^{\kappa_2}(t,x,v).
\ee
From the estimate of coefficients in (\ref{noveq220}), we have
\be\label{feb12eqn151}
 \| \omega_{\beta}^\alpha(x,v)H_k^1(t,x,v) \|_{L^2_xL^2_v}^2  \lesssim \sum_{\begin{subarray}{c}
\kappa, \gamma\in \mathcal{B},|\kappa|\leq N_0,|\gamma|\leq |\alpha|\\
  |\rho|+|\gamma|\leq N_0-12,  \\
 \end{subarray}} \sum_{u\in \{\widetilde{\p_t \phi^\kappa}, \nabla_x \widetilde{\phi^\kappa} \}}  (1+t)^2    \int_{\R^3} \int_{\R^3}  G_{\rho }^{\gamma}(t,x,v) | u_{k }(t,x+\hat{v}t) |^2   d x d v, 
\ee
where $ G_{\rho }^{\gamma}(t,x,v)$ is some determined function that satisfies the following estimate, 
\[
 |G_{\rho}^{\gamma}(t,x,v)|\lesssim |\omega_{\beta}^\alpha (x,v) (1+|x|^2+|v|^2)^{|\alpha|+2|\beta|-2|\rho|-|\gamma|+5} g_{\rho}^\gamma(t,x,v)|^2 .
\]
Recall (\ref{feb1eqn10}). From the estimate (\ref{feb12eqn151}),  the  multilinear estimate (\ref{multilinear1jan16}) in Lemma \ref{multilinearlemma1}, the estimate of modified profiles $\widehat{\widetilde{h}^\alpha}(t,\xi)$ in (\ref{octe741}) and (\ref{octe742}), and the hierarchy between the different order of weight functions, we have
\[
 \|\omega_{\beta}^\alpha(x,v)H_k^1(t,x,v) \|_{L^2_xL^2_v}\lesssim \sum_{|\rho|+|\kappa|\leq N_0- 7}   \| (1+|x|^2+|v|^2)^{|\alpha|+2|\beta|-2|\rho|-|\gamma|+22}   {\omega_{\beta}^\alpha}(x  ,v) g_{\rho}^\kappa(t,x ,v)\|_{L^2_x L^2_v} 
\]
\be\label{jan16eqn62}
\times  t^{-3/2}  2^{-k_{}/2}\big(\epsilon_0  + |t|^{2\delta} 2^{k_{-}}\epsilon_0 \big)   
 \lesssim |t|^{-3/2+\delta/2} 2^{-k/2} \epsilon_1^2 + |t|^{-3/2+3\delta} 2^{-k/2+k_{-}} \epsilon_1^2.
\ee

It remains to  estimate   $H_k^2(t,x,v)$. Recall (\ref{noveq291}),   (\ref{noveq510}), and  (\ref{feb1eqn1}). Note that the terms inside $H_k^2(t,x,v)$ have the same structure as the bilinear form that we will define in (\ref{noveq260}).  From the estimate of coefficients in (\ref{noveq220}) and the bilinear estimate (\ref{bilineardensity}) in Lemma \ref{bilineardensitylemma} and the estimate of correction term in (\ref{may17eqn51}), we have
\be\label{jan16eqn64}
 \|\omega_{\beta}^\alpha(x,v)H_k^2(t,x,v) \|_{L^2_xL^2_v} \lesssim \big(|t|^{-1+\delta/2}+ |t|^{-2+\delta/2}2^{-k} + |t|^{-2+2\delta}2^{k } \big)\epsilon_0. 
 \ee
To sum up, from the decomposition  (\ref{jan16eqn69}) and the  estimates (\ref{jan16eqn61}), (\ref{jan16eqn62}), and (\ref{jan16eqn64}), we have
\[
\sum_{k\in \mathbb{Z}, k\leq 0} \|\omega_{\beta}^\alpha(x,v)H_k (t,x,v) \|_{L^2_xL^2_v} \lesssim  \big[\big(\sum_{  2^{k}\leq  |t| ^{-1 }}   |t|^{-1/2+\delta/2 }2^{k/2}\big(1  + |t|  2^{ k }\big)\epsilon_0 \big) + \big( \sum_{  |t|^{-1}\leq 2^k\leq 2}       \big(|t|^{-1+\delta/2}
\]
 \be\label{jan16eqn72}
+|t|^{-3/2+\delta/2} 2^{-k/2} + |t|^{-2+\delta/2}2^{-k} + |t|^{-3/2+2\delta}2^{k/2} \big)\epsilon_0 \big]     \lesssim  (1+t)^{-1+\delta/2}\log(1+|t|) \epsilon_0 .
\ee 

\noindent $\bullet$\quad If $k\geq 0$.

 From the estimate of coefficients in (\ref{noveq220}) and the bilinear estimate (\ref{bilineardensitylargek}) in Lemma \ref{bilineardensitylemma}, we have 
\be\label{feb12eqn161}
\sum_{k\geq 0, k\in \mathbb{Z}} \|\omega_{\beta}^\alpha(x,v)H_k^2(t,x,v) \|_{L^2_xL^2_v} \lesssim \sum_{k\geq 0 } (1+t)^{-2 +2\delta } 2^{-k} \epsilon_1^2\lesssim (1+|t|)^{  -2 +2\delta }\epsilon_0.
\ee
Hence, it would be sufficient to estimate ``$H_k^1(t,x,v)$''. Similar to the strategy used in  deriving  the  estimate (\ref{feb12eqn151}), we  have 
  \be\label{feb12eqn201}
 \|\sum_{k\in \mathbb{Z}, k\geq 0}\omega_{\beta}^\alpha(x,v)H_k^1(t,x,v) \|_{L^2_xL^2_v}^2 \lesssim   
 \sum_{ k_1, k_2\in \mathbb{Z},   k_1, k_2\geq 0  }  (1+t)^2  K_{k_1,k_2},   
  \ee
  where
\be\label{jan16eqn191}
  K_{k_1,k_2}:= \sum_{\begin{subarray}{c}
\kappa, \gamma\in \mathcal{B},|\kappa|\leq N_0\\
  |\rho|+|\gamma|\leq N_0-12,  |\gamma|\leq |\alpha|\\
 \end{subarray}}  \sum_{u, v\in \{\widetilde{\p_t \phi^\kappa}, \nabla_x \widetilde{\phi^\kappa} \}}   \Big| \int_{\R^3} \int_{\R^3}  G_{\rho;u}^{\gamma;v}(t,x,v)      u_{k_1}(t,x+\hat{v}t) v_{k_2}(t,x+\hat{v}t) d x d v \Big|,
\ee
where $ G_{\rho;u}^{\gamma;v}(t,x,v) $, $u, v\in \{\widetilde{\p_t \phi^\kappa}, \nabla_x \widetilde{\phi^\kappa} \}$, are some determined function that satisfies the following estimate, 
\be\label{feb13eqn1}
\sum_{
\begin{subarray}{l}
|a|+|b|\leq 5\\
u, v\in \{\widetilde{\p_t \phi^\kappa}, \nabla_x \widetilde{\phi^\kappa} \}\\ 
\end{subarray}} |\nabla_x^a \nabla_v^b G_{\rho;u}^{\gamma;v}(t,x,v) |\lesssim \sum_{|\iota|+|\kappa|\leq |\rho|+|\gamma|+5} |\omega_{\beta}^\alpha (x,v) (1+|x|^2+|v|^2)^{|\alpha|+2|\beta|-2|\iota|-|\kappa|+10} g_{\iota}^{\kappa}(t,x,v)|^2.
\ee

We first consider the case when $|k_1-k_2|\geq 5$. Recall (\ref{jan16eqn191}). From the orthogonality on the Fourier side, we know  that the frequency of ``$G_{\rho;i}^{\gamma}(t,x,v)$'' is localized around $2^{\max\{k_1,k_2\}}$ when $|k_1-k_2|\geq 5$. Hence,   from the above estimate (\ref{feb13eqn1}), and  the trilinear estimate (\ref{trilinearfeb20est1}) in Lemma \ref{multilinearlemma1}, we have 
\[
\sum_{
\begin{subarray}{c}
k, k_1,k_2\in \mathbb{Z}, 
 k_1,k_2\geq 0\\
 |k_1-k_2|\geq 5\\
 \end{subarray} } | K_{k_1,k_2}| 
\lesssim  
\sum_{\begin{subarray}{c}
i=1,2,3,|a|\leq 4\\
\rho\in\mathcal{S}, \gamma\in \mathcal{B},   |\gamma|\leq |\alpha|,
  |\rho|+|\gamma|\leq N_0-12\\
 \end{subarray}} \sum_{
\begin{subarray}{c}
k, k_1,k_2\in \mathbb{Z}, 
 k_1,k_2\geq 0\\
 |k_1-k_2|\geq 5\\
 \end{subarray} }
2^{-\max\{k_1k_2\}}  |t|^{-5+2\delta } \epsilon_1^2
\]
\be\label{feb13eqn9}
 \times \|(1+|x|^2)(1+|v|^{25}) \nabla_x \nabla_{v}^a G_{\rho; i}^{\gamma}(t,x,v) \|_{L^1_{x,v}}  \lesssim |t|^{-5+4\delta}\epsilon_0^2. \ee

Lastly, we consider the case when $|k_1-k_2|\leq 5$. Recall (\ref{jan16eqn191}). Again, from   the   estimate (\ref{feb13eqn1}),  the trilinear estimate (\ref{multilinear1jan16}) in Lemma \ref{multilinearlemma1}, and the Cauchy-Schwarz inequality,   we know that the following  estimate holds, 
\[
\sum_{
\begin{subarray}{c}
k, k_1,k_2\in \mathbb{Z}\\
 k_1,k_2\geq 0\\
 |k_1-k_2|\leq 5\\
 \end{subarray} } |K_{k_1,k_2}| \lesssim 
\sum_{
\begin{subarray}{c}
k_1,k_2\in \mathbb{Z}\\  k_1,k_2\geq 0\\ 
|k_1-k_2|\leq 5 \\ 
\end{subarray}}
\sum_{\begin{subarray}{c}
i=1,2,3,|a|\leq 4\\
\rho\in\mathcal{S}, \gamma\in \mathcal{B},   |\gamma|\leq |\alpha|\\
  |\rho|+|\gamma|\leq N_0-12\\
 \end{subarray}} (1+|t|)^{-5} \|(1+|x|^2)(1+|v|^{25})  \nabla_{v}^\alpha  G_{\rho; i}^{\gamma}(t,x,v)  \|_{L^1_{x,v}}
\]
\[
\times \big(2^{-k_{1}/2} E_{\textup{high}}^{\phi}(t) +\sum_{\iota\in \mathcal{B}, |\iota|\leq N_0} \| \widehat{h^\iota}(t, \xi)\psi_{k_1}(\xi)\|_{L^2} )   \big(2^{-k_{2}/2} E_{\textup{high}}^{\phi}(t) +\sum_{\iota\in \mathcal{B}, |\iota|\leq N_0} \| \widehat{h^\iota}(t, \xi)\psi_{k_2}(\xi)\|_{L^2} ) 
\]
\be\label{feb13eqn10}
  \lesssim    |t|^{-5}     \big(E_{\textup{high}}^{\phi}(t)E_{\textup{high}}^{f}(t)\big)^2\lesssim  |t|^{-5+4\delta} \epsilon_0^2.
\ee
From   the estimates (\ref{feb12eqn201}), (\ref{feb13eqn9}), and (\ref{feb13eqn10}), we have
\be\label{feb13eqn19}
   \|\sum_{k\in \mathbb{Z}, k\geq 0}\omega_{\beta}^\alpha(x,v)H_k^1(t,x,v) \|_{L^2_xL^2_v} \lesssim   |t|^{-3/2+2\delta} \epsilon_0. 
\ee

To sum up,  recall the decompositions (\ref{jan16eqn67}) and (\ref{jan16eqn69}), the following estimate holds   from the estimates (\ref{jan16eqn72}), (\ref{feb12eqn161}), and (\ref{feb13eqn19}),
\be\label{feb13eqn24}
\sum_{\alpha\in\mathcal{B}, \beta\in \mathcal{S}, |\alpha|+|\beta|=  N_0} \| \omega_{\beta}^\alpha(x,v) \textit{l.o.t}_{\beta;3}^\alpha(t,x,v) \|_{L^2_{x,v}}\lesssim (1+|t|)^{-1+\delta} \epsilon_0. 
\ee
Therefore, our desired estimate    (\ref{jan15eqn91}) holds   from the estimates (\ref{jan16eqn71})  and  the above estimate  (\ref{feb13eqn24}).  

Since the correction term $ \widetilde{g}_{\alpha,\gamma}(t, v)$ defined in (\ref{correctionterm2}), which contributes the logarithmic growth in the estimate (\ref{jan16eqn72}),  equals zero if $|\alpha|< N_0$, with minor modifications in the above argument, our desired estimate (\ref{may17eqn120}) holds  similarly. 
\end{proof}
 
 \subsection{The low order energy estimate for the Vlasov part}\label{lowordervlasovsec}

In this subsection, as summarized in Proposition \ref{loworderenergy},  we show that the  low order energy $E_{\textup{low}}^f(t)$ will be remained bounded over time. 

The main ideas of proving this fact are same as the one we used in a relatively simpler system, Vlasov-Poisson system, in \cite{wang3}. For the sake of readers, we explain concisely main ideas behind as follows: (i) The main obstruction of  preventing the nonlinearities of the Vlasov-Nordstr\"om system decay faster comes from the low frequency part of the scalar field, which, intuitively speaking,  of size $1/t$.  (ii) Note that, the frequency of the Vlasov part is exactly same as the frequency of the scalar field since the output frequency is zero in the definition of low order energy, see (\ref{octeqn1896}). Due to the presence of spatial derivative in the bulk term $t\nabla_v \hat{v}\cdot \nabla_x g(t,x,v)$, in the aforementioned worst case,  the decay rate is compensated by the spatial derivative.

 \begin{proposition}\label{loworderenergy}
Under the bootstrap assumption \textup{(\ref{bootstrapassumption})},  the following estimate holds for any $t\in[0,T]$, 
\be\label{feb20eqn10}
\big| E_{\textup{low}}^{f}(t  ) \big| \lesssim \epsilon_0 +\int_0^t (1+s)^{-3/2+3\delta} \epsilon_1^2 d s\lesssim \epsilon_0.
\ee
\end{proposition} 
\begin{proof}
Recall the definition of the low order energy $E_{\textup{low}}^f(t)$ in   (\ref{octeqn1896}),  the definition of correction term in (\ref{correctionterm2}),   the equation satisfied by $ {g}(t, x, v)$ in  (\ref{eqnn5}) and the equation satisfied by $g^{\gamma}(t,x,v)$ in (\ref{jan30eqn1}). For the sake of simplicity in notation, we focus on the case   $|\gamma|=0$, i.e., $\gamma=Id$. With minor modification, we can show  general case with $\gamma\in \mathcal{B}, |\gamma|\leq N_0$ similarly.

Let $\widetilde{\omega_{ }^\alpha}(v):=\widetilde{\omega_{Id}^\alpha}(v)$. Note that,  the following equality holds for any $\alpha$ s.t., $|\alpha|\leq N_0$,
\[
\p_t \big( \nabla^\alpha   \widehat{g^\gamma}(t,0 , v) - \nabla_v\cdot g_{\alpha,\gamma}(t,v)\big) = \sum_{\beta_1 + \gamma_2=\alpha} \int_{\R^3} 4 \nabla_{v}^{\beta_1}\big((\p_t +\hat{v}\cdot \nabla_x) \phi(t,x+\hat{v}t)\big)     \nabla_{v}^{\gamma_1} g(t,x,v) \]
\[
+\sum_{\beta_2 + \gamma_2=\alpha, |\gamma_2|\leq N_0-1} \nabla_v^{\beta_2} (\p_t \phi(t, x+\hat{v}t))\big(c_{\beta_2, \gamma_2}^{\alpha, 1}(v) \nabla_v^{\gamma_2 + 1} g(t, x,v) - t c_{\beta_2, \gamma_2}^{\alpha, 2}  \nabla_x\nabla_v^{\gamma_2  }  g(t, x,v)\big)
\]
\be\label{may27eqn301}
+\nabla_v^{\beta_2} (\nabla_x \phi(t, x+\hat{v}t))\big(c_{\beta_2, \gamma_2}^{\alpha, 3}(v) \nabla_v^{\gamma_2 + 1} g(t, x,v) - t c_{\beta_2, \gamma_2}^{\alpha, 4}(v) \nabla_x \nabla_v^{\gamma_2  } g(t, x,v)\big) d x,
\ee
where $c_{\beta_2, \gamma_2}^{\alpha,i}(v), i\in\{1,2,3,4\},$ are some uniquely determined coefficients that satisfy the following estimate, 
\be\label{may16eqn91}
 |c_{\beta_2, \gamma_2}^{\alpha, i}(v) |\lesssim (1+|v|), \quad \textup{for any\,\,}  i \in\{1,2,3,4\}. 
\ee 

Since our desired estimate (\ref{feb20eqn10}) holds straightforwardly for 	the case   $|t|\leq 1$, we focus on the case when $|t|\geq 1$. 
 Note that, after using the equality (\ref{sepeqn190}) repeatedly, the following equality holds,
 \be\label{may16eqn21}
\nabla_v^\beta f(t,x+\hat{v}t)= \sum_{|\gamma|\leq |\beta|} t^{\gamma} c_{\beta, \gamma}(v)\nabla_x^\gamma f(t,x+\hat{v}t)= \sum_{\gamma\in \mathcal{B}, |\gamma|\leq |\beta|} t^{\gamma} ||t|-|x+\hat{v}t||^{-\gamma} C_{\beta, \gamma}( x,v ) \Gamma^\gamma f(t,x+\hat{v}t),
 \ee
where the uniquely determined coefficients $C_{\beta,\gamma}(x,v)$ satisfy the following estimate, 
\be\label{may16eqn92}
|C_{\beta, \gamma}(x,v)|\lesssim 1, \quad |\nabla_x C_{\beta, \gamma}(x,v)|\lesssim 1/(|t|+|x+\hat{v}t|). 
\ee

Note that we have   $||t|-|x+\hat{v} t ||\gtrsim |t|(1+|v|^2)^{-1}$ if $|x|(1+|v|^2) \leq 2^{-10}|t|$. This fact motivates us to decompose (\ref{may27eqn301}) into two parts as follows, 
\be\label{may16eqn151}
 \p_t \big( \nabla^\alpha   \widehat{g}(t,0 , v) - \nabla_v\cdot g_{\alpha}(t,v)\big)= \sum_{\kappa\in \mathcal{B}, |\kappa|, |\gamma_1|+|\gamma_2|\leq N_0, |\gamma_2|\leq N_0-1} H_{\gamma_1, \gamma_2}^{\alpha, \kappa}(t,v) + K_{\gamma_1, \gamma_2}^{\alpha, \kappa}(t,v), 
\ee
where
\[
H_{\gamma_1, \gamma_2}^{\alpha, \kappa}(t,v) = \int_{\R^3} \psi_{> -10}(|x|(1+|v|^2)/(1+|t|)) |t|^{\gamma_1}\big[ \nabla_x^{\gamma_1}\p_t \phi(t,x+\hat{v}t)\big(\alpha_{\gamma_1, \gamma_2}^{\alpha, \kappa;1}(v)\nabla_v^{\gamma_2}  g(t,x,v) - t \alpha_{\gamma_1, \gamma_2}^{\alpha, \kappa;2}( v)\]
\[
 \times \nabla_x \nabla_v^{\gamma_2}g(t,x,v)  \big) + \nabla_x^{\gamma_1+1} \phi(t,x+\hat{v}t)\big(\alpha_{\gamma_1, \gamma_2}^{\alpha, \kappa;3}(v)\nabla_v^{\gamma_2}  g(t,x,v)  - t \alpha_{\gamma_1, \gamma_2}^{\alpha, \kappa;4}( v) \nabla_x \nabla_v^{\gamma_2}g(t,x,v)  \big)\big] d x 
\]
\[
K_{\gamma_1, \gamma_2}^{\alpha, \kappa}(t,v) = \int_{\R^3}c_{\gamma_1}( t,x,v)  \big[ \p_t \Gamma^{\gamma_1} \phi(t,x+\hat{v}t)\big(b_{\gamma_1, \gamma_2}^{\alpha, \kappa;1}(x, v)\nabla_v^{\gamma_2}  g(t,x,v) - t b_{\gamma_1, \gamma_2}^{\alpha, \kappa;2}(x, v)\nabla_x \nabla_v^{\gamma_2}g(t,x,v)  \big)\]
\[
  +  \nabla_x \Gamma^{\gamma_1} \phi(t,x+\hat{v}t)\big(b_{\gamma_1, \gamma_2}^{\alpha, \kappa;3}(x, v)\nabla_v^{\gamma_1}  g(t,x,v)  - t b_{\gamma_1, \gamma_2}^{\alpha, \kappa;4}(x, v) \nabla_x \nabla_v^{\gamma_2}g(t,x,v)  \big)\big] d x,
\]
where 
\be
c_{\gamma_1}( t,x,v):=\psi_{\leq  -10}(|x|(1+|v|^2)/(1+|t|)) |t|^{|\gamma_1|}||t|-|x+\hat{v}t||^{-|\gamma_1|}, \quad \Longrightarrow | c_{\gamma_1}( t,x,v)|\lesssim (1+|v|)^{2|\gamma_1|},
\ee
and  the coefficients $\alpha_{\gamma_1, \gamma_2}^{\alpha, \kappa,i}(v) $ and $b_{\gamma_1, \gamma_2}^{\alpha, \kappa,i}(x,v) $, $i\in \{1,2,3,4\}$, satisfy the following estimate,
\be\label{may16eqn61}
|\alpha_{\gamma_1, \gamma_2}^{\alpha, \kappa,i}(v) | +  |b_{\gamma_1, \gamma_2}^{\alpha, \kappa,i}(x,v) | + |t| |\nabla_x b_{\gamma_1, \gamma_2}^{\alpha, \kappa,i}(x,v) | \lesssim (1+|v|), \quad 
\ee
which can be derived from the estimates (\ref{may16eqn91}) and (\ref{may16eqn92}).

Due to the high order weight function $\omega_{\beta}^{\vec{0}}(t,x,v)$ (see (\ref{highorderweight}) ) we associated with $\Lambda^{\beta} g(t,x,v)$, the inverse of  weight function provides fast decay rate if $|x|(1+|v|^2)\gtrsim (1+|t|)$. As a result, the following estimate holds after using the $L^2_x-L^2_xL^2_v$ type bilinear estimate, 
\be\label{may16eqn194}
 \sum_{\kappa\in \mathcal{B}, |\kappa|, |\gamma_1|+|\gamma_2|\leq N_0, |\gamma_2|\leq N_0-1} \| \widetilde{\omega_{ }^{\alpha}}(  v) H_{\gamma_1, \gamma_2}^{\alpha, \kappa}(t,v)\|_{L^2_v}\lesssim (1+|t|)^{-2} E_{\textup{high}}^{\phi}(t)E_{\textup{high}}^{f}(t)\lesssim (1+|t|)^{-2+2\delta}\epsilon_1^2.  
\ee

 It remains to estimate $K_{\gamma_1, \gamma_2}^{\alpha, \kappa}(t,v)$.  We emphasize that, due to the cutoff function $\psi_{\leq  -10}(|x|(1+|v|^2)/(1+|t|))$ in $K_{\gamma_1, \gamma_2}^{\alpha, \kappa}(t,v)$,  the following estimate holds, 
 \[
||t|-|x+\hat{v} t ||\gtrsim |t|(1+|v|^2)^{-1}.
 \]
  For fixed  $\alpha,\kappa, \gamma_1, \gamma_2$, we do
dyadic decomposition for the scalar field. As a result, we have
\be\label{may16eqn121}
 K_{\gamma_1, \gamma_2}^{\alpha, \kappa}(t,v)= \sum_{k\in \mathbb{Z}} J_{\gamma_1, \gamma_2}^{\alpha, \kappa;k}(t,v),
 \ee
 \[
J_{\gamma_1, \gamma_2}^{\alpha, \kappa;k}(t,v):=\int_{\R^3}c_{\gamma_1}( t,x,v)  \big[P_k\big( \p_t \Gamma^{\gamma_1} \phi\big)(t,x+\hat{v}t)\big(b_{\gamma_1, \gamma_2}^{\alpha, \kappa;1}(x,v)\nabla_v^{\gamma_2}  g(t,x,v) - t b_{\gamma_1, \gamma_2}^{\alpha, \kappa;2}(x, v)\nabla_x \nabla_v^{\gamma_2}g(t,x,v)  \big)
 \]
\be\label{may16eqn19}
   +  P_k\big(\nabla_x \Gamma^{\gamma_1} \phi\big)(t,x+\hat{v}t)\big(b_{\gamma_1, \gamma_2}^{\alpha, \kappa;3}(x, v)\nabla_v^{\gamma_2}  g(t,x,v)  - t b_{\gamma_1, \gamma_2}^{\alpha, \kappa;4}( x,v) \nabla_x \nabla_v^{\gamma_2}g(t,x,v)  \big)\big] d x.
\ee

 Based on the size of $k$, we separate into two case as follows.

\noindent $\bullet$\quad If $k\geq 0$. 

If $|\gamma_1|\leq 10$,   we use the  $L^\infty_{x }-L^1_{x}L^2_v$-type bilinear estimate by putting $P_k[\p_t \Gamma^{\gamma_1}\phi](t, x+\hat{v}t)$ and $P_k[\nabla_x\Gamma^{\gamma_1} \phi ](t,$ $ x+\hat{v}t)$ in $L^\infty_x$
and  the linear decay estimate (\ref{noveqn935}) in Lemma \ref{sharplinftydecay}. If $|\gamma_1|\geq 10$, we redo the argument used in the estimate $H_k$ (see (\ref{jan16eqn67})) in the proof of Lemma \ref{fixedtimeestimate2}. Recall the estimates (\ref{feb12eqn161}) and (\ref{feb13eqn19}).  As a result, from the estimate of coefficients in (\ref{may16eqn61}), we have 
\[
\sum_{\kappa\in \mathcal{B}, |\kappa|, |\gamma_1|+|\gamma_2|\leq N_0, |\gamma_2|\leq N_0-1} \sum_{k\geq 0} \|\widetilde{\omega_{ }^{\alpha}}(  v) J_{\gamma_1, \gamma_2}^{\alpha, \kappa;k}(t,v) \|_{L^2_v}\lesssim \sum_{k\geq 0}(1+|t|)^{-2} 2^{-k} \big(E_{\textup{high}}^{\phi}(t)  +E_{\textup{low}}^{\phi}(t)  \big)E_{\textup{high}}^{f}(t)
\]
\be\label{may16eqn110}
 +  |t|^{-3/2+2\delta} \epsilon_0\lesssim  |t|^{-3/2+2\delta} \epsilon_0. 
\ee
 \noindent $\bullet$\quad If $k\leq 0$.

 For this case,    we  first do integration by parts in ``$x$'' in (\ref{may16eqn19}) to move around the spatial derivative ``$\nabla_x$'' in front of   $ \nabla_x \nabla_v^{\gamma_2  } g(t, x,v)$ .  

 If $|\gamma_1|\leq 10$,  similar to the strategies we used in the case   $k\geq 0 $, we use the $L^\infty_{x }-L^1_{x}L^2_v$-type bilinear estimate. As a result,  from    the linear decay estimates (\ref{noveqn555}) in Lemma \ref{twistedlineardecay}
and  the estimate  (\ref{noveqn935}) in Lemma \ref{sharplinftydecay} for the scalar field and the estimate of coefficients in (\ref{may16eqn61}),  we have 
\[
 \sum_{\kappa\in \mathcal{B}, |\kappa|, |\gamma_1|+|\gamma_2|\leq N_0, |\gamma_2|\leq N_0-1, |\gamma_1|\leq 10}\sum_{k\in \mathbb{Z}, k\leq 0} \| \widetilde{\omega_{ }^{\alpha}}(  v)  J_{\gamma_1, \gamma_2}^{\alpha, \kappa;k}(t,v)\|_{L^2_v}\lesssim\sum_{k\in \mathbb{Z}, k\leq 0} \min\{(1+|t|)^{-2}, (1+|t|)^{-1} 2^{k} \} 
\]
\be\label{may27eqn1}
\times  E_{\textup{low}}^{\phi}(t)  E_{\textup{high}}^{f}(t) \lesssim (1+|t|)^{-2+4\delta}  \epsilon_1^2.	
\ee

If   $|\gamma_1|\geq 10$, then we  rerun the argument used in the estimate of $H_k$, $k\leq 0$, in the proof of Lemma    \ref{fixedtimeestimate2}. If moreover $2^k\leq |t|^{-1}$, then with minor modification in the estimate (\ref{jan16eqn61}), we have
\[
 \sum_{\kappa\in \mathcal{B}, |\kappa|, |\gamma_1|+|\gamma_2|\leq N_0, |\gamma_2|\leq N_0-1, |\gamma_1|> 10}\sum_{k\in \mathbb{Z}, 2^k\leq |t|^{-1}} \| \widetilde{\omega_{ }^{\alpha}}(  v)  J_{\gamma_1, \gamma_2}^{\alpha, \kappa;k}(t,v)\|_{L^2_v}\lesssim \sum_{k\in \mathbb{Z}, 2^k\leq |t|^{-1}}  \big(2^{k/2}  + |t|2^{3k/2} \big)\epsilon_0 \]
\be\label{may27eqn2}
 \times(1+2^k|t|) (1+|t|)^{-3/2} E_{\textup{high}}^{f;2}(t) \lesssim (1+|t|)^{-2+\delta}\epsilon_0. 
\ee

Alternatively, if $2^{k}\geq |t|^{-1}$, then as what we did in  (\ref{jan16eqn69}), in (\ref{may16eqn19}), we   use the decomposition of the profile in (\ref{feb1eqn1}) for $\phi^{\gamma_1}$ and $\p_t \phi^{\gamma_1}$. From the estimate (\ref{multilinear1jan16})  in Lemma \ref{multilinearlemma1} and the estimate (\ref{bilineardensitylargek}) in Lemma \ref{bilineardensitylemma},  the following estimate holds  because of  the extra spatial derivative that comes from the integration by parts in ``$x$'',
\[
 \sum_{\kappa\in \mathcal{B}, |\kappa|, |\gamma_1|+|\gamma_2|\leq N_0, |\gamma_2|\leq N_0-1, |\gamma_1|> 10}\sum_{k\in \mathbb{Z},   |t|^{-1}\leq 2^k\leq 1} \| \widetilde{\omega_{ }^{\alpha}}(  v)  J_{\gamma_1, \gamma_2}^{\alpha, \kappa;k}(t,v)\|_{L^2_v}\lesssim \sum_{k\in \mathbb{Z},   |t|^{-1}\leq 2^k\leq 1}  |t|^{-3/2+3\delta} 2^{ k/2 } \epsilon_1^2.
\]
\be\label{may27eqn4}
+ |t|^{-2}\epsilon_1^2 + |t|^{-3}2^{-k} \epsilon_1^2\lesssim  |t|^{-3/2+3\delta}\epsilon_0.
\ee
To sum up, from the estimates  (\ref{may27eqn1}), (\ref{may27eqn2}), and (\ref{may27eqn4}), the following estimate holds, 
\be\label{may17eqn161}
 \sum_{\kappa\in \mathcal{B}, |\kappa|, |\gamma_1|+|\gamma_2|\leq N_0, |\gamma_2|\leq N_0-1    } \sum_{k\in \mathbb{Z}, k\leq 0} \| \widetilde{\omega_{ }^{\alpha}}(  v)  J_{\gamma_1, \gamma_2}^{\alpha, \kappa;k}(t,v)\|_{L^2_v} \lesssim  |t|^{-3/2+3\delta}\epsilon_0.
\ee
From the estimates (\ref{may16eqn110}) and (\ref{may17eqn161}) and the decomposition (\ref{may16eqn121}), we have
\be\label{may16eqn196}
  \sum_{\kappa\in \mathcal{B}, |\kappa|, |\gamma_1|+|\gamma_2|\leq N_0, |\gamma_2|\leq N_0-1    } \| \widetilde{\omega_{ }^{\alpha}}(  v) K_{\gamma_1, \gamma_2}^{\alpha, \kappa}(t,v)\|_{L^2_v}\lesssim  |t|^{-3/2+3\delta}\epsilon_0.
\ee
Recall the decomposition (\ref{may16eqn151}). Our desired estimate (\ref{feb20eqn10}) follows directly from the estimates  (\ref{may16eqn194}) and (\ref{may16eqn196}).
\end{proof}
\subsection{Toolkit}\label{multilinearestimate}

In this subsection, we prove some basic estimates used in the previous two subsections.

The first basic estimate is a linear decay estimate for the scalar field part, which show that the nonlinear solutions decays sharply at rate $1/\big( (1+|t| )(1+||t|-|x||) \big)$ as long as the low order energy of scalar  field doesn't grow over time.  It is a natural application of our general  linear decay estimate of the half wave in Lemma  \ref{twistedlineardecay}.

\begin{lemma}\label{sharplinftydecay}
Given  Fourier multiplier operator   $T$ with Fourier multiplier symbol $m(\xi)\in \mathcal{S}^\infty$, the following estimate holds,
\be\label{noveqn935}
\sum_{\begin{subarray}{c}
u\in\{\p_t\phi ^\rho,\nabla_x \phi^\rho\},
\rho \in \mathcal{B}, |\rho|\leq 13\\
\end{subarray}} | T(u)(t, x)|\lesssim  (1+|t|)^{-1} (1+||t|-|x||)^{-1} \|m(\xi)\|_{\mathcal{S}^\infty }  E_{\textup{low}}^{\phi}(t). 
\ee
\end{lemma}
\begin{proof}
We first do dyadic decomposition for the frequency of $u$. As a result, we have
\[
T(u)(t, x)=\sum_{k\in \mathbb{Z}} T_k(u), \quad T_k=T\circ P_k.
\]
Recall (\ref{feb8eqn31}) and (\ref{feb8eqn32}).  From the linear decay estimate (\ref{noveqn555}) in Lemma \ref{twistedlineardecay},  we have
\be\label{noveqn930}
\sum_{\begin{subarray}{c}
u\in\{\nabla \phi^\rho,\p_t \phi^\rho\},
\rho \in \mathcal{B}, |\rho|\leq 13\\
\end{subarray}} | T_k(u)(t,x)| \lesssim \sum_{
\begin{subarray}{c}
i=0, 1,2,
\alpha\in \mathcal{B},|\alpha|\leq  15\\
\end{subarray} } (1+|t|+|x|)^{-1} 2^{k-2k_{+	}}\| m(\xi)\|_{\mathcal{S}^\infty_k} \| P_{k}h^\alpha (t) \|_{X_i}. 
\ee
From the above estimate,   our desired estimate (\ref{noveqn935}) holds straightforwardly if $||t|-|x||\leq 1$. Hence, it would be sufficient to consider the case $||t|-|x||\geq 1$.

From the equalities (\ref{noveqn171}) and (\ref{feb12eqn39}) in Lemma \ref{tradethreetimes},  we  can trade regularity for the weight of the   modulation ``$|t|-|x|$''. More precisely, from the  estimates of coefficients in  (\ref{noveqn152}),  (\ref{feb12eqn2}), and (\ref{feb12eqn10}), the estimates of symbols in (\ref{feb12eqn1}) and  (\ref{feb12eqn10}),  and the linear decay estimate (\ref{noveqn555}) in Lemma \ref{twistedlineardecay},  we have 
\[
\sum_{\begin{subarray}{c}
u\in\{\nabla\phi^\rho,\p_t \phi^\rho\},
\rho \in \mathcal{B}, |\rho|\leq 13\\
\end{subarray}} \big|\big(|t|-|x| \big)^3 T_k(u)(t,x)\big|\lesssim \sum_{\alpha\in\mathcal{B}, |\alpha|\leq 16} 2^{-2k} (1+|t|+|x|)^{-1} \|m(\xi)\|_{\mathcal{S}^\infty_k}\big(\| P_k h^\alpha(t)\|_{X_0} \]
\be\label{noveqn938}
 + \| P_k h^\alpha(t)||_{X_1}  \big)+ \big( 2^{-5k} 2^{k+3k} + ||t|-|x|| 2^{-4k}2^{k+3k}  +||t|-|x||^2 2^{-3k}2^{k+3k}  \big)\|m(\xi)\|_{\mathcal{S}^\infty_k}\|\widehat{\p_t h^\alpha}(t, \xi)\psi_k(\xi) \|_{L^\infty_\xi}.
\ee

Recall the low order energy of the nonlinear wave part  $E_{\textup{low}}^{\phi}(t)$ in (\ref{secondorderloworder}), from the above estimate, we have 
\[
\sum_{\begin{subarray}{c}
u\in\{\nabla\phi^\rho,\p_t \phi^\rho\},
\rho \in \mathcal{B}, |\rho|\leq 13\\
\end{subarray}} \big| T_k(u)(t,x)\big|\lesssim (||t|-|x||)^{-1}   \|m(\xi)\|_{\mathcal{S}^\infty_k}  E_{\textup{low}}^{\phi}(t)
\]
\be\label{noveqn940}
\times  \big[ (1+|t|)^{-1} 2^{ -2k}(||t|-|x||)^{-2}  + (1+|t|)^{-1} 2^{ -k} (||t|-|x||)^{-1} + (1+|t|)^{-2} 2^{  -k}\big]. 
\ee
Therefore, after optimizing  the estimates (\ref{noveqn930}) and (\ref{noveqn940}), we have
\[
\sum_{\begin{subarray}{c}
u\in\{\nabla\phi^\rho,\p_t \phi^\rho\},
\rho \in \mathcal{B}, |\rho|\leq 13\\
\end{subarray}}  | T(u)(t, x)| \lesssim    (1+|t|)^{-1}(||t|-|x||)^{-1}\|m(\xi)\|_{\mathcal{S}^\infty_k}  E_{\textup{low}}^{\phi}(t) .
\]
Hence finishing the proof of the desired estimate (\ref{noveqn935}).  
\end{proof}

Next, we prove   trilinear  estimates, which are used in the energy estimate of the Vlasov part for the case when the scalar field has relatively more derivatives than the Vlasov part in the nonlinearity of the Vlasov equation. More precisely, the estimate of $ H_k^1(t,x,v)$  in (\ref{2020feb13eqn1}).  In our desired  trilinear estimates, instead of using the decay estimate of density type function in Lemma \ref{decayestimateofdensity} directly,  we exploit the smallness of the space-resonance set which enable us to get a improved decay estimate.

 \begin{lemma}\label{multilinearlemma1}
Given any fixed signs $\mu, \nu\in\{+,-\}$, fixed time $t\in \R_{+}$, fixed $k_1,k_2 \in \mathbb{Z}$.  Moreover, given     any functions $f_1, f_2:\R_t\times \R_x^3 \rightarrow \mathbb{C}$, and any distribution function $g:\R_t\times \R_x^3\times \R_v^3\rightarrow \mathbb{R}$, we define a trilinear form as follows,
\be\label{jan16eqn1}
T(f_1,f_2, g):= \int_{\R^3} \int_{\R^3} e^{-i\mu t\d} P_{k_1}[f_1](t,x+\hat{v}t )e^{-i\nu t\d}  P_{k_2}[f_2](t,x+\hat{v}t ) g(t,x,v) d x d v,
\ee
then the following estimate holds, 
\[
|T(f_1,f_2,g)| \lesssim \sum_{|\alpha|\leq 4} (1+|t|)^{-5} \|(1+|x|)^{2}(1+|v|)^{25} \nabla_v^\alpha g(t,x,v)\|_{L^1_{x,v}}\big(  2^{-k_{1,-}}\|\widehat{f_1}(t,\xi)\psi_{k_1}(\xi) \|_{L^2} 
\]
\be\label{multilinear1jan16}
 +\|\nabla_\xi \widehat{f_1}(t,\xi)\psi_{k_1}(\xi) \|_{L^2} \big) \big(  2^{-k_{2,-}}\|\widehat{f_2}(t,\xi)\psi_{k_2}(\xi) \|_{L^2}  +\|\nabla_\xi \widehat{f_2}(t,\xi)\psi_{k_2}(\xi) \|_{L^2} \big).
\ee
Moreover,  if $|k_1-k_2|\geq 5$, then the following estimate holds, 
\[
|T(f_1,f_2,g)| \lesssim \sum_{|\alpha|\leq 4} (1+|t|)^{-5} 2^{-\max\{k_1,k_2\}} \|(1+|x|)^{2}(1+|v|)^{25} \nabla_x \nabla_v^\alpha g(t,x,v)\|_{L^1_{x,v}}\big(  2^{-k_{1,-}}\|\widehat{f_1}(t,\xi)\psi_{k_1}(\xi) \|_{L^2} 
\]
\be\label{trilinearfeb20est1}
 +\|\nabla_\xi \widehat{f_1}(t,\xi)\psi_{k_1}(\xi) \|_{L^2} \big) \big(  2^{-k_{2,-}}\|\widehat{f_2}(t,\xi)\psi_{k_2}(\xi) \|_{L^2}  +\|\nabla_\xi \widehat{f_2}(t,\xi)\psi_{k_2}(\xi) \|_{L^2} \big).
\ee
 \end{lemma}
 \begin{proof}
Recall (\ref{jan16eqn1}). Note that the following equality holds on the Fourier side, 
\be\label{feb20eqn30}
T(f_1,f_2, g):= \int_{\R^3}  \int_{\R^3}\int_{\R^3} e^{it\hat{v}\cdot \xi - i t \mu |\xi-\eta|-i t \nu |\eta|} \overline{\widehat{g}(t,\xi,v)}    \widehat{f_1}(t,\xi-\eta)\widehat{f_2}(t, \eta)\psi_{k_1}(\xi-\eta)\psi_{k_2}(\eta) d\eta d \xi  d v.
\ee
From the above formula of $T(f_1,f_2, g)$,  our desired estimate (\ref{multilinear1jan16}) holds straightforwardly if $|t|\leq 1$. Hence, we restrict ourself to the case when $|t|\geq 1$.

Firstly, we do integration by parts in ``$\xi$'' once. As a result, we have
\be\label{jan16eqn29}
T(f_1,f_2, g)=T_1(f_1,f_2, g)+ T_2(f_1,f_2, g),
 \ee
where
\[
T_1(f_1,f_2, g):= \frac{i}{t}\int_{\R^3} \int_{\R^3} \int_{\R^3} e^{it\hat{v}\cdot \xi - i t \mu |\xi-\eta|-i t \nu |\eta|}  \nabla_\xi\overline{\widehat{g}(t,\xi,v)} \cdot \big(  \frac{\hat{v}- \mu(\xi-\eta)/|\xi-\eta|}{\big|\hat{v}- \mu(\xi-\eta)/|\xi-\eta|\big|^2}\widehat{f_1}(t,\xi-\eta) \]
\[
\times\psi_{k_1}(\xi-\eta) \big)    \widehat{f_2}(t, \eta) \psi_{k_2}(\eta) d\eta d \xi  d v,
\]
\[
T_2(f_1,f_2, g):= \frac{i}{t}\int_{\R^3} \int_{\R^3} \int_{\R^3} e^{it\hat{v}\cdot \xi - i t \mu |\xi-\eta|-i t \nu |\eta|} \overline{\widehat{g}(t,\xi,v)}  \nabla_\xi \cdot \big(  \frac{\hat{v}- \mu(\xi-\eta)/|\xi-\eta|}{\big|\hat{v}- \mu(\xi-\eta)/|\xi-\eta|\big|^2} \widehat{f_1}(t,\xi-\eta)  \]
\[
\times  \psi_{k_1}(\xi-\eta)\big)  \widehat{f_2}(t, \eta) \psi_{k_2}(\eta) d\eta d \xi  d v.
\]
For $T_1(f_1,f_2, g)$, we do integration by parts in ``$\xi$'' one more time. As a result, we have
\[
T_1(f_1,f_2, g):= \frac{-1}{t^2}\int_{\R^3} \int_{\R^3} \int_{\R^3} e^{it\hat{v}\cdot \xi - i t \mu |\xi-\eta|-i t \nu |\eta|} \widehat{f_2}(t, \eta)\psi_{k_2}(\eta) \nabla_\xi\cdot\big[ \frac{\hat{v}- \mu(\xi-\eta)/|\xi-\eta|}{\big|\hat{v}- \mu(\xi-\eta)/|\xi-\eta|\big|^2}     \]
\be\label{jan16eqn5}
\times  \nabla_\xi \overline{\widehat{g}(t,\xi,v)} \cdot \big(  \frac{\hat{v}- \mu(\xi-\eta)/|\xi-\eta|}{\big|\hat{v}- \mu(\xi-\eta)/|\xi-\eta|\big|^2} \psi_{k_1}(\xi-\eta) \widehat{f_1}(t,\xi-\eta) \big)  \big]   d\eta d \xi  d v.
\ee
For $T_2(f_1,f_2,g)$, we first switch the role of ``$\xi-\eta$'' and ``$\eta$'' and then do integration by parts in ``$\xi$'' once. As a result, we have
 \[
T_2(f_1,f_2, g):= \frac{i}{t}\int_{\R^3} \int_{\R^3} \int_{\R^3} e^{it\hat{v}\cdot \xi - i t \mu | \eta|-i t \nu |\xi-\eta|} \overline{\widehat{g}(t,\xi,v)}  \nabla_\eta \cdot \big(  \frac{\hat{v}- \mu  \eta /| \eta|}{\big|\hat{v}- \mu  \eta /| \eta|\big|^2}   \widehat{f_1}(t, \eta)\]
\[
\times\psi_{k_1}( \eta)\big)  \widehat{f_2}(t, \xi-\eta) \psi_{k_2}(\xi-\eta) d\eta d \xi  d v  
\]
\[
= \frac{-1}{t^2}\int_{\R^3}  \int_{\R^3} \int_{\R^3} e^{it\hat{v}\cdot \xi - i t \mu |\eta|-i t \nu |\xi-\eta|}  \nabla_\xi\cdot \big[\frac{\hat{v}- \nu(\xi-\eta)/|\xi-\eta|}{\big|\hat{v}- \nu(\xi-\eta)/|\xi-\eta|\big|^2}     \overline{\widehat{g}(t,\xi,v)}  \widehat{f_2}(t,\xi- \eta)\psi_{k_2}(\xi-\eta) \big]\]
\[
\times   \nabla_\eta \cdot \big(  \frac{\hat{v}- \mu  \eta /| \eta|}{\big|\hat{v}- \mu  \eta /| \eta|\big|^2}   \widehat{f_1}(t, \eta) \psi_{k_1}( \eta)\big)  d\eta d \xi  d v. 
\]
Now, we do dyadic decomposition for the size of frequency of ``$\xi$'' and have the following decomposition, 
\be\label{jan16eqn26}
T_1(f_1,f_2, g)= \sum_{k\in \mathbb{Z}} T_{1}^k (f_1,f_2, g), \quad T_2(f_1,f_2, g)= \sum_{k\in \mathbb{Z}} T_{2}^k (f_1,f_2, g),
\ee
where
\[
T_1^k(f_1,f_2, g):= \frac{-1}{t^2}\int_{\R^3} \int_{\R^3} \int_{\R^3} e^{it\hat{v}\cdot \xi - i t \mu |\xi-\eta|-i t \nu |\eta|} \widehat{f_2}(t, \eta)\psi_{k_2}(\eta) \nabla_\xi\cdot\big[ \frac{\hat{v}- \mu(\xi-\eta)/|\xi-\eta|}{\big|\hat{v}- \mu(\xi-\eta)/|\xi-\eta|\big|^2}     \]
\be\label{jan16eqn10}
\times  \nabla_\xi \overline{\widehat{g}(t,\xi,v)} \cdot \big(  \frac{\hat{v}- \mu(\xi-\eta)/|\xi-\eta|}{\big|\hat{v}- \mu(\xi-\eta)/|\xi-\eta|\big|^2} \psi_{k_1}(\xi-\eta) \widehat{f_1}(t,\xi-\eta) \big)  \big]  \psi_k(\xi) d\eta d \xi  d v,
\ee
\[
T_2^k (f_1,f_2, g):=  \frac{-1}{t^2}\int_{\R^3}  \int_{\R^3} \int_{\R^3} e^{it\hat{v}\cdot \xi - i t \mu |\eta|-i t \nu |\xi-\eta|}  \nabla_\xi\cdot \big[\frac{\hat{v}- \nu(\xi-\eta)/|\xi-\eta|}{\big|\hat{v}- \nu(\xi-\eta)/|\xi-\eta|\big|^2}   \overline{\widehat{g}(t,\xi,v)}  \widehat{f_2}(t,\xi- \eta)  \]
\be\label{jan16eqn12}
\times  \psi_{k_2}(\xi-\eta) \big] \nabla_\eta \cdot \big(  \frac{\hat{v}- \mu  \eta /| \eta|}{\big|\hat{v}- \mu  \eta /| \eta|\big|^2}   \widehat{f_1}(t, \eta) \psi_{k_1}( \eta)\big) \psi_k(\xi) d\eta d \xi  d v. 
\ee
 
On one hand, if we use the volume of support of ``$\xi$'' and the Cauchy-Schwarz inequality for the integration with respect to ``$\eta$'', we have
\[
|T_1^k (f_1,f_2, g)| + |T_2^k (f_1,f_2, g)| \lesssim \sum_{0\leq |\beta|\leq 2} (1+|t|)^{-2} 2^{3k} \big(\| (1+|v|)^{10}  \nabla_\xi^\beta \widehat{g}(t,\xi, v)\|_{L^\infty_\xi L^1_v} \big(  2^{-k_{1,-}}\|\widehat{f_1}(t,\xi)\psi_{k_1}(\xi) \|_{L^2} 
\]
\be\label{jan16eqn20}
  +\|\nabla_\xi \widehat{f_1}(t,\xi)\psi_{k_1}(\xi) \|_{L^2} \big)\big(  2^{-k_{2,-}}\|\widehat{f_2}(t,\xi)\psi_{k_2}(\xi) \|_{L^2}  +\|\nabla_\xi \widehat{f_2}(t,\xi)\psi_{k_2}(\xi) \|_{L^2} \big).
\ee
On the other hand, if we do integration by parts in ``$v$''four times for $T_1^k (f_1,f_2, g)$ and $T_2^k (f_1,f_2, g)$, then the following estimate holds after using the volume of support of ``$\xi$'' and  the Cauchy-Schwarz inequality for the integration with respect to ``$\eta$'',
\[
|T_1^k (f_1,f_2, g)| + |T_2^k (f_1,f_2, g)| \lesssim \sum_{|\alpha|\leq 4, |\beta|\leq 2} (1+|t|)^{-6} 2^{-k}  \| (1+|v|)^{25 } \nabla_v^\alpha \nabla_\xi^\beta \widehat{g}(t,\xi, v)\|_{L^\infty_\xi L^1_v} 
\]
\be\label{jan16eqn21}
 \times \big(  2^{-k_{1,-}}\|\widehat{f_1}(t,\xi)\psi_{k_1}(\xi) \|_{L^2} +\|\nabla_\xi \widehat{f_1}(t,\xi)\psi_{k_1}(\xi) \|_{L^2} \big) \big(  2^{-k_{2,-}}\|\widehat{f_2}(t,\xi)\psi_{k_2}(\xi) \|_{L^2}  +\|\nabla_\xi \widehat{f_2}(t,\xi)\psi_{k_2}(\xi) \|_{L^2} \big).
\ee
Recall the decompositions in (\ref{jan16eqn29}) and (\ref{jan16eqn26}). From the estimates (\ref{jan16eqn20}) and (\ref{jan16eqn21}), we have
\[
|T(f_1,f_2,g)| \lesssim \sum_{|\alpha|\leq 4} (1+|t|)^{-5} \|(1+|x|)^{2}(1+|v|)^{25} \nabla_v^\alpha g(t,x,v)\|_{L^1_{x,v}}\big(  2^{-k_{1,-}}\|\widehat{f_1}(t,\xi)\psi_{k_1}(\xi) \|_{L^2} 
\]
\be
 +\|\nabla_\xi \widehat{f_1}(t,\xi)\psi_{k_1}(\xi) \|_{L^2} \big) \big(  2^{-k_{2,-}}\|\widehat{f_2}(t,\xi)\psi_{k_2}(\xi) \|_{L^2}  +\|\nabla_\xi \widehat{f_2}(t,\xi)\psi_{k_2}(\xi) \|_{L^2} \big).
\ee
Hence finishing the proof of the desired estimate (\ref{multilinear1jan16}). 

 Due to the symmetry between $k_1$ and $k_2$, it would be sufficient to consider the case when $k_2\leq k_1-10$. Recall (\ref{feb20eqn30}), we have ``$|\xi|\sim 2^{k_1}$''   for this case. With minor modifications, our desired estimate (\ref{trilinearfeb20est1}) holds from redoing the proof of the estimate (\ref{multilinear1jan16})
 \end{proof}

Lastly, we prove bilinear estimates for the Vlasov-Vlasov type interaction, which are used in the energy estimate of the Vlasov part for the case when the scalar field has relatively more derivatives than the Vlasov part in the nonlinearity of the Vlasov equation. More precisely, the estimate of $ H_k^2(t,x,v)$  in (\ref{noveq291}).  

 For any fixed sign $\mu\in\{+,-\}$, any two distribution functions $f_1(t,x,v)$ and $f_2(t,x,v)$, any fixed $k\in\mathbb{Z}$, any symbol $m(\xi,v)\in L^\infty_v \mathcal{S}^\infty_k$, and  any differentiable coefficient $c(v)$,  we define a bilinear operator as follows, 
\be\label{noveq260}
B_k  (f_1,f_2)(t,x,v):= f_1(t,x,v) E (P_{k}[f_2(t)])(x+a(v)t), 
\ee
where
\[
E(P_k[f])(t,x):=\int_{\R^3}\int_{\R^3} e^{i x \cdot \xi} e^{-i \mu t \hat{u}\cdot \xi} {c(u)  m(\xi,u) \psi_k(\xi)}  \widehat{f}(t, \xi, u) d\xi d u. 
\]

As stated in the following Lemma, we have two bilinear estimates for  the above defined bilinear operator. 
 \begin{lemma}\label{bilineardensitylemma}
For any fixed $t\in \R, |t|\geq 1$, and any localized differentiable function $f_3(t,v):\R_t\times \R_v^3\longrightarrow \mathbb{C}$,
   the following bilinear estimate  holds for the bilinear operators defined in  \textup{(\ref{noveq260})},
 \[
\| B_k (f_1,f_2)(t,x  ,v)\|_{L^2_x L^2_v}  \lesssim \sum_{|\alpha|\leq 5} \big( \|m(\xi,v)\|_{L^\infty_v\mathcal{S}^\infty_k} + \|m(\xi,v)\|_{L^\infty_v\mathcal{S}^\infty_k}\big) \big[ |t|^{-2}2^{k} \|\big(|c(v)|+|\nabla_v c(v)|\big) f_3(t,v) \|_{L^2_v} 
     \]
 \[
+ |t|^{-3}  2^{ k }    \|(1+ |v|+|x|)^{20}c(v) f_2(t,x,v) \|_{  L^2_x L^2_v} +  |t|^{-3} \| c(v)\big( \widehat{f_2}(t,0,v ) -\nabla_v\cdot f_3(t, v)\big)\|_{L^2_v}   \big]
 \]
 \be\label{bilineardensity}
  \times  \|(1+ |v|+|x|)^{20}  \nabla_v^\alpha f_1(t,x,v)\|_{L^2_xL^2_v}, \quad \textit{if\,\,} k\in \mathbb{Z}, |t|^{-1}\lesssim  2^{k}\leq 1.
 \ee
Alternatively,  the following bilinear estimate holds for any $k\in \mathbb{Z}$,
 \[
\| B_k (f_1,f_2)(t,x  ,v)\|_{L^2_x L^2_v}  \lesssim \sum_{|\alpha|\leq 5}\min\{ |t|^{-3}    ,2^{3k}\}   \|m(\xi,v)\|_{L^\infty_v \mathcal{S}^\infty_k}    
 \|(1+ |v|+|x|)^{20}c(v) f_2(t,x,v) \|_{  L^2_x L^2_v}   \]
\be\label{bilineardensitylargek}
 \times \|(1+ |v|+|x|)^{20} \nabla_v^\alpha f_1(t,x,v)\|_{L^2_xL^2_v} .
\ee

 \end{lemma}

  \begin{proof}
  See \cite{wang3}[Lemma 3.2\& Lemma 3.3].
  \end{proof}

\appendix
\section{Commutation rules}\label{toolkit}
  Although it is tedious to compute the commutation rules and check all the first order derivatives in  $\mathcal{S}$ acting on the inhomogeneous modulation $\tilde{d}(t,x,v)$ and , for the sake of readers, we do detailed computations and  prove   Lemma \ref{firstordercommutationlemma} and   Lemma \ref{derivativesofcoefficient} in this appendix.

In the following lemma, we compute several basic quantities to be used later, which also directly imply our desired results in    Lemma \ref{derivativesofcoefficient}. 
\begin{lemma}\label{derivativeofmodulation}
For any $\rho \in \mathcal{S}, |\rho|=1$, the following equality holds, 
\be\label{jan26eqn100}
\Lambda^\rho(\tilde{d}(t,x,v)) = e^\rho_1(x,v) \tilde{d}(t,x,v) + e^\rho_2(x,v),  \quad 
D_v(\tilde{d}(t,x,v)) = \hat{e}_1(x,v) \tilde{d}(t,x,v) + \hat{e}_2(x,v),
\ee
where the coefficients  satisfy the following estimate, 
\be\label{jan26eqn101}
|e^\rho_1(x,v)|+ |e^\rho_2(x,v)| + |\hat{e}_1(x,v)|+ |\hat{e}_2(x,v)|   \lesssim 1,\quad |\hat{e}_2(x,v)|\psi_{\geq 2}(|x|)=0.
\ee
Moreover, the following rough estimate holds for any $\beta\in \mathcal{S}$, 
\be\label{jan26eqn401}
\sum_{i=1,2}|\Lambda^\beta e^\rho_i(x,v) | + |\Lambda^\beta \hat{e}_i (x,v) |\lesssim (1+|x|)^{|\beta|} (1+|v|)^{|\beta|}.
\ee
\end{lemma}
\begin{proof}
First of all, we compute   two basic quantities which will be used later. Recall (\ref{eqn19}). For any $\mu\in\{+,-\}$, we have
\[
\big(\Omega^v_i-\frac{t}{\sqrt{1+|v|^2}}\Omega^x_i\big)\omega_{\mu}(x,v)= \big[\big(\Omega^v_i-\frac{t}{\sqrt{1+|v|^2}}\Omega^x_i\big) \big( (x\cdot v) + \mu \sqrt{(x\cdot v)^2+|x|^2}\big)\big] 
\vspace{-0.5\baselineskip}
\]
\[
=\big[(\tilde{V}_i \cdot x)+ \mu \frac{(x\cdot v)(\tilde{V}_i \cdot x)}{\sqrt{(x\cdot v)^2+|x|^2}} -\mu \frac{t}{\sqrt{1+|v|^2}}\frac{\tilde{V}_i\cdot x}{\sqrt{(x\cdot v)^2+|x|^2}}\big] 
\]
\be\label{jan25eqn51}
= \frac{-\mu \tilde{V}_i\cdot x }{\sqrt{(x\cdot v)^2+|x|^2} }\big(\frac{t}{\sqrt{1+|v|^2}}- ( x\cdot v    + \mu \sqrt{(x\cdot v)^2+|x|^2} ) \big)= \frac{-\mu \tilde{V}_i\cdot x }{\sqrt{(x\cdot v)^2+|x|^2} }\big(\frac{t}{\sqrt{1+|v|^2}}- \omega_{\mu}(x,v)\big),
\vspace{-0.5\baselineskip}
\ee
\[
\big(S^v -\frac{t}{ (1+|v|^2)^{3/2}}S^x 	\big) \omega_{\mu}(x,v)=   \big(S^v -\frac{t}{ (1+|v|^2)^{3/2}}S^x 	\big) \big( (x\cdot v) + \mu \sqrt{(x\cdot v)^2+|x|^2}\big)  
\]
\[
=(\tilde{v}\cdot x +\mu \frac{(x\cdot v )\tilde{v}\cdot x }{\sqrt{(x\cdot v)^2 +|x|^2}}) - \frac{t}{(1+|v|^2)^{3/2}}(\tilde{v}\cdot v  + \mu \frac{(x\cdot v)(\tilde{v}\cdot v ) + x \cdot \tilde{v}}{\sqrt{(x\cdot v)^2 +|x|^2}}  )
\]
\be\label{jan25eqn52}
=   \frac{-\mu x\cdot \tilde{v}   }{\sqrt{(x\cdot v)^2+|x|^2} }\big(\frac{t}{\sqrt{1+|v|^2}}- \omega_{\mu}(x,v)\big) - \frac{t|v|}{(1+|v|^2)^{3/2}}. 
\ee
Therefore,
\be\label{jan25eqn71}
 \big(S^v -\frac{t}{ (1+|v|^2)^{3/2}}S^x 	\big)\big(  \frac{t}{\sqrt{1+|v|^2}}- \omega_{\mu}(x,v) \big)
=\frac{\mu \tilde{v} \cdot x }{\sqrt{(x\cdot v)^2+|x|^2} }\big(\frac{t}{\sqrt{1+|v|^2}}- \omega_{\mu}(x,v)\big),
\ee
\be\label{jan25eqn710}
  \big(\Omega^v_i-\frac{t}{\sqrt{1+|v|^2}}\Omega^x_i\big) \big(  \frac{t}{\sqrt{1+|v|^2}}- \omega_{\mu}(x,v) \big)
=\frac{\mu \tilde{V}_i \cdot x }{\sqrt{(x\cdot v)^2+|x|^2} }\big(\frac{t}{\sqrt{1+|v|^2}}- \omega_{\mu}(x,v)\big).
\ee 
Now, we consider the case with the cutoff function defined in (\ref{sepeqn975}). From the equalities (\ref{jan25eqn51}), we have, 
\[
\big(\Omega^v_i-\frac{t}{\sqrt{1+|v|^2}}\Omega^x_i\big)\omega(x,v)=    \frac{- \tilde{V}_i\cdot x }{\sqrt{(x\cdot v)^2+|x|^2} }\big(\frac{t}{\sqrt{1+|v|^2}}- \omega_{+}(x,v)\big) \psi_{\geq 0 }((x\cdot v)^2 +|x|^2)
\]
\be\label{sepeq21}
+ 2 \omega_{+}(x,v) {\psi}'_{\geq 0 }((x \cdot v )^2 +|x|^2)\big((x\cdot v)(\tilde{V}_i\cdot x ) -\frac{t}{\sqrt{1+|v|^2}}\tilde{V}_i \cdot x 	 \big)=c_{i}(x, v) \tilde{d}(t,x,v) + e_{i}(x, v),
\ee
where $c_i(x,v)$ and  $e_i(x,v)$, $i\in\{1,2,3\}$, are defined as follows,
\be\label{jan26eqn500}
c_i(x,v)=  -\sqrt{1+|v|^2}  \big[\psi_{\geq 0 }((x\cdot v)^2 +|x|^2)\frac{ \tilde{V}_i\cdot x}{\sqrt{(x\cdot v)^2+|x|^2}} +2 \omega_{+}(x, v){\psi}'_{\geq 0}((x\cdot v)^2 +|x|^2) \tilde{V}_i \cdot  x\big],
\ee
\[
  e_i(x,v)=\frac{ \tilde{V}_i\cdot x }{\sqrt{(x\cdot v)^2+|x|^2} } \omega_{+}(x,v)\psi_{<0}((x\cdot v)^2+|x|^2)  \psi_{\geq 0 }((x\cdot v)^2 +|x|^2)
\]
\be\label{jan26eqn501}
+ 2 \omega_{+}(x,v) {\psi}'_{\geq 0 }((x \cdot v )^2 +|x|^2)\big((x\cdot v)(\tilde{V}_i\cdot x ) -\omega(x,v) \tilde{V}_i \cdot x 	 \big).
\ee
From the detailed formulas of $c_i(x,v)$ and $e_i(x,v)$ in (\ref{jan26eqn500}) and (\ref{jan26eqn501}),  we have 
\be\label{feb11eqn1}
|c_i(x,v)|  \lesssim 1+|v|, \quad |e_i(x,v)|\lesssim 1, \quad |e_i(x,v)|\psi_{\geq 2}(|x|)=0.
\ee

From the equality (\ref{jan25eqn52}), we have
\[
\big(S^v -\frac{t}{ (1+|v|^2)^{3/2}}S^x 	\big)\omega(x,v)= \big[\frac{-  \tilde{v} \cdot x }{\sqrt{(x\cdot v)^2+|x|^2} }\big(\frac{t}{\sqrt{1+|v|^2}}- \omega_{+}(x,v)\big) - \frac{t|v|}{(1+|v|^2)^{3/2}}\big]\]
\[
\times \psi_{\geq 0}(|x|^2+(x\cdot v)^2)
+ 2\omega_{+}(x,v)\psi'_{\geq 0}((x\cdot v)^2 +|x|^2)\big( (x\cdot v)\tilde{v}\cdot x - \frac{t}{(1+|v|^2)^{1/2}} \tilde{v}\cdot x  \big)
\]
\be\label{sepeq6}
= \tilde{c}(x,v)\tilde{d}(t,x,v) -\frac{t|v|}{(1+|v|^2)^{3/2}}\psi_{\geq 0}( (x\cdot v)^2+|x|^2) + \tilde{e}(x,v),
\ee
where
\be\label{jan26eqn502}
\tilde{c}(x,v)= - x\cdot \tilde{v}  \sqrt{1+|v|^2} \big[  \frac{\psi_{\geq 0}((x\cdot v)^2+|x|^2)}{\sqrt{(x\cdot v)^2+|x|^2}} + 2
 \psi'_{\geq 0}((x\cdot v)^2+|x|^2) { \omega_{+}(x, v)}   \big],
\ee
\[
\tilde{e}(x, v)= 2\omega_{+}(x,v)\psi'_{\geq 0}((x\cdot v)^2 +|x|^2)(\tilde{v}\cdot x )\big( x\cdot v-\omega(x,v) \big) \]
\be\label{jan26eqn503}
+\omega_+(x,v) \frac{\tilde{v}\cdot x}{\sqrt{|x|^2 +(x\cdot v)^2}} \psi_{\geq 0}((x\cdot v)^2 +|x|^2) \psi_{< 0}((x\cdot v)^2 +|x|^2).
\ee
From the detailed formulas of $\tilde{c}(x,v)$ and $\tilde{e}(x,v)$ in (\ref{jan26eqn502}) and (\ref{jan26eqn503}), we have
\be\label{feb11eqn2}
|\tilde{c}(x,v)| \lesssim 1, \quad |\tilde{e}(x,v)| \lesssim 1, \quad |\tilde{e}(x,v)|\psi_{\geq 2}(|x|)=0.
\ee

Now, we are ready to compute the quantity $\Lambda^\rho( \tilde{d}(t,x,v))$, where $\rho\in \mathcal{S}, |\rho|=1$. Recall that
\be\label{feb11eqn8}
\tilde{d}(t,x,v)= \frac{1}{\sqrt{1+|v|^2}}\big(  \frac{t}{\sqrt{1+|v|^2}}- \omega(x,v)\big)
\ee
\be\label{jan26eqn129}
= \frac{1}{\sqrt{1+|v|^2}}\big(  \frac{t}{\sqrt{1+|v|^2}} - (x\cdot v + \sqrt{(x\cdot v)^2 +|x|^2}) \psi_{\geq 0}((x\cdot v)^2 +|x|^2).
\ee
A direct computation gives us the following equality,  
\be\label{feb11eqn3}
\nabla_x \tilde{d}(t,x,v) = \frac{-1}{\sqrt{1+|v|^2}} \big(    {v} + \frac{(x\cdot v) v + x}{\sqrt{(x\cdot v)^2+|x|^2}}  \big)\psi_{\geq 0}\big((x\cdot v)^2+|x|^2\big) - \frac{ \omega_{+}(x,v) }{\sqrt{1+|v|^2}}	\psi'_{\geq 0}((x\cdot v)^2 +|x|^2)\big((x\cdot v) v + x  \big).
\ee
Therefore, from the above equality (\ref{feb11eqn3}), we have 
\be\label{feb11eqn11}
|\tilde{v} \cdot \nabla_x \tilde{d}(t,x,v)|\lesssim 1, \quad \sum_{i=1,2,3} |\tilde{V}_i \cdot \nabla_x \tilde{d}(t,x,v)  | \lesssim (1+|v|)^{-1}. 
\ee

Moreover, from the 	equalities (\ref{sepeq21}) and (\ref{sepeq6}), we know that the following two equalities hold, 
\be\label{jan26eqn191}
\big(\Omega^v_i-\frac{t}{\sqrt{1+|v|^2}}\Omega^x_i\big)\big(\frac{t}{\sqrt{1+|v|^2}}- \omega(x,v)\big) = - \big(\Omega^v_i-\frac{t}{\sqrt{1+|v|^2}}\Omega^x_i\big)  \omega(x,v)=-\big( c_{i}(x, v) \tilde{d}(t,x,v) + e_{i}(x, v)\big), 
\ee
\[
\big(S^v -\frac{t}{(1+|v|^2)^{3/2}}S^x \big)\big(\frac{t}{\sqrt{1+|v|^2}}- \omega(x,v)\big) =  
-\frac{t|v|}{(1+|v|^2)^{3/2}} - \big(S^v -\frac{t}{(1+|v|^2)^{3/2}}S^x \big)  \omega(x,v) 
 \]
 \be\label{jan26eqn181}
 =-\tilde{c}(x,v)\tilde{d}(t,x,v) -\frac{t|v|}{(1+|v|^2)^{3/2}}\psi_{< 0}( (x\cdot v)^2+|x|^2) - \tilde{e}(x,v)
 = \tilde{\hat{c}}(x,v)\tilde{d}(t,x,v) + \tilde{\hat{e}}(x,v),
 \ee
 where
 \be\label{jan26eqn182}
 \tilde{\hat{c}}(x,v):= -\tilde{c}(x,v) - \frac{|v|}{\sqrt{1+|v|^2}} \psi_{< 0}( (x\cdot v)^2+|x|^2),
 \ee
 \be\label{jan26eqn183}
\tilde{\hat{e}}(x,v):= - \tilde{e}(x,v) - \frac{|v|}{1+|v|^2}\omega_{+}(x,v)\psi_{\geq 0}( (x\cdot v)^2+|x|^2) \psi_{< 0}( (x\cdot v)^2+|x|^2).
 \ee
From the detailed formulas in  above equalities    (\ref{jan26eqn182}) and (\ref{jan26eqn183}) and the estimate of coefficients in (\ref{feb11eqn2}), we have 
\be\label{feb11eqn6}
|\tilde{\hat{c}}(x,v)| + |\tilde{\hat{e}}(x,v)| \lesssim 1, \quad |\tilde{\hat{e}}(x,v)|\psi_{\geq 2}(|x|)=0.
\ee

Recall (\ref{feb11eqn8}). To sum up,  from  the equalities  (\ref{feb11eqn3}), (\ref{jan26eqn191}), and (\ref{jan26eqn181}) and the   decomposition of ``$D_v$'' in (\ref{eqq13}),  we know  that the desired equalities in  (\ref{jan26eqn100}) hold for some uniquely determined coefficients.

Moreover, our  desired estimates (\ref{jan26eqn101}) and (\ref{jan26eqn401}) hold   from the estimate (\ref{feb11eqn11}) and  the estimates of coefficients in (\ref{feb11eqn1}) and (\ref{feb11eqn6}).

\end{proof}
 
Our desired results in Lemma  \ref{firstordercommutationlemma} follow directly from the following  Lemma. 

\begin{lemma}\label{proofoffirstordercommutationrule}
For any $\rho \in \mathcal{K}, |\rho|=1$, and $i\in \{1,\cdots, 7\}$, the following commutation rule holds, 
\be\label{jan25eqn81}
[X_i, \Lambda^\rho]= \sum_{\kappa\in \mathcal{K},|\kappa|=1} \big(\tilde{c}_{i}^{\rho, \kappa}( x,v) \tilde{d}(t,x,v) + \hat{c}_{i}^{\rho, \kappa}( x,v)\big) \Lambda^\kappa, 
\ee
where the coefficients $\tilde{c}_{i}^{\rho, \kappa}(t,x,v) $ and $\hat{c}_{i}^{\rho, \kappa}(t,x,v)$ satisfy the following  rough estimates,
\be\label{jan26eqn1}
 |\tilde{c}_{i}^{\rho, \kappa}( x,v)  | +  |\hat{c}_{i}^{\rho, \kappa}( x,v)  |  \lesssim \min\{(1+|v|)^{1+c_{\textup{vn}}(\kappa)-c_{\textup{vn}}(\rho)},(1+|v|)^{ c_{\textup{vm}}(\kappa)-c_{\textup{vm}}(\rho)}\},
\ee
\be\label{feb11eqn21}
| \Lambda^\beta \big( \tilde{c}_{i}^{\rho, \kappa}( x,v)\big)|+ | \Lambda^\beta \big( \hat{c}_{i}^{\rho, \kappa}( x,v)\big)|\lesssim (1+|v|)^{|\beta|+2} (1+|x|)^{|\beta|+2}, \quad \beta \in \mathcal{S}.
\ee
If   $i(\kappa)-i(\rho)>0$,  where $i(\kappa)$ denotes the total number of vector fields $\Omega_i^x$ in $\Lambda^\kappa$, then the following improved estimate holds for the coefficients $\hat{c}_{i}^{\rho, \kappa}( x,v)$,
\be\label{april5eqn16}
  |\hat{c}_i^{\rho, \kappa}(x,v)|\lesssim  (1+|v|)^{-1+c_{\textup{vm}}(\kappa)-c_{\textup{vm}}(\rho)}.
\ee
Moreover, for the case when $i=1,$ the following improved estimate holds, 
\be\label{jan31eqn71}
 |\tilde{c}_{1}^{\rho, \kappa}(  x,v)  | +  |\hat{c}_{1}^{\rho, \kappa}( x,v)  |  \lesssim  (1+|v|)^{-1+c_{\textup{vn}}(\kappa)-c_{\textup{vn}}(\rho)}.
\ee
\end{lemma}
\begin{proof}
Note that the following commutation rules hold for any differentiable coefficients ``$c_1(v)$'' and  ``$c_2(v)$'' and any two vector fields ``$X$'' and ``$Y$''
\be\label{jan26eqn3100}
[c_1(v)X, c_2(v) Y] = c_1(v)( X c_2(v)) Y + c_1(v) c_2(v)[X, Y]  - c_2(v)( Y c_1(v)) X, 
\ee
\[
[c_1(v)\big(S^v-\frac{t}{(1+|v|^2)^{3/2}}S^x \big) , c_2(v)\cdot \nabla_x ] = c_1(v)\big(\tilde{v}\cdot \nabla_v c_2(v)\big)\cdot \nabla_x,
\]
\[
[c_1(v)\big(\Omega_i^v-\frac{t}{(1+|v|^2)^{1/2}}\Omega_i^x \big) , c_2(v)\cdot \nabla_x ] =c_1(v) \big(\tilde{V}_i\cdot \nabla_v c_2(v)\big)\cdot \nabla_x, \quad i\in\{1,2,3\}.
\]
Moreover,   the following commutation rules hold  for the vector field $\tilde{\Omega}_i$,$i \in\{1,2,3\}$,
 \[
[ S^v- \frac{t}{(1+|v|^2)^{3/2}}S^x,X_i\cdot\nabla_x + V_i\cdot \nabla_v]= \frac{1}{|v|}\big(\Omega_i^v- \frac{t}{(1+|v|^2)^{3/2}}\Omega_i^x\big)- \frac{1}{|v|}\Omega_i^v + \frac{t}{(1+|v|^2)^{3/2}}\Omega_i^x=0,
\]
\[
[ \Omega_j^v- \frac{t}{(1+|v|^2)^{1/2}}\Omega_j^x,X_i\cdot\nabla_x + V_i\cdot \nabla_v]=\frac{1}{|v|}\big(V_j\cdot \nabla_v V_i-V_i\cdot \nabla_v V_j)\cdot \nabla_v\]
\be\label{april5eqn41}
-\frac{t}{(1+|v|^2)^{1/2}}\frac{1}{|v|}\big(V_j\cdot \nabla_x(X_i)-V_i\cdot\nabla_v V_j  \big)\cdot \nabla_x=\sum_{k=1,2,3} {\epsilon_{ij}^k} \big(\Omega_k^v-\frac{t}{(1+|v|^2)^{1/2}}\Omega_k^x \big),
\ee
where $\epsilon_{ij}^k\in\{0,1\}$ are some uniquely determined coefficients.

Hence, to prove the desired equality (\ref{jan25eqn81}) and the desired estimates (\ref{jan26eqn1}), (\ref{feb11eqn21}), (\ref{april5eqn16}), and (\ref{jan31eqn71}),   it would be sufficient to consider the case when  $\Lambda^{\rho}\in \{\psi_{\geq 1}(|v|)\widehat{S}^{v},\psi_{\geq 1}(|v|) \widehat{\Omega}^v_i, \psi_{< 1}(|v|) K_{v_i}, i\in\{1,2,3\}\}$. Moreover, from the equality (\ref{jan26eqn3100}),  {it would be sufficient to compute the commutation rules without the cutoff functions $\psi_{\geq 1}(|v|)$ and $\psi_{\leq 0}(|v|)$  during computations. }
 
   We first consider the case when $\Lambda^\rho =\psi_{\geq 1}(|v|) \widehat{S}^v $ and $X_i= \psi_{\geq 1}(|v|) \tilde{v}\cdot D_v$.  From the equality (\ref{sepeq6}), we have 
\[
[\big(S^v-\frac{t}{(1+|v|^2)^{3/2}}S^x \big) ,  \big(S^v  -  \frac{\omega(x,v)}{1+|v|^2} S^x  \big) ] 
\]
\[
=   - [\frac{t}{(1+|v|^2)^{3/2}}S^x, S^v] -[S^v, \frac{\omega(x,v)}{1+|v|^2} S^x ]+ [\frac{t}{(1+|v|^2)^{3/2}}S^x,   \frac{\omega(x,v)}{1+|v|^2} S^x ]   
\]
\[
=    \tilde{v}\cdot\nabla_v\big(\frac{t}{(1+|v|^2)^{3/2}} \big) S^x- \big(S^v(\frac{\omega(x,v)}{1+|v|^2}) \big) S^x+ \frac{t}{(1+|v|^2)^{5/2}} \big( S^x(\omega(x,v))\big) S^x    
\]
\[ 
= -\frac{3 t |v|}{(1+|v|^2)^{5/2}}S^x +\frac{2|v| \omega(x,v)}{(1+|v|^2)^2} S^x - \frac{1}{1+|v|^2}\big[\big(S^v -\frac{t}{ (1+|v|^2)^{3/2}}S^x 	\big)\omega(x,v)\big] S^x  
\]
\[ 
 = -\frac{3 t |v|}{(1+|v|^2)^{5/2}}S^x +\frac{2|v| \omega(x,v)}{(1+|v|^2)^2} S^x  -\frac{1}{1+|v|^2}
\big[ \tilde{c}(x,v) \tilde{d}(t,x,v)+ \tilde{e}(x,v) -\frac{t|v|}{(1+|v|^2)^{3/2}}\]
\be\label{jan26eqn32} 
 \times \psi_{\geq 0}( (x\cdot v)^2+|x|^2) \big] S^x=  	 \big( \widehat{c}(x,v)\tilde{d}(t,x,v) + \widehat{e}(x,v) )S^x,
\ee
where
\be\label{sepeq9}
\widehat{c}(x,v)=-  \frac{1}{1+|v|^2}  \big( \tilde{c}(x,v) +\frac{2|v|}{(1+|v|^2)^{1/2}}  +\frac{|v| }{(1+|v|^2)^{1/2}}\psi_{< 0}((x\cdot v)^2+|x|^2)\big),
\ee
\be\label{sepeq10} 
\widehat{e}(x,v)= -\frac{1}{1+|v|^2}\big(\tilde{e}(x,v)+ \frac{|v|}{1+|v|^2}\omega(x,v)\psi_{<0}((x\cdot v)^2+|x|^2)\big),
\ee
where $ \tilde{c}(x,v)$ and $\tilde{e}(x,v)$ are defined in (\ref{jan26eqn502}) and (\ref{jan26eqn503}) respectively.

From  the detailed formulas of $\widehat{c}(x,v)$ and $\widehat{e}(x,v)$ in (\ref{sepeq9}) and (\ref{sepeq10}) and the estimate of coefficients in (\ref{feb11eqn2}),  we have 
\be\label{feb11eqn31}
|\widehat{c}(x,v)|   + |\widehat{e}(x,v)|   \lesssim (1+|v|)^{-2}.
\ee

Next, we consider the case   $\Lambda^\rho =\psi_{\geq 1}(|v|) \widehat{\Omega}^v_i, X_i=\psi_{\geq 1}(|v|) \tilde{v}\cdot D_v$. Recall (\ref{sepeq6}), we have
\[
[\big(S^v-\frac{t}{(1+|v|^2)^{3/2}}S^x \big) , \big(\Omega^v_i -  \omega(x,v) \Omega^x_i \big) ]=   [S^v, \Omega^v_i] - [S^v, \omega(x,v) \Omega^x_i]-[\frac{t}{(1+|v|^2)^{3/2}}S^x, \Omega^v_i] 
\]
\[
+ [\frac{t}{(1+|v|^2)^{3/2}}S^x, \omega(x,v) \Omega^x_i]  = -\frac{1}{|v|}\Omega^v_i -(S^v \omega(x,v)) \Omega_i^x  - \omega(x,v) [S^v, \Omega_i^x]  + \frac{t}{(1+|v|^2)^{3/2}}[\Omega_i^v, S^x]
\]
\[
+\frac{t}{(1+|v|^2)^{3/2}}\big(S^x \omega(x,v)\big)\Omega_i^x  = -\frac{1}{|v|}\Omega^v_i +  \frac{1}{|v|}\frac{t}{(1+|v|^2)^{3/2}}\Omega_i^x
-\big(  \tilde{c}(x,v)\tilde{d}(t,x,v) + \tilde{e}(x,v) -\frac{t|v|}{(1+|v|^2)^{3/2}}\]
\be\label{jan26eqn34}
\times \psi_{\geq 0}( (x\cdot v)^2+|x|^2) \big) \Omega_i^x =  -\frac{1}{|v|}  \widehat{\Omega}_i + \big(\hat{c}_i (x,v) \tilde{d}(t,x,v) + \hat{e}_i (x,v) \big) \Omega_i^x, 
\ee
where
\be\label{sepeq15}
 \hat{c}_i  (x,v)=  -\tilde{c} (x,v)  + \frac{\sqrt{1+|v|^2}}{|v|} -\frac{|v|}{(1+|v|^2)^{1/2}}\psi_{<0}((x\cdot v)^2+|x|^2)  ,
 \ee
 \be\label{sepeq16}
  \hat{e}_i(x, v)=  -\tilde{e}(x,v)-\frac{|v|}{1+|v|^2}\omega(x,v)\psi_{<0}((x\cdot v)^2+|x|^2) . 
\ee
From the above detailed formulas and the estimate of coefficients in (\ref{feb11eqn2}),  we have 
\be\label{feb11eqn39}
\big( | \hat{c}_i  (x,v) | + | \hat{e}_i  (x,v)|\big)\psi_{\geq -10}(|v|)\lesssim 1.
\ee

 Next, we consider the case   $\Lambda^\rho = \psi_{\geq 1}(|v|) \widehat{S}^v, X_i= \psi_{\geq 1}(|v|) \tilde{V}_i\cdot D_v$. Recall (\ref{sepeq21}), we have
\[
[\big(\Omega^v_i-\frac{t}{\sqrt{1+|v|^2}}\Omega^x_i\big) , \big(S^v - \frac{\omega(x,v)}{(1+|v|^2)^{ }} S^x \big) ] = [\Omega^v_i, S^v ] - [\frac{t}{\sqrt{1+|v|^2}}\Omega^x_i, S^v ] - [\Omega^v_i, \frac{\omega(x,v)}{(1+|v|^2)^{ }} S^x  ] 
\]
\[
+ [\frac{t}{\sqrt{1+|v|^2}}\Omega^x_i, \frac{\omega(x,v)}{(1+|v|^2)^{ }} S^x]=\frac{1}{|v|}\Omega^v_i + \tilde{v}\cdot \nabla_v \big(\frac{t}{\sqrt{1+|v|^2}}\big) \Omega_i^x + \frac{t}{\sqrt{1+|v|^2}}[S^v, \Omega_i^x]
\]
\[
 - \big( \Omega^v_i(\frac{\omega(x,v)}{(1+|v|^2)^{ }})\big) S^x- \frac{\omega(x,v)}{1+|v|^2} [\Omega_i^v, S^x]  + \frac{t}{\sqrt{1+|v|^2}} \big( \Omega_i^x\big(\frac{\omega(x,v)}{1+|v|^2}\big)\big) S^x   
\]
 \[
= \frac{1}{|v|}\big(\Omega^v_i- \frac{\omega(x,v)}{1+|v|^2}\Omega^x_i -\frac{t|v|^2}{(1+|v|^2)^{3/2}}  \Omega_i^x \big)- \frac{1}{1+|v|^2}\big( c_{i}(x, v) \tilde{d}(t,x,v) + e_{i}(x, v)\big)S^x  
\]
\be\label{jan26eqn2}
= \frac{1}{|v|}\widehat{\Omega}^v_i+ \big[   \widehat{\tilde{c}}_i(x,v))  \tilde{d}(t,x,v)+\widehat{\tilde{e}}_i(x,v)) \big]\cdot (S^x, \Omega_1^x,\Omega_2^x, \Omega_3^x),
\ee
where
\be\label{sepeq27}
\widehat{\tilde{c}}_i(x,v))= - \big(\frac{c_i(x,v)}{1+|v|^2}, \frac{ |v| \delta_{1i}}{(1+|v|^2)^{1/2}} ,  \frac{  |v|\delta_{2i}}{(1+|v|^2)^{1/2}} ,  \frac{ |v| \delta_{3i}}{(1+|v|^2)^{1/2}}\big), \quad 
\widehat{\tilde{e}}_i(x,v)=- \big( \frac{e_i(x,v)}{1+|v|^2}, 0,0,0 \big).
\ee
From the above detailed formulas and the estimate (\ref{feb11eqn1}), we have
\be\label{feb11eqn51}
| \widehat{\tilde{c}}_{i;1}(x,v))|\lesssim (1+|v|)^{-1}, \quad \sum_{l=2,3,4} | \widehat{\tilde{c}}_{i;l}(x,v))|\lesssim 1,  \quad | \widehat{\tilde{e}}_i(x,v))| \lesssim (1+|v|)^{-2}. 
\ee

 Next, we consider the case   $\Lambda^\rho =\psi_{\geq 1}(|v|) \widehat{\Omega}_j^v, X_i=\psi_{\geq 1}(|v|) \tilde{V}_i\cdot D_v$, $i, j\in \{1,2,3\}$. Recall (\ref{sepeq21}), we have
\[
[\big(\Omega^v_i-\frac{t}{\sqrt{1+|v|^2}}\Omega^x_i\big) , \big(\Omega^v_j- \omega(x,v) \Omega^x_j\big) ]
=  [\Omega_i^v, \Omega_j^v] -  [\frac{t}{\sqrt{1+|v|^2}}\Omega^x_i,\Omega^v_j ] +  \frac{t}{\sqrt{1+|v|^2}}\big(\Omega^x_i   \omega(x,v) \big)\Omega^x_j  
\]
\[
- [\Omega^v_i, \omega(x,v) \Omega^x_j ]=  \big(\tilde{V}_i \cdot\nabla_v \tilde{V}_j -\tilde{V}_j \cdot\nabla_v \tilde{V}_i \big)\cdot \nabla_v  +\frac{t}{\sqrt{1+|v|^2}}\big(\tilde{V}_j\cdot \nabla_v \tilde{V}_i\big)\cdot \nabla_x- \omega(x,v)\big(\tilde{V}_i\cdot \nabla_v \tilde{V}_j\big)\cdot \nabla_x \]
\[  - \big(\Omega_i^v \omega(x,v) \big)  \Omega^x_j+ \frac{t}{\sqrt{1+|v|^2}} 
\big(\Omega^x_i \omega(x,v)\big)\Omega^x_j = \frac{(e_j\times e_i )\times v}{|v|^2}\cdot\big(\nabla_v - \omega(x,v)\nabla_x \big)
\]
\be\label{sepeq39}
 +  \sqrt{1+|v|^2}\tilde{d}(t,x,v) \big((\tilde{V}_j\cdot \nabla_v) \tilde{V}_i \big)\cdot \nabla_x-   \big[c_{i}(x, v)\tilde{d}(t,x,v) + e_{i}(x, v)\big]\Omega_j^x.
\ee
 Note that
\[
\nabla_x = \tilde{v} S^x + \sum_{k=1,2,3} \tilde{V}_k \Omega_k^x,   \Longrightarrow
(\tilde{V}_j\cdot\nabla_v)\tilde{V}_i \cdot\nabla_x= \big((   \tilde{V}_j \cdot\nabla_v )\tilde{V}_i\big)\cdot \tilde{v}  S^x + \sum_{k=1,2,3}  \big((   \tilde{V}_j \cdot\nabla_v )\tilde{V}_i\big)\cdot \tilde{V}_k\Omega^x_k,
\]
\[
(e_j \times e_i)\times v= \epsilon_{ij}^k \tilde{V}_k, \quad 
\]
where $\epsilon_{ij}^k\in \{0,1\}$, are some uniquely determined constants.  Hence, 
\be\label{jan26eqn35}
\textup{(\ref{sepeq39})}=  \frac{ \epsilon_{ij}^k   }{|v|}\widehat{\Omega}^v_k + \big[a_{i,j}(x,v)\tilde{d}(t,x,v) + e_{i,j}(x,v) \big]\cdot (S^x, \Omega_1^x,\Omega_2^x, \Omega_3^x),
\ee
where
\be\label{sepeq30}
a_{i,j}(x,v)=(a_{i,j}^0(x,v), a_{i,j}^1(x,v),a_{i,j}^2(x,v),a_{i,j}^3(x,v) ),\quad a_{i,j}^0(x,v)=  \sqrt{1+|v|^2}\big((   \tilde{V}_j \cdot\nabla_v )\tilde{V}_i\big)\cdot \tilde{v},
\ee
\[
  a_{i,j}^k(x,v)=  {\sqrt{1+|v|^2}}\big((   \tilde{V}_j \cdot\nabla_v )\tilde{V}_i\big)\cdot \tilde{V}_k -c_i(x,v) \delta_{jk} ,\quad k=1,2,3,\]
 \be\label{sepeq31} 
  e_{i,j}(x,v)=- (0,   e_i(x,v) \delta_{j1},  e_i(x,v) \delta_{j2},   e_i(x,v) \delta_{j3} ).
\ee
From the above detailed formulas and the estimate of coefficients in (\ref{feb11eqn1}), we have
\be\label{feb11eqn71}
|a_{i,j}(x,v)|\psi_{\geq -10}(|v|) \lesssim 1,  \quad |e_{i,j}(x,v)| \lesssim 1. 
\ee

 Lastly, we consider the case when $X_i=\psi_{\leq 0}(|v|) D_{v_i}, i \in \{1,2,3\}$. For this case we have $|v|\lesssim 1$, which means that our desired estimates in (\ref{jan26eqn1}), (\ref{feb11eqn21}), (\ref{april5eqn16}), and (\ref{jan31eqn71}) are trivial. Hence, it would be sufficient to verify that the desired equality (\ref{jan25eqn81}) holds.

  Note that   the following commutation rules hold for any  $i,j\in\{1,2,3\}$,
\[
[\p_{v_i}- t \p_{v_i}\hat{v}\cdot \nabla_x, \p_{v_j}-\omega(x,v)\sqrt{1+|v|^2}\p_{v_j}\hat{v}\cdot \nabla_x ]= -[t \p_{v_i}\hat{v}\cdot \nabla_x, \p_{v_j}]-[\p_{v_i}, \omega(x,v)\sqrt{1+|v|^2}\p_{v_j}\hat{v}\cdot \nabla_x ]
\]
\[
+ [t \p_{v_i}\hat{v}\cdot \nabla_x, \omega(x,v)\sqrt{1+|v|^2}\p_{v_j}\hat{v}\cdot \nabla_x]= t\p_{v_i}\p_{v_j}\hat{v}\cdot \nabla_x  -\p_{v_i}\big(\omega(x,v)\sqrt{1+|v|^2} \p_{v_j}\hat{v}\big)\cdot \nabla_x 
\]
\[
+ t\p_{v_i}\hat{v}\cdot \nabla_x \omega(x,v) \sqrt{1+|v|^2}\p_{v_j}\hat{v}\cdot \nabla_x
=( t - \omega(x,v)\sqrt{1+|v|^2}) \p_{v_i}\p_{v_j}\hat{v}\cdot \nabla_x \]
\be\label{eqq20}
+ \big[ t\p_{v_i}\hat{v}\cdot \nabla_x \omega(x,v)\sqrt{1+|v|^2}-\p_{v_i}\big( \omega(x,v)\sqrt{1+|v|^2}\big)\big]\p_{v_j}\hat{v}\cdot \nabla_x. 
\ee
 Note that
\[
 t\p_{v_i}\hat{v}\cdot \nabla_x \omega(x,v)\sqrt{1+|v|^2}-\p_{v_i}\big( \omega(x,v)\sqrt{1+|v|^2}\big)
\]
\[
= t \Big[\big(v_i+\frac{(x\cdot v)v_i + x_i}{\sqrt{(x\cdot v)^2 +|x|^2}}\big)- \frac{v_i}{(1+|v|^2)^{ }}\big(|v|^2 + \frac{(x\cdot v)(1+|v|^2)}{\sqrt{(x\cdot v)^2 +|x|^2}} \big)\Big]\psi_{\geq 0}((x\cdot v)^2 +|x|^2)
\]
\[
+   2t \omega_{+}(x,v)\Big(\psi'_{\geq 0}((x\cdot v)^2+|x|^2)((x\cdot v)v_i + x_i)- \frac{v_i}{1+|v|^2}\psi'_{\geq 0}((x\cdot v)^2+|x|^2)(x\cdot v)(1+|v|^2) \Big)
\]
\[
  -\frac{v_i}{\sqrt{1+|v|^2}} \omega(x,v)- \sqrt{1+|v|^2}\big(x_i + \frac{(x\cdot v) x_i}{\sqrt{(x\cdot v)^2 +|x|^2}} \big) \psi_{\geq 0} ((x\cdot v)^2+|x|^2)
\]
\[
-2\omega_{+}(x,v)\sqrt{1+|v|^2}\psi'_{\geq 0}((x\cdot v)^2+|x|^2)(x\cdot v) x_i
\]
\[
= t\big(\frac{v_i}{1+|v|^2} + \frac{x_i}{\sqrt{(x\cdot v)^2+|x|^2}} \big)\psi_{\geq 0}((x\cdot v)^2 +|x|^2)
+ 2 t \omega_{+	}(x,v) x_i\psi'_{\geq 0}((x\cdot v)^2+|x|^2)-\frac{v_i}{\sqrt{1+|v|^2}} \omega(x,v)
\]
\[
 -\sqrt{1+|v|^2}\omega(x,v)\frac{x_i}{\sqrt{(x\cdot v)^2+|x|^2}} 
-2 \sqrt{1+|v|^2}\omega_{+}(x,v)\psi'_{\geq 0}((x\cdot v)^2+|x|^2)(x\cdot v) x_i
\]
\[
=\psi_{\geq 0}((x\cdot v)^2+|x|^2) \big(t-\sqrt{1+|v|^2}\omega_{+}(x,v) \big)\big(\frac{v_i}{1+|v|^2} + \frac{x_i}{\sqrt{(x\cdot v)^2+|x|^2}} \big)
\]
\be\label{jan26eqn10}
+2x_i\omega_{+}(x,v)\psi'_{\geq 0}((x\cdot v)^2+|x|^2)\big(t-\sqrt{1+|v|^2}(x\cdot v)\big) 
=\mathfrak{c}_i(x,v)\tilde{d}(t,x,v) + \mathfrak{e}_i(x,v),
\ee
where
\be\label{jan26eqn11}
\mathfrak{c}_i(x,v)=  \psi_{\geq 0}((x\cdot v)^2+|x|^2)  \big( {v_i}  + \frac{x_i (1+|v|^2)}{\sqrt{(x\cdot v)^2+|x|^2}} \big)
+ 2x_i\omega_{+}(x,v) \psi'_{\geq 0}((x\cdot v)^2+|x|^2) (1+|v|^2), 
\ee
\[
\mathfrak{e}_i(x,v)=- \omega_{+}(x,v)\psi_{\geq 0}((x\cdot v)^2+|x|^2)\psi_{< 0}((x\cdot v)^2 +|x|^2 ) \big( \frac{v_i}{\sqrt{1+|v|^2}}  + \frac{x_i  \sqrt{1+|v|^2}}{\sqrt{(x\cdot v)^2+|x|^2}} \big)\]
\be\label{jan26eqn12}
+ 2x_i\omega_{+}(x,v)\psi'_{\geq 0}((x\cdot v)^2+|x|^2)\sqrt{1+|v|^2}\big(\omega(x,v)-x\cdot v\big). 
\ee
Combining equalities in  (\ref{eqq20}) and (\ref{jan26eqn10}), we have 
\be\label{jan26eqn41}
[\p_{v_i}- t \p_{v_i}\hat{v}\cdot \nabla_x, \p_{v_j}-\omega(x,v)\sqrt{1+|v|^2}\p_{v_j}\hat{v}\cdot \nabla_x ]= \big[\tilde{\mathfrak{c}}_{i,j}(x,v) \tilde{d}(t,x,v) + \tilde{\mathfrak{e}}_{i,j}(x,v)\big]\cdot \nabla_x,
\ee
where
\be\label{jan26eqn31}
\tilde{\mathfrak{c}}_{i,j}(x,v)= (1+|v|^2) \p_{v_i}\p_{v_j}\hat{v} + \mathfrak{c}_i(x,v)  \p_{v_j}\hat{v}, \quad 
\tilde{\mathfrak{e}}_{i,j}(x,v)= \mathfrak{e}_i(x,v)  \p_{v_j}\hat{v}.
\ee

To sum up, our desired equality (\ref{jan25eqn81}) holds from the equalities (\ref{april5eqn41}), (\ref{jan26eqn32}), (\ref{jan26eqn34}), (\ref{jan26eqn2}), (\ref{jan26eqn35}), and (\ref{jan26eqn41}). Moreover, recall the definition of indexes in (\ref{countingnumber}),  our  desired estimates  (\ref{jan26eqn1}), (\ref{feb11eqn21}),  (\ref{april5eqn16}), and (\ref{jan31eqn71}) hold  from the estimates of coefficients in (\ref{feb11eqn31}), (\ref{feb11eqn39}), (\ref{feb11eqn51}), and (\ref{feb11eqn71}). 
\end{proof}


\begin{thebibliography}{99}

 

 
\bibitem{bardos} C. Bardos, P. Degond. Global existence for the Vlasov-Poisson equations, space variable with small initial data,  \textit{Ann. Inst. H. Poincar\'e Anal. Non Lin\'eaire}, 2(\textbf{1985}), 101--118.

\bibitem{bigorgne} L. Bigorgne. Asymptotic properties of small data solutions of the Vlasov-Maxwell system in high dimensions, arXiv:1712.09698, \textit{preprint}. 

\bibitem{calogero1} S. Calogero. Spherically symmetric steady states of galactic dynamics in scalar gravity. \textit{Classical Quantum Gravity}, 20(\textbf{2003}), no. 9 , 1729--1741.
\bibitem{calogero2} S. Calogero. Global classical solutions to the $3$-D Nordstr\"om-Vlasov system, \textit{Comm. Math. Phys}., 266 (\textbf{2006}), 343--353. 

\bibitem{calogero3} S. Calogero, G. Rein. Global weak solutions to the Nordstr\"om-Vlasov system, \textit{J. Differential Equations}, 204(\textbf{2004}), no. 2, 323--338. 

\bibitem{diperna} P. Diperna, P.L. Lions. Global solutions of Vlasov-Maxwell system, \textit{Comm. Pure Appl. Math.}, 42(\textbf{1989}), 729--757.
\bibitem{smulevic1} D. Fajman, J. Joudioux, J. Smulevici. A vector field method for relativistic transport equations with applications, \textit{Anal. PDE}, 10(\textbf{2017}), no.7, 1539--1612.


\bibitem{smulevic2} D. Fajman, J. Joudioux, J. Smulevici. Sharp asymptotics for small data solutions of the Vlasov-Nordstr\"om system in three dimensions,   arXiv:1704.05353v1, \textit{preprint}.

\bibitem{smulevic3} D. Fajman, J. Joudioux, J. Smulevici. The stability of the Minkowski space for the Einstein-Vlasov system, arXiv:1707.06141, \textit{preprint}.  

\bibitem{friedrich} S. Friedrich. Global small solutions of the Vlasov-Nordstr\"om  system, \textit{arXiv:0407023}, \textbf{2004}. 



\bibitem{glasseys2} R. Glassey, J. Schaeffer. Global existence for the relativistic Vlasov-Maxwell system with nearly neutral initial data, \textit{Comm. Math. Phys.}, 119(\textbf{1988}), 353--384.

\bibitem{glasseys1} R. Glassey, J. Schaeffer. The two and one half dimensional relativistic Vlasov-Maxwell system, \textit{Comm. Math. Phys.}, 185(\textbf{1997}), 257--284.



\bibitem{glassey3} R. Glassey, W. Strauss. Singularity formation in a collisionless plasma could only occur
at large velocities, \textit{Arch. Rat. Mech. Anal.}, 92(\textbf{1986}), 59--90.


 

\bibitem{glassey1} R. Glassey, W. Strauss. High velocity particles in collitionless plasma, \textit{Math. Meth. in Appl. Sci.}, 9(\textbf{1987}), 46--52.

\bibitem{glassey2} R. Glassey, W. Strauss.  Absence of shocks in an initially dilute collisionless plasma, \textit{Comm. Math. Phys.}, 113(\textbf{1987}), 191--208. 

\bibitem{glassey4} R. Glassey, W. Strauss.  Large velocities in the relativistic Vlasov-Maxwell equations. \textit{J. Fac. Sci. Univ. Tokyo Sect. IA Math.} 36(\textbf{1989}), no. 3, 615--627.


\bibitem{hwang} H-J. Hwang, A. Rendall, J. Vel\'azquez. Optimal gradient estimates and asymptotic behaviour for the Vlasov-Poisson system with small initial data, \textit{Arch. Rational Mech. Anal.}, 200(\textbf{2011}), 313--360.

\bibitem{Klainerman1} S. Klainerman. Uniform decay estimates and the Lorentz invariance of the classical wave equation, \textit{Comm. Pure Appl. Math.}, 38(\textbf{1985}), no. 3, 321--332.

\bibitem{Klainerman2} S. Klainerman. Remark on the asymptotic behavior of the Klein-Gordon equation in $\R^{n+1}$, \textit{Comm. Pure Appl. Math.}, 46(\textbf{1993}), no. 2, 137--144.

\bibitem{Klainerman3} S. Klainerman, G. Staffilani. A new approach to study the Vlasov-Maxwell system, \textit{Comm. Pure Appl. Anal.}, 1(\textbf{2002}), no. 1, 103--125.


\bibitem{kunze} M. Kunze. Yet another criterion for global existence in the $3D$ relativistic Vlasov-Maxwell system, \textit{J. Differential Equations}, 259(\textbf{2015}), no. 9, 4413--4442.

\bibitem{lindblad} H. Lindblad, M. Taylor.  Global Stability of Minkowski space for the Einstein-Vlasov system in the harmonic gauge, arXiv:1707.06141, \textit{preprint}.

\bibitem{Lions} P.-L. Lions and B. Perthame. Propagation of moments and regularity for the $3$-dimensional Vlasov-Poisson system, \textit{Invent. Math.}, 105(\textbf{1991}), 415--430.


\bibitem{luk2} J. Luk, R. Strain. A new continuation criterion for the relativistic Vlasov-Maxwell system, \textit{Comm. Math. Phys.} 331(\textbf{2014}), no. 3, 1005--1027.

\bibitem{luk} J. Luk, R. Strain. Strichartz estimates and moment bounds for the relativistic Vlasov-Maxwell system, \textit{Arch. Rat. Mech. Anal.}, 219(\textbf{2016}), no. 1, 1--120. 





\bibitem{pallard1} C. Pallard. On global smooth solutions to the $3D$ Vlasov-Nordstr\"om system, \textit{Ann. Inst. H. Poincar\'e Anal. Non Lin\'eaire}, 23 (\textbf{2006}), no.1, 85--96.


\bibitem{pallard2} C. Pallard. On the boundedness of the momentum support of solutions to the relativistic Vlasov-Maxwell system, \textit{Indiana Univ. Math. J.}, 54(\textbf{2005}), no. 5, 1395--1409.

\bibitem{pallard3} C. Pallard. A refined existence criterion for the relativistic Vlasov-Maxwell system, \textit{Comm. Math. Sci.}, 13(\textbf{2015}), no. 2, 347--354.

\bibitem{patel} N. Patel. Three new results on continuation criteria for the $3D$ relativistic Vlasov-Maxwell system, \textit{J. Differential Equations}., 264(\textbf{2018}), no. 3, 1841--1885.


\bibitem{rein} G. Rein, A.D. Rendall. Global existence of solutions of the spherically symmetric Vlasov-Einstein system with small initial data, \textit{Comm. Math. Phys.} 150 (\textbf{1992}), no. 3, 561--583.

\bibitem{schaeffer2} J. Schaffer. Global existence of smooth solutions to the Vlasov-Poisson system in three dimensions, \textit{Comm. P.D.E}, 16(\textbf{1991}), 1313--1335. 

\bibitem{schaeffer1} J. Schaeffer. A small data theorem for collisionless plasma that includes high velocity particles, \textit{Indiana Univ. Math. J.}, 53(\textbf{2004}), no.1, 1--34.
\bibitem{smulevici4} J. Smulevici. Small data solutions of the Vlasov-Poisson system and the vector field method, \textit{Ann. PDE.,}  (\textbf{2016}), 2:11.
\bibitem{alfonso} R. Sospedra-Alfonso, R. Illner. Classical solvability of the relativistic Vlasov-Maxwell system with bounded spatial density, \textit{Math. Methods Appl. Sci.}, 33(\textbf{2010}), no. 6, 751--757. 

\bibitem{taylor} M. Taylor. The global nonlinear stability of Minkowski space for the massless Einstein-Vlasov system, \textit{Ann. PDE.},  (2017) 3: 9. 

\bibitem{wang} X. Wang.   Propagation of  regularity and long time behavior of   $3D$ massive relativistic transport equation II: Vlasov-Maxwell  system, arXiv:1804.06566, \textit{preprint}.

\bibitem{wang3} X. Wang. Decay estimates for the $3D$ relativistic and non-relativistic Vlasov-Poisson systems, arXiv:1805.10837, \textit{preprint}.

\bibitem{wong} W.W.Y. Wong. A commutating-vector-field approach to some dispersive estimates, \textit{preprint}, arXiv:1701.01460.

\end{thebibliography}
\end{document}